\documentclass[11pt,a4paper, reqno]{amsart}
\usepackage{mathrsfs}
\usepackage{amsmath,amsfonts,verbatim, amsthm}
\usepackage{latexsym}
\usepackage{amssymb,leftidx}
\usepackage{enumitem}
\usepackage{extarrows}
\usepackage{overpic}
\usepackage{xypic}
\usepackage{color}
\usepackage{epsfig}
\usepackage{subfigure}
\usepackage{tikz}
\usetikzlibrary{graphs}
\usetikzlibrary{positioning}
\usepackage[backref=page]{hyperref}
\usepackage{cleveref}
\usepackage{caption}
\usepackage{yhmath}

\newcommand\rv[1]{{\color{blue} #1}}

\newcommand{\WF}{\mathrm{WF}}
\newcommand\WFsc{\mathrm{WF}_{\mathrm{sc}}}

\newcommand{\Ell}{\ensuremath{\mathrm{Ell}_{\mathrm{sc}}}}
\newcommand{\Char}{\ensuremath{\mathrm{Char}}}
\newcommand{\supp}{\ensuremath{\mathrm{supp}}}
\newcommand\scH{ {}^{\mathrm{sc}} H}
\newcommand{\sct}{\ensuremath{\mathrm{sc}}}

\newcommand{\la}{\ensuremath{\langle}}
\newcommand{\ra}{\ensuremath{\rangle}}

\newcommand{\rmb}{\mathrm{b}}

\hypersetup{urlcolor=blue, citecolor=red}

\usepackage{amssymb}
\usepackage{mathrsfs}
\usepackage{amscd}
\usepackage{bbm}
\usepackage{cite}

\newcommand\R{\mathbb{R}}

\newcommand\Z{\mathbb{Z}}
\newcommand\N{\mathbb{N}}
\newcommand\Id{\mathrm{Id}}

\newcommand{\Schw}{\mathcal S}
\newcommand{\SX}{\mathcal X}
\newcommand{\SY}{\mathcal Y}

\newcommand{\cu}{\ensuremath{\mathrm{cu}}}

\newcommand\sca{\mathrm{sc}}
\newcommand\cl{\mathrm{cl}}
\newcommand\ang[1]{\langle #1 \rangle}
\newcommand\aang[1]{\langle \! \langle #1 \rangle \! \rangle}
\newcommand\bang[1]{\big\langle #1 \big\rangle}
\newcommand\compn{\overline{\R^n}}
\newcommand\comphase{\overline{T^* \R^n}}
\newcommand\Psisc{\Psi_{\sca}}
\newcommand\Op{\operatorname{Op}}
\newcommand\spr{\sigma_{\mathrm{pr}}}
\renewcommand\Im{\operatorname{Im}}

\newcommand\Poi{\mathcal{P}}

\hypersetup{urlcolor=blue, citecolor=red}

\usepackage{amssymb}
\usepackage{mathrsfs}
\usepackage{amscd}
\usepackage{bbm}
\usepackage{cite}

\newcommand{\SV}{\mathcal V}

\renewcommand\Im{\operatorname{Im}}

\newcommand\Rin{\mathcal{R}_{\mathrm{in}}}
\newcommand\Rout{\mathcal{R}_{\mathrm{out}}}
\newcommand\Uin{\mathcal{U}_{\mathrm{in}}}
\newcommand\Uout{\mathcal{U}_{\mathrm{out}}}
\newcommand\signum{\operatorname{sgn}}
\newcommand\ml{\mathsf{l}}
\newcommand{\mr}{\mathsf{r}}
\newcommand\parb{\mathrm{par}}
\newcommand\comparphase{\overline{T^* \R^{n+1}_{\parb}}}
\newcommand\Psipar{\Psi_{\parb}}
\newcommand\Hpar{H_{\parb}}
\newcommand{\msf}{\ensuremath{\mathsf}}
\newcommand{\el}{l}

\usepackage{graphicx}
\usepackage{float}

\usepackage{graphicx}
\usepackage{float}

\numberwithin{equation}{section}
\newtheorem{theorem}{Theorem}[section]
\newtheorem{proposition}[theorem]{Proposition}
\newtheorem{definition}[theorem]{Definition}
\newtheorem{lemma}[theorem]{Lemma}
\newtheorem{corollary}[theorem]{Corollary}
\newtheorem*{example}{Example}

\theoremstyle{definition}

\newtheorem{problem}{Problem}[section]

\newtheorem{remark}[theorem]{Remark}

\begin{document}

\title[Lecture notes on non-elliptic Fredholm theory]{Lecture notes on non-elliptic Fredholm theory}
\author{Andrew Hassell}
\author{Qiuye Jia}
\author{Ethan Sussman}
\begin{abstract}These are lecture notes from the Austral Winter School on Microlocal Analysis and Non-elliptic Fredholm Theory, held at the Australian National University, Canberra, June 30 -- July 11, 2025. 
\end{abstract}

\maketitle

\tableofcontents

\setcounter{section}{-1}

\section{Introduction}

The Austral Winter School on Microlocal Analysis and Non-elliptic Fredholm Theory was a two-week school taught at the Australian National University, Canberra, from June 30 -- July 11, 2025. The lectures were delivered by the three authors of these lecture notes. Problem sessions were an integral part of the school, and there are problems and exercises at the end of each lecture. There were 28 students from a variety of countries, including Australia, the USA, China, Canada and Finland. 

The prerequisite for attending the school, and for reading these lecture notes, is a first course in microlocal analysis, covering the fundamentals of pseudodifferential operators on $\R^n$ and on closed manifolds, principal symbols, ellipticity, and elliptic parametrix. This could be obtained, for example, by reading the final chapter of Folland's book \cite{Folland-book} \emph{Introduction to Partial Differential Equations} (Princeton, 1995), or by reading the first 5 chapters of Hintz' recent book \cite{Hintz-book} \emph{An Introduction to Microlocal Analysis} (Springer, 2025), or lecture notes of Richard Melrose, available from his website. Other monographs or sources, most going far beyond what is needed to read these notes, include the third volume \cite{hormander2007analysis} of  H\"ormander's four volume treatise \emph{The analysis of linear partial differential operators}, Springer, 1983; Taylor, \emph{Pseudodifferential Operators}, Princeton, 1981 \cite{Taylor-book-2} and \emph{Partial Differential Equations II} \cite{Taylor-book-pseudo}, Springer, 1996; Tr\`eves, \emph{Introduction to Pseudodifferential and Fourier Integral operators vol. 1}, Plenum, 1980 \cite{Treves-book};  Dimassi-Sj\"ostrand, \emph{Spectral asymptotics in the semi-classical limit}, Cambridge, 1999 \cite{dimassi-sjostrand1999spectral}; Zworski, \emph{Semiclassical Analysis}, AMS, 2012 \cite{zworski2022semiclassical}. 

The aim of the Winter School was to explain the Fredholm approach to solving non-elliptic problems on complete manifolds, as developed during the 2010s principally by Andr\'as Vasy, although the first instance is due to Faure and Sj\"ostrand \cite{faure-sjostrand2011anosov}. We elected to do this in the context of the scattering calculus, as this is technically less demanding than the approach first used by Vasy and Hintz-Vasy, which was dedicated to studying the wave equation and used the b-calculus. Using the scattering calculus we can analyze several operators including the Helmholtz operator, the Klein-Gordon operator and  the Schr\"odinger operator (using a parabolic variant of the scattering calculus that is indistinguishable from it at finite frequency). Fredholm analysis of the wave operator, however, really does require the b-calculus due to a degeneracy at zero frequency, as explained in Problem~\ref{prob:wave} of Lecture 11. This topic is beyond the scope of these lectures. 

In the first 4 lectures we cover the basic theory of the scattering calculus on $\R^n$, including a quick review of abstract Fredholm theory. We show that elliptic elements of the scattering calculus are Fredholm on weighted Sobolev spaces (here, `elliptic' means `fully elliptic' which is a stronger notion than ellipticity in the uniform H\"ormander calculus, as it requires invertibility of the principal symbol at spatial infinity). We use a compactified perspective, in which the phase space (cotangent bundle) is compactified to a compact manifold with corners, and microlocal analysis takes place, in some sense, at the boundary of this compactification. 
In Lectures 5 and 6 we prove H\"ormander's `Propagation of singularities' in the setting of the scattering calculus (more properly called the Propagation of Regularity theorem, or a microlocal propagation estimate). In lectures 7 and 8, we describe a more general setting for the scattering calculus on the interior of any compact manifold with boundary, as presented by Melrose, and give some applications in inverse problems, where this viewpoint is essential. Lecture 8 completed the syllabus for the first week of the school. 

In Lecture 9, we introduce the Helmholtz operator and analyze its characteristic set and Hamilton vector field on the scattering phase space (or more precisely, on the boundary of its compactification). 
We then introduce variable order Sobolev spaces (Lecture 10) and extend the Propagation of Regularity result to this setting. Lecture 11 is the technical heart of the school, where we extend the propagation estimates to radial points, where the Hamilton vector field $H_p$ of our operator $P$  vanishes at the boundary of compactified phase space. These propagation estimates are slightly more intricate in that they take a different form depending on whether the Sobolev regularity/decay of the weighted Sobolev space being employed is more or less than a threshold value (above/below threshold radial point estimates). In Lecture 12, all these microlocal propagation estimates, together with a straightforward elliptic estimate, are assembled into a global semi-Fredholm estimate for the operator $P$. When a similar, dual estimate can be proven for the formal adjoint of $P$, one obtains a Fredholm realization of $P$ between certain spaces $\SX \to \SY$, where $\SY$ is a variable order Sobolev space, and $\SX$ is a `domain space' (that is, there is a condition on both $u$ and $Pu$, that they should lie in the appropriate variable order Sobolev space). In the remaining three lectures we give some applications of the general theory: to the Helmholtz operator in Lecture 13, to the Klein-Gordon operator in Lecture 14 and to the Schr\"odinger operator in Lecture 15. We also give, in Lecture 15, an alternative technical approach to building Fredholm realizations of operators that does not require variable order spaces. 

We mention some of the research articles closely related to these lectures are based. As mentioned, the first instance of a Fredholm realization of a non-elliptic operator is in a 2011 paper of Faure and Sj\"ostrand \cite{faure-sjostrand2011anosov}. They had the essential insight that, on certain judiciously chosen variable order spaces, a global escape function for the bicharacteristic flow could be constructed. Independently, Vasy \cite{vasy2013asympthyp} provided a more general framework for proving Fredholm estimates, and used this to give an asymptotic expansion of waves on a Kerr-de Sitter-like Lorentzian background. This framework was used by Vasy, with a variety of collaborators including Baskin, Gell-Redman, Haber, Hintz and Wunsch to analyze linear and nonlinear wave equations on various Lorentzian backgrounds \cite{baskin-vasy-wunsch2015asymptrad}, \cite{hintz-vasy2015semilinear}, \cite{gell-redman-haber-vasy2016feynman}, \cite{hintz-quasi}, including the celebrated work of Hintz-Vasy on the nonlinear stability of the Kerr-de Sitter family of black hole spacetimes \cite{hintz-vasy-KdSstability}. All these works applied the `Fredholm method' to the indicial family of operators arising from the b-calculus. 

Meanwhile, the Fredholm method in the scattering calculus was exposited in Vasy's Grenoble lecture notes \cite{vasy-monicourse} and applied by Gell-Redman, Hassell, Shapiro and Zhang to the nonlinear Helmholtz equation \cite{Helmholtz-1}, and by Gell-Redman, Gomes and Hassell to the Schr\"odinger equation (linear and nonlinear) \cite{gell2022propagation}, \cite{gell2023scattering}. The three authors of these lecture notes, with Vasy, also recently applied the Fredholm method to the analysis of the nonrelativistic limit, in which the Klein-Gordon equation degenerates to two decoupled Schr\"odinger equations \cite{HJSV1, HJSV2}.

\section{\texorpdfstring{Lecture 1: The scattering calculus on $\R^n$ introduced}{Lecture 1: The scattering calculus on Rn introduced}}
\label{sec:lecture1}

Recall that the standard H\"ormander pseudodifferential calculus $\Psi^m(\R^n)$ is the space of operators $A$ given by quantizations of symbols $a \in S^m(\R^n \times \R^n)$ as follows:
\begin{multline}
A \phi(x) = (2\pi)^{-n} \int e^{i(x-y) \cdot \xi} a(x, \xi) \phi(y) \, dy \, d\xi, \quad \phi \in \Schw(\R^n) \\
= (2\pi)^{-n} \int e^{ix \cdot \xi} a(x, \xi) \hat\phi(\xi) \, d\xi. 
\label{eq:def}
\end{multline}
Here the symbol class $S^m(\R^n \times \R^n)$ is the class of smooth functions $a(x, \xi)$ satisfying the standard symbol estimates 
\begin{equation}
    \Big| D_x^\alpha D_\xi^\beta a(x, \xi) \Big| \leq C_{\alpha \beta} \ang{\xi}^{m-|\beta|}.
\end{equation}
Our notation is that $D_{x_j} = -i \partial_{x_j}$, $D_x^\alpha = D_{x_1}^{\alpha_1} \dots D_{x_n}^{\alpha_n}$ where $\alpha = (\alpha_1, \dots, \alpha_n) \in \N^n$ is a multi-index, and $\ang{\xi} := \sqrt{1 + |\xi|^2}$ is the `Japanese bracket'. When \eqref{eq:def} holds, we will say that $A$ is the (left, or standard) quantization of $a$, and write $A = \Op(a)$. Occasionally we write $A = \Op_L(a)$, where $L$ stands for `left'. The corresponding right quantization $\Op_R(a)$ is given by a similar formula to \eqref{eq:def} where $a(x, \xi)$ is replaced by $a(y, \xi)$, that is, $a$ depends on the right variable $y$ of the Schwartz kernel instead of the left variables $x$. 

We will assume that the reader is familiar with basic elements of pseudodifferential calculus, including that the definition \eqref{eq:def}, initially for Schwartz $\phi$, can be extended to all tempered distributions.

The scattering calculus on $\R^n$, denoted $\Psi_{\sca}^{*, *}(\R^n)$, is a sub-calculus that is more symmetric in $(x, \xi)$. Given two orders $(m,l)$, the set of scattering symbols of order $(m,l)$, denoted $S_{\sca}^{m,l}(\R^n \times \R^n)$ are by definition the smooth functions $a(x, \xi)$ satisfying the estimates 
\begin{equation} \label{eq:sc-symbol-Rn}
    \Big| D_x^\alpha D_\xi^\beta a(x, \xi) \Big| \leq C_{\alpha \beta} \ang{x}^{l-|\alpha|} \ang{\xi}^{m-|\beta|}.
\end{equation}
These symbol estimates endow $S_{\sca}^{m,l}(\R^n \times \R^n)$ with the structure of a Fr\'echet space, with an increasing sequence of norms defined by 
\begin{equation}
\| a \|_{m,l; k} = \sup_{(x, \xi) \in \R^{2n}, |\alpha| + |\beta| \leq k} \Big| \ang{x}^{-l+|\alpha|} \ang{\xi}^{-m+|\beta|} D_x^\alpha D_\xi^\beta a(x, \xi)  \Big|. 
\end{equation}
The quantization of such symbols, as in \eqref{eq:def}, gives us the space of scattering pseudodifferential operators of order $(m,l)$, denoted $\Psi_{\sca}^{m,l}(\R^n)$. 

We call $S^{-\infty, -\infty}_{\sca}(\R^n \times \R^n) := \bigcap_{m,l} S^{m,l}_{\sca}(\R^n \times \R^n)$ the \emph{residual space of symbols}; similarly, we call a symbol $a \in S^{-\infty, -\infty}_{\sca}(\R^n \times \R^n)$ a \emph{residual symbol}, and its quantization a \emph{residual operator}. 
We write $\Psisc^{-\infty, -\infty}(\R^n)$ to denote the class of residual operators, thus 
$$
\Psisc^{-\infty, -\infty}(\R^n) = \cap_{m,l} \Psi^{m,l}(\R^n). 
$$
We also denote $\Psisc(\R^n) = \cup_{m,l} \Psisc^{m,l}(\R^n)$ to be the space of all scattering pseudodifferential operators. This is variously denoted $\Psi_{\sca}^{*, *}(\R^n)$ or $\Psi_{\sca}^{\infty, \infty}(\R^n)$ in different articles. 

The Schwartz kernel of a scattering pseudodifferential operator of order $(m,l)$ is thus given by 
\begin{equation}
K(x, y) = (2\pi)^{-n} \int e^{i(x-y) \cdot \xi} a(x, \xi) \, d\xi
\end{equation}
as an `oscillatory integral' (see \cite[Section 7.8]{hormanderbookvolI}) where $a \in S_{\sca}^{m,l}(\R^n \times \R^n)$ (and as a convergent integral if $m< -n$). It follows from this formula that we can express $a$ in terms of the kernel $K$ by performing a Fourier transform in the variable $x-y$:
\begin{equation}\label{eq:symbol from K}
a(x, \xi) = \int e^{-iw \cdot \xi} K(x, x-w) \, dw. 
\end{equation}

The point of the scattering calculus is 
\begin{itemize}
    \item It is an ideal tool for doing scattering theory, i.e. the study of the generalized eigenfunctions of operators on $\R^n$ and their large-distance asymptotics;
    \item It is the right framework for realizing certain non-elliptic operators as Fredholm maps between suitable Sobolev-like spaces. 
\end{itemize}

Some examples of scattering pseudodifferential operators:
\begin{itemize}
    \item Constant coefficient differential operators of order $m$ on $\R^n$ are scattering pseudodifferential operators of order $(m,0)$. To give specific examples, the Laplacian $\Delta = \sum_j D_{x_j} D_{x_j}$, the wave operator $D_t^2 - \Delta$ and the Klein-Gordon operator $D_t^2 - \Delta - m^2$ (where $m > 0$ is constant) are scattering pseudodifferential operators. (Notice that we use the `microlocal' convention that $\Delta$ is a positive operator, instead of the `PDE' convention that $\Delta$ is negative.)
    \item More generally, pseudodifferential operators of order $m$ in the H\"ormander class with symbols depending only on $\xi$ (thus, Fourier multipliers) are scattering pseudodifferential operators of order $(m,0)$. 
    \item Differential operators of order $m$ with coefficients that are themselves symbols in $x$ of order $l$ are in $\Psi_{\sca}^{m,l}(\R^n)$. 
    \item Suppose that $g= \sum_{ij} g_{ij}(x) dx_i dx_j$ is an `asymptotically Euclidean' metric on $\R^n$, i.e. 
    $$
    g_{ij}(x) - \delta_{ij} \in S^{-1}(\R^n),
    $$
    then the associated Laplace operator, 
    $$
    \Delta_g = \sum_{ij} \frac1{\sqrt{g(x)}} D_{x_i} g^{ij}(x) \sqrt{g(x)} D_{x_j}, \quad g = \det g_{ij}(x),   $$
    is in $\Psi^{2,0}_{\sca}(\R^n)$. (Here $g^{ij}(x)$ is the inverse matrix to $g_{ij}(x)$.)
    \item A non-example: differential operators with periodic (and non-constant) coefficients are in the H\"ormander class, but not in the scattering class of pseudodifferential operators. 
\end{itemize}

\subsection{Basic Properties of Symbols}
\begin{itemize}
    \item If $m' \leq m$ and $l' \leq l$ then $S^{m',l'}_{\sca}(\R^n \times \R^n) \hookrightarrow S^{m,l}_{\sca}(\R^n \times \R^n)$ is continuous. 
    \item $D_x^\alpha D_\xi^\beta$ maps $S^{m,l}_{\sca}(\R^n \times \R^n) \to S^{m-|\beta|,l-|\alpha|}_{\sca}(\R^n \times \R^n)$ continuously. 
    \item Pointwise multiplication is continuous 
    $$
S^{m,l}_{\sca}(\R^n \times \R^n) \times S^{m',l'}_{\sca}(\R^n \times \R^n) \to S^{m+m',l+l'}_{\sca}(\R^n \times \R^n).     
    $$
    \item Density. The residual space of symbols $S^{-\infty, -\infty}_{\sca}$ is not dense in $S^{m,l}_{\sca}$. However, if $a \in S^{m,l}_{\sca}$, then there exists a sequence $a_j \in S^{-\infty, -\infty}_{\sca}$ that is uniformly bounded in $S^{m,l}_{\sca}$ and which converges to $a$ in the (slightly weaker) topology of $S^{m',l'}_{\sca}$ for any $m' > m$ and $l' > l$. 
    \item Asymptotic summation: given a sequence of orders $(m_j, l_j), j \geq 0$ with $m_j \searrow -\infty$, $l_j \searrow -\infty$, and a sequence of scattering symbols $a_j \in S^{m_j, l_j}_{\sca}$, there exists $a \in S^{m_0, l_0}_{\sca}$ such that 
    $$
    a - \sum_{j=0}^{J-1} a_j \in S^{m_J, l_J}_{\sca}.
    $$
    We call $a$ an asymptotic sum of the $a_j$; it is unique modulo $S^{-\infty, -\infty}_{\sca}$.
\end{itemize}

\subsection{Compactification of phase space}\label{subsec:compactification}
We define the radial compactification of $\R^n$ as follows: using polar coordinates, we can represent $\R^n \setminus \{ 0 \}$ as $(0, \infty)_r \times S^{n-1}_{\hat x}$, where $r = |x|$ and $\hat x = x/|x|$. We then let $s = 1/r$, so that $[1, \infty)_r = (0, 1]_s$. We compactify $(0, 1]_s$ to $[0, 1]_s$ and then take the product with $S^{n-1}_{\hat x}$; we have thus added a `sphere at infinity'. Formally our compactification $\compn$ is given by 
$$
\compn = \R^n \sqcup [0,1]_s \times S^{n-1}_{\hat x} / \sim,
$$
where the equivalence relation $\sim$ identifies $(r, \hat x)$ (when $r \geq 1$) and $(s, \hat x)$ exactly when $r = 1/s$. Topologically $\compn$ is a closed ball. It has the differentiable structure of a compact $C^\infty$ manifold with boundary. 

The compactification of phase space $\R^n_x \times \R^n_\xi$ is then defined to be $\comphase := \compn_x \times \compn_\xi$. It is a compact manifold with corners of codimension 2. 

We can reformulate the symbol estimates using the compactification $\comphase$. We note that $a \in S^{m,l}_{\sca}$ if and only if $a$ is smooth on $(T^* \R^n)$, and if 
\begin{equation}\label{eq:symbol est equiv}
\big( \ang{x} D_x \big)^\alpha \big( \ang{\xi} D_{\xi} \big)^\beta a \in \ang{\xi}^m \ang{x}^l L^\infty(T^* \R^n)
\end{equation}
for all multi-indices $(\alpha, \beta)$. We then claim that $\ang{x} D_x$ and $\ang{\xi} D_{\xi}$ generate all smooth vector fields on $\comphase$ tangent to the boundary. Let's show this. We can write
$$
\ang{x} D_{x_k} = \frac1{\ang{x}} D_{x_k} + \sum_j \frac{x_j}{\ang{x}} x_j D_{x_k}
$$
and since $\frac{x_j}{\ang{x}}$ is a smooth function on $\compn$ (cf. Problem~\ref{prob:smooth fn}), we see that \eqref{eq:symbol est equiv} is equivalent to requiring that repeated applications of $x_j D_{x_k}$ and $\xi_j D_{\xi_k}$ (as $(j, k)$ range between $1$ and $n$ independently) to $a$ remain in the fixed space $\ang{\xi}^m \ang{x}^l L^\infty(T^* \R^n)$.

We claim that the vector fields $x_j D_{x_k}$ on $\compn$ generate (over $C^\infty(\compn)$) all smooth vector fields on $\compn$ tangent to the boundary. To see this we notice that such vector fields are homogeneous of degree zero under dilations $x \mapsto ax$, $a > 0$. In polar coordinates $r, y$, where $y = (y_1, \dots, y_{n-1})$ are angular coordinates, i.e. functions of $\hat x$, vector fields that are homogeneous of degree zero take the form 
$$
\sum_j b_j(y) D_{y_j} + c(y) r D_r.
$$
Changing to $s = 1/r$ this reads
$$
\sum_j b_j(y) D_{y_j} - c(y) s D_s,
$$
and this is evidently smooth in $(s, y)$ (which is a smooth coordinate system on $\compn$ near the boundary), and tangent to the boundary since the coefficient of $D_s$ vanishes when $s=0$. So the reformulation of the definition of the symbol space is 
\begin{multline}\label{eq:conormal reg}
S^{m,l}_{\sca}(\R^n \times \R^n) = \Big\{ a \in \ang{\xi}^m \ang{x}^l L^\infty(\comphase) \mid  V_1 \dots V_k a \in \ang{\xi}^m \ang{x}^l L^\infty(\comphase) \text{ for all} \\  \text{smooth vector fields $V_1, \dots V_k$}  \text{ on $\comphase$ tangent to the boundary, and all $k \in \N$} \Big\}.
\end{multline}
It's clear from this formulation that $C^\infty(\comphase)$ (or more precisely the set of restrictions of such functions to the interior $T^* \R^n$) is contained in $S^{0,0}_{\sca}$; we call these symbols \emph{classical} symbols of order $(0,0)$ and denote them $S^{0,0}_{\sca, \cl}(T^* \R^n)$. We similarly define classical symbols of order $(m,l)$ to be $S^{m,l}_{\sca, \cl}(T^* \R^n) := \ang{\xi}^m \ang{x}^l S^{0,0}_{\sca, \cl}(T^* \R^n)$.  Classical symbols have a Taylor-like series at both spatial and frequency infinity. Choose boundary defining functions $s = 1/r$ for the boundary at spatial infinity and $\rho = 1/|\xi|$ at frequency infinity. (A boundary defining function for a boundary hypersurface $H$ of a manifold with corners is a smooth function that is nonnegative, vanishes at $H$ and whose differential is nonzero at $H$, i.e. it vanishes simply at $H$.) Then $a \in S^{m,l}_{\sca, \cl}$ iff $a$ admits asymptotic expansions as $s \to 0$ and as $\rho \to 0$ of the form 
\begin{equation}
a \sim \sum_{j=0}^\infty s^{-l+j} a_j(\hat x, \xi) \text{ where $|\xi|$ is bounded }, \ a_j \in C^\infty, s \to 0, 
\end{equation}
as well as 
\begin{equation}
a \sim \sum_{k=0}^\infty \rho^{-m+k} b_k(x, \hat \xi) \text{ where $|x|$ is bounded }, \ b_k \in C^\infty, \rho \to 0. 
\end{equation}
Near the corner, we have a Taylor-like expansion in two variables $(\rho, s)$:
\begin{equation}
a \sim \sum_{j,k=0}^\infty s^{-l+j} \rho^{-m+k} c_{jk}(\hat x, \hat \xi) c_{jk} \in C^\infty, \rho \to 0, \ s \to 0. 
\end{equation}
Not all symbols are classical! A simple example of a non-classical symbol in $S^{0,0}_{\sca}$ is $\ang{x}^{i\lambda}$ or $\ang{\xi}^{i\lambda}$ where $\lambda \neq 0 \in \R$. Clearly, this symbol does not have a limit at the boundary of phase space, as classical symbols of order $(0,0)$ do. 

\subsection{Exercises}
\begin{problem}\label{prob:density symbols}
    \item Prove the statements about density of symbols above. To prove the density statement, take a function $\chi \in C_c^\infty(\R^n)$, such that $\chi = 1$ on $B(0,1)$ and $\chi = 0$ outside $B(0,2)$. Let $a_j(x, \xi)$ be the function 
    $$
    a_j(x, \xi) = \chi(x/j) \chi(\xi/j) a(x, \xi)
    $$
and prove the properties stated above for the sequence $a_j$. 
\end{problem}

\begin{problem} Prove the statement about asymptotic summation above. In this case, we may construct $a$ as a sum of the form 
$$
a = \sum_{j=0}^\infty a_j(x, \xi) \big( 1 - \chi(\epsilon_j x) \big) \big( 1 - \chi(\epsilon_j \xi) \big),
$$
where the $\epsilon_j$ are a sequence of small parameters converging to zero sufficiently fast. Notice that this sum is locally finite for any fixed $(x, \xi)$ since only finitely many terms are nonzero, so the sum converges to a smooth function. However, we need to show more, we need convergence in the symbol space. Assume, mostly for ease of notation, that $m_j = m-j$ and $l_j = l-j$. 

(i) Prove that a sufficient condition for convergence, both of the series above to $a$ in the topology of $S^{m, l}_{\sca}$, and of the series $a - \sum_{j=0}^J a_j$ in the topology of $S^{m -J, l - J}_{\sca}$, is that for all $k \geq 0$, $J \geq 0$, 
$$
\| a_{J+k}(x, \xi) \big( 1 - \chi(\epsilon_{J+k} x) \big) \big( 1 - \chi(\epsilon_{J+k} \xi) \big) \|_{m-J, l-J; J} \leq 2^{-k}. 
$$

(ii) Show that these conditions are satisfied given a finite number of smallness conditions on $\epsilon_j$ for each $j$. Hence we can satisfy all conditions and thus construct the asymptotic sum. 
\end{problem}

\begin{problem}\label{prob:smooth fn} Show that $\frac{x_j}{\ang{x}}$ is a smooth function on $\compn$, by writing it explicitly as a smooth function in a coordinate system $(s, y)$, where $s=1/r$ and $y = (y_1, \dots, y_{n-1})$ are local coordinates on the sphere $S^{n-1}$. (Equivalently, each $y_i$ is a function homogeneous of degree zero on a conic open set of $\R^n$, and their differentials are linearly independent.)
\end{problem}

\begin{problem}
    Show that if $a$ and $b$ are classical symbols, then $\partial_{x_j} a$, $\partial_{\xi_k} a$, and $ab$ are classical. Hence show that the scattering pseudodifferential operators with classical symbols forms a subalgebra of $\Psisc^{*,*}(\R^n)$. It is denoted $\Psi_{\sca, \cl}^{*,*}(\R^n)$. 
\end{problem}
   
\section{Lecture 2: Composition, principal symbol, ellipticity}

\subsection{Composition}
As the term `calculus' suggests, the scattering pseudodifferential calculus is a bigraded algebra, i.e. 
\begin{equation}
    \Psisc^{m,l} \circ \Psisc^{m',l'} \subset \Psisc^{m+m', l+l'}.
\end{equation}
We will give a partial proof of this fact, taking for granted the graded algebra property of the H\"ormander calculus. First we treat the case $l \leq 0, l' \leq 0$; it is easy to reduce to this case (see the first exercise). Let $A \in \Psisc^{m,l}$ and $B \in \Psi^{m',l'}$. Then $A \in \Psi^m$ and $B \in \Psi^{m'}$ (the usual H\"ormander calculus), so we know that $A \circ B \in \Psi^{m+m'}$. We need to show that the symbol of $A \circ B$ is in the space $S^{m+m', l+l'}_{\sca}$, which is a stronger condition. Let's consider the first the asymptotic expansion of the symbol of the composition:
\begin{equation}\label{eq:comp expansion}
    \sigma_L(A \circ B) \sim \sum_\alpha \frac{\big(D_\xi^\alpha a_L(x, \xi) \big)\big( \partial_x^\alpha b_L(x, \xi)\big)}{\alpha!}
\end{equation}
Consider all terms with $|\alpha| = j$. We have 
$$
D_\xi^\alpha a_L(x, \xi) \in S^{m-|\alpha|, l}_{\sca}, \quad \partial_x^\alpha b_L(x, \xi) \in S^{m', l' - |\alpha|}_{\sca}
$$
so taking the pointwise product gives a symbol in $S^{m+m'-|\alpha|, l+l'-|\alpha|}_{\sca}$ for all such terms. We can asymptotically sum such terms since both exponents are tending to $-\infty$ as $|\alpha| \to \infty$, obtaining a symbol in $S^{m+m, l+l'}_{\sca}$. This is a good sign, but is not a proof as we also have to consider the properties of the remainder term.

There are three main approaches to composition: Gauss transforms, used in H\"or\-mander's treatise \cite{hormander2007analysis}, as well as Zworski \cite{zworski2022semiclassical}, Dimassi-Sj\"ostrand \cite{dimassi-sjostrand1999spectral}, etc; left-right reduction, used by lecture notes of Melrose, Vasy \cite{vasy-monicourse} and Hintz \cite{Hintz-book}; and using pullback and pushforward of distributions via double and triple spaces, due to Melrose (e.g. \cite{melrose1996scattering}). We will follow the left/right reduction method. This involves quantizing symbols that depend on both the left and right variables of the Schwartz kernel of the operator, i.e. we quantize symbols $a(x, y, \xi)$ in a certain symbol class to the operator with Schwartz kernel
\begin{equation}\label{eq:double quant}
    K_a(x, y) = (2\pi)^{-n} \int e^{i(x-y) \cdot \xi} a(x, y, \xi) \, d\xi.
\end{equation}
The relevant class here is symbols that behave in a scattering-like manner in the $x$ and $y$ variables independently: that is, we consider classes $S^{m,l,\tilde l}_{\sca}(\R^n_x \times \R^n_y \times \R^n_\xi)$ satisfying product-type estimates 
\begin{equation}
       \Big| D_x^\alpha D_y^{\tilde \alpha}D_\xi^\beta a(x, \xi) \Big| \leq C_{\alpha \tilde{\alpha} \beta} \ang{x}^{l-|\alpha|} \ang{y}^{\tilde l - \tilde \alpha}\ang{\xi}^{m-|\beta|}.
\end{equation}
It turns out that these quantize to elements of $\Psisc^{m, l + \tilde l}$. That is, only the sum $l + \tilde l$ is relevant to the spatial order of the operator, which is not surprising as only the jet of the symbol $a(x, y, \xi)$ on the diagonal $x=y$ affects the operator (modulo residual terms). This is shown in Vasy's Grenoble lecture notes \cite{vasy-monicourse}. 

To show composition, then, we follow the following standard strategy:
\begin{itemize}
    \item We show that each symbol $a(x, y, \xi) \in S^{m,l,\tilde l}_{\sca}$ can be left or right reduced in a unique manner. That is, there are uniquely determined symbols $a_L(x, \xi)$ and $a_R(y, \xi)$, both in $S^{m, l + \tilde l}_{\sca}$, that quantize to the same operator as \eqref{eq:double quant}. 
    \item We observe that the composition $A \circ B$ is a pseudodifferential operator with symbol $a_L(x, \xi) b_R(y, \xi)$. 
    \item We then perform, say, left reduction on this symbol to show that we obtain something in $S^{m+m', l + l'}_{\sca}$. One expands the symbol in a Taylor series around the diagonal $x=y$.  As we have already seen, the terms in the asymptotic expansion are no problem, but we need to consider the remainder term after subtracting the first $J$ terms in the expansion. There is a subtlety, as this remainder term is obtained from the remainder term in the Taylor expansion of $a_L(x, \xi) b_R(y, \xi)$ at $x=y$, and involves an integral along the line segment from $x$ to $y$ (the usual Taylor formula integral remainder term). The problem is that if $|x|$ is large, and $y$ is approximately $-x$, then the line segment passes near the origin. It means that the remainder might, ostensibly, decay much less than the terms in the expansion, since spatial decay happens when $x$ and/or $y$ are large. To circumvent this difficulty, we can decompose our symbol $a_L(x, \xi) b_R(y, \xi)$ into a part far from the diagonal and a part near the diagonal. The part far from the diagonal produces a residual operator as is shown by standard integration-by-parts tricks -- see exercises. The part near the diagonal has the good property that the line segment from $x$ to $y$ has distance from the origin comparable to $|x|$ and $|y|$, and allows us to prove uniform decay of the remainder term as required. See Vasy's Grenoble notes \cite[Section 5.3.2]{vasy-monicourse} for the details. 
\end{itemize}

We obtain the bigraded algebra property of the scattering calculus, and we have also confirmed that the symbol of the composition admits the usual asymptotic expansion \eqref{eq:comp expansion}, which is an \emph{even better expansion} than the usual expansion in the sense that we gain spatial decay as well as frequency decay with each additional term. Indeed, in the scattering calculus, if we write $R_N$ for the remainder term of expansion of $A \circ B$, then $R_N \in \Psisc^{m+m'-N, l+l'-N}$ and every norm of $R_N$ in this space (by which we mean the norms of $\sigma_L(R_N)$ in $S^{m+m'-N, l+l'-N}_{\sca}$) is bounded by a constant times a suitable norm of $A$ in $\Psisc^{m,l}$ and a suitable norm of $B$ in $\Psisc^{m',l'}$. 

\subsection{Principal symbol}
Let $A = \Op_L(a_L) \in \Psisc^{m,l}$. The principal symbol, $\spr^{m,l}(A)$, is defined to be the equivalence class of $a_L$ in $S^{m,l}_{\sca} / S^{m-1. l-1}_{\sca}$. It is the same as the equivalence class of $a_R$ if $A = \Op_R(a_R)$. 

Notice the difference with the H\"ormander calculus: the principal symbol determines the leading behaviour of the symbol \emph{both} as $|\xi| \to \infty$ for all $x$, and as $|x| \to \infty$ for all $\xi$ (even $\xi = 0$). 

For classical symbols, we can be more explicit. Assume that $(m,l) = (0,0)$ for convenience. Then a classical symbol $a \in S^{0,0}_{\sca, \cl}$ is by definition $C^\infty$ on $\comphase$, and in particular has a limit at $\partial \comphase$. On the other hand, if $b \in S^{-1, -1}_{\sca}$, then $b = s \rho \tilde b$, where $\tilde b \in C^\infty(\comphase)$ and $s = \ang{x}^{-1}, \rho = \ang{\xi}^{-1}$. Such symbols are \emph{precisely} those symbols of order $(0,0)$ that \emph{vanish} at $\partial \comphase$. Thus, the principal symbol of $a \in S^{0,0}_{\sca, \cl}$ can be identified with the boundary value of $a$ on $\partial \comphase$. The boundary value of $a$ is exactly knowing the limit of $a$ as $|\xi| \to \infty$ for all $x$ and $\hat \xi$, and the limit as $|x| \to \infty$ for all $\xi$ and $\hat x$. This boundary value is `smooth' in the sense that it is smooth on each boundary hypersurface, and is consistent at the corner. Next lecture, we will see that there is an important symmetry between these two boundary hypersurfaces of the manifold with corners $\comphase$. 

Fundamental properties of the principal symbol:
\begin{itemize}
    \item By definition of the principal symbol there is a short exact sequence:
    \begin{equation}\label{eq:short exact}
        0 \rightarrow \Psisc^{m-1, l-1} \rightarrow \Psisc^{m,l} \rightarrow S^{m,l}_{\sca} / S^{m-1,l-1}_{\sca} \rightarrow 0,
    \end{equation}
    where the second arrow is inclusion and the third arrow is the principal symbol map. 
    \item The principal symbol is multiplicative:
    \begin{equation}
        \spr^{m+m', l+l'}(A \circ B) = \spr^{m,l}(A) \spr^{m', l'}(B).
    \end{equation}
    This can be seen from the asymptotic expansion \eqref{eq:comp expansion}, in which every term other than the leading term (which is the pointwise product) lies in $S^{m+m'-1,l+l'-1}_{\sca}$ which is zero in the quotient space where $\spr^{m+m', l+l'}(A \circ B)$ lives. 
    \item We also have the formula for the principal symbol of a commutator $[A, B] = AB - BA$. Recall some background to this identity: the cotangent bundle of any manifold, here $\R^n$, is a symplectic manifold with a canonically defined symplectic form $\omega = \sum_j d\xi_j \wedge dx_j$. This defines a Poisson bracket $\{ \cdot, \cdot \}$ and for any real function $p$ on $T^* \R^n$, a Hamilton vector field $H_p$ on $T^* \R^n$ defined invariantly by 
    $$
    \omega(H_p, V) = dp(V) = V(p). 
    $$
    Then we have 
    \begin{equation}\label{eq:Poisson b}
        \spr^{m+m'-1, l+l'-1}(i[A, B]) = \{ a, b \} = H_a (b) = - H_b(a), \quad a = \spr^{m,l}(A), \ b = \spr^{m', l'}(B). 
    \end{equation}
    This also can be seen from the asymptotic expansion \eqref{eq:comp expansion}; in this case, the leading terms cancel and the subleading terms give the Poisson bracket \eqref{eq:Poisson b}. 
\end{itemize}

\subsection{Ellipticity}
We say that $A \in \Psisc^{m,l}$ is (totally) elliptic if its principal symbol is invertible in the sense that there exists $B_0 \in \Psisc^{-m,-l}$ such that 
$$
\spr^{m,l}(A) \cdot \spr^{-m,-l}(B_0) = 1 \in S^{0,0}_{\sca} / S^{-1,-1}_{\sca}.
$$
In other words, 
\begin{equation}\label{eq:B0 symbol}
\sigma_L(A) \cdot \sigma_L(B_0) - 1 \in S^{-1,-1}_{\sca}.
\end{equation}

\begin{remark}\label{rem:ellipticity} This is a very strong version of ellipticity! The standard Laplacian $\Delta$  is \emph{not} elliptic according to this definition, for example: although the principal symbol at frequency infinity is elliptic, the principal symbol at spatial infinity is not; in fact, it vanishes at $\xi = 0$. However, $\Delta + \lambda$ is totally elliptic in the scattering calculus provided $\lambda \in \mathbb{C} \setminus \R_{\leq 0}$. 
\end{remark}

The standard parametrix construction, adapted to the scattering calculus, shows that if $A$ is an elliptic scattering pseudodifferential operator, then there is an inverse modulo residual operators in the scattering calculus. Let us show this. 

\begin{definition} Let $A \in \Psisc^{m,l}$ be elliptic. A parametrix for $A$ is an operator $B$ such that 
\begin{equation}
    A B - \Id \in \Psisc^{-\infty, -\infty} \ni BA - \Id. 
\end{equation}
\end{definition}

\begin{proposition} Every elliptic element of the scattering calculus has a parametrix. 
\end{proposition}

\begin{proof} Given $A \in \Psisc^{m,l}$ elliptic, choose $B_0 \in \Psisc^{-m,-l}$ satisfying \eqref{eq:B0 symbol}. From \eqref{eq:B0 symbol} and the symbol exact sequence \eqref{eq:short exact}, we see that 
\begin{equation}
    AB_0 - \Id = -R_1 \in \Psisc^{-1,-1}. 
\end{equation}
Thus $AB_0 = \Id - R_1$. If we can invert $\Id - R_1$ then we can compose with this inverse on the right, and hence (exactly) invert $A$. Thus, 
we would like to invert $\Id - R_1$. This is not possible in general, but \emph{modulo residual operators} we can do this with a Neumann series: let $R$ be an asymptotic sum 
$$
R \sim R_1 + R_1^2 + R_1^3 + \dots ,
$$
which exists as $R_1^j \in \Psisc^{-j,-j}$. Then we have (exercise)
\begin{equation}\label{eq:Neumann inv}
(\Id - R_1) (\Id + R) = \Id + R_r, \quad R_r \in \Psisc^{-\infty, -\infty}.
\end{equation}
So defining $B_r = B_0(\Id + R)$ we have 
$$
A B_r = \Id + R_r.
$$
In the same way, putting $B_0$ on the left of $A$ instead of to the right, we can find $B_l$ such that 
$$
B_l A = \Id + R_l, \quad R_l \in \Psisc^{-\infty, -\infty}.
$$
We then note that $B_l$ and $B_r$ differ by a residual operator. This follows from the following calculation:
$$
B_l A B_r = B_r + R_l B_r = B_l + B_l R_r. 
$$
It follows that either $B_l$ or $B_r$ is a parametrix for $A$.
\end{proof}

We remark that for classical symbols of order $(0,0)$, ellipticity is equivalent to the condition that $a \big|_{\partial \comphase}$ is everywhere nonzero. Since any classical $a$ is continuous on $\partial \comphase$, and $\partial \comphase$ is compact it follows that if $a$ is classical and elliptic, then $a$ is bounded away from zero on $\partial \comphase$ and hence $a^{-1}$ is also smooth. By extending $a^{-1}$ into the interior arbitrarily as a smooth function, we obtain an symbol $b$ satisfying $ab - 1 \in S^{-1,-1}_{\sca, \cl}$ so quantizing $b$ leads to an operator $B_0$ as above. 

For a general operator $A$, we define the elliptic and characteristic sets of $A$, which are subsets of $\partial \comphase$. This is easiest to do in the case of classical operators. If $A \in \Psisc^{m,l}$ is classical, let $\tilde a = \ang{\xi}^{-m} \ang{x}^{-l} \sigma_L(A)$. Then $\tilde a$ is smooth on $\comphase$ and we can define 
\begin{equation}\begin{aligned}
    \Ell^{m,l}(A) &= \{ q \in \partial \comphase \mid \tilde a(q) \neq 0 \}, \\
    \Char^{m,l}(A) &= \{ q \in \partial \comphase \mid \tilde a(q) = 0 \}.
\end{aligned}\label{eq:ellchar}\end{equation}
More generally, for a non-necessarily-classical operator $A\in \Psisc^{m,l}$, we say that a point $q \in \partial \comphase$ is in $\Ell^{m,l}(A)$ iff there exists a symbol $b \in S^{-m,-l}_{\sca}$ such that $\sigma_L(A) \cdot b - 1$ is in $S^{-1, -1}_{\sca}$ in a neighbourhood of $q \in \comphase$. More precisely, there is a smooth function $\chi$ on $\comphase$, identically $1$ near $q$, such that $\chi (\sigma_L(A) \cdot b - 1)$ is in $S^{-1, -1}_{\sca}$. The characteristic set $\Char^{m,l}(A)$ is defined to be the complement of the elliptic set in $\partial \comphase$. Notice that the elliptic set is an open subset of $\partial \comphase$, and hence the characteristic set is closed. 

\begin{remark}
A remark about notation: the elliptic set depends on the indices $(m,l)$. To see this, notice that if $A \in \Psisc^{m,l}$ then certainly $A \in \Psisc^{m', l'}$ if $m' > m$ and $l' > l$; however, we will always have $\Ell^{m', l'}(A) = \emptyset$ if $A \in \Psisc^{m,l}$. However, we commonly abuse notation and write $\Ell(A)$ instead of $\Ell^{m,l}(A)$ when the orders $(m,l)$ of $A$ have been specified. In this convention, the orders of the elliptic set would only be given if they differ from the orders that are implicit from the context. 
\end{remark}

\begin{remark}
 This is one of the first places where we see that it is the \emph{boundary} of compactified phase space that is of primary interest in microlocal analysis: fundamental objects such as the elliptic set, characteristic set and (operator) wavefront sets (see Lecture 4), are all subsets of the boundary of $\comphase$. This concept can be summarized by the slogan that ``microlocal analysis takes place at the boundary of phase space". (This slogan implicitly assumes that phase space has been compactified.)   
\end{remark}

\subsection{Mapping properties of residual operators}\label{subsec:residual} We will discuss the boundedness properties of scattering operators more fully in Lecture 3. But here, we make an elementary observation about the mapping properties of residual operators. It is easy to see that a symbol $a$ of order $(-\infty, -\infty)$ is precisely a Schwartz function on $\R^{2n}$. It follows that the Schwartz kernel of $\Op(a)$ is also a Schwartz function on $\R^{2n}$, as it involves taking the Fourier transform of $a$ in $\xi$. Such Schwartz kernels map distributions to Schwartz functions, since if $u_1$ and $u_2$ are distributions, and $K(x,y)$ is Schwartz in $\R^{2n}_{x,y}$, then 
$$
u_2 \mapsto \bang{ u_2, \ang{K(x, \cdot), u_1}}, \quad u \in \Schw'(\R^n)_x, 
$$
is a continuous linear functional on $\Schw'(\R^n)$, and thus $\ang{K(x, \cdot), u_1}$ is Schwartz. Thus, residual operators map distributions to Schwartz functions. 

It follows that the null space of any elliptic operator $A$ is Schwartz. In fact, if $B$ is a parametrix, with $BA = \Id + R$, $R$ residual, then $A u = 0$ implies $(\Id + R) u = 0$, so $u = - Ru$ is Schwartz. This has some consequences for essential self-adjointness of symmetric elliptic operators. 

\subsection{Applications of the elliptic parametrix construction}

We give a few corollaries of the elliptic parametrix construction. Fix real-valued $V\in S^{-\varepsilon}(\mathbb{R}^n)$ for some $\varepsilon>0$. 
Recall that a \emph{bound state} of a Schr\"odinger operator $\triangle + V$ on $\mathbb{R}^n$ is an $L^2$-eigenfunction with non-positive eigenvalue. Our discussion above shows that  
\begin{proposition}
    Bound states with strictly negative eigenvalues are necessarily all Schwartz functions.
\end{proposition}

\begin{proof}
Let $-\kappa < 0$ be the eigenvalue, and $u$ be the eigenfunction. Thus $(\Delta + V + \kappa) u = 0$. But the operator $\Delta + V + \kappa$ is fully elliptic, so the result follows from the discussion in Ssection~\ref{subsec:residual}.     
\end{proof}



Next, we give a spectral-theoretic application. Recall that a symmetric, densely defined unbounded operator on a Hilbert space is said to be \emph{essentially self-adjoint} if its closure is self-adjoint. Then, the spectral theorem can be brought to bear. In particular, all eigenfunctions of an essentially self-adjoint operator have real eigenvalues. 

To show that essential self-adjointness is subtle:
\begin{example}
    Fix $\varepsilon \in \{-1,+1\}$ and $k\in \mathbb{N}$. Consider $P= D_x^2 + \varepsilon x^k \in \Psi_{\mathrm{sc}}^{2,k}(\mathbb{R})$. This is certainly symmetric: 
    \begin{equation}
        \langle P u,v \rangle_{L^2} = \langle u,Pv \rangle_{L^2} 
        \label{eq:L2_symmetry}
    \end{equation}
    whenever $u,v\in \mathcal{S}(\mathbb{R})$. Question: is it essentially self-adjoint, when considered as an unbounded operator on $L^2(\mathbb{R})$ with domain $C_{\mathrm{c}}^\infty(\mathbb{R})$? The answer is rather suprising.

    \begin{itemize}
        \item If $k$ is even and $\varepsilon=+1$, or if $k\in \{0,1,2\}$, then $P$ is essentially self-adjoint.
        \item Otherwise, $P$ is not essentially self-adjoint. In fact, for any complex number $\lambda$, there exists an $L^2$-eigenfunction of $P$ with that eigenvalue. So, $P$ is not essentially self-adjoint for \emph{any} domain $\mathcal{D}\subseteq H^2(\mathbb{R})$ containing that eigenfunction. (And a duality argument shows that it is not essentially self-adjoint on the minimal domain $C_{\mathrm{c}}^\infty(\mathbb{R})$ either.)
    \end{itemize}
    Most of this is a consequence of some classical ODE theory (Liouville--Green, or JWKB), which tells us that there exist solutions with prescribed exponential behavior as $x\to\pm \infty$. For example, if $k \geq 4$ is even and $\varepsilon$ is $-1$, then any solution of the ODE $Pu= \lambda u$ is asymptotic to a linear combination of 
    \begin{equation}
        \frac{1 }{|x|^{k/4}} e^{\pm \frac{2 i}{k+2} |x|^{(k+2)/2}  }
        \label{eq:LG_as}
    \end{equation}
    as $|x|\to \infty$. As long as $k\geq 3$, this lies in $L^2(\mathbb{R})$. If instead $k\geq 3$ is odd, then there exists a $u$ with the asymptotic \cref{eq:LG_as} on one side of the real line and superpolynomial decay on the other. Either way, we have an $L^2$-eigenfunction with eigenvalue $\lambda$. 
    Similar arguments apply to the remaining cases. (If $k=0,1,2$, then the asymptotic form \eqref{eq:LG_as} is invalid, receiving a correction from the $\lambda$ term.)
\end{example}

The following reformulation is useful: 
\begin{lemma}
    Fix $m,\el \in \mathbb{R}$. 
    Let $\mathcal{D}$ denote a subspace $C_{\mathrm{c}}^\infty(\mathbb{R}^n)\subseteq \mathcal{D} \subseteq H^{m,\el}(\mathbb{R}^n)$.
    Suppose that $A\in \Psi_{\mathrm{sc}}^{m,\el}$ is $L^2$-symmetric (\cref{eq:L2_symmetry}). Then, $A$ is essentially self-adjoint on $L^2$ with domain $\mathcal{D}$ if
    \[ 
        \operatorname{ker}_{L^2} (A\pm i) = \{0\},
    \]
       for both choice of signs.
\end{lemma}
    This reduces questions of the essential self-adjointness of differential operators $P$ to a concrete question about the well-posedness of the PDE $(P\pm i)u=f$.
    
\begin{proof}

    Note that if $P$ is $L^2$-symmetric, then we in fact have $\langle Pu,v \rangle_{L^2} = \langle u,Pv \rangle_{L^2}$ for all $u\in \mathcal{S}$ and $v\in \mathcal{S}'$. Here, the $L^2$-inner product is understood distributionally: 
    \begin{equation}
        \langle \phi,\psi \rangle_{L^2} = 
        \begin{cases}
            \phi^*(\psi) & \phi\in \mathcal{S}' \\ 
            \psi(\phi^*)  & \psi\in \mathcal{S}'.
        \end{cases}
    \end{equation}
    Indeed, if we fix $u\in \mathcal{S}$, then $\langle Pu,v \rangle_{L^2}, \langle u,Pv \rangle_{L^2}$ depend continuously on $v$ with respect to the topology of $\mathcal{S}'$ and agree when $v\in \mathcal{S}$.
    
    If $A$ is not symmetric \emph{on all of $\mathcal{D}$}, then it is not essentially self-adjoint, trivially, so assume otherwise.
    Deficiency index theory (see e.g.\ [RS72, Chp. VIII §2]]) says that the essential self-adjointness of $A$ on $\mathcal{D}$ is equivalent to the range of $A\pm i$, acting on $\mathcal{D}$, being dense in $L^2$, for both choices of sign. Therefore, it is sufficient to consider only $\mathcal{D}=C_{\mathrm{c}}^\infty(\mathbb{R}^n)$. Now, 
    \begin{align}
    \begin{split}
        \overline{ \operatorname{Ran}_{C_{\mathrm{c}}^\infty(\mathbb{R}^n) } (A\pm i)} = L^2 &\iff  \operatorname{Ran}_{C_{\mathrm{c}}^\infty(\mathbb{R}^n) } (A\pm i)^\perp = \{0\} \\
        &\iff \{u\in L^2:  \langle (A\pm i) v , u \rangle_{L^2}=0\text{ for all } v\in C_{\mathrm{c}}^\infty(\mathbb{R}^n)\}= \{0\} \\
        &\iff \{u\in L^2:  \langle v , (A\mp i)  u \rangle_{L^2}=0\text{ for all } v\in C_{\mathrm{c}}^\infty(\mathbb{R}^n)\}= \{0\} \\ 
        &\iff \underbrace{\{u\in L^2:   (A\mp i)  u =0\}}_{\operatorname{ker}_{L^2}(A\mp i)}= \{0\}.
        \end{split} 
    \end{align}
\end{proof}

Why is it not obvious that $\ker_{L^2} (A\pm i) = \{0\}$? This is the statement that $\pm i$ is not an eigenvalue of $A$, but we saw in the example above that this can happen, e.g.\ for $A=D_x^2 + x^3$. What goes wrong if we try to apply to the unbounded operator $A$ the usual argument that bounded symmetric operators have only real eigenvalues? If $u\in L^2$ satisfies $Au=\lambda u$, \emph{and if $A$ were bounded}, then 
\begin{equation}
       \lambda^* \lVert u\rVert^2 =  \langle Au, u\rangle = \langle u,A u\rangle = \lambda \lVert u \rVert^2, 
\end{equation}
which implies $\lambda=\lambda^*$ if $u\neq 0$. The problem is that if $A$ is not bounded, then we only know $\langle Au,v \rangle_{L^2} = \langle u, A v\rangle_{L^2}$ for ``sufficiently nice'' $u,v$, and it may not be the case that an $L^2$-eigenfunction is nice. For example, for $A=D_x^2 + x^3$, the $\ker_{L^2}(A\pm i)$ consists of functions which are smooth but which do not decay particularly rapidly at infinity. They are in $L^2$, but if we try to carry out the proof that $\langle Au, u\rangle_{L^2} = \langle u,A u\rangle_{L^2}$, the integration-by-parts produces boundary terms. In fact, 
\begin{equation}\label{eq:non ibp}
    \langle Au, u\rangle_{L^2} \neq \langle u,A u\rangle_{L^2}
\end{equation}
for such $u$. In effect, there is a nonzero `boundary term at infinity' that is the difference between the LHS and RHS of \eqref{eq:non ibp}. 
    
\begin{proposition}
    If $V \in S^{-\varepsilon}(\overline{\mathbb{R}^n})$ is real, then the operator $P=\triangle+V$ is essentially self-adjoint on $L^2$ with domain $C_{\mathrm{c}}^\infty$ .
\end{proposition}
\begin{proof}
    By the argument above, it suffices to show that any $L^2$-eigenfunction of $P$ with eigenvalue $\pm i$ is sufficiently nice for $\langle Pu,u \rangle_{L^2} = \langle u,Pu \rangle_{L^2}$ to hold.  

    Note that $P\pm \lambda$ is an elliptic element of $\Psi_{\mathrm{sc}}^{2,0}$ if $\lambda$ is non-real. Consequently, any element of its null space must be Schwartz. 
\end{proof}

\subsection{Exercises}
\begin{problem} Show that the composition result above in the case $l \leq 0, l' \leq 0$ implies composition for arbitrary orders. 
\end{problem}

\begin{problem}
 Consider a symbol in the case $S^{m,l,\tilde l}_{\sca}$ with the property that it vanishes whenever $|x-y| \leq \ang{x}/4$. Show that the quantization of such a symbol is residual. (Integrate by parts suitably.)
\end{problem}

\begin{problem} Prove \eqref{eq:Neumann inv}, by showing that the LHS $- \, \Id$ is in $\Psi_{\mathrm{sc}}^{-k,-k}$ for every $k$. To do this, split $R$ into the finite sum $R_1 + \dots R_1^{k-1}$, and the remainder. 
\end{problem}

\begin{problem}
    Suppose that the characteristic set of $A\in \Psi^{m,0}_{\mathrm{sc}}$ is disjoint from fiber infinity. Suppose that $A$ is symmetric, meaning that $\langle A u,v \rangle_{L^2} = \langle u,A v \rangle_{L^2}$
    for all $u,v \in \mathcal{S}$.

    Show that $A$, with domain $C_{\mathrm{c}}^\infty(\mathbb{R}^n)$, is essentially self-adjoint on $L^2(\mathbb{R}^n)$.  
\end{problem}

\begin{problem}
    Suppose that $V\in C^\infty(\overline{\mathbb{R}})$ satisfies $\lim_{x\to-\infty} V(x)=0$.
    Suppose that $u$ is a bound state, with energy $E<0$. Show that $u$ is smooth everywhere and Schwartz on the left- half-line.
\end{problem}

\begin{problem}
    If $k>0$ and $V\in \mathcal{S}(\mathbb{R})$, then $\triangle  + k x^2 + V$ is essentially self-adjoint acting on $C_{\mathrm{c}}^\infty(\mathbb{R})$. Hint:  a coordinate change from $x$ to some  
    \begin{equation}
        \tilde{x} = \begin{cases} 
            x & (x\leq 1) \\
            x^2 & (x \geq 2)
        \end{cases}
    \end{equation}
    may be useful. Note: $\triangle=-\partial_x^2$.
\end{problem}

\begin{problem}
    Use the asymptotic form \eqref{eq:LG_as} for an eigenfunction $u$ of $A = D_x^2 + x^3$, and assume that the derivative of the function enjoys the corresponding asymptotic, i.e. the derivative of \eqref{eq:LG_as}. Show that there is a `boundary term at infinity' when we attempt to integrate by parts to show that 
    \begin{equation*}
    \langle Au, u\rangle_{L^2} = \langle u,A u\rangle_{L^2}. 
\end{equation*}
\end{problem}

\begin{problem}
    Consider the differential operator on the real line given by $P=-\partial_x^2 + 1$. A student makes the following argument: 
    \begin{itemize}
        \item[]  
    \textit{$P$ is an elliptic element of $\Psi_{\mathrm{sc}}^{2,0}(\mathbb{R})$, and so, via elliptic regularity, $Pu=0 \Rightarrow u\in \mathcal{S}(\mathbb{R})$.}
    \end{itemize}
    But, $u(x)=e^{x}$ solves $Pu=0$ and is not Schwartz. What mistake did the student make?
\end{problem}

\section{Lecture 3: Weighted Sobolev spaces}

\subsection{Intertwining with the Fourier transform}

In the first lecture, we defined $\Psi_{\mathrm{sc}}^{0,0}=\Psi_{\mathrm{sc}}^{0,0}(\mathbb{R}^n)$, which is a certain proper subset 
\[
\Psi_{\mathrm{sc}}^{0,0} \subsetneq \Psi^0 
\]
of H\"ormander's uniform pseudodifferential operator class of order $0$, which we denote simply $\Psi^0 = \Psi^0(\mathbb{R}^n)$. Roughly, one should think of sc-pseudodifferential operators\footnote{We will often abbreviate scattering to sc-.} as those ordinary (i.e. H\"ormander-class) pseudodifferential operators with ``good'' behavior at infinity. 

The purpose of this subsection is to provide another motivation for the sc-calculus, namely that it is, in some sense, the simplest pseudodifferential calculus that \emph{treats momentum/frequency and position on equal footing}. This is an attractive feature of the sc-calculus not possessed by H\"ormander's uniform pseudodifferential algebra.

One can develop this theme in many ways. Here is one:
\begin{proposition} 
If $\mathcal{F}$ denotes the Fourier transform, then $\mathcal{F} \circ \Psi_{\mathrm{sc}}^{m,\el} \circ \mathcal{F}^{-1} = \Psi_{\mathrm{sc}}^{\el,m}  $ . That is, 
\begin{equation}
   \mathcal{F} \circ A \circ \mathcal{F}^{-1} \in \Psi_{\mathrm{sc}}^{\el,m}
\end{equation}
for all  $A\in \Psi_{\mathrm{sc}}^{m,\el}$.
\label{prop:Fourier}
\end{proposition}
\begin{remark}
    In contrast: $\mathcal{F} \circ \Psi^{0} \circ \mathcal{F}^{-1}$ is not a subset of $\Psi^0$. (We will not prove this, but it follows via a similar computation to that below.) 
\end{remark}
\begin{proof}
Fix $A=\operatorname{Op}_L(a)$, $a\in S_{\mathrm{sc}}^{m,\el}$. 
Let us compute the action of $\mathcal{F} \circ A \circ \mathcal{F}^{-1}$ on $\phi\in \mathcal{S}(\mathbb{R}^n)$. Directly plugging into the definition \cref{eq:def}, 
\begin{equation}
  \mathcal{F}\circ  A \circ \mathcal{F}^{-1} \phi(\zeta)  = (2\pi)^{-n} \int \int e^{ix\cdot (\xi-\zeta)} a(x,\xi) \phi(\xi) d\xi d x = \int K(\zeta,\xi) \phi(\xi) d \xi  , 
\end{equation}
where 
\begin{equation}
    K(\zeta,\xi) = (2\pi)^{-n} \int e^{ix \cdot (\xi - \zeta)} a (x,\xi) d x. 
\end{equation}
To make this more recognizable, we can rename dummy variables: 
\begin{equation}
    K(x,y) = (2\pi)^{-n} \int e^{i\xi \cdot (y - x)} a (\xi,y) d \xi =  (2\pi)^{-n} \int e^{i\xi \cdot (x - y)} a (-\xi,y) d \xi. 
\end{equation}
We now recognize this as $ K_{\tilde{a}}(x,y)$, where $\tilde{a}(x,y,\xi) =  a(-\xi,y)$. 
Let $R(\pi/2) : \mathbb{R}^{2n}\to \mathbb{R}^{2n}$ denote the $90^\circ$ rotation 
\begin{equation}\label{eq:rotation}
R(\pi/2): (x,\xi) \mapsto (-\xi,x).
\end{equation}
Then, $\tilde{a}=  a\circ R(\pi/2)$. 
So, 
\begin{equation}
    \mathcal{F}\circ  A \circ \mathcal{F}^{-1} = \operatorname{Op}_{R}(a\circ R ).  
\end{equation}

The key point is that $a\circ R$ is a sc-symbol, specifically one in $S^{\el,m}_{\mathrm{sc}}$, \emph{because $R(\pi/2)$ is (obviously) a map} \[
R(\pi/2):S^{\el,m}_{\mathrm{sc}} \to S_{\mathrm{sc}}^{m,\el},
\]
owing to the fact that the estimates used to define $S_{\mathrm{sc}}^{0,0}$ treat position and frequency on equal footing.
\end{proof}

As part of the previous proposition, we proved:
\begin{proposition}
    If $A\in \Psi_{\mathrm{sc}}$, then, if $a$ denotes the full left symbol of $A$, then $ a\circ R$ is the full right symbol of $\mathcal{F}\circ  A \circ \mathcal{F}^{-1}$. 
    \label{prop:FAF_Symbol}
\end{proposition}

\begin{corollary}
The principal symbol of $\mathcal{F} \circ A \circ \mathcal{F}^{-1}$ is $a \circ R$. 
\end{corollary} 
\begin{proof}
    For either left or right quantization, we have $A=\operatorname{Op}_{L/R}(a) \Rightarrow a \in \sigma(A)$. So, this follows immediately from \Cref{prop:FAF_Symbol}.
\end{proof}

\begin{remark}
    For the reader acquainted with Fourier integral operators (FIOs), let us remark that the Fourier transform can be thought of as an FIO quantizing the symplectomorphism $R^{\pm 1}$. 
    
    Then, since the composition of FIOs is an FIO, and since pseudodifferential operators are FIOs whose underlying symplectomorphism (/canonical relation) is the identity, $\mathcal{F}\circ A \circ \mathcal{F}^{-1}$ should be an FIO for any $A\in \Psi_{\mathrm{sc}}$, whose underlying symplectomorphism is $R^{\pm 1}\circ \operatorname{id} \circ R^{\mp 1} = \operatorname{id}$, i.e.\ a pseudodifferential operator.
\end{remark}

\subsection{Weighted Sobolev spaces}

The ordinary ($L^2$-based) Sobolev spaces $H^s(\mathbb{R}^n) = \Psi^{-s}(\mathbb{R}^n) L^2(\mathbb{R}^n)$ play a fundamental role in the ordinary microlocal analysis of PDEs. The analogous role in scattering theory is played by the (polynomially) weighted Sobolev spaces 
\begin{equation}
    H_{\mathrm{sc}}^{s,r}(\mathbb{R}^n) = L^2 (\mathbb{R}^n) = \langle x \rangle^{-r}  H^s(\mathbb{R}^n) \subset \mathcal{S}'  .
\end{equation}
These come with a natural Hilbertizable topology, namely that coming from $H^s$, such that multiplication by $\langle x \rangle^{-\el}$ defines a continuous map $H^s\to H_{\mathrm{sc}}^{s,r}$.

\begin{remark}
    Warning: different conventions exist for how $H_{\mathrm{sc}}^{\bullet,\bullet}$ should be indexed. We are following the convention that higher orders mean more regularity and more decay.
\end{remark}

In this subsection, we will present a few basic facts about weighted Sobolev spaces. Proofs will only be provided when they differ in some essential way from the corresponding facts about ordinary Sobolev spaces.

\begin{proposition}\label{prop:duality}
    The map $\iota: \mathcal{S} \to (H_{\mathrm{sc}}^{s,r})^*$ given by 
\begin{equation}\label{eq:L2pairing}
\iota(f)(u) = \int u \overline{f}, \quad f \in \mathcal{S}, \quad u \in H_{\mathrm{sc}}^{s,r}
\end{equation}
    extends boundedly to a (conjugate linear) isomorphism $H_{\mathrm{sc}}^{-s,-r} \cong (H_{\mathrm{sc}}^{s,r})^*$.  
\end{proposition}
\begin{proof}[Proof sketch]
    Follows from the duality $(L^2)^* \cong L^2$.
\end{proof}

\begin{remark} This may be confusing: isn't every Hilbert space its own dual, so shouldn't the statement be that $H_{\mathrm{sc}}^{s,r} \cong (H_{\mathrm{sc}}^{s,r})^*$? The confusion arises as to the implied pairing being considered. If we consider the $L^2$ pairing \eqref{eq:L2pairing} then we find that the dual of $H_{\mathrm{sc}}^{s,r}$ is $H_{\mathrm{sc}}^{-s,-r}$ as in the Proposition. If we take the pairing to be the inner product in the Hilbert space $H_{\mathrm{sc}}^{s,r}$, then the dual of $H_{\mathrm{sc}}^{s,r}$ is identified with itself, according to the usual Riesz representation theorem. In PDE, scattering theory, etc, the $L^2$ pairing is usually the more relevant pairing (for example, adjoint operators are defined with respect to the $L^2$ pairing), so this leads to the identification of the dual space as in Proposition~\ref{prop:duality}. 
\end{remark}

\begin{proposition}[Schwartz representation theorem] \hfill

    \begin{itemize}
        \item $\bigcap_{s,r\in \mathbb{R}} H_{\mathrm{sc}}^{s,r} = \mathcal{S}$, 
        \item $\bigcup_{s,r\in \mathbb{R}} H_{\mathrm{sc}}^{s,r} = \mathcal{S}'$.
    \end{itemize}
    \label{prop:Schwartz_rep}
\end{proposition}
\begin{proof}[Proof sketch]
    The first fact follows from Sobolev embedding. The second follows by dualizing the first.
\end{proof}

In the spirit of the previous subsection, let us remark:
\begin{proposition}
    $\mathcal{F} : H_{\mathrm{sc}}^{s,r} \to H_{\mathrm{sc}}^{r,s}$ is a bounded map. 
    \label{prop:F_sob}
\end{proposition}

\begin{example}
    Suppose $M\subseteq \mathbb{R}^n$ is a closed codimension $k$-submanifold of $\mathbb{R}^n$. Then, the Dirac-$\delta$ function $\delta_M$ lies in $H^{s,\el}_{\mathrm{sc}}$ whenever $s<-k/2$. So, $\mathcal{F} \delta_M \in H_{\mathrm{sc}}^{s,r}$ whenever $\el<-k/2$. 

    For example, if $M=\{0\}$, then $\mathcal{F} \delta_M = 1$. We are claiming that $1 \in \langle r \rangle^{\el} H^s$ whenever $\el<-n/2$, which is true.  
\end{example}

\begin{proposition}
    If $A\in \Psi_{\mathrm{sc}}^{m,\el}$, then, for all $s,r\in \mathbb{R}$, $A$ maps 
    \begin{equation}
        A: H_{\mathrm{sc}}^{s,r} \to H_{\mathrm{sc}}^{s-m,r-\el}, 
    \end{equation}
    and does so boundedly.
\end{proposition}
So,  if $m,\el>0$, then applying $A\in \Psi_{\mathrm{sc}}^{m,\el}$ reduces regularity by $m$ orders \emph{and} reduces decay by $\el$ orders. Similar statements apply if $m\leq 0$ or $\el\leq 0$, switching ``reduces'' to ``increases'' where necessary.

The proposition is an easy corollary of the $L^2$-boundedness of ordinary $\Psi$DOs:
\begin{proof}
    We have a commutative diagram 
    \begin{equation}
        \xymatrix{ 
        H^{s,r}_{\mathrm{sc}} \ar[rr]^A \ar[d]_{\times \langle x \rangle^{r} } & & H_{\mathrm{sc}}^{s-m,r-\el} \\ 
        H^{s}  \ar[rr]_{   \langle x \rangle^{r-\el} A\langle x \rangle^{-r} } & &H^{s-m} \ar[u]_{\times \langle x \rangle^{\el-r} } 
        }
    \end{equation}
    in which all of the maps except the top map are known to be bounded.  (The reason we know that $\langle x \rangle^{r-\el} A\langle x \rangle^{-r}$ maps $H^{s}\to H^{s-m}$ boundedly is that, since $\langle x \rangle^b \in \Psi_{\mathrm{sc}}^{0,b}$ for any $b$, 
    \begin{equation}
            \langle x \rangle^{r-\el} A\langle x \rangle^{-r} \in \Psi_{\mathrm{sc}}^{0,r-\el} \Psi_{\mathrm{sc}}^{m,\el} \Psi_{\mathrm{sc}}^{0,-r} \subseteq \Psi_{\mathrm{sc}}^{m,0} \subset \Psi^m,  
    \end{equation}
    and we already know, as part of the theory of standard (H\"ormander-class) pseudodifferential operators, that  elements of $\Psi^m$ map $H^{s}\to H^{s-m}$  boundedly.)

    So, $A$ maps $H_{\mathrm{sc}}^{s,r} \to H_{\mathrm{sc}}^{s-m,r-\el}$, and does so boundedly.
\end{proof}

\begin{remark}
    It is sometimes useful to note that $\operatorname{Op}(a)$, viewed as an element of $\mathcal{L}(H_{\mathrm{sc}}^{s,r},H_{\mathrm{sc}}^{s-m,r-\el})$, depends continuously on $a$ in the sense that if $a_1,a_2,\dots \in S_{\mathrm{sc}}^{m,\el}$ is some sequence such that $a_n\to a$ in $S_{\mathrm{sc}}^{m,\el}$, then 
    \[\operatorname{Op}(a_n)\to \operatorname{Op}(a)\] in $\mathcal{L}(H_{\mathrm{sc}}^{s,r},H_{\mathrm{sc}}^{s-m,r-\el})$.
\end{remark}

\subsection{The Rellich theorem}

Recall the \emph{Rellich compactness theorem}, which states that if $X$ is a closed manifold (i.e. compact with no boundary), then the inclusion $H^{s'}(X)\hookrightarrow H^{s}(X)$ is compact whenever $s<s'$. 

Importantly, this fails when $X=\mathbb{R}^n$. Indeed, fix $u\in C_{\mathrm{c}}^\infty(\mathbb{R}^n)$, not identically zero. Then, for any nonzero $\mathbf{k}\in \mathbb{R}^n$, the translates $u_n = u(\bullet +n\mathbf{k} )$ converge weakly to $0$ in $H^{s'}$ for every $s \in \mathbb{R}$ as $n \to \infty$. (Why?) But, $\lVert u_n \rVert_{H^s} \neq 0$ is independent of $n$, and therefore not converging to $0$. So, 
$H^{s'}(\mathbb{R}^n)\hookrightarrow H^{s}(\mathbb{R}^n )$
is never compact.
Instead, the ``correct'' analogue of the Rellich compactness theorem on $\mathbb{R}^n$ is:
\begin{proposition}\label{prop:comp emb}
    For any $s,s',r,r'$, with $s<s'$ and $r<r'$, the embedding $H_{\mathrm{sc}}^{s',r'} \hookrightarrow H_{\mathrm{sc}}^{s,r}$ is compact.
\end{proposition}

\begin{proof}[Proof Idea]
    The full result follows easily from the $s,r=0$ case.
    
    So, we want to show that if $\varepsilon>0$, then the inclusion $H_{\mathrm{sc}}^{\varepsilon,\varepsilon} \hookrightarrow L^2$ is compact. In other words, we want to show that if $u_1,u_2,\dots$ is a sequence of elements of $H_{\mathrm{sc}}^{\varepsilon,\varepsilon}$ converging weakly to $0$ in this space, then $\lVert u_n \rVert_{L^2}\to 0$.

    Via the Rellich compactness of the inclusion $H^{\varepsilon} \hookrightarrow L^2$ on \emph{compact manifolds} (say, the torus $\mathbb{R}^n / \Lambda \mathbb{Z}^n$ for $\Lambda$ large) (the Rellich compactness theorem), it must be the case that $\lVert \chi u_n \rVert_{L^2}\to 0$ for any $\chi \in C_{\mathrm{c}}^\infty(\mathbb{R}^n)$. This means that the $L^2$-mass of the $u_n$ is leaving every compact subset. But we want to say that the $L^2$-mass of the $u_n$ is going to zero. So, what we need to rule out is that the $L^2$ mass of the $u_n$ ``escapes'' to spatial infinity without decaying. Intuitively, it makes sense that this is ruled out by weak convergence in $H_{\mathrm{sc}}^{\varepsilon,\varepsilon}$; were $u_n$ to escape,  it would be expected to be possible to construct an adversarial $v \in H_{\mathrm{sc}}^{-\varepsilon,-\varepsilon}  = (H_{\mathrm{sc}}^{\varepsilon,\varepsilon})^*$ such that $\langle v, u_n \rangle \not\to 0$.  
    
    Rather than argue along these lines, it is somewhat easier to use an alternative characterization of compact maps between Hilbertizable spaces, namely that the compact maps are those in the closure (under the operator norm) of the set of finite-rank operators.  
\end{proof}
\begin{proof}
We will find a finite rank operator $F_R$ that approximates $1: H_{\mathrm{sc}}^{\epsilon,\epsilon} \to L^2$.
    For each $R>0$, consider the multiplication operator $M_R=\chi(x/R)$, where $\chi \in C_c^\infty(\R^n)$ is identically $1$ on the unit ball. By the Rellich compactness theorem of compact manifolds, there exists a finite-rank operator $F_R:H_{\mathrm{sc}}^{\varepsilon,\varepsilon} \to L^2$ such that $\lVert M_R - F_R \rVert_{ H_{\mathrm{sc}}^{\varepsilon,\varepsilon} \to L^2}\leq 1/R$. Now, 
    \begin{equation}
        \lVert 1 - F_R \rVert_{H_{\mathrm{sc}}^{\varepsilon,\varepsilon} \to L^2} \leq   \lVert 1 - M_R \rVert_{H_{\mathrm{sc}}^{\varepsilon,\varepsilon} \to L^2}+ \lVert M_R - F_R \rVert_{H_{\mathrm{sc}}^{\varepsilon,\varepsilon} \to L^2} \leq \lVert 1 - M_R \rVert_{H_{\mathrm{sc}}^{\varepsilon,\varepsilon} \to L^2} + 1/R .
    \end{equation}
    But, $\chi(\bullet/R)$ converges, as $R\to\infty$, to $1$ in $S_{\mathrm{sc}}^{\varepsilon,\varepsilon}$, for every $\varepsilon>0$ --- see Problem~\ref{prob:density symbols}. It follows that $\lVert 1 - M_R \rVert_{H_{\mathrm{sc}}^{\varepsilon,\varepsilon} \to L^2} \to 0$ as $R\to\infty$.
    Alternatively, it suffices to note that $\chi(\bullet/R)$ converges to $1$, as $R\to\infty$,  as a map $H_{\mathrm{sc}}^{0,\varepsilon} = \ang{x}^{-\epsilon} L^2 \to L^2$, for every $\varepsilon>0$. In fact, this only requires observing that $\ang{x}^{-\epsilon}(\chi(\cdot/R) - 1)$ tends to zero in $L^\infty$, which is obvious. 
\end{proof}

Suppose that $A,A'\in \Psi_{\mathrm{sc}}^{0,0}$ have the same principal symbol. Then, it follows that 
\begin{equation}
    A-A' \in \Psi_{\mathrm{sc}}^{-1,-1}. 
\end{equation}
So, the difference $K=A-A'$ is compact acting on any fixed sc-Sobolev space. \emph{In the sc-calculus, principal symbols capture operators modulo compact errors}. 
This is the single most important property that the sc-calculus has that H\"ormander's $\Psi$ does not --- it is why $\Psi_{\mathrm{sc}}$ is as useful as it is when doing scattering theory.

\subsection{A quick review of abstract Fredholm theory}

It is a fact of life that many PDEs one cares about are not well-posed. For example, if $\triangle_g$ denotes the Laplace--Beltrami operator on a closed Riemannian manifold $(M,g)$, then, for $\lambda\in \mathbb{R}$, the problem 
\begin{equation} 
\begin{cases}
u\in H^2(M), \\ 
\triangle_g u - \lambda u = f \in L^2
\end{cases}  
\end{equation}
is only well-posed if $\lambda$ is not an eigenvalue of $\triangle_g$. 

After well-posedness, Fredholmness is the next best thing.
Recall that if $\mathcal{X},\mathcal{Y}$ are two Banach spaces and $P:\mathcal{X}\to \mathcal{Y}$ is a bounded linear map, then $P$ is said to be \emph{Fredholm} if the following three conditions are all satisfied: 
\begin{enumerate}[label=(\roman*)]
    \item $P$ has closed range, 
    \item $P$ has finite-dimensional null space,
    \item $P$ has finite-dimensional cokernel $\mathcal{Y}/ P \mathcal{X}\cong (P\mathcal{X})^\perp$. 
\end{enumerate}

Roughly, being Fredholm means being invertible modulo a finite-dimensional obstruction.

\begin{example}
If $(M,g)$ is a closed Riemannian manifold, then, for any first-order differential operator $Q$ with smooth coefficients, the operator $P=\triangle_g + Q$ is Fredholm as a map $H^s(M) \to  H^{s-2}(M)$, for any $s\in \mathbb{R}$. Thus, every eigenvalue of $\triangle_g + Q$ has finite multiplicity.
\end{example} 

\begin{example}
    On $\mathbb{R}^n$, if 
    \[ 
        \triangle = \sum_{j=1}^n D_{x_j}^2 = -\sum_{j=1}^n \partial_{x_j}^2
    \]
    denotes the \emph{positive semi-definite} Laplacian and $\lambda>0$, then $\triangle+\lambda$ is Fredholm as a map $H^{s}(M) \to  H^{s-2}(M)$, for any $s\in \mathbb{R}$. This is easily seen by passing to Fourier transforms. 

    This is not true if $\lambda\leq 0$. Indeed, then $P=\triangle+\lambda$, considered as a map $H^2\to L^2$, has range 
    \[
        \operatorname{range} P = \mathcal{F} \{u\in L^2 : (\xi^2 - \lambda)^{-1} u \in L^2(\mathbb{R}^n) \} 
    \]
    which is dense in $L^2$ but not all of $L^2$, and therefore not closed.
\end{example}

\textbf{Moral.} Unlike on closed manifolds, Fredholmness on non-compact spaces can depend on seemingly lower-order operators. 
The reason for ``seemingly'' in the previous sentence is that $\triangle \in \Psi_{\mathrm{sc}}^{2,0}$ and $+\lambda \in \Psi_{\mathrm{sc}}^{0,0}$; so in terms of differential order, it is true that $\triangle$ is higher-order, but in terms of \emph{decay order}, they are both $0$th order.

Let $P:\mathcal{X}\to \mathcal{Y}$ denote a bounded linear map. A \emph{semi-Fredholm} estimate is one of the form 
        \begin{align}\label{eq:semiFred}
            \lVert u \rVert_{\mathcal{X}} &\leq C( \lVert Pu \rVert_{\mathcal{Y}} + \lVert  \iota u \rVert_{\mathcal{Z}})
        \end{align}
for $\iota: \mathcal{X}\to \mathcal{Z}$ a compact injection from $\mathcal{X}$ to some other Banach space $\mathcal{Z}$, and $C>0$ a constant (independent of $u$). This essentially means that we can control $u$ in terms of $Pu$ and some \emph{weak} norm $\lVert \iota u \rVert_{\mathcal{Z}}$ of $u$. 
\begin{remark}
    If $P$ had no null space, then we would have an estimate $\lVert u \rVert_{\mathcal{X}} \leq C \lVert Pu \rVert_{\mathcal{Y}}$. So, a semi-Fredholm estimate is almost as good, except we have to allow a ``small'' error on the right-hand side to accomodate the fact that $P$ may have null space.
\end{remark}

The following fact about Fredholm operators is ``standard.''
\begin{proposition}
    Fix Banach spaces $\mathcal{X},\mathcal{Y}$. Then, the following are equivalent: 
    \begin{enumerate}[label=(\roman*)]
        \item  $P$ is Fredholm. 
        \item $P$ is invertible modulo a compact error. That is, there exists a bounded linear operator $Q:\mathcal{Y}\to \mathcal{X}$ such that 
        \begin{equation}\label{eq:parametrix}
        Q P - \operatorname{id}_{\mathcal{X}}, PQ - \operatorname{id}_{\mathcal{Y}} 
        \end{equation}
        are compact. 
        \item (Semi-Fredholm estimates.) $P$ and $P^*:\mathcal{Y}^* \to \mathcal{X}^*$ both satisfy a semi-Fredholm estimate.
    \end{enumerate}
    \label{prop:Fred}
\end{proposition}
\begin{proof}
    The equivalence of (i), (ii) really is a standard fact, explained in many functional analysis texts, so we will focus on the equivalence of (iii) with the other two.
    \begin{itemize}
        \item (iii) $\Rightarrow$ (i): Consider the semi-Fredholm estimate \eqref{eq:semiFred} restricted to the kernel of $P$. This gives us 
        $$
        \lVert u \rVert_{\mathcal{X}} \leq C\lVert  \iota u \rVert_{\mathcal{Z}}, \quad u \in \ker P. 
        $$
        Thus $\iota$ has a bounded left inverse. As $\iota$ is compact, this is only possible if the kernel of $P$ is finite-dimensional. 

        To see that the range of $P$ is closed: suppose that $u_1,u_2,u_3,\dots$ is an infinite sequence of elements of $\mathcal{X}$ such that $Pu_j \to f$. Suppose we knew that the $u_j$ (or a subsequence thereof) stayed within some fixed ball in $\mathcal{X}$. Then, by compactness of $\iota$, we can pass to a subsequence (which we still denote $(u_j)$) for which $\iota u_j$ converges in $Z$. The semi-Fredholm estimate shows that this subsequence is Cauchy and hence converges in $\mathcal{X}$, say to $u_\infty$. This implies that $P u_j$ converges to $P u_\infty$. Thus, $f = P u_\infty$ is in the range of $P$, so the range is closed. 

        Of course, it need not be the case that $u_j$ stay bounded in $\mathcal{X}$, since we always could add an arbitrarily large element of $\ker P$ to $u_j$ without changing its image under $P$. To fix this we can choose $u_j$ to lie in some complementary subspace $W$ to the kernel of $P$ (which exists as the kernel is finite-dimensional). Now suppose that still $\lVert u_j \rVert \to \infty$. Let $\hat{u}_j = u_j / \lVert u_j \rVert$. 
        Since these stay bounded in $\mathcal{X}$, we can apply the argument above to $\hat u_j$. Thus, after passing to a subsequence, we may arrange that $\hat u_j$ converges in $\mathcal{X}$, say to $v$. Since $\| \hat u_j \| = 1$ for all $j$, we have $\| v \| = 1$, so in particular, $v \neq 0$. Since each $\hat u_j$ lies in the closed subspace $W$, we have $v \in W$. And since $P u_j \to f$, and $\| u_j \| \to \infty$, we have $P v = \lim P \hat u_j = \lim P u_j / \| u_j \| = 0$. So now we have $v \neq 0$ and is in both $W$ and the kernel of $P$. This is a contradiction, that shows that the sequence $u_j$ must remain in some fixed ball, so the earlier argument applies and shows that the range of $P$ is closed.

        From this, it follows that the cokernel is isomorphic to $\ker P^*$ (see \Cref{prob:kerPadj}). The argument above shows that the finite-dimensionality of $\ker P^*$ follows from the semi-Fredholm estimate for $P^*$.

        \item (ii) $\Rightarrow$ (iii). Let $QP-\operatorname{id}_{\mathcal{X}} = K$, which we assume is a compact operator on $\mathcal{X}$. Then, 
        \begin{equation}
            \lVert u \rVert_{\mathcal{X}} \leq  \lVert QP u \rVert_{\mathcal{X}}+ \lVert K u \rVert_{\mathcal{X}} \lesssim \lVert Pu \rVert_{\mathcal{Y}} + \lVert K u \rVert_{\mathcal{X}}.
        \end{equation}
        This is almost itself a semi-Fredholm estimate, except that $K$ is not injective. To remedy this (although the issue is actually inconsequential, as a semi-Fredholm-like estimate without the injectivity condition on $\iota$ is morally \emph{stronger} than with the injectivity condition!)
        we let $\mathcal{Z} := \mathcal{X} \oplus \mathcal{X}$, choose a compact injection $\tilde K : \mathcal{X} \to \mathcal{X}$, and define $\iota(x) = (Kx, \tilde K x)$. Then $\iota$ is a compact injection and we have 
        \begin{equation}
            \lVert u \rVert_{\mathcal{X}}\lesssim \lVert Pu \rVert_{\mathcal{Y}} + \lVert \iota u \rVert_{\mathcal{Z}}. 
        \end{equation}
        So, we get a semi-Fredholm estimate for $P$. 

        Applying the same argument to $P^*$, we conclude that it satisfies a semi-Fredholm estimate.
    \end{itemize}
    
\end{proof}
The operator $Q$ in \eqref{eq:parametrix} is called a \emph{parametrix} for $P$. 

\subsection{Elliptic Operators are Fredholm}

Now the main result of this lecture: 
\begin{proposition}
    If $A\in \Psi_{\mathrm{sc}}^{m,\el}$ is elliptic, then it is Fredholm as a map $H^{s,r}_{\mathrm{sc}} \to H_{\mathrm{sc}}^{s-m,r-\el}$. 
\end{proposition}

Before the proof, we remind the reader that `elliptic' is meant in the scattering sense, which is a stronger condition than ellipticity in the H\"ormander calculus (see Remark~\ref{rem:ellipticity}).

\begin{proof}
    Follows immediately from the elliptic parametrix construction, the characterization \Cref{prop:Fred} of Fredholmness, and the compactness of the embedding $H_{\mathrm{sc}}^{-1,-1} \hookrightarrow L^2(\mathbb{R}^n)$
\end{proof}

Consequently: 

\begin{proposition}
    If $A\in \Psi_{\mathrm{sc}}^{m,\el}$ is elliptic, then $\ker_{\mathcal{S}'} A$ is finite-dimensional. 
\end{proposition}
\begin{proof}
    By the elliptic parametrix construction, $\ker_{\mathcal{S}'} A$ consists entirely of Schwartz functions. Thus, they lie in the kernel of $A$ restricted to any individual sc-Sobolev space, say $L^2$. By the the previous proposition, the Fredholmness of $A$ acting on $L^2$ implies that the kernel is finite-dimensional.
\end{proof}

\begin{corollary}
    Consider $P=\triangle + V$, $V\in S^{-\varepsilon}$. For each $E<0$, there are at most finitely many bound states (eigenfunctions) with that energy (eigenvalue).
\end{corollary}

\subsection{Problems}

\begin{problem}
    Prove \Cref{prop:Schwartz_rep}.
\end{problem}

\begin{problem}
    Prove \Cref{prop:F_sob}.
\end{problem}

\begin{problem}
    \begin{enumerate}
        \item 
    Show that if $P:\mathcal{X}\to \mathcal{Y}$ is a bounded linear map between Hilbert spaces $\mathcal{X},\mathcal{Y}$ with closed range, then $\mathcal{Y}/ P \mathcal{X}\cong \operatorname{ker} P^*$. 
    \item Is this true if we do not assume that $P$ has closed range?
    \end{enumerate}
    \label{prob:kerPadj}
\end{problem}

\begin{problem}
    Consider $u(x) = e^x \sin(e^x)$. 
    \begin{enumerate}
        \item Show that $u\in H_{\mathrm{sc}}^{-1,r}(\mathbb{R})$ for $r$ sufficiently negative. 
        \item (Optional.) For which  $H_{\mathrm{sc}}^{s,r}(\mathbb{R})$ is $u$ in? ($s$ is an integer.)
    \end{enumerate}
    
\end{problem}

\begin{problem}
    Consider a repulsive Coulomb-like potential $V$, 
    \[
        V = \frac{1}{\langle r \rangle}  + S^{-1-\varepsilon}(\overline{\mathbb{R}^n} ).  
    \] 
    Show that $\operatorname{ker}_{\mathcal{S}'} (\triangle + V) $ is finite-dimensional. Hint: a coordinate change might be useful.  
\end{problem}

\section{Lecture 4: Microlocalization}

Recall that the \emph{wavefront set} $\operatorname{WF}(u)$ of a distribution $u\in \mathcal{D}'(\mathbb{R}^n)$ is a subset of \footnote{Usually, one defines wavefront set to be a fiberwise conic subset of $T^* \mathbb{R}^n = \mathbb{R}^{n}\times \mathbb{R}^n$. The compactified perspective is that, instead of taking the wavefront set to consist of lines in the cotangent bundle, to only take their ``endpoints'' at fiber infinity.} 
\begin{equation} 
\underbrace{\mathbb{R}^n}_{\text{base}} \times \underbrace{\partial \overline{\mathbb{R}^n}}_{\mathbb{S}^{n-1}\text{ fiber}} = \mathbb{S}^* \mathbb{R}^n 
\end{equation} 
obstructing smoothness --- the wavefront set over a point $x\in \mathbb{R}^n$ in the base is a subset of the fiber $\mathbb{S}^{n-1}$ over $x$ describing the directions where the germ of $u$ at $x$ fails to be smooth. A distribution is smooth if and only if it has empty wavefront set.

In this lecture, we talk about the \emph{scattering} wavefront set. This assigns to each tempered distribution $u\in \mathcal{S}'$ a subset 
\begin{equation} 
    \operatorname{WF}_{\mathrm{sc}}(u)\subset \partial (\overline{\mathbb{R}^n}\times  \overline{\mathbb{R}^n})
\end{equation} 
of the ``square boundary'' $\partial (\overline{\mathbb{R}^n}\times  \overline{\mathbb{R}^n})$
which obstructs $u$ being Schwartz. Over the ``interior'' $\mathbb{R}^n\subset \overline{\mathbb{R}^n}$, this is just the ordinary wavefront set $\operatorname{WF}(u)$. The novel thing about the sc-wavefront $\operatorname{WF}_{\mathrm{sc}}(u)$ is that it contains, in addition to the ordinary wavefront set, a set of points over spatial infinity. 
Namely, over $\infty \theta$, $\theta\in \mathbb{S}^{n-1}$\footnote{Here $\infty\theta$ means the point on $\partial \overline{\R^n}$ on the direction $\theta$.}, the sc-wavefront set contains the \emph{frequencies} at which $u$, in that direction, fails to be Schwartz. So, while the ordinary wavefront set is detecting a failure to be smooth, the sc-wavefront set is also detecting failure to decay.

Obviously, we need to make this mathematically precise. We will do this in the next subsection. Afterwards, we will discuss examples. 
\subsection{The wavefront set of a distribution}

In this section, if $S$ is a subset of $\partial( \overline{\mathbb{R}^n}\times  \overline{\mathbb{R}^n})$, then we use $S^\complement$ to denote the complement of $S$ within this set: 
\begin{equation}
    S^\complement = (\partial( \overline{\mathbb{R}^n}\times \overline{\mathbb{R}^n}))\backslash  S.
\end{equation}

The sc-wavefront set of a tempered distribution $u$ is defined by 
\begin{equation}\label{eq:WF defn}
    \operatorname{WF}_{\mathrm{sc}}(u) = \bigcap_{A\in \Psi_{\mathrm{sc}}^{0,0}\text{ s.t. }Au\in \mathcal{S} }  \operatorname{char}_{\mathrm{sc}}(A) = \bigg( \bigcup_{A\in \Psi_{\mathrm{sc}}^{0,0}\text{ s.t. }Au\in \mathcal{S} }  \Ell(A)  \bigg)^\complement .
\end{equation}
Since this is an intersection of closed sets, it is closed.

\begin{proposition}
    \begin{itemize}
    \item The portion of $\operatorname{WF}_{\mathrm{sc}}(u)$ not over base infinity is just the ordinary wavefront set $\operatorname{WF}(u)$. 
    \item  If $u$ is compactly supported, then its sc-wavefront set is the same as its ordinary wavefront set. 
    \end{itemize} 
\end{proposition}
\begin{proof}
    Exercise.
\end{proof}

\begin{proposition}
    $u\in \mathcal{S} \iff \operatorname{WF}_{\mathrm{sc}}(u) = \varnothing$.
    \label{prop:WF_sc_basic}
\end{proposition}
\begin{proof}
    \begin{itemize}
       \item ($\Rightarrow$): take $A=1$. 
        \item ($\Leftarrow$): if $\operatorname{WF}_{\mathrm{sc}}(u) = \varnothing$, then, for each point $p$ in $\partial (\overline{\mathbb{R}}^n \times \overline{\mathbb{R}^n})$, there exists an $A_p \in \Psi_{\mathrm{sc}}^{0,0}$ such that $\Ell(A_p)\ni p$ and $A_p u$ is Schwartz. Because the sets $\Ell(A_p)$ are all open, and because $\partial (\overline{\mathbb{R}}^n \times \overline{\mathbb{R}^n})$ is compact, we can choose a finite number $J\in \mathbb{N}^+$ of points $p_j$ such that the sets $\Ell(A_{p_j})$ form an open cover of $\partial (\overline{\mathbb{R}^n} \times \overline{\mathbb{R}^n})$. Now consider 
        \begin{equation}
            A = \sum_{j=1}^J A_{p_j}^* A_{p_j} \in \Psi_{\mathrm{sc}}^{0,0}. 
        \end{equation}
        This satisfies $A u \in \mathcal{S}$.        
        The principal symbol of $A$ is 
        \[
            a=\sum_{j=1}^J |a_{p_j}|^2,
        \]
        where the $a_{p_j}$'s are the principal symbols of the $A_{p_j}$'s. Because the sets $\Ell(A_{p_j})$ form an open cover of $\partial (\overline{\mathbb{R}^n} \times \overline{\mathbb{R}^n})$, $a$ is totally elliptic. So, it follows that
        \begin{equation}
            Au \in \mathcal{S} \Rightarrow u\in \mathcal{S}. 
        \end{equation}
    \end{itemize}
\end{proof}
A local version of this is:
\begin{proposition}\label{prop:WF absence}
    Let $\theta\in \mathbb{S}^{n-1}$ be a spatial direction. Then, $\operatorname{WF}_{\mathrm{sc}}(u)$ is disjoint from the fiber over $\infty \theta$ if and only if there exists some $\chi \in C^\infty(\mathbb{S}^{n-1})$ identically $=1$ near $\theta$ and some $\psi\in C_{\mathrm{c}}^\infty(\mathbb{R})$ identically $=1$ near the origin such that the product $\chi(r^{-1} x )\psi(1/r) u$ is Schwartz.
\end{proposition}

\subsection{Operator wavefront set}
\label{subsec:op-WF}
Given a symbol $a \in S^{m,l}_{\sca}(\R^n \times \R^n)$, we define its \emph{essential support}, a subset of $\partial \comphase$,  as follows: the complement of the essential support (let us call this the \emph{residual set}, NB not a standard term) is the subset of $\partial \comphase$ consisting of points $q$ such that $a$ is a residual symbol in a neighbourhood of $q$. More precisely, this condition is the existence of $\chi \in C^\infty(\comphase)$ such that $\chi(q) \neq 0$ and $\chi a$ is a residual symbol. By definition, the residual set is an open subset of $\partial \comphase$, and therefore the essential support is a closed subset of $\partial \comphase$. 

Then, given $A \in \Psisc^{m,l}$, we define the \emph{operator wavefront set}, $\WF'(A)$, also known as its \emph{microlocal support} or \emph{essential support}, to be the essential support of its (full) symbol. We emphasize that the operator wavefront set depends on the full symbol, not just the principal symbol of $A$.  Thus we now have two decompositions of $\partial \comphase$ depending on a scattering pseudodifferential operator $A$: it is the disjoint union of the elliptic set and the characteristic set, and the disjoint union of the operator wavefront set and the residual set (of its full symbol). See Problem~\ref{prob:op WF set} for some simple properties of these two decompositions.

Proposition~\ref{prop:WF absence} can be generalized as follows, using the operator wavefront set: 

\begin{proposition}
    Suppose that $A\in \Psi_{\mathrm{sc}}$ has $\operatorname{WF}'_{\mathrm{sc}}(A)$  disjoint from $\operatorname{WF}_{\mathrm{sc}}(u)$. Then $Au$ is Schwartz.
    \label{prop:cones}
\end{proposition}
\begin{proof}
    See \Cref{ex:cones}.
\end{proof}

\begin{proposition}[Microlocality] \label{prop:microlocality}
        If $A\in \Psi_{\mathrm{sc}}$ and $u\in \mathcal{S}'$, then  $\operatorname{WF}_{\mathrm{sc}}(Au) \subseteq \operatorname{WF}_{\mathrm{sc}}(u)  \cap \operatorname{WF}'_{\mathrm{sc}}(A)$.
\end{proposition}
\begin{proof}
    \begin{itemize}
        \item First, we show that  $\operatorname{WF}_{\mathrm{sc}}(Au) \subseteq\operatorname{WF}'_{\mathrm{sc}}(A)$, i.e.\ $\operatorname{WF}_{\mathrm{sc}}(Au)^\complement \supseteq \operatorname{WF}'_{\mathrm{sc}}(A)^\complement$. Suppose that $p \in \operatorname{WF}'_{\mathrm{sc}}(A)^\complement$. Then, $\exists B\in \Psi_{\mathrm{sc}}^{0,0}$ elliptic at $p$ but whose operator wavefront set is disjoint from that of $A$. Then, $BA\in\Psi_{\mathrm{sc}}^{-\infty,-\infty}$. Thus $BAu\in \mathcal{S}$, so the elliptic set of $B$, including $p$, is contained in  $\operatorname{WF}_{\mathrm{sc}}(Au)^\complement$. 
        \item Second, we show that $\operatorname{WF}_{\mathrm{sc}}(Au) \subseteq\operatorname{WF}_{\mathrm{sc}}(a)$, i.e.\  $\operatorname{WF}_{\mathrm{sc}}(Au)^\complement \supseteq \operatorname{WF}_{\mathrm{sc}}(u)^\complement$. If $p\in \operatorname{WF}_{\mathrm{sc}}(u)^\complement$, then there exists a $B\in \Psi_{\mathrm{sc}}^{0,0}$ elliptic at $p$ but with $\operatorname{WF}'_{\mathrm{sc}}(B)$ disjoint from $\operatorname{WF}_{\mathrm{sc}}(u)$. It follows that $\operatorname{WF}'_{\mathrm{sc}}(BA)$ is disjoint from $\operatorname{WF}_{\mathrm{sc}}(u)$. So, \Cref{prop:cones} implies that $BAu \in \mathcal{S}$. It follows that $p\in \operatorname{WF}_{\mathrm{sc}}(Au)^\complement$. 
    \end{itemize}
\end{proof}

Just as differential operators do not spread supports, pseudodifferential operators do not spread singular supports or wavefront set.

The following proposition will be useful in reducing computations of sc-wavefront set to computations of ordinary wavefront set:
\begin{proposition}
    Except at the corner of the square $\partial \overline{\mathbb{R}^n}\times \partial \overline{\mathbb{R}^n}$, $\operatorname{WF}_{\mathrm{sc}}(u)$ consists of the union of the ordinary wavefront set $\operatorname{WF}(u)$ set of $u$ and the rotation by $90^\circ$ of the ordinary wavefront set of $\mathcal{F} u$. 
    \label{prop:Fourier_WF}
\end{proposition}
By rotation by $90^\circ$ we mean the map \eqref{eq:rotation}.
\begin{proof}
    Follows from \Cref{prop:Fourier}.
\end{proof}

Similarly, if we want to measure $u$'s failure to lie in some sc-Sobolev space $H_{\mathrm{sc}}^{s,r}$, then we consider
\begin{equation}
    \operatorname{WF}_{\mathrm{sc}}^{s,\el}(u) = \bigcap_{A\in \Psi_{\mathrm{sc}}^{0,0}\text{ s.t. }Au\in H^{s,\el}_{\mathrm{sc}} }  \operatorname{char}_{\mathrm{sc}}(A) .
\end{equation}

\begin{proposition}
    $\operatorname{WF}_{\mathrm{sc}}(u) = \overline{\bigcup_{s,\el\in \mathbb{R}}  \operatorname{WF}_{\mathrm{sc}}^{s,\el}(u)}$. 
\end{proposition}
\begin{remark}
    The closure here is necessary! (Exercise: Why?)
\end{remark}

The propositions above apply, mutatis mutandis, to the $\operatorname{WF}_{\mathrm{sc}}^{s,\el}$.
For example, microlocality reads:
\begin{proposition}
    If $A\in \Psi_{\mathrm{sc}}^{m,\el}$ and $u\in \mathcal{S}'$, then 
    $\operatorname{WF}_{\mathrm{sc}}^{s,r}(Au) \subseteq \operatorname{WF}_{\mathrm{sc}}^{s+m,r+\el}(u)  \cap \operatorname{WF}'_{\mathrm{sc}}(A)$.
\end{proposition}

\subsection{One-dimensional examples}

\begin{example} Fix $\sigma>0$. 
    Consider $u\in \mathcal{S}'(\mathbb{R})$ defined by $u(x)= e^{i \sigma x}$. Then, $\operatorname{WF}_{\mathrm{sc}}(u)$ consists of exactly two points, one over each of the two points in $\infty \mathbb{S}^0$.  
\end{example}
\begin{proof}
    The Fourier transform of $u$ is a Dirac $\delta$-function over the single point $\sigma\in \mathbb{R}$. The ordinary wavefront set of such a $\delta$-function is the whole cosphere $\mathbb{S}^*_\sigma \mathbb{R}$ over the point where the $\delta$-function is located. The sc-wavefront set must be the same --- we cannot have any sc-wavefront set at base infinity --- since the $\delta$-function is compactly supported. So, the claim follows from \Cref{prop:Fourier_WF}.
\end{proof}

This is consistent with the intuition that $\operatorname{WF}_{\mathrm{sc}}$ is measuring which frequencies fail to decay.

\begin{example}[Wavefront set at the corner]
        How do we interpret sc-wavefront set at the corner 
        \begin{equation}
            \partial \overline{\mathbb{R}^n} \times \partial \overline{\mathbb{R}^n} \subseteq \partial (\overline{\mathbb{R}^n} \times \overline{\mathbb{R}^n})? 
        \end{equation}
        Let us get some intuition in the $n=1$ case.
        \begin{itemize}
            \item Consider the Dirac comb $u(x)= \sum_{k\in \mathbb{Z} }\delta(x-k)$. Since the singular support of $u$ is $\mathbb{Z}$, and since the sc-wavefront set is a closed set, the sc-wavefront set of $u$ must contain points in the corner.  

            This applies also to 
            \[ u(x)= \sum_{k\in \mathbb{Z} }f(x) \delta(x-k)
            \] 
            for $f$ 
            Schwartz.
            \item Consider $u(x) =e^{ix^2}$. This tempered distribution is not Schwartz, so it has to have sc-wavefront set \emph{somewhere}. But it is smooth, so it has no ordinary wavefront set. Moreover, it's Fourier transform has the same form (up to rescaling) and therefore has no ordinary wavefront set either. 
            
            Therefore, by \Cref{prop:Fourier_WF}, the wavefront set of $u$ must be entirely at the corner. 
        \end{itemize}
\end{example}

Intuition: sc-wavefront set at the corner corresponds to oscillations with infinite frequency which fail to decay.

\subsection{Higher-dimensional examples: it's about the cones}

\begin{example}
    A \emph{plane wave} has the form $u(x)=e^{ik\cdot x}$ for $k\in \mathbb{R}^n$. \textit{Claim:} the sc-wavefront set of $u$ consists of a single point over each point at spatial infinity. 
\end{example}
\begin{proof}
    Same as in the one-dimensional case.
\end{proof}

Ignoring fiber infinity, we can think of the portion of $\partial( \overline{\mathbb{R}^n}\times \overline{\mathbb{R}^n} )$ over $\infty\theta$, $\theta \in \mathbb{S}^{n-1}$, as consisting of one point for each plane wave.

\begin{example}
    A \emph{spherical wave} has the form $u(x) = e^{i\sigma \langle r \rangle }$ for $\sigma \in \mathbb{R}$. The sc-wavefront set over $\infty \theta$ consists of a single point, that in the sc-wavefront set of the plane wave $e^{i k\cdot x}$ for $k=\sigma \theta$. 

    \textbf{Q.} Fix a plane wave $e^{ik\cdot x}$. At how many points do $\operatorname{WF}_{\mathrm{sc}}(e^{ik\cdot x})$, $\operatorname{WF}_{\mathrm{sc}}(e^{i \sigma \langle r \rangle } )$ intersect?
    \textbf{A.} $0$ if $\sigma\neq \lVert k \rVert$, $1$ otherwise.
\end{example}

\begin{example}
    Let $n=2$ (for notational simplicity). 
    Consider a ``beam'' of the form $u(x) = \chi(x_2 )e^{i\sigma x_1}$ for nonnegative $\chi \in C_{\mathrm{c}}^\infty(\mathbb{R})$ not identically zero.  

    \textbf{Q.} What is the sc-wavefront set of $u$?
    \textbf{A.} The sc-wavefront set is entirely over the forward/backwards directions $\pm \infty \mathbf{e}_1$, $\mathbf{e}_1=(1,0)$, but, rather than consisting of a single point over each, it consists of (the closure of) all of the wavevectors $k\in \mathbb{R}^2$ of the form $k=(\sigma,\eta)$, $\eta\in \mathbb{R}$. 
\end{example}
\begin{proof}
    
    First of all, $u$ is smooth --- it has no ordinary wavefront set. Its sc-wavefront set therefore sits entirely over base infinity. Some of this wavefront set could be in the corner. Let's ignore this (but see \Cref{ex:beam_helper}). The question we are then asking is about the sc-wavefront set of $u$ over base infinity at finite frequency. This is equivalent to asking about the ordinary wavefront set of the Fourier transform $\hat{u}$.

    The Fourier transform $\hat{u}$ is $\hat{\chi}(\xi_2) \delta(\xi_1 - \sigma)$. Because $\chi$ is compactly supported, the support of $\hat{\chi}$ must be the whole real line. (This is because the Fourier transform of any $C_{\mathrm{c}}^\infty$ function on the real line extends to an entire, nonzero function on the complex plane, so cannot vanish on any nonempty open sets.) So, the ordinary wavefront set of $\hat{u}$ is 
    \begin{equation}
        \operatorname{WF}(\hat{u}) = N^* \{\xi_1 = \sigma \}.
    \end{equation}
    That is, it consists of the entire conormal bundle of the line  $\{\xi_1 = \sigma \}$. Rotating this by $90^\circ$, we conclude the claim.
\end{proof}

\textbf{Moral:} We cannot determine the plane waves out of which a distribution is built if we restrict attention to a rectangular prism $R=\{\lVert y \rVert<C\}$, because, if we only have access to $u|_R$, then we cannot distinguish different functions of the form $e^{ik\cdot x} A_k$ for $A_k\in C^\infty(\overline{\mathbb{R}^n})$ for $k\in \mathbb{R}^n$ with the same first component. 
Instead, we need access to the whole cone $\{ \lVert y \rVert < C x_1 \}$. 

\subsection{The microlocal elliptic parametrix construction}

\begin{proposition} \label{prop:micro-elliptic}
    Suppose that $A\in \Psi_{\mathrm{sc}}^{m,\el}$ is elliptic at $p\in \partial( \overline{\mathbb{R}^n} \times \overline{\mathbb{R}^n} )$. Then, there exists a $B\in \Psi_{\mathrm{sc}}^{-m,-\el}$ such that $AB-1,BA-1$ have $\operatorname{WF}'_{\mathrm{sc}}$ disjoint from $p$.  

    This also applies with $p$ replaced by any closed subset.
\end{proposition}
\begin{proof}
    Analogous to construction of parametrix for totally elliptic operators.
\end{proof}

Microlocality says that pseudodifferential operators cannot spread wavefront set. Elliptic regularity says that pseudodifferential operators cannot kill wavefront set except where they are characteristic. Specifically:
\begin{proposition}
    Suppose $A \in \Psi_{\mathrm{sc}}^{0,0}$. Then,
    \[
    \operatorname{WF}_{\mathrm{sc}}(Au)\supseteq \operatorname{WF}_{\mathrm{sc}}(u)\backslash \operatorname{char}_{\mathrm{sc}}^{m,\el}(A)  . 
    \]
\end{proposition}
\begin{proof}
    If $B\in \Psi_{\mathrm{sc}}^{0,0}$ and $K=I-BA$, then, from $u = BAu+Ku$, we get 
    \begin{equation}
        \operatorname{WF}_{\mathrm{sc}}(u) \subseteq \operatorname{WF}_{\mathrm{sc}}(BAu) \cup \operatorname{WF}_{\mathrm{sc}}(Ku) \subseteq \operatorname{WF}_{\mathrm{sc}}(Au)\cup \operatorname{WF}'_{\mathrm{sc}}(K). 
    \end{equation}
    So, 
    \begin{equation}
        \operatorname{WF}_{\mathrm{sc}}(u) \subseteq \operatorname{WF}_{\mathrm{sc}}(Au)\cup \bigcap_{B\in \Psi_{\mathrm{sc}}^{0,0} } \operatorname{WF}'_{\mathrm{sc}}(I-BA). 
    \end{equation}
    If $p\notin \operatorname{char}_{\mathrm{sc}}^{s,\el}(A)$, then the microlocal elliptic parametrix construction says we can find some $B$ such that the operator wavefront set of $I-BA$ does not contain $p$. So, $p$ is not in the intersection above. So, 
    \begin{equation}
        \bigcap_{B\in \Psi_{\mathrm{sc}}^{0,0} } \operatorname{WF}'_{\mathrm{sc}}(I-BA) \supseteq \operatorname{char}_{\mathrm{sc}}^{s,\el}(A). 
    \end{equation}
    So we end up with $\operatorname{WF}_{\mathrm{sc}}(u) \subseteq \operatorname{WF}_{\mathrm{sc}}(Au)\cup \operatorname{char}_{\mathrm{sc}}^{m,\el}(A)$, which is a restatement of the desired result. 
\end{proof}

\subsection{Problems and exercises}

\begin{problem}\label{prob:op WF set}
(a) Show that the elliptic set of $A$, $\Ell(A)$, is contained in the operator wavefront set. Show that if $\Ell(A) = \WF'(A)$ then $A$ is either elliptic or residual. 

(b) Show that, for any constant-coefficient differential operator, not identically zero, the operator wavefront set is equal to $\partial \comphase$. 
\end{problem}

\begin{problem}
    Prove \Cref{prop:cones}. Hint: you cannot use microlocality, since we used \Cref{prop:cones} to prove microlocality. Instead, use the microlocal elliptic parametrix construction.
    \label{ex:cones}
\end{problem}

\begin{problem}
    Suppose that $u$ is real-valued. Show that $\operatorname{WF}_{\mathrm{sc}}(u)$ is invariant under the fiberwise antipodal map.
\end{problem}

\begin{problem}
    \begin{itemize}
        \item Prove $\operatorname{WF}_{\mathrm{sc}}(u+v) \subseteq \operatorname{WF}_{\mathrm{sc}}(u)\cup \operatorname{WF}_{\mathrm{sc}}(v)$.   
        \item Let $u= \sum_{j=1}^J e^{i k_j \cdot x}$ for some distinct $k_j\in \mathbb{R}^n$. What is $\operatorname{WF}_{\mathrm{sc}}(u)$?
    \end{itemize}
\end{problem}

\begin{problem}
    Let $p(x)$ denote a nonzero polynomial of $x\in \mathbb{R}^n$. What can the sc-wavefront set of $p$ be? Can you guess the answer before doing any work?
\end{problem}

\begin{problem}
    Consider $u(x) = e^{ix^3} \in \mathcal{S}'(\mathbb{R})$. What is the sc-wavefront set of this $u$? Can you guess the answer before doing any work? 
\end{problem}

\begin{problem}
    Consider the Bessel function $J_\nu(x)$, $\nu\in \mathbb{R}$. What is the sc-wavefront set of $u(x)=1_{x>0} \chi(1/x) J_\nu(x)$? 
    
    Hint: use Bessel's ODE, $x^2 u''(x) +x u'(x) + (x^2- \nu^2) u(x) = 0$.
\end{problem}

\begin{problem}
    Consider the Airy function $A(x)$. What is the sc-wavefront set of $A$?

    Hint: use Airy's ODE $A''(x) = x A(x)$. 
\end{problem}

\begin{problem}
    Consider the beam 
    $u(x) = \chi(x_2 )e^{i\sigma x_1}$ on $\mathbb{R}^2$. Show that the portion of its sc-wavefront set at the corner is the closure of the sc-wavefront set at finite frequency.
    \label{ex:beam_helper}
\end{problem}

\begin{problem}
    Repeat a nonempty subset of the problems above for the Sobolev wavefront sets $\operatorname{WF}_{\mathrm{sc}}^{m,s}$. 
\end{problem}

\begin{problem}
    Let $j>0$, $\chi \in C_{\mathrm{c}}^\infty(\mathbb{R})$ be identically $1$ near the origin. What is the sc-wavefront set of $\chi(y/x^j_1) e^{ik\cdot x}$? Here $x=(x_1,y)$
\end{problem}

\section{Microlocal propagation estimates I}

\subsection{Null bicharacteristics and microlocal propagation}
Consider $P \in \Psisc^{m,l}(\R^n)$ which is not elliptic. The failure of ellipticity 
could be at frequency infinity, at spatial infinity, or both. For example, the Helmholtz operator $\Delta - \lambda^2$, where $\lambda > 0$, is elliptic at frequency infinity but not at spatial infinity, while the Klein Gordon operator $D_t^2 - \Delta - m^2$, $m > 0$, is non-elliptic both at frequency infinity and at spatial (or more precisely spacetime) infinity.  We'll assume that $P$ has a real, classical principal symbol $p \in S^{m,l}_{\sca, \cl}$.

Recall some basic symplectic geometry: the vector field $H_p$ is given by the formula
\begin{equation}\label{eq:Hvf}
    H_p = \sum_j \Big( \frac{\partial p}{\partial \xi_j} \frac{\partial}{\partial x_j} - \frac{\partial p}{\partial x_j} \frac{\partial}{\partial \xi_j} \Big).
\end{equation}
Although we write this using the usual Cartesian coordinates on $\R^n$, this is actually invariant under coordinate changes: if $(x_1, \dots x_n)$ are \emph{any} local coordinates and $(\xi_1, \dots, \xi_n)$ are the dual coordinates induced by $(x_1, \dots x_n)$ on the fibres of the cotangent bundle, then $H_p$ takes the form \eqref{eq:Hvf} in these coordinates. 

If $P$ happens to have order $(1,1)$, then the Hamilton vector field is a smooth vector field on $\comphase$, tangent to the boundary, and hence restricts to a smooth vector field on $\partial \comphase$. This is straightforward to see from \eqref{eq:Hvf}, which we write as follows:
\begin{equation}\label{eq:Hvf2}
    H_p = \sum_j \Big( \ang{x}^{-1} \frac{\partial p}{\partial \xi_j} \big(\ang{x} \frac{\partial}{\partial x_j}\big) - \ang{\xi}^{-1} \frac{\partial p}{\partial x_j} \big( \ang{\xi} \frac{\partial}{\partial \xi_j} \big) \Big).
\end{equation}
As we have seen before, the vector fields $\ang{x} \partial_{x_j}$ and $\ang{\xi}\partial_{\xi_j}$
are smooth on $\comphase$ and tangent to the boundary, while the coefficients in \eqref{eq:Hvf2} are classical of order $(0,0)$ and hence smooth on $\comphase$. 
For operators of general order $(m,l)$, it is convenient to rescale the Hamilton vector field as follows:
\begin{equation}
\scH_p^{m,l} = \ang{x}^{-l+1} \ang{\xi}^{-m+1} H_p.
\end{equation}
Then $\scH_p^{m,l}$ is again smooth on $\comphase$ and tangent to the boundary. Because of the tangency to the boundary, the flow of $\scH_p^{m,l}$ exists for all `time' (the `time' here is the flow parameter). An easy calculation (see exercises) shows that this rescaled vector field is tangent to $\mathsf{p} = 0$, where $\mathsf{p} = \ang{x}^{-l}\ang{\xi}^{-m} p$ is a smooth function on $\comphase$ vanishing at $\Char(P)$. It follows that if one point of an integral curve of $H_p$ is contained in $\Char(P)$, then the whole integral curve is contained in $\Char(P)$. Such integral curves play an important role in microlocal analysis. 

\begin{definition} A (null-)bicharacteristic of $P$ is an integral curve of $\scH_p^{m,l}$ contained in $\Char(P) = \{ \mathsf{p} = 0 \} \cap \partial \comphase$. We will usually omit the prefix `null-'. 
\end{definition}

Simple example: take $P = D_{x_1} \in \Psi^{1, 0}$. Then $p = \xi_1$ and $H_p = \partial_{x_1}$. The regularized Hamilton vector field is $\scH_p^{1,0}$ is $\ang{x} \partial_{x_1}$ and this is tangent to $\partial \comphase$; see Problem~\ref{prob:D1 tangency}. Note that if we are only interested in frequency infinity, then over a bounded region in $x$-space there is no need for the regularization. 

\subsection{Propagation theorems}

A major theorem (or really a meta-theorem, with many versions in different contexts) in microlocal analysis is called `Propagation of Singularities', or perhaps more accurately `Propagation of Regularity'. Our first version of this theorem is:

\begin{theorem}\label{thm:pos1} Suppose that $P \in \Psisc^{m,l}(\R^n)$ and admits a real, classical principal symbol $p$. Suppose that $u \in \Schw'(\R^n)$ satisfies  $Pu \in \Schw(\R^n)$. Then 
\begin{itemize}
    \item $\WFsc(u) \subset \Char(P)$, and 
    \item $\WFsc(u)$ is a union of bicharacteristics of $\scH_p^{m,l}$. 
\end{itemize}
That is, if $q, q' \in \Char(P)$ are on the same bicharacteristic of $P$, then $q \in \WFsc(u)$ iff $q' \in \WFsc(u)$, i.e. we have `propagation of singularities'. Equivalently, $q \notin \WFsc(u)$ iff $q' \notin \WFsc(u)$, i.e. we have `propagation of regularity' along bicharacteristics. 
\end{theorem}

We can also give a more quantitative version of this, in which we look at wavefront set relative to $H^{s,r}$:

\begin{theorem}\label{thm:pos2} Suppose that $P \in \Psisc^{m,l}(\R^n)$ and admits a real, classical principal symbol $p$. Suppose that $u \in \Schw'(\R^n)$ satisfies  $Pu \in H^{s,r}$. Then 
\begin{itemize}
    \item $u$ is in $H^{s+m, r+l}$ microlocally on $\Ell(P)$, i.e. 
    $$
    \WFsc^{s+m, r+\el}(u) \cap \Ell(P) = \emptyset;
    $$
    \item $\WFsc^{s+m-1, r+\el-1}(u) \subset \Char(P)$ is a union of bicharacteristics of $\scH_p^{m,l}$. 
\end{itemize}
\end{theorem}

\begin{remark}\label{rem:gain}
Notice the discrepancy of the two orders in the two statements of Theorem~\ref{thm:pos2}. On the elliptic set, we gain regularity of order $(m, l)$, the same as the order of $P$, where `gain' refers to the orders of the Sobolev space containing $u$ compared to the orders of the Sobolev space containing $Pu$. On the characteristic variety, there is no `automatic' gain as in the elliptic case, but there is a `conditional' gain, i.e. if one has regularity of a certain order of $Pu$ and regularity of $u$ at some point along a bicharacteristic, then one has regularity along the whole bicharacteristic. However, the `non-elliptic' gain here is only $(m-1, l-1)$, one less in each exponent relative to the elliptic gain. This is always the case and is explained by the method of proof below. More on this later in Lecture 6. 
\end{remark}

We will prove a \emph{microlocal propagation estimate} that implies Theorem~\ref{thm:pos2} and is a strictly stronger result, as it is `fully microlocal': instead of the global assumption that $Pu \in H^{s,r}$, we only assume this in some microlocal region. To set up this theorem, we suppose as before that $P \in \Psisc^{m,l}$ has real, classical principal symbol, and $\gamma: [0, s_0] \to \Char(P)$ is a nontrivial bicharacteristic, i.e. nonconstant (equivalent to assuming that $\scH_p^{m,l}$ is nonvanishing at, and therefore, by continuity, also near, $\gamma(0)$). Let $U_0$ and $U$ be open subsets of $\partial \comphase$ such that $U_0$ contains $\gamma(0)$ and $U$ contains $\gamma([0, s_0])$. 

\begin{theorem}\label{thm:mpe} There exist operators $B, E, G \in \Psisc^{0,0}$ such that 
\begin{equation}\begin{gathered}
    \gamma([0, s_0]) \subset \Ell(B) \subset \WF'(B) \subset \Ell(G) \subset \WF'(G) \subset U, \\
    \gamma(0) \subset \Ell(E) \subset \WF'(E) \subset U_0, 
\end{gathered}\end{equation}
such that for all $s,r, N$ there exists $C > 0$ such that for all $u \in \Schw'(\R^n)$ we have 
\begin{equation}\label{eq:mpe}
\| Bu \|_{H^{s,r}} \leq C \Big( \| GPu \|_{H^{s-m+1, r-l+1}} + \| Eu \|_{H^{s,r}} + \| u \|_{H^{-N, -N}} \Big). 
\end{equation}
This inequality holds in the strong sense that if the RHS is finite, then so is the LHS, and the inequality holds. In particular it means that if $u \in H^{-N, -N}$, $Pu$ is in $H^{s-m+1, r-l+1}$ microlocally in $U$, and $u$ is in $H^{s,r}$ microlocally in $U_0$, then $u$ is in $H^{s,r}$ microlocally on $\Ell(B)$, in particular on $\gamma([0, s_0])$. 
\end{theorem}

\begin{figure}
    \centering
    \includegraphics[width=0.8\textwidth]{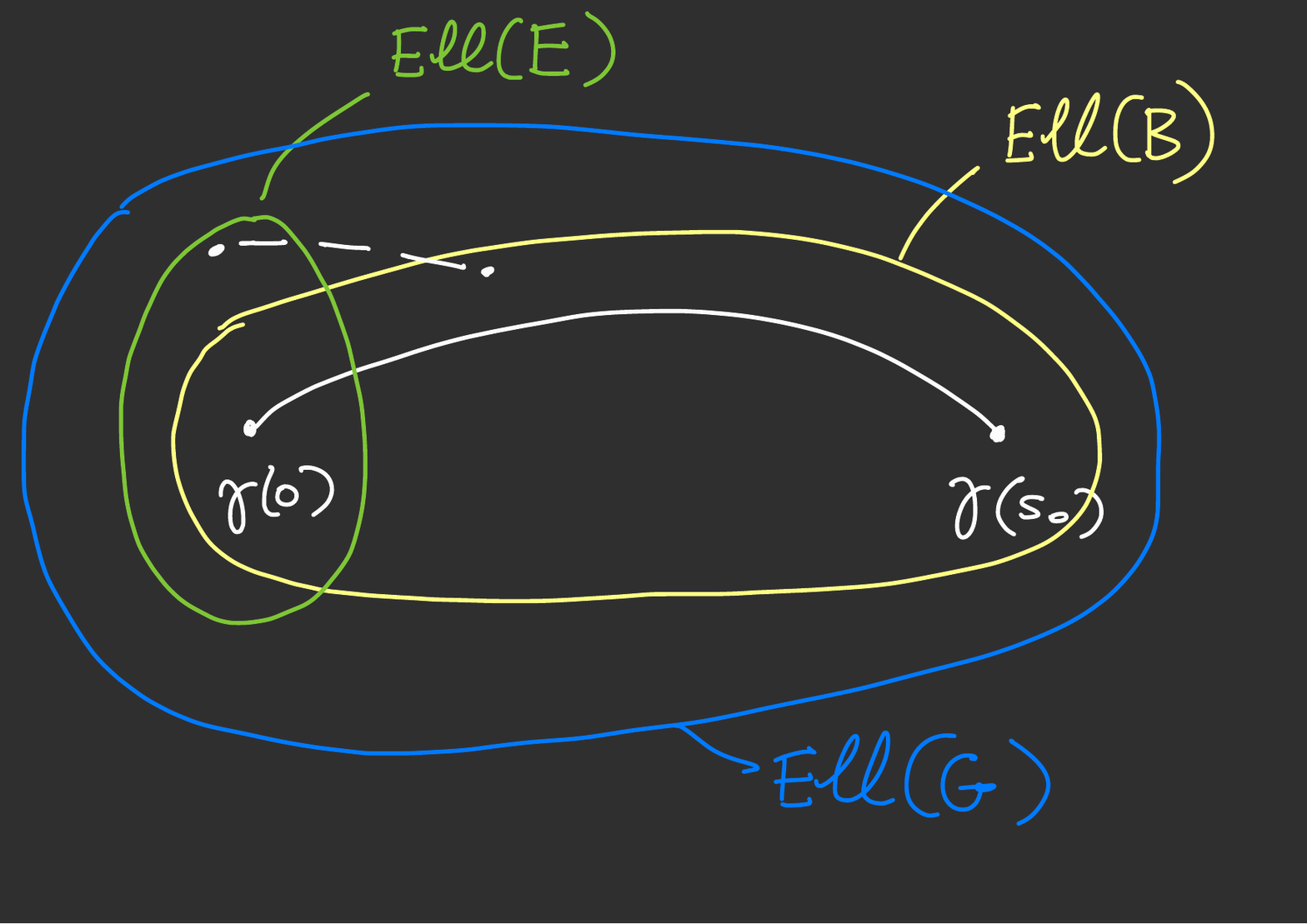}
    \caption{The geometric setup of Theorem~\ref{thm:mpe}. The bicharacteristic $\gamma$ is depicted with a solid white line, and the dotted white line shows a bicharacteristic with endpoint in $\Ell(B)$ and starting point in $\Ell(E)$, as required in the proof.}
    \label{fig:pos}
\end{figure}

\begin{remark} Note that this estimate is `fully microlocal' in the sense that the assumptions both on $u$ and $Pu$ are microlocalized near the region of interest, which is the bicharacteristic segment $\gamma([0, s_0])$. The only exception is the assumption that $u \in H^{-N, -N}$ but this is really no assumption at all, since every tempered distribution is in $H^{-N, -N}$ for sufficiently large $N$. 
\end{remark}

We will prove this result next lecture. 

\subsection{A simple example}
In the remainder of this lecture we will present a proof in a simple model situation. Here will be working at frequency infinity, over a bounded region of $\R^2$. Let $P = D_{x_1} \in \Psisc^{1,0}(\R^2)$, and we suppose that $u$ is a tempered distribution such that $Pu = f \in L^2(\R^2)$. We will assume that 
$$
u \in L^2([-2, 0]_{x_1} \times [-2, 2]_{x_2}). 
$$
The problem is to prove that $u$ is $L^2$ on a rectangle that is larger in the $x_1$ direction, although (for convenience) a bit smaller in the $x_2$ direction: show that 
$$
u \in L^2([-2, 2]_{x_1} \times [-1, 1]_{x_2}). 
$$
That is, regularity of $u$ `propagates' in the $x_1$-direction. 
It will be enough, in view of the assumption on $u$, to show it is $L^2$ on $L^2([-1, 2]_{x_1} \times [-1, 1]_{x_2})$. 

We choose a cutoff function $\chi_1(x_1)$ that is identically $1$ on $[-1, 2]$, supported on $[-2, 3]$, and monotone nondecreasing between $[-2, -1]$, monotone nonincreasing on $[2, 3]$. We also choose a cutoff $\chi_2(x_2)$ that is identically $1$ on $[-1, 1]$, supported on $[-2, 2]$, and monotone nondecreasing between $[-2, -1]$, monotone nonincreasing on $[1, 2]$. We define
$$
a(x_1, x_2) = \chi_1(x_1) \chi_2(x_2) e^{-x_1}. 
$$
We compute
\begin{equation}\begin{gathered}
    i[P, a(x_1, x_2)] = \partial_{x_1} a(x_1, x_2) \\
    = - b(x_1, x_2) + e(x_1, x_2) \text{ where }\\
    b(x_1, x_2) = a(x_1, x_2) - \chi_1'(x_1) 1_{x_1 \geq 0} e^{-x_1} \chi_2(x_2) \geq 0, \\
    e(x_1, x_2) = \chi_1'(x_1) 1_{x_1 \leq 0} e^{-x_1} \chi_2(x_2) \geq 0.
\end{gathered}\label{eq:be}
\end{equation}
Then we compute (for suitably regular $u$, say $u \in C^1(\R^2)$) the commutator
\begin{equation}\label{eq:commutator}
i \big\langle i[P, a(x_1, x_2)] u, u \big\rangle_{L^2(\R^2)}
\end{equation}
in two different ways. 

\emph{First}, we unwrap the commutator, obtaining 
\begin{equation}\begin{gathered}
\eqref{eq:commutator} = i \big\langle  a(x_1, x_2) u, P u \big\rangle - i \big\langle Pu,  a(x_1, x_2) u \big\rangle \\
= 2 \Im \big\langle Pu,  a(x_1, x_2) u \big\rangle  = 2 \Im \big\langle f,  a(x_1, x_2) u \big\rangle.
\end{gathered}
\end{equation}

\emph{Second}, we use \eqref{eq:be} to write 
\begin{equation}
    \eqref{eq:commutator} = - \ang{bu, u} + \ang{eu, u}. 
\end{equation}
Using the support property of $e$, we put these two identities together and use Cauchy-Schwarz on \eqref{eq:commutator}, together with the inequality $2|xy| \leq  x^2/\epsilon + \epsilon y^2$,  to find that 
$$
\ang{bu, u} \leq \ang{eu, u} + \frac1{\epsilon}  \| f \|_{L^2}^2 + \epsilon  \| au \|_{L^2}^2
$$
for any $\epsilon > 0$. Here, the $\ang{eu, u}$ term can be controlled by $C' \| u \|_{L^2([-2, 1]_{x_1} \times [-2, 2]_{x_2})}^2$, but we do not make this step just yet. 

Notice that $0 \leq a \leq 9$. So $a^2 \leq 9a \leq 9b$. Choosing $\epsilon = 1/18$, we find 
$$
\ang{bu, u} \leq \ang{eu,u} + 18  \| f \|_{L^2}^2 + \frac1{2} \ang{bu, u}.
$$
We now absorb the $\ang{bu, u}$ term on the RHS and obtain, after multiplying by 2, 
\begin{equation}\label{eq:quant est}
\ang{bu, u} \leq 2 \ang{eu,u} + 36  \| f \|_{L^2}^2 .
\end{equation}

Unfortunately, this argument does \emph{not} prove that $u$ is $L^2$ on $[-1, 2] \times [-1, 1]$ as it requires that $\ang{bu, u} < \infty$ (for the absorption step), which is a stronger assumption! So it might seem that what we've done so far is completely pointless. This is not the case however, as we have proved a \emph{quantitative} estimate on $\ang{bu, u}$ making only the \emph{qualitative} assumption that this quantity is finite. Quantitative estimates are very powerful as we shall see shortly: our quantitative estimate can be combined with an approximation argument to achieve our aim.  

To do this, we choose $N$ large enough so that $u \in H_{\sca}^{-N, -N}$; one can choose such $N$ for any tempered distribution. Fixing such an $N$, we let 
$$
u_{\tilde{r}} = (1 + {\tilde{r}} \Delta)^{-M} u := T_{\tilde{r}} u, \quad M \geq N/2, {\tilde{r}} > 0, 
$$
and consider the limit as ${\tilde{r}} \to 0$. Intuitively, as ${\tilde{r}} \to 0$, $T_{\tilde{r}}$ tends to the identity operator. In fact this is rigorously true in the topology of $\Psisc^{\eta, \eta}$ for any $\eta > 0$; moreover, $T_{\tilde{r}}$ is uniformly bounded in $\Psisc^{0,0}$ as ${\tilde{r}} \to 0$ and converges strongly to the identity operator in $\mathcal{B}(L^2, L^2)$ (exercise). We have 
$$
u_{\tilde{r}} \in L^2 \text{ for all } {\tilde{r}} > 0, \text{ and } Pu = T_{\tilde{r}} f = f_{\tilde{r}}
$$
where for the second identity we used the fact that $P$ and $T_{\tilde{r}}$ commute (both are Fourier multipliers). Our quantitative estimate \eqref{eq:quant est} applies to $u_{\tilde{r}}$ for each ${\tilde{r}} > 0$, and we obtain 
\begin{equation}\label{eq:ineqr}
\ang{bu_{\tilde{r}}, u_{\tilde{r}}} \leq 2 \ang{eu_{\tilde{r}}, u_{\tilde{r}}} + 36  \| f_{\tilde{r}} \|_{L^2}^2 .
\end{equation}
Now we take the limit as ${\tilde{r}} \to 0$. Since $f \in L^2$, we have $f_{\tilde{r}} \to f$ in $L^2$ and hence 
$$
\| f_{\tilde{r}} \|_{L^2}^2  \to \| f \|_{L^2}^2 . 
$$
Next we consider the limit of the $\ang{eu_{\tilde{r}}, u_{\tilde{r}}}$ term. This is slightly more complicated as $u$, unlike $f$, is not globally in $L^2$. We write this term by choosing a new cutoff function $\psi$, so that $\psi \equiv 1$ on the support of $e$, and $\psi$ is supported in $[-2, 1]_{x_1} \times [-2, 2]_{x_2}$. We have 
\begin{equation}\begin{gathered}
    \ang{eT_{\tilde{r}} u, T_{\tilde{r}} u} = \ang{eT_{\tilde{r}} \psi u, T_{\tilde{r}} \psi u} + \ang{eT_{\tilde{r}} (1 - \psi) u, T_{\tilde{r}} \psi u} + \ang{ T_{\tilde{r}} u,  eT_{\tilde{r}} (1 - \psi) u}. 
\end{gathered}
\end{equation}
The first term on the RHS converges to $\ang{e\psi u, \psi u}$ since $\psi u \in L^2$ and  $T_{\tilde{r}} \to \Id$ strongly. Moreover, since $\psi \equiv 1$ on the support of $e$, this is just $\ang{eu, u}$. We would like to show that the other two terms tend to zero. 

Consider the operator $e T_{\tilde{r}} (1-\psi)$. Since  $e (1 -\psi) = 0$ (as the supports of $e$ and $1-\psi$ are disjoint) this operator is equal to $e (T_{\tilde{r}} - \Id) (1-\psi)$. We recall that $T_{\tilde{r}} - \Id$ converges to zero in the topology of $\Psisc^{\eta, \eta}$ for all $\eta > 0$. However, we can say much more by writing the operator $e T_{\tilde{r}} (1-\psi)$ as the expansion of this composition to $K$ terms, plus a remainder term. The expansion vanishes identically, and the remainder (which is $e T_{\tilde{r}} (1 - \psi)$ itself) tends to zero  in $\Psisc^{-K+\eta, -K+\eta}$ for all $\eta > 0$. From this we see that $e T_{\tilde{r}} (1 - \psi) u$ converges to zero in every Sobolev space, showing that the second and third terms on the RHS tend to zero. 

Using \eqref{eq:ineqr} and the reasoning above, we have 
\begin{equation}\label{eq:ineqr2}
\limsup_{{\tilde{r}} \to 0} \ang{bu_{\tilde{r}}, u_{\tilde{r}}} \leq 2\ang{eu, u} + 36 \| f \|_{L^2}^2.
\end{equation}
Now we would like to distribute $\sqrt{b}$ on both sides of the inner product on the LHS. This can be done either by choosing $b$ to be a square of a smooth function from the outset, or we can give up a little, and choose $\chi \in C_c^\infty(\R^2)$ such that $0 \leq \chi^2 \leq b$ and $\chi \equiv 1$ on $[-2, 1] \times [-1, 1]$. Then we have from \eqref{eq:ineqr2} that 
\begin{equation}\label{eq:ineqr3}
\limsup_{{\tilde{r}} \to 0} \| \chi T_{\tilde{r}} u \|_{L^2}^2 \leq 2\ang{eu, u} + 36 \| f \|_{L^2}^2.
\end{equation}
The uniform boundedness of $\chi T_{\tilde{r}} u$ in $L^2$ means that there is a weakly convergent subsequence, (weakly) converging to $v \in L^2$, say. A fundamental property of weak limits is that $\| v \|_{L^2}^2$ also satisfies the inequality \eqref{eq:ineqr3}.  On the other hand, we have $\chi T_{\tilde{r}} u \to \chi u$ in a weaker topology, that of $H_{\sca}^{-N-\eta, -N - \eta}$, using $u \in H_{\sca}^{-N, -N}$, and using the convergence of $T_{\tilde{r}} \to \Id$ in $\Psisc^{\eta, \eta}$ for all $\eta > 0$. Weak convergence in $L^2$ implies strong convergence in $H_{\sca}^{-N-\eta, -N - \eta}$ due to compact embeddings (see Proposition~\ref{prop:comp emb}).  Uniqueness of limits shows that $v = \chi u$. We deduce that 
\begin{equation}\label{eq:ineqr4}
\| \chi u \|_{L^2}^2 \leq 2\ang{eu, u} + 36 \| f \|_{L^2}^2 \leq 2 \| u \|_{L^2([-2, 1]_{x_1} \times [-2, 2]_{x_2})}^2 + 36 \| f \|_{L^2}^2,
\end{equation}
and since $\chi = 1$ on $[-1, 2] \times [-1, 1]$, we have achieved our goal. 

\subsection{Exercises}
\begin{problem}
 Check that the rescaled Hamilton vector field $\scH^{m,l}_p$ of $P = \Op(p)$ is tangent to $\mathsf{p} = 0$, using the notation in the lecture notes above. 
 \end{problem}

 \begin{problem}\label{prob:D1 tangency}
Show, by writing in smooth coordinates near the boundary of spacetime, that the vector field $\ang{x} D_{x_1}$ is tangent to the boundary of phase space. 
\end{problem}

 \begin{problem}
 Show that $T_{\tilde{r}} = (1 + {\tilde{r}} \Delta)^{-M}$ is a uniformly bounded family of operators in $\Psisc^{0,0}$, for any $M > 0$. Use this to show that $T_{\tilde{r}}$ converges strongly to the identity operator in $\mathcal{B}(L^2, L^2)$. Hint: for $f \in L^2$, for any $\epsilon > 0$ one can approximate $f$ by an element of a positive order Sobolev space, up to an $\epsilon$ error (in $L^2$-norm). 
 \end{problem}
 
\begin{problem} Suppose that $u$ is a tempered distribution in $\R^2$ satisfying the PDE
$$
D_{x_1} u = \delta + f,
$$
where $\delta$ is the two-dimensional delta distribution, supported at the origin in $\R^2$, and $f$ is an $L^2$ function of $(x_1, x_2) \in \R^2$ supported where $x_1 > 0$. You are given that $u$ is $C^\infty$ for $x_1 < 0$. 

(i) First show that $\delta \in H^s(\R^2)$ for all $s < -1$, but it is not in $H^{-1}(\R^2)$. 

Now use the result of Theorem~\ref{thm:pos2} to answer the following questions:

(ii) Suppose that $s < -1$. Show that $u$ is `locally' in $H^s(\R^2)$, in the sense that for any $\phi \in C_c^\infty(\R^2)$, we have $\phi u \in H^s(\R^2)$.

(iii) Suppose that $-1 < s < 0$. What bound on $\WF^s(\phi u)$ does Theorem~\ref{thm:pos2} provide? (Hint: it might help to first consider the case that $f = 0$.) Here $s$ is the differential index; the spatial index is omitted as $\phi$ has compact support, so there is no wavefront set at spatial infinity. 

(iv) Suppose that $0 < s < 1$. Give an `optimal' bound (not depending on $f$, other than using the information $f \in L^2$ and its support property)  on $\WF^s(\phi u)$. (Your bound may not be optimal for every individual choice of $f$, but should be the optimal bound that works for all $f$ as specified in the problem.)   
\end{problem}



\section{Microlocal propagation estimates II}\label{lec:propagation}
In this Lecture we prove Theorem~\ref{thm:mpe} for a general operator $P$. However, we shall assume that $P \in \Psisc^{1,1}$ which simplifies some of the numerology. There is no loss of generality, as solving an equation $Pu = f$ is equivalent to solving $TPu = Tf$ where $T$ is an invertible operator. We can always choose an invertible scattering pseudodifferential operator $T$ with real, classical symbol, such that $TP \in \Psisc^{1,1}$ (see below) and this reduces us to this case. 

For such a $P$, with real, classical principal symbol, we want to construct $B, E, G \in \Psisc^{0,0}$ as in the theorem, so that the estimate
\begin{equation}\label{eq:mpe2}
\| Bu \|_{H^{s, r}} \leq C \Big( \| GPu \|_{H^{s, r}} + \| Eu \|_{H^{s,r}} + \| u \|_{H^{-N, -N}} \Big) 
\end{equation}
holds. The nontrivial case of \eqref{eq:mpe2} is that $u$ is in  $H^{-N, -N}$, that $Pu$ is in $H^{s,r}$ microlocally in $U$, and that $u$ is in $H^{s,r}$ microlocally in $U_0$, otherwise the RHS is infinite (or can be made infinite by choosing $E, G$ carefully) and the result is vacuously true. So we assume this for the remainder of the proof. 

\subsection{Proof of Theorem~\ref{thm:mpe}, Step 1}
 We construct $A \in \Psisc^{2s, 2r}$ with real symbol $a$, with $A = A^*$ such that 
\begin{equation}\label{eq:operatorA}
    i(P^*A - AP) = - \tilde B^* \tilde B - (\Lambda^{-s,-r} A)^* \Lambda^{-s,-r} A + E' + R, 
\end{equation}
where 
\begin{itemize}
    \item $\Lambda^{s,r} = \ang{x}^r \ang{D}^s$ is an invertible elliptic operator, with real, classical symbol, of order $(s,r)$. Note that $\ang{D}^s$ is an abbreviation for $\Op(\ang{\xi}^s)$. 
    \item $\tilde B = \Lambda^{s,r} B$, where $B$ is as in \eqref{eq:mpe2}.
    \item $E' \in \Psisc^{2s, 2r}$ and $\WF'(E') \subset \Ell(E)$.
    \item $R \in \Psisc^{2s-1, 2r-1}$, $\WF'(R) \subset \Ell(G)$. 
    \item $G$ is microlocally equal to the identity on $\WF'(A)$. 
\end{itemize}
Notice that in \eqref{eq:operatorA}, the LHS has order $(2s, 2r)$ although the individual terms $P^* A$ and $AP$ have higher order. In fact, we can write this as 
$$
i [P, A] + i (P^* - P)A,
$$
and $[P, A]$ has order $(2s, 2r)$ (recall we are assuming $P \in \Psisc^{1,1}$)  since commutators drop order by $(1,1)$ relative to the sum of the orders of the individual operators; while $(P^* - P)$ has order $(0,0)$ since $\spr^{1,1}(P) = \overline{\spr^{1,1}(P)} = \spr^{1,1}(P^*)$. On the RHS side, the first three terms have order $(2s, 2r)$, while the last term has order $(2s-1, 2r-1)$. Thus, if we view the $R$ term as a remainder term of lesser interest (we only care about its microsupport, as in the fourth bullet point above), then the operator equation \eqref{eq:operatorA} can be arranged by making it hold at a principal symbol level. 

The reason for being interested in the quantity $i(P^* A - AP)$ is that 
$$
\bang{i(P^* A - AP) u, u} = 2 \Im \ang{Pu, Au}. 
$$
Given the identity \eqref{eq:operatorA}, and a sufficiently nice $u$ (meaning it is in a Sobolev space with sufficiently large orders), we can make the following calculation:
\begin{equation}\label{eq:comm-identity}
    2 \Im \ang{Pu, Au} = - \| \tilde B u \|_{L^2}^2 - \| \Lambda^{-s,-r} A u \|_{L^2}^2 + \ang{E' u, u} + \ang{Ru, u}. 
\end{equation}
We can microlocalize the $Pu$ term by using the condition that $G = \Id$ microlocally on $\WF'(A)$. Therefore, by redefining $R$ by the addition of a residual operator (which we do not indicate in notation), we obtain from \eqref{eq:operatorA} the variant 
\begin{equation}\label{eq:operatorA2}
    i((GP)^*A - AGP) = - \tilde B^* \tilde B - (\Lambda^{-s,-r} A)^* \Lambda^{-s,-r} A + E' + R, 
\end{equation}
and this leads to (with the same redefinition of $R$)
\begin{equation}\label{eq:comm-identity-G}
    2 \Im \ang{GPu, Au} = - \| \tilde B u \|_{L^2}^2 - \| \Lambda^{-s,-r} A u \|_{L^2}^2 + \ang{E' u, u} + \ang{Ru, u}. 
\end{equation}
in place of \eqref{eq:comm-identity}. 
We estimate the LHS as follows: we write $Au = (\Lambda^{-s,-r})^{-1} \Lambda^{-s,-r} Au$ and move the left factor to the other side of the inner product, noting that the adjoint of $(\Lambda^{-s,-r})^{-1}$ is $\Lambda^{s,r}$. We then apply Cauchy-Schwarz to the inner product. We obtain 
\begin{equation}
     \| \tilde B u \|_{L^2}^2 +  \| \Lambda^{-s,-r} A u \|_{L^2}^2 \leq \| \Lambda^{s,r} GPu \|_{L^2}  \| \Lambda^{-s,-r} Au \|_{L^2} 
    + \ang{E' u, u} + \ang{Ru, u}. 
\end{equation}
We now apply the inequality $ab \leq a^2 + b^2$ to the first term on the RHS. We also apply an elliptic estimate to the $\ang{E'u, u}$ term, noting that $(\Lambda^{s,r} E)^* \Lambda^{s,r} E$ is elliptic on $\WF'(E')$ (see Problem~\ref{prob:micro ell estimate} exercises at the end of the lecture). We can do the same with the $R$ term. In fact, we choose a $B' \in \Psisc^{0,0}$ such that $\WF'(R) \subset \Ell(B') \subset \WF'(B') \subset \Ell(G)$, and estimate it in the same way. 
\begin{multline}\label{eq:microlocal energy est}
     \| \Lambda^{s,r}  B u \|_{L^2}^2 +  \| \Lambda^{-s,-r} A u \|_{L^2}^2 \leq  \| \Lambda^{s,r} GPu \|_{L^2}^2 +  \| \Lambda^{-s,-r} Au \|_{L^2}^2  \\
    + C \| \Lambda^{s,r} Eu \|_{L^2}^2 + C \| \Lambda^{s-1/2,r-1/2} B'u \|_{L^2}^2 + C \| u \|_{H^{-N, -N}}^2,
\end{multline}
where the final term arises from the remainder term in the elliptic estimates. 
The $\| \Lambda^{-s, -r} Au \|_{L^2}^2$ terms cancel. Now eliminating the $\Lambda^{\bullet, \bullet}$ factors and writing norms in terms of weighted Sobolev norms, we have obtained 
\begin{equation}
     \|  B u \|_{H^{s,r}}^2  \leq C \Big(  \| GPu \|_{H^{s,r}}^2 
    +  \|  Eu \|_{H^{s,r}}^2 +  \|  B'u \|_{H^{s-1/2,r-1/2}}^2 + \| u \|_{H^{-N, -N}}^2 \Big),
\end{equation}
which is almost the estimate we are aiming for: the only discrepancy is the term $C \|  B'u \|_{H^{s-1/2,r-1/2}}^2$ on the RHS. Notice that this term is just like the $Bu$ term on the LHS that we are estimating, but it has orders lower by $(1/2, 1/2)$. We can thus proceed by induction, estimating the $B'$ term as above, up to an additional $B''u$ term that would be measured in the Sobolev space $H^{s-1, r-1}$, and so on. Each time, we need to enlarge the microlocal support of $B', B'', ...$ relative to the operator before (and also enlarge the microlocal support of the $E'$ term correspondingly), but we can do this while always remaining in the region where $G$ is microlocally equal to the identity. After a finite number of steps, the order of the extra term is reduced to $(-N, -N)$ and then this term can be absorbed in the $\| u \|_{H^{-N, -N}}^2$ term, at which point the argument is terminated and the estimate is proved. 

To complete the proof, we need to accomplish two more steps: arrange the operator identity \eqref{eq:operatorA2}, and eliminate the assumption that $u$ is in a sufficiently nice Sobolev space (just like we had to eliminate the assumption $\ang{bu,u} < \infty$ in the previous Lecture). 

\begin{remark}
    Returning to Remark~\ref{rem:gain}, the reason that there is a loss of 1 in both orders relative to the elliptic gain is that, in the propagation proof, the estimate arises from the ellipticity of $B$, which comes from the positivity (actually negativity, as we presented it!) of the commutator $i[P, A]$. On the other hand, in the elliptic estimate the ellipticity is directly from the ellipticity of $P$. Since the commutator drops order by $(1,1)$ relative to the composition of the two operators, this causes the propagation estimate to be weaker, i.e. we need to measure $Pu$ by a stronger norm to deduce a fixed Sobolev norm of $Bu$. 
\end{remark}

\subsection{Proof of Theorem~\ref{thm:mpe}, Step 2} Constructing the operators. We do this first in the case that $(s, r) = (0, 0)$ (the general case is only notationally more complicated). In this case, $A$, $B = \tilde B$ and $E'$ all have order $(0,0)$. 

We need to arrange \eqref{eq:operatorA}. Since we are not much interested in the properties of $R$ (other than its microlocal support, which is anyway bounded above by the union of the  microlocal supports of the other operators), to have \eqref{eq:operatorA} it is enough to ensure it holds at a principal symbol level. Let $p_1 = i\sigma_L(P^* - P) \in S^{0,0}_{sca}$; this is a classical symbol, and real since $i(P^* - P)$ is symmetric.  At the principal symbol level, \eqref{eq:operatorA} reads 
\begin{equation}\label{eq:ham eqn}
    H_p(a) + p_1 a = - b^2 - a^2 + e' . 
\end{equation}
In our construction, these operators will be classical, thus their symbols will be smooth on the compactified phase space $\comphase$. Since we are only interested in their principal symbols, that means we only care about their value at the boundary of $\comphase$. Thus in Step 2, we work purely at the boundary of phase space, and define $a$, $b$ and $e'$ such that \eqref{eq:ham eqn} holds at the boundary.

Using the assumed nonvanishing of $H_p$ at $\gamma(0)$, and assuming for expository clarity that $\gamma(0)$ is not contained in the corner\footnote{A small, mostly notational, modification of this construction serves to treat the case that $\gamma(0)$ is contained in the corner of $\comphase$; we do not pursue this further here.} of $\comphase$, there are local coordinates $(z_1, z')$, $z' \in \R^{2n-2}$,  on $\partial \comphase$ near $\gamma([0, s_0])$ so that $\gamma(0) = (0, 0)$ and $H_p = \partial_{z_1}$ in this coordinate system, i.e. bicharacteristics are given by $z' = $ constant.

Choose $\epsilon > 0$ sufficiently small so that, in the coordinates $(z_1, z')$, we have
\begin{equation}\label{eq:bich nbhd}\begin{aligned}
[-2\epsilon, s_0 + 2\epsilon] \times \{ |z'| \leq 2\epsilon \} &\subset U \subset WF^{0,0}(Pu)^\complement, \\
[-2\epsilon, 2\epsilon] \times \{ |z'| \leq 2\epsilon \} &\subset U_0 \subset WF^{0,0}(u)^\complement. 
\end{aligned}\end{equation}
This is possible because, by our assumption of nontriviality, we have $Pu \in H^{0,0} = L^2$ microlocally near $\gamma([0, s_0])$, hence $\WF^{0,0}(Pu)$ is disjoint from $\gamma([0, s_0])$, which is given by $\{ z_1 \in [0, s_0], z' = 0 \}$ in our coordinate system. Since both $\gamma([0, s_0])$ and $\WF^{0,0}(Pu)$ are compact sets, one can find a neighbourhood $U$ of $\gamma([0, s_0])$ disjoint from $\WF^{0,0}(Pu)$, and it necessarily contains the set $[-2\epsilon, s_0 + 2\epsilon] \times \{ |z'| \leq 2\epsilon \}$ for sufficiently small $\epsilon > 0$. A similar argument applies to the second line of \eqref{eq:bich nbhd}. 

We now choose several cutoff functions. Let $\psi(z') \in C_c^\infty(\R^{2n-2})$ be such that $\psi(z') = 1$ for $|z'| \leq \epsilon$ and $\psi(z') = 0$
if $|z'| \geq 2\epsilon$. Then, in the $z_1$ direction we choose turn-on and turn-off functions. The turn-on function is $\chi_1(z_1)$ which is smooth, monotone, equal to $0$ for $z_1 \leq -\epsilon$ and $1$ for $z_1 \geq \epsilon$. The turn-off function we specify a bit more explicitly, as we need to take square-roots and show smoothness of these. We let $\chi_0$ be the smooth function defined by 
\begin{equation}\label{eq:chi_0}
    \chi_0(t) = \begin{cases} 0, t \leq 0 \\ e^{-\digamma /t}, t \geq 0 \end{cases} 
\end{equation}
where $\digamma > 0$ (the symbol $\digamma$ is `digamma') will be chosen sufficiently large. Notice that 
\begin{equation}\label{eq:chi_0 2}
    \chi_0'(t) = \digamma \chi_0(t)/t^2.
\end{equation} 
The turn-off function is then defined by
\begin{equation}\label{eq:turnoff}
    \chi(z_1) = \chi_0(s_0 + \epsilon - z_1) \Longrightarrow \chi'(z_1)= - \frac{\digamma}{(s_0+\epsilon-z_1)^2}\chi(z_1). 
\end{equation}
We now define
\begin{equation}
    a(z, z') = \chi(z_1) \chi_1(z_1)^2 \psi(z')^2,
\end{equation}
which is supported in $[-\epsilon, s_0 + \epsilon] \times \{ |z'| \leq 2\epsilon \} \subset U$. 
We compute
\begin{equation}\label{eq:Hpa}\begin{gathered}
    H_p a + p_1 a = \chi'(z_1) \chi_1(z_1)^2 \psi(z')^2 \\
    + p_1(z, z') \chi(z_1) \chi_1(z_1)^2 \psi(z')^2 \\
    + 2 \chi(z_1) \chi_1(z_1) \chi_1'(z_1) \psi(z')^2
\end{gathered}\end{equation}
where the sum of the first two lines is nonpositive (for $\digamma$ sufficiently large) thanks to \eqref{eq:turnoff} and the third line is nonnegative. We will define $b$ so that the first two lines are $-b^2 - a^2$, and the third line is $e'$. Using \eqref{eq:turnoff}, this requires that 
\begin{equation}
    b^2 = \Big( \frac{\digamma}{(s_0+\epsilon-z_1)^2} + p_1 \Big)\chi(z_1) \chi_1(z_1)^2 \psi(z')^2 - \chi(z_1)^2 \chi_1(z_1)^4  \psi(z')^4. 
\end{equation}
Noting that $\sqrt{t^{-2}\chi_0(t)}$ is a smooth function, we have 
\begin{equation}
    b = \sqrt{\frac{\chi_0(s_0 + \epsilon - z_1)}{(s_0+\epsilon-z_1)^2}}\chi_1(z_1) \psi(z') \sqrt{\digamma - (s_0+\epsilon-z_1)^2 \Big( p_1 + \chi(z_1) \chi_1(z_1)^2  \psi(z')^2 \Big)}
\end{equation}
Clearly, for $\digamma$ sufficiently large, the argument of the square root is bounded away from zero on the support of $\chi(z_1) \chi_1(z_1) \psi(z')^2$ and therefore $b$ is smooth,  with support contained in the closed set $[-\epsilon, s_0 + \epsilon] \times \{ |z'| \leq 2\epsilon \} \subset U$. 

As mentioned above, we define $e'$ to be the third line of \eqref{eq:Hpa}:
\begin{equation}
    e' = 2 \chi_1(z_1) \chi_1'(z_1) \chi(z_1) \psi(z')^2.
\end{equation}
We see that the support of $e'$ is contained in $[-\epsilon, \epsilon] \times \{ |z'| \leq 2\epsilon \} \subset U_0$. 
We choose $g$, the principal symbol of $G$, to be a smooth function equal to $1$ on the support of $a$ (hence also on the support of $b$ and $e$) and supported in $U$. 

We then extend the smooth functions $a, b, e', g$ into the interior and quantize to obtain operators 
$$
A = \frac{\Op_L(a) + \Op_L(a)^*}{2}, \quad B = \Op_L(b), \quad E' = \Op_L(e'), \quad G = \Op_L(g). 
$$
so that \eqref{eq:operatorA} is satisfied, noting that this makes $R \in \Psisc^{-1, -1}$. We must extend $g$ into the interior so that it is identically $1$ in a neighbourhood (in the ambient space, i.e. not just at $\partial \comphase$) of $\WF'(A)$. We also note that $\WFsc'(R)$ is bounded by the union of $\WFsc'(A), \WFsc(B)$ and $\WFsc(E')$. Thus \eqref{eq:operatorA}, and the conditions listed below that equation, are all satisfied. 

In the case of general $(s,r)$ we redefine $a$ to be 
$$
a = \ang{\xi}^{2s} \ang{x}^{2r} \chi_1(z_1)^2 \chi(z_1) \psi(z')^2
$$
and proceed as before. In the calculation \eqref{eq:Hpa}, as well as multiplying the whole identity through by the factor $\ang{\xi}^{2s} \ang{x}^{2r}$ we will obtain extra terms, while arise from differentiating these factors. The new terms, in effect, have the effect of replacing the $p_1$ term by 
$$
p_1 + \ang{\xi}^{-2s} \ang{x}^{-2r} H_p \big( \ang{\xi}^{2s} \ang{x}^{2r} \big).
$$
It is easy to see that this additional term is a symbol of order $(0,0)$; in fact, it is the principal symbol of $\Lambda^{-2s, -2r} [P, \Lambda^{2s, 2r}]$ which has order $(0,0)$. 
So, effectively the difference is to replace $p_1$ by some other symbol $\tilde p_1$ of order $(0,0)$, so this allows the construction to proceed exactly as before, with $b$ and $e'$ redefined suitably. It would be a useful exercise for the reader to compute the exact formulae for $b$ and $e'$ in this setting of general orders. 

\subsection{Proof of Theorem~\ref{thm:mpe}, Step 3} At this stage we have proved the required inequality for sufficiently regular/decaying $u$. As with the simple example proved last lecture, this is far from satisfactory as we need to assume at least as much regularity/decay as is proved. Indeed, the $\| \Lambda^{-s, -r} Au \|_{L^2}^2$ terms need to be absorbed  in \eqref{eq:microlocal energy est}, which requires them to be finite a priori. Nevertheless we have proved a quantitative estimate for sufficiently regular/decaying functions, and as with the example last lecture, we can use a regularization argument to obtain the full result. 

In the simple example last time, we used $T_{\tilde{r}} = (1 + {\tilde{r}} \Delta)^{-M}$ as a regularizing operator. This worked well as it commuted with the operator $P = D_{x_1}$. In the more general case, the lack of commutation is an issue and we have to be more careful. Moreover, $T_{\tilde{r}}$ regularizes only in the differential sense, not the decay sense. To emphasize the point here, we will henceforth assume that $\gamma$ lies in the spatial boundary, at finite frequency. Thus the appropriate `regularization' is to multiply by $(1 + {\tilde{r}} |x|^2)^{-M}$. We will thus define a family $a_{\tilde{r}}$, $\tilde r \geq 0$, (again sticking to the case $s=r=0$ for notational simplicity; note that the $s$ parameter is actually irrelevant for us if $\gamma$ is at finite frequency, and only the $r$ is relevant) by
\begin{equation}\label{eq:reg x}
    a_{\tilde{r}} = a (1 + {\tilde{r}} |x|^2)^{-M}. 
\end{equation}
This will allow us to apply the regularized operators $A_{\tilde{r}} = (\Op_L(a_{\tilde{r}}) + \Op_L(a_{\tilde{r}})^*)/2$, etc, to $u$ (notice that although $u$ is only assumed to be in $H^{-N, -N}$ the lack of differential regularity is not a problem as our operators are all microsupported near $\gamma$, and in particular away from frequency infinity; that is, they can be taken to be order $-\infty$ in the differential sense, so only the spatial order needs to be improved). Then, when we compute $H_p a_{\tilde{r}} + p_1 a_{\tilde{r}}$, we obtain an extra term from differentiating the regularizer. 

To compute it, recall that for $P \in \Psisc^{1,1}$ classical, the principal symbol at spatial infinity, and away from fibre infinity, may be taken homogeneous of degree 1 under spatial dilations $(x, \xi) \mapsto (ax, \xi)$. Its Hamilton vector field $H_p$ is, correspondingly, homogeneous of degree 0 under spatial dilations. Let us extend coordinates $(z_1, z')$ into the interior by letting them be homogeneous of degree 0 under spatial dilations, and choose boundary defining function $\rho_b = 1/|x|$ for the spatial boundary. In these coordinates $(z_1, z', \rho_b)$ in the ambient space, $H_p$ takes for the form 
\begin{equation}\label{eq:H_p hom}
H_p = \frac{\partial}{\partial z_1} + q(z_1, z') \rho_b \frac{\partial}{\partial \rho_b}
\end{equation}
for some smooth, homogeneous of degree zero function $q$. Therefore, 
$$
H_p (1 + {\tilde{r}} |x|^2)^{-M} = \frac{2M q {\tilde{r}} |x|^2}{(1 + {\tilde{r}} |x|^2)} (1 + {\tilde{r}} |x|^2)^{-M} := f_{\tilde{r}} (1 + {\tilde{r}} |x|^2)^{-M}.
$$
One can check that the factor $f_{\tilde{r}} = 2M q {\tilde{r}} |x|^2/(1 + {\tilde{r}}|x|^2)$ is a symbol of order $(0,0)$, uniformly as ${\tilde{r}} \to 0$. Thus, this additional term is similar to the additional term from varying the order of the symbol $a$: it in effect changes $p_1$ to a new symbol of order $(0,0)$, although now it is ${\tilde{r}}$-dependent (but in a uniform way). We notice that it is proportional to $M$, which could be large. However, we first fix $M$ (sufficiently large depending on $N$, where $u \in H^{-N, -N}$) and \emph{then} choose $\digamma$ large enough so that the square root defining $b_{\tilde{r}}$ is well-defined and smooth. It is a subtle but important technical point that we can only do a finite (although arbitrarily large) amount of regularization in this argument. 

The upshot of this is that we can repeat the previous construction with ${\tilde{r}}$-dependent operators and obtain families of operators $A_{\tilde{r}}, B_{\tilde{r}}, E'_{\tilde{r}}, R_{\tilde{r}}$ satisfying the identity (in the case of general orders $(s,r)$, but assuming that $\gamma$ is located away from frequency infinity)
\begin{equation}\label{eq:operatorAr}
    i(P^*A_{\tilde{r}} - A_{\tilde{r}}P) = - \tilde B_{\tilde{r}}^* \tilde B_{\tilde{r}} - (\Lambda^{-s,-r} A_{\tilde{r}})^* \Lambda^{-s,-r} A_{\tilde{r}} + E_{\tilde{r}}' + R_{\tilde{r}}, 
\end{equation}
where these operators satisfy, for every $\eta > 0$, and every $S \in \R$, 
\begin{itemize}
    \item $A_{\tilde{r}} \in \Psisc^{S, 2r}$ uniformly, and $A_{\tilde{r}} \in \Psisc^{S, 2r-2M}$ for each $r > 0$, 
    \item $\tilde B_{\tilde{r}} \in \Psisc^{S, r}$ uniformly, and $\tilde B_{\tilde{r}} \in \Psisc^{S, r-M}$ for each $r > 0$, $\tilde B_{\tilde{r}} \to \tilde B$ in $\Psisc^{S, r + \eta}$, 
    \item $E'_{\tilde{r}} \in \Psisc^{S, 2r}$ uniformly, and $E'_{\tilde{r}} \in \Psisc^{S, 2r-2M}$ for each $r > 0$, $E'_{\tilde{r}} \to E'$ in $\Psisc^{S, 2r + \eta}$, 
    \item $R_{\tilde{r}} \in \Psisc^{S, 2r-1}$ uniformly, and $R_{\tilde{r}} \in \Psisc^{S, 2r-1-2M}$ for each $r > 0$, $R_{\tilde{r}} \to R$ in $\Psisc^{S, 2r -1 + \eta}$.
\end{itemize}
Because of the regularization, we can now take expectation values with $u \in H^{-N, -N}$ (for sufficiently large $M$), and we find that 
\begin{equation}\label{eq:reg-est}
     \|  \tilde B_{\tilde{r}} u \|_{L^2}^2  \leq C \Big(  \| GPu \|_{H^{s,r}}^2 
    +  \ang{E_{\tilde{r}}'u, u}   + \| u \|_{H^{-N, -N}}^2 \Big),
\end{equation}
We now `take a limit' as ${\tilde{r}} \to 0$, or more precisely investigate the uniform properties of this inequality in this limit. Intuitively, we expect that the term $\ang{E_{\tilde{r}}'u, u}$ should be uniformly bounded, since $u$ is in $H^{s,r}$ microlocally on $\WF'(E')$. This is correct, and is one of the exercises for this lecture. Thus the RHS of \eqref{eq:reg-est} is uniformly bounded, and it follows that $\tilde B_{\tilde{r}} u$ is uniformly bounded in $L^2$, as ${\tilde{r}} \to 0$. We can therefore find a sequence $r_j$ tending to zero such that $\tilde B_{r_j} u$ converges weakly, say to $v \in L^2$. Similarly to last lecture, we have 
\begin{equation}\label{eq:vineq}
  \| v \|_{L^2}^2 \leq \limsup_{{\tilde{r}} \to 0}    \|  \tilde B_{\tilde{r}} u \|_{L^2}^2  \leq C \Big(  \| GPu \|_{H^{s,r}}^2 
    +  \|  Eu \|_{H^{s,r}}^2  + \| u \|_{H^{-N, -N}}^2 \Big). 
\end{equation}
On the other hand, we know that $\tilde B_{\tilde{r}} u$ converges to $\tilde Bu$ in a weak Sobolev norm, since $\tilde B_{\tilde{r}} \to \tilde B$ in $\Psisc^{s, r + \eta}$. So we have $\tilde B_{\tilde{r}} u$ converges to both $\tilde Bu$ and $v$ distributionally. By uniqueness of distributional limits, $\tilde Bu = v$ so we have the inequality \eqref{eq:vineq} for $\tilde Bu$ in $L^2$, or equivalently for $Bu$ in $H^{s, r}$. This completes the proof of Theorem~\ref{thm:mpe}.

\subsection{Exercises}
\begin{problem}\label{prob:micro ell estimate} Let $E' \in \Psisc^{2s,2r}$ and suppose that $E \in \Psisc^{0,0}$ is elliptic on $\WF'(E')$. Show that for all $N$ there is a constant $C_N$ such that, for all $u \in H^{s,r}$, 
    $$
    |\ang{E'u, u}| \leq C_N \Big( \| Eu \|_{H^{s,r}}^2 + \| u \|_{-N, -N}^2 \Big). 
    $$
    Hint: find an operator $F \in \Psisc^{0,0}$ such that $E' = (\Lambda^{s,r}E)^* F \Lambda^{s,r}E + R$, where $R$ has sufficiently negative orders. 
\end{problem}

\begin{problem} Extend the previous result as follows: Suppose that $\sigma_L(E'_{\tilde{r}}) = \sigma_L(E') (1+ {\tilde{r}} |x|^2)^{-M}$. Prove the above inequality with $E'_{\tilde{r}}$ in place of $E'$ on the LHS, with a constant $C_N$ independent of ${\tilde{r}}$ for all $\tilde r \in (0, 1]$. Hint: use the fact that the operator norm of $F$ is controlled by a suitable $S^{0,0}_{\sca}$-norm of its left-reduced symbol. 
\end{problem}

\begin{problem} Show that the function $\chi_0$ in \eqref{eq:chi_0} is $C^\infty$ at $t=0$. Moreover, $t^{-k} \chi_0(t)$ is $C^\infty$ for all $k \in \mathbb{N}$. 
\end{problem}

\begin{problem}
Consider the coordinates $(z_1, z', \rho_b)$ defined above \eqref{eq:H_p hom}. Show that any vector field that is homogeneous of degree zero under spatial dilations takes the form 
$$
q_1 \partial_{z_1} + \sum q'_j \cdot \partial_{z'_j} + q \rho_b \partial_{\rho_b},
$$
where the coefficients $q_1, q'_j, q$ are independent of $\rho_b$. 
\end{problem}

\section{Lecture 7: Scattering calculus on manifolds (with boundary)}

We introduce the scattering calculus on manifolds, first introduced in this setting by Melrose \cite{melrose1994spectral}, and discuss its basic properties in this part. Some of our discussion is derived from the lecture notes of Peter Hintz, which is now a part of \cite{Hintz-book}.

We only deal with operators acting on scalar functions with details, but a significant portion of those general statements in this section have fairly straightforward generalization to vector bundles. Of course, difficulties in dealing with vector bundles will arise in concrete problems.

\subsection{The scattering cotangent bundle} \label{sec:sc-bundle}
Let $M$ be a $n$-dimensional manifold with boundary and denote its boundary by $\partial M$ and its interior by $M^\circ$. 
We have studied the case $M = \overline{\R^n}$ with $\partial M = \mathbb{S}^{n-1}$ in Lecture~\ref{sec:lecture1}.
 The smooth structure of $M$ determines a (equivalent class of) boundary defining function $\msf{x}$. (We will use $\msf{x}$ to denote such a boundary defining function, with the sans-serif font to distinguish it from the Euclidean $x \in \R^n$. Furthermore, we will use $z = (z_1, \dots, z_n)$ instead of $(x_1, \dots, x_n)$ for Euclidean coordinates on $\R^n$, and $\zeta = (\zeta_1, \dots, \zeta_n)$ for the dual coordinates, to avoid confusion with $\msf{x}$.) 
All of our bundles, symbol classes, calculi will be the same as the classical ones on manifolds without boundary on the interior of $M$, but not uniformly down to $\partial M$.
Suppose $y = (y_1,...,y_{n-1})$ is a coordinate system near $p \in \partial M$, then we can extend the $y_i$ into a collar neighbourhood of $\partial M$ such that 
\begin{align*}
(\msf{x},y_1,...,y_{n-1})
\end{align*}
forms a coordinate system of $M$ near $p$. 
A local frame of the space\footnote{Rigorously speaking, generators of the left $C^\infty(M)$-module.} of vector fields that are tangent to $\partial M$ will be
\begin{align} \label{eq:b-vec-frame}
\msf{x}\partial_{\msf{x}}, \, \partial_{y_1}, \, ... \; \partial_{y_{n-1}}.
\end{align}
We denote their $C^\infty(M)$-span by\footnote{Be sure that you appreciate the huge difference between the $C^\infty(M)$-span and the $C^\infty(M^\circ)$-span!} $\mathcal{V}_{\rmb}(M)$, which is called the space of b-vector fields on $M$. 
A direct computation shows that this is a Lie algebra with the Lie bracket being the standard commutator:
\begin{align} \label{eq:commutator_defn}
[V_1,V_2]f = V_1V_2f - V_2V_1f, \; f \in C^\infty(M).
\end{align}
Then we define the space of scattering vector fields to be
\begin{align} \label{eq:sc-vec-defn}
\mathcal{V}_{\sct}(M) = {\msf{x}} \mathcal{V}_{\rmb}(M).
\end{align}
From \eqref{eq:b-vec-frame} and the definition, we know that locally $\mathcal{V}_{\sct}(M)$ is the span of
\begin{align} \label{eq:sc-vec-frame}
{\msf{x}}^2\partial_{\msf{x}}, \, {\msf{x}}\partial_{y_1}, \, ...\; {\msf{x}}\partial_{y_{n-1}}.
\end{align}
In the case $M = \overline{\R^n}$, with $z_i$ being Euclidean, coordinates on $\R^n$, one can calculate that the frame $\partial_{z_i}$ will be equivalent to the frame ${\msf{x}}^2\partial_{\msf{x}}, \; {\msf{x}}\partial_{y_i}$ in the sense that the linear transformations between them have uniformly bounded and smooth coefficients in both directions. 
For example, in the region of $\overline{\R^n}$ such that $z_n$ dominates all other components and $z_n>0$, we can take 
\begin{equation} \label{eq:z-xy-transform}
{\msf{x}} = \frac{1}{z_n}, \; y_i = \frac{z_i}{z_n},\, i = 1,2,..\,n-1
\end{equation}
and it is straightforward to verify this claim by computing one frame in terms of the other. 
See Problem~\ref{ex:Rn-sc-vec} at the end of this lecture. 

This $\mathcal{V}_{\sct}(M)$ again a Lie algebra with the Lie bracket being the commutator as in \eqref{eq:commutator_defn}. In fact, it is even better, in the sense that
\begin{align*}
[\mathcal{V}_{\sct}(M),\mathcal{V}_{\sct}(M)]
=[{\msf{x}}\mathcal{V}_{\rmb}(M),{\msf{x}}\mathcal{V}_{\rmb}(M)] \subset {\msf{x}}^2\mathcal{V}_{\rmb}(M) = {\msf{x}}\mathcal{V}_{\sct}(M),
\end{align*}
which means commutators has one order better decay compared with $\mathcal{V}_{\sct}(M)$ and this corresponds to the more general fact that the scattering calculus is commutative on the top level in both the differential and decay sense, and this is one of the major reasons that makes it more tractable\footnote{Of course, when it is applicable!} compared with other calculi (for example, the b-calculus) when one wants to obtain the Fredholm property or invertibility.

The Lie algebra of scattering vector fields uniquely determines\footnote{See Problem~\ref{ex:sc-bundle-construction}.} a vector bundle ${}^{\sct}TM$, which is called the \emph{scattering tangent bundle}. The bundle of importance for our analysis is its dual bundle ${}^{\sct}T^*M$, which is called the \emph{scattering cotangent bundle}.
Its local frame\footnote{Again, by frame below, we mean the generator of a left $C^\infty(M)$-module.} is given by
\begin{align} \label{eq:sc-cotangent-frame}
\frac{d{\msf{x}}}{{\msf{x}}^2}, \, \frac{dy_1}{{\msf{x}}}, \, ... \; \frac{dy_{n-1}}{{\msf{x}}}.
\end{align}
This is going to be the phase space in which we work and coefficients in terms of this frame gives the `scattering frequencies'.
Concretely, this means that  we write the (extension of) tautological one-form as
\begin{equation}
\alpha =  \tau \frac{d{\msf{x}}}{{\msf{x}}^2} + \mu \cdot \frac{dy}{{\msf{x}}},
\end{equation}
where $\mu = (\mu_1,\, ... \; \mu_{n-1})$ and $\mu \cdot \frac{dy}{{\msf{x}}} = \sum_{i=1}^{n-1} \mu_i \frac{dy_i}{{\msf{x}}}$.

In the case $M = \overline{\R^n}$, they are closely related to the frequency in classical Fourier analysis. 
Let $(z,\zeta)$ be coordinates of $T^*\R^n$, then one has
\begin{equation} \label{eq:tautological-one-form}
\tau \frac{d{\msf{x}}}{{\msf{x}}^2} + \mu \cdot \frac{dy}{{\msf{x}}} = \zeta \cdot dz,
\end{equation}
with ${\msf{x}} = |z|^{-1}$, say,  for $|z|$ large and $y \in \R^{n-1}$ is obtained from a local parametrization of $\mathbb{S}^{n-1}$. So one can think of $\tau$ as minus the radial component of the classical frequency while $\mu$ is the tangential component.

As we did on $\R^n$ in Lecture~\ref{sec:lecture1}, to facilitate the high frequency analysis, we will compactify the fiber to be $\overline{\R^n}$ and denote the corresponding bundle by $\overline{{}^{\sct}T^*M}$. We call it the \emph{compactified scattering cotangent bundle}, which is a $n$-ball bundle over $M$. 
Its boundary consists of two parts: its restriction to $\partial M$, which we denote by $\overline{{}^{\sct}T^*_{\partial M}M}$, and the `fiber infinity', which is locally of the form $U \times \partial \overline{\R^n}$, where $U$ is an open set in $M$ on which you can trivialize ${}^{\sct}T^*M$.
We will use $\rho_{\mathrm{df}}$ ($\mathrm{df}$ stands for `differential face') to denote a boundary defining function of this boundary hypersurface, which is just the boundary defining function of $\overline{\R^n}$ formed by compactifying the $\R^n$ of $(\tau,\mu)$.
Even when $M$ is a manifold with boundary having no corner, this will have a corner formed by the intersection of these two parts.


\begin{remark}
If one is familiar with the so-called double space construction of those calculi, then in general the conceptually correct phase space to work with should be the conormal bundle of the lifted diagonal of the appropriate double space. 
See \cite[Lemma~4.6]{Melrose-APS} for the construction of the double space in the case of b-calculus and \cite[Section 21]{melrose1994spectral} for the scattering calculus.  This is a quite general principle that applies to many settings. 
\end{remark}

\subsection{The symbol class and the quantization}
\label{subsec:symbol-quantization}

We will define the scattering symbol class and the corresponding pseudodifferential algebra in this lecture.

As we have seen in the case $M = \overline{\R^n}$ in Lecture~\ref{sec:lecture1}, compared with the classical symbol class, the main property of the scattering symbol class is that it gains frequency/spatial decay when you differentiate in frequency/spatial variables. 

We will define scattering symbols in terms of smooth functions on the interior of $\overline{{}^{\sct}T^*M}$ with certain prescribed growth or decay rate when we approach $\partial M$ or fiber infinity, but of course, one should think of them as objects living on this compactified phase space.


Let $M^\circ$ be the interior of $M$ and we denote the restriction of ${}^{\sct}T^*M$ to $M^\circ$ by ${}^{\sct}T_{M^\circ}^*M$.
The space of scattering symbols of order $(m,r)$, which we denote by $S^{m,\el}_{\sct}(M)$, consists of smooth functions $a$ on ${}^{\sct}T_{M^\circ}^*M$ such that in a coordinate chart:
\begin{equation} \label{eq:sc-symbol-def}
|({\msf{x}}\partial_{\msf{x}})^{\alpha}\partial_y^\beta \partial_\tau^\gamma \partial_{\mu}^\delta \partial_{\mu} a({\msf{x}},y,\tau,\mu) |  \leq C_{\alpha\beta\gamma\delta} \la \tau,\mu \ra^{m-\gamma-|\delta|} {\msf{x}}^{-\el}.
\end{equation}
We will call $m$ the differential order, and $\el$ the decay\footnote{In fact, this is a growth order on the symbol or operator side if one think about how the requirement changes as $r$ increases. It will become clear that this is a good name after we define corresponding Sobolev spaces.} order.
For fixed indices, the infimum\footnote{We fix an atlas of coordinate charts that is locally finite.} of $C_{\alpha\beta\gamma\delta}$ on the right hand side gives a seminorm on this symbol class. All those seminorms together gives the Fr\'echet topology on $S^{m,\el}_{\sct}(M)$.

Recalling the equivalence between $\msf{x}^2\partial_{\msf{x}},\msf{x}\partial_y$ and $\partial_{z_i}$ mentioned before \eqref{eq:z-xy-transform},
the requirement in \eqref{eq:sc-symbol-def}, i.e. remaining the same growth under repeated application of ${\msf{x}}\partial_{\msf{x}}, \; \partial_{y_i}$, which are generators (say, as a left $C^\infty(M)$-module) of vector fields that are tangent to $\partial M$,
is indeed the same as gaining one order decay in $\la z \ra^{-1}$ when we apply $\partial_{z_i}$, which is our definition in \eqref{eq:sc-symbol-Rn}.

We can also interpret \eqref{eq:sc-symbol-def} in terms of vector fields on $\overline{{}^{\sct}T^*M}$ tangent to the boundary, as discussed in Section~\ref{subsec:compactification} in the case $M = \overline{\R^n}$. To do so it is more transparent to rewrite \eqref{eq:sc-symbol-def} in the form 
\begin{equation} \label{eq:sc-symbol-def-2}
|({\msf{x}}\partial_{\msf{x}})^{\alpha}\partial_y^\beta (\ang{\tau, \mu}\partial_\tau)^\gamma (\ang{\tau, \mu}\partial_{\mu})^\delta \partial_{\mu} a({\msf{x}},y,\tau,\mu) |  \leq C_{\alpha\beta\gamma\delta} \la \tau,\mu \ra^{m} {\msf{x}}^{-\el}.
\end{equation}
Here, all the vector fields on the LHS are tangent to the boundary of $\overline{{}^{\sct}T^*M}$ and we have a fixed weighted $L^\infty$ bound on the RHS. Such regularity (with respect to vector fields tangent to the boundary) is sometimes referred to as conormal regularity, thus, symbols have infinite conormal regularity with respect to a given weighted $L^\infty$ space. 

Let $\mathcal{S}(M^\circ) = \bigcap_{N \in \mathbb{N} } {\msf{x}}^NC^\infty(M)$ be the Schwartz function class on $M^\circ$.\footnote{This is indeed the usual Schwartz function class $\mathcal{S}(\R^n)$ when $M = \overline{\R^n}$.} 
Then for $a \in S^{m,\el}_{\sct}(M)$ that is supported in a single coordinate chart, the (left) quantization of $a \in S_{\sct}^{m,\el}(M)$ is defined to be the operator acting on functions in $\mathcal{S}(M^\circ)$ and supported in the same chart as follows. If the coordinate patch does not meet the boundary of $M$ then in terms of local coordinates $z_1, \dots, z_n$ on this patch, we define 
\begin{align} \label{eq:quant-sc-0}
\operatorname{Op}(a)f(z) = 
(2\pi)^{-n} \int e^{i (z-z') \cdot \zeta } a(z, \zeta) f(z')  dz' d\zeta,
\end{align}
while if the coordinate patch is localized near $\partial M$ with coordinates $(\msf{x}, y)$ as above, then we define
\begin{align} \label{eq:quant-sc}
\operatorname{Op}(a)f(x,y) = 
(2\pi)^{-n} \int \int e^{i (\tau \frac{{\msf{x}'}-{\msf{x}}}{{\msf{x}}{\msf{x}}'} + \mu \cdot (\frac{y}{{\msf{x}}} - \frac{y'}{{\msf{x}}'}) ) } a({\msf{x}},y,\tau,\mu)  f({\msf{x}}',y') \frac{d{\msf{x}}'dy'}{({\msf{x}}')^{n+1}} d\tau d\mu. 
\end{align}
And it acts on functions with support disjoint from this chart by a kernel that is a Schwartz function on $M^\circ \times M^\circ$. Although this definition depends on the choice of coordinates, one can verify that the \emph{class} of such operators, as $a$ ranges over $S^{m,\el}_{\sct}(M)$, is independent of the choice of coordinates modulo Schwartz functions on $M^\circ \times M^\circ$. 
See \cite[Chapter~5]{Hintz-book} for the Kuranishi trick that is used to prove such invariance. It is also from this change of variables involved in the Kuranishi trick that it becomes clear that we should view frequencies as living in the cotangent bundle, and specifically, the \emph{scattering} cotangent bundle near $\partial M$. 

The volume form used in \eqref{eq:quant-sc} near $\partial M$ is given by wedging together the frame \eqref{eq:sc-cotangent-frame}. More geometrically, one could specify a smooth non-degenerate metric on the scattering cotangent bundle, and take the corresponding Riemannian volume, which will be of this form. Melrose defined a `scattering metric' to be such a metric, specifically of the form  
\begin{align} \label{eq:sc-metric}
g = \frac{d{\msf{x}}^2}{{\msf{x}}^4} + \frac{h({\msf{x}},dy)}{{\msf{x}}^2},
\end{align}
where $h$ is a metric on $\partial M$ depending on ${\msf{x}}$ smoothly. They are also called `asymptotically conic metrics' in the literature. The reason for the name `asymptotically conic' is that when $h$ is independent of $\msf{x}$, then under the transformation $r = \msf{x}^{-1}$, \eqref{eq:sc-metric} becomes $g = dr^2+r^2h$, which is the metric of a cone. Since we required $h$ to be smooth in $\msf{x}$, $g$ in \eqref{eq:sc-metric} is asymptotic to such a cone with metric $h(0,dy)$ on the cross-section. It is, however, not necessary to have a metric specified to define the quantization \eqref{eq:quant-sc} above. 

Generally, we define $\Psi_{\sct}^{m,\el}(M)$ to be the space of operators having Schwartz kernels of the form
\begin{equation}
 \sum_{i} A_i + R,
\end{equation}
where $A_i$ acts like \eqref{eq:quant-sc} in various charts and $R \in \mathcal{S}(M^\circ) \otimes \mathcal{S}(M^\circ)$. That is, near the diagonal of $M \times M$, it acts by \eqref{eq:quant-sc} and off the diagonal, it acts by a kernel that is Schwartz in both variables.
This definition gives a local reduction to $\Psi_{\sct}^{m,\el}(\overline{\R^n})$ modulo $\mathcal{S}(M^\circ) \otimes \mathcal{S}(M^\circ)$. Now it might be illuminating to look at the phase $\tau \frac{{\msf{x}}-{\msf{x}}'}{{\msf{x}}{\msf{x}}'} + \mu \cdot (\frac{y}{{\msf{x}}} - \frac{y'}{{\msf{x}}'})$ in the $\overline{\R^n}$ case.
Using the transformation in \eqref{eq:z-xy-transform}, we have
\begin{equation}
\alpha = \tau \frac{d{\msf{x}}}{{\msf{x}}^2} + \mu \cdot \frac{dy}{{\msf{x}}} = 
(-\tau - \frac{\mu \cdot \tilde{z}}{z_n})dz_n + \sum_{i=1}^{n-1}\mu_idz_i,
\end{equation}
where $\tilde{z}=(z_1,..,z_{n-1})$.
Comparing with $\alpha = \zeta \cdot dz$, we see
\begin{equation}
\zeta_i = \mu_i,\, 1 \leq i \leq n-1, \; \zeta_n = -\tau - \frac{\mu \cdot \tilde{z} }{z_n}.
\end{equation}
Then we have
\begin{equation*}
\zeta \cdot (z-z') = \mu \cdot (\frac{y}{{\msf{x}}} - \frac{y'}{{\msf{x}}'})
+ (-\tau - \mu \cdot y ) (\frac{1}{{\msf{x}}} - \frac{1}{{\msf{x}}'}),
\end{equation*}
which agrees with the phase in \eqref{eq:quant-sc} after a change of coordinates in frequencies that is smooth in both directions. Moreover, at any given point of $\partial M$ we can arrange coordinates so that $y=0$ at that point, and then we have the simple correspondence $\zeta_i = \mu_i$, $1 \leq i \leq n-1$ and $\zeta_n = -\tau$ (but only at that one point, not in a neighbourhood). 

In addition, this phase is the reasonable one to use by composing elements in $\mathcal{V}_{\sct}(M)$ from the left. 
If we apply ${\msf{x}}^2D_{\msf{x}}, \, {\msf{x}}D_y$, then it brings down an $\tau,\, \mu$ factor respectively.

We will call $\Psi_{\sct}(M) = \bigcup_{m,\el} \Psi_{\sct}^{m,\el}(M)$ the scattering pseudodifferential algebra on $M$.
We will discuss its structure as a multi-graded algebra in the next subsection.
Lastly, we set
\begin{equation}
\Psi_{\sct}^{-\infty,-\infty}(M) = \bigcap_{m,\el} \Psi_{\sct}^{m,\el}(M),
\end{equation}
which consists of operators with kernels being Schwartz functions on $M^\circ \times M^\circ$, and we will call them residual.

Classical symbols are defined in the same manner as the $\R^n$-case:
\begin{equation} 
	S_{\mathrm{sc, cl}}^{m,\el}(M) =  \rho_{\mathrm{df}}^{-m} {\msf{x}}^{-\el}  C^\infty(\overline{{}^{\sct}T^*M}) \subset S_{\mathrm{\sct}}^{m,\el}(M).
\end{equation} 
If $a$ is classical, then $A=\operatorname{Op}(a)$ is called classical.

Now we turn to the global quantization map. Considering the setting of a general manifold $M$ without any restriction will face the issue that one need to compare the volume growth of the manifold with the off-diagonal decay of the kernel or other objects one is summing/integrating over. In this notes we make the following assumption:
\begin{center}
$M$ is compact.
\end{center}
However, be aware that the actual physical space is the interior of $M$, which is of course non-compact.
For example, in the $\R^n$-case studied in Lecture 1-6, $M = \overline{\R^n}$ is a compact ball but $M^{\circ} = \R^n$. 



As we will see in Theorem~\ref{thm:Groenewold}, there can't be a `perfect' quantization. So we just define it via gluing \eqref{eq:quant-sc}. Concretely, choose a (finite) cover $U_i$  of $M$ by open sets, and let $\phi_i: U_i \to \overline{\R^n}$ or $U_i \to \overline{\R^n_+}$, if $U_i$ meets $\partial M$, be a coordinate chart. Let $\chi_i$ be a partition of unity that is subordinate to this cover and let $\tilde{\chi} \in C_c^\infty(U_i)$ be identically $1$ on $\supp \chi_i$.
Then we define
\begin{equation} \label{eq:sc-quant-global}
\operatorname{Op}(a) = \sum_i \operatorname{Op}(\chi_ia)\tilde{\chi}_i,
\end{equation}
where $\operatorname{Op}(\chi_i a)$ is defined via \eqref{eq:quant-sc}\footnote{Here we packaged pull-backs, trivialization induced by $\phi_i$ into the definition of $\operatorname{Op}$.} and it can act on $\tilde{\chi}_i f$ since it has support contained in $U_i$.
Here the localizer $\tilde{\phi}_i$ differs with $\phi_i$ because in \eqref{eq:quant-sc} we allow $f$ to have support slightly larger than $a$. Keep in mind that $\operatorname{Op}$ is far from canonical, as it depends on the choice of open sets, coordinate charts and partitions of unity. For simplicity we do not indicate this in notation. Despite this lack of uniqueness, we will be able to define a canonical principal symbol map in Definition~\ref{def:sc-pr-symbol} below.

\subsection{Basic properties of the calculus} \label{subsec:sc-basic-property}


First we define the principal symbol of scattering pseudodifferential operator.
Consider $A_{\mathrm{\mathrm{loc}}}=\chi A \chi$ with $\chi \in C_c^\infty(M)$ with support contained in a single coordinate chart $\phi: U \to U' \subset \overline{\R^n}$ and is identically $1$ on some smaller $V \subset U$ (of course, it will be $1$ on $\overline{V}$ then). 
By the invariance of $\Psi_{\sct}$ discussed after \eqref{eq:quant-sc},
we know that $A_{\mathrm{\mathrm{loc}}}$ can be identified with 
\begin{equation}
   (\phi^{-1})^* A_{\mathrm{\mathrm{loc}}} \phi^*,
\end{equation}
which is a scattering pseudodifferential operator on $\overline{\R^n}$ (verify this!).
So locally this $(\phi^{-1})^* A_{\mathrm{\mathrm{loc}}} \phi^*$ can be written as $\operatorname{Op}(a_{U'})$ for some $a_{U'} \in S^{m,\el}(\overline{\R^n})$.
Finally, we set
\begin{equation}
a_V = \tilde{\phi}^*a_{U'}|_{{}^{\sct}T^*_VM},
\end{equation}
where $\tilde{\phi}:{}^{\sct}T_U^*M  \to \phi(U) \times \R^n$ is the trivialization induced by $\phi$. And the equivalent class
\begin{equation}
[a_V] \in S^{m,\el}_{\sct}(V) / S^{m-1,\el-1}_{\sct}(V)
\end{equation}
is independent of the choice of $\chi$ and the coordinate system $\phi$, by checking the `coordinate invariance' as in the classical setting.\footnote{This is not very trivial. Read Section~6.1 of Hintz's notes if you haven't seen this.} 
This defines the principal symbol of $A$ locally and the actual `global' principal symbol will be the equivalent class modulo $S^{m-1,\el-1}_{\sct}(M)$ obtained by gluing those $a_{U'}$.
It is not hard to check the following fact: if $V_1 \subset V$, then we have
\begin{equation}
[a_V]|_{T^*_{V_1}M} = [a_{V_1}].
\end{equation}

\begin{definition}\label{def:sc-pr-symbol}
The principal symbol of $A \in \Psi_{\sct}^{m,\el}(M)$ is the unique equivalent class
\begin{equation}
\sigma^{m,\el}(A) \in S^{m,\el}_{\sct}(M)/S^{m-1,\el-1}_{\sct}(M)
\end{equation}
such that: let $a \in S^{m,\el}_{\sct}(M)$ be any representative of it and $V, \, [a_V]$ as above, then $[a|_{{}^{\sct}T^*_V M}]=[a_V]$ in $S^{m,\el}_{\sct}(V) / S^{m-1,\el-1}_{\sct}(V)$.
\end{definition}

\begin{proof}

The uniqueness is easy: the restriction of $\sigma^{m,\el}(A)$ to any coordinate chart is uniquely determined.

To show the existence, we use a partition of unity $\{ \phi_i \}$ subordinate to a locally finite\footnote{It can be taken to be finite by our assumption.} open cover $\{V_i\}$ as in the definition of $[a_V]$ above.
Now we take a representative $a_{V_i}$ for each $[a_{V_i}]$ and set
\begin{equation}
a = \sum_i \phi_i a_{V_i}, \quad \sigma^{m,\el}(A) = [a] \in S_{\sct}^{m,\el}(M)/S_{\sct}^{m-1,\el-1}(M).
\end{equation}
Now we verify that $[a]$ satisfies desired properties. Take any $V$ in the definition, since for any $p \in V$, there is a $\phi \in C_c^\infty(V)$ that is identically $1$ near $p$, so we only need to show that for all $\phi \in C_c^\infty(V)$ we have
\begin{align*}
[(\phi a)|_{V}] = [(\phi a)_{V}].
\end{align*}
As aforementioned, making further restriction does not change the equivalent class, so we restrict $[\phi \phi_i a_{V_i}]$ and $[\phi\phi_i a_V]$ to $V_i \cap V$ to see that 
\begin{align*}
\phi \phi_i a_{V_i} = \phi \phi_i a_V + e_i,
\end{align*}
with $e_i \in S_{\sct}^{m-1,\el-1}(M)$ and is supported in ${}^{\sct}T^*_{V\cap V_i} M$.
Then we have
\begin{align*}
\phi a = \phi a_V + \sum_i e_i,
\end{align*}
while $\sum_i e_i \in S_{\sct}^{m-1,\el-1}(M)$, as desired.

\end{proof}

Now we turn to the ellipticity. For $a \in S^{m,\el}_{\sct}(M)$, it is called elliptic (in $S^{m,\el}_{\sct}(M)$) at $q \in \partial(\overline{{}^{\sct}T^*M})$ if 
\begin{align*}
|{\msf{x}}^\el\rho_{\mathrm{df}}^ma({\msf{x}},y,\tau,\mu)| \geq C>0
\end{align*}
in a neighborhood of $q$. We say $a$ is elliptic (in $S^{m,\el}_{\sct}(M)$) if it is elliptic at every point of $\partial(\overline{{}^{\sct}T^*M})$. Of course, this property descends to a property of $S^{m,\el}_{\sct}(M)/S^{m-1,\el-1}_{\sct}(M)$ and we say $A$ is elliptic if $\sigma_{\sct}^{m,r}(A)$ is elliptic.
This is equivalent to the following fact: $a \in S^{m,\el}_{\sct}(M)$ is elliptic if and only if there exists $b \in S^{-m,-\el}_{\sct}(M)$ such that
\begin{equation}
ab - 1 \in S^{-1,-1}_{\sct}(M).
\end{equation}
Also, $A$ is elliptic in $\Psi_{\sct}^{m,\el}(M)$ if and only if there is a $B \in \Psi_{\sct}^{-m,-r}(M)$ such that
\begin{equation} \label{eq:elliptic-1-parametrix}
AB - \mathrm{Id}, \; BA - \mathrm{Id} \in \Psi^{-1,-1}_{\sct}(M).
\end{equation}
In fact, the error term can be improved to be in $\Psi^{-\infty,-\infty}_{\sct}(M)$ and such $B$ is called the parametrix of $A$.

\begin{proposition}
The principal symbol map gives the following short exact sequence\footnote{Say, as left $C^\infty(M)$-modules. But this is not that important in our analysis.}:
\begin{align*}
0 \to \Psi_{\sct}^{m-1,l-1}(M) \to 
\Psi_{\sct}^{m,\el}(M) \xrightarrow{\sigma^{m,\el}(\cdot)} S^{m,\el}_{\sct}(M)/S^{m-1,\el-1}_{\sct}(M) \to 0.
\end{align*}
\end{proposition}
This means that $\sigma^{m,\el}(\cdot)$ captures the leading order behaviour in both the differential sense and the decay sense. If two operators have the same principal symbol, then the coincide up to $\Psi_{\sct}^{m-1,\el-1}(M)$-level (in the sense that their difference lies in here).

Then the proof of the composition law can be done by local reduction to the $\overline{\R^n}$-case. We state the conclusion below.
\begin{proposition}
$\Psi_{\sct}$ is a multi-graded $*$-algebra in the following sense (let $*$ denote the formal adjoint)
\begin{align} \label{sc_composition}
\begin{split}
A \in \Psi_{\sct}^{m,\el}(M)  & \implies A^* \in \Psi_{\sct}^{m,\el}(M),\\
A \in \Psi_{\sct}^{m_1,\el_1}(M), \; B \in \Psi_{\sct}^{m_2,\el_2}(M)  & \implies A \circ B \in \Psi_{\sct}^{m_1+m_2,\el_1+\el_2}(M).
\end{split}
\end{align}
In addition, the principal symbol map `preserves products':
\begin{equation} \label{eq:principal-product}
\sigma^{m_1+m_2,\el_1+\el_2}(AB) = \sigma^{m_1,\el_1}(A)\sigma^{m_2,\el_2}(B).
\end{equation}
Of course, the product on the right hand side is defined via taking representatives to form a product and then take the equivalence class again. One can check that this does not depend on the choice of representatives.

\end{proposition}

As we will see in the proof of propagation estimates, commutators will play a very important role in microlocal analysis. In scattering calculus, commutators has the following property that is analogous to the classical case:
\begin{proposition}
Let $A \in \Psi_{\sct}^{m_1,\el_1}(M), \; B \in \Psi_{\sct}^{m_2,\el_2}(M)$, then $$[A,B] \in \Psi_{\sct}^{m_1+m_2-1,\el_1+\el_2-1}(M)$$ and
\begin{equation} \label{eq:principal-sc-commutator}
\sigma^{m_1+m_2-1,\el_1+\el_2-1}(i[A,B]) =  \{ \sigma^{m_1,\el_1}(A), \sigma^{m_2,\el_2}(B) \}
\end{equation}
where $\{ \cdot, \cdot \}$ denotes the Poisson bracket. 
\end{proposition}

The proof comes down to local reduction to the $\R^n$ case and use the asymptotic expansion. 
The Poisson bracket on the right hand side is defined via taking representatives and compute using the definition that is almost the same as the $\R^n$-case which we recall below.  
For $p, a \in C^\infty({}^{\sct}T^*M)$, $\{p,a\}$ is defined via
\begin{equation}
\{ p , a \} = H_pa,  
\end{equation}
where $H_p$ is the Hamilton vector field uniquely determined by
\begin{equation} \label{eq:Hp-defn}
dp(H') = \omega_{\sct}(H_p,H'),
\end{equation}
for any smooth vector field $H'$ on ${}^{\sct}T^*M$. Here $\omega_{\sct}$ is the symplectic form
\begin{align}  \label{eq: omega sc}
\begin{split}
\omega_{\sct} & = -d(\tau \frac{d{\msf{x}}}{{\msf{x}}^2} + \mu \cdot \frac{dy}{{\msf{x}}} )
\\ &= - d \tau \wedge \frac{d{\msf{x}}}{{\msf{x}}^2} - d\mu \wedge \frac{dy}{{\msf{x}}} + {\msf{x}} \frac{d{\msf{x}}}{{\msf{x}}^2} \wedge \frac{\mu \cdot dy}{{\msf{x}}}.
\end{split}
\end{align}

Of course, both of the Poisson bracket and the Hamiltonian vector field are the same as the classical one in the interior and one can compute them in terms of local coordinates.
But notice that our contact form is the one in \eqref{eq:tautological-one-form}, not $\tau d{\msf{x}} + \mu \cdot dy$, so there are ${\msf{x}}$-factors involved and the expression in terms of those coordinates will look different from the classical case.


By condition \eqref{eq:Hp-defn} solving for the undetermined coefficients, we have
\begin{align} \label{eq: sc Hamilton vector}
\begin{split}
H_p = & \partial_{\tau}p ({\msf{x}}^2\partial_{\msf{x}}) + \partial_{\mu}p({\msf{x}}\partial_y)
-({\msf{x}}^2\partial_{\msf{x}}p + {\msf{x}}\mu \cdot \partial_{\mu}p) \partial_{\tau} 
\\ & -({\msf{x}}\partial_yp-{\msf{x}}(\partial_\tau p )\mu) \cdot \partial_{\mu}.
\end{split}
\end{align}

One point we would like to emphasize here is that one can see from \eqref{eq:principal-sc-commutator} that Hamiltonian dynamics will enter naturally if one wants to run arguments involving commutators (like propagation estimates) ---  as we have already done in Lecture~\ref{lec:propagation}.

 In fact, one can check that this Poisson bracket satisfies the Jacobi's identity and makes $C^\infty({}^{\sct}T^*M)$ a Lie algebra.
 On the other hand $[\cdot,\cdot]$ also makes $\Psi_{\sct}$ a Lie algebra. Then \eqref{eq:principal-sc-commutator} says that the principal symbol map (modulo the $i$-factor) preserves those structures on the top level. 
 Of course, one would like to ask can we preserve this structure completely?
The answer is no, even in a much more limited setting. This is known as Groenewold's theorem:

\begin{theorem} \label{thm:Groenewold}
On $\R^n$, as long as $S \subset C^\infty(T^*\R^n)$ includes symbols that is polynomial in positions and frequencies, then there is no quantization map $Q: S \to \Psi(\R^n)$ such that
$Q(x_j) = x_j,\, Q(\xi_j)=D_{x_j}$ and:
\begin{equation} \label{eq:7-9}
-i Q(\{f,g\}) = [Q(f),Q(g)].
\end{equation}
\end{theorem}
A sketch of the proof is as following: \eqref{eq:7-9} for $f,g$ being polynomials for the first two orders determines that it have to be the Weyl quantization, at least restricted to polynomials up to order three.
And then the quantized version of ($x_1,\xi_1$ stands for the first component of position and frequency respectively)
\begin{equation*}
x_1^2\xi_1^2 = \frac{1}{9} \{ x_1^3,\xi_1^3 \} = \frac{1}{3} \{x_1^2\xi_1,x_1\xi_1^2\}
\end{equation*}
will lead to contradiction.
See \cite[Section~13.4]{hall2013quantum} for details.



\subsection{Sobolev spaces, mapping properties}

Recall we have the Schwartz function class on $M$ defined by
\begin{equation}
\mathcal{S}(M^\circ) = \bigcap_{N \in \mathbb{N} } {\msf{x}}^NC^\infty(M),
\end{equation} 
and it is equipped with the Fr\'echet topology induced by this intersection (for example: $\sup_{({\msf{x}},y) \in \phi(U)} |{\msf{x}}^{-N}\partial_{{\msf{x}},y}^{\alpha}u|$ for a fixed set of charts $(U,\phi)$ forming an finite open cover of $M$).

Then we denote $\mathcal{S}'(M^\circ)$ to be the dual space of it (continuous linear functionals on it). It is called the space of tempered distributions on $M$ (or $M^\circ$).

A metric of the form \eqref{eq:sc-metric} gives a volume form, or more precisely a density $|d\nu|$,  that is a smooth multiple, bounded away from $0$, of $|\frac{d{\msf{x}}dy}{{\msf{x}}^{n+1}}|$ and this defines $L^2_{\sct}(M)$.
Concretely, let $\chi_i$ be a partition subordinate to coordinate charts $\{ (U_i,\phi_i) \}$ consists of finitely $U_i$ and we define
\begin{equation}
\| u \|_{L^2_{\sct}(M)} =  ( \sum_i \int |\chi_i u({\msf{x}},y)|^2 |d\nu| )^{1/2}
\end{equation}
for $u \in \mathcal{S}(M^\circ)$. Here $|d\nu|$ is $|\frac{d{\msf{x}}dy}{{\msf{x}}^{n+1}}|$ if we are near the boundary and if we are away from the boundary, then it is any smooth positive density.  
One can verify that this is a norm and $L^2_{\sct}(M)$ is defined to be the completion of $\mathcal{S}(M^\circ)$ with respect to this norm.

Let $\{ (U_i,\phi_i,\chi_i) \}$ be as above \eqref{eq:sc-quant-global}, and set 
\begin{equation}
\| u \|_{H_{\sct}^{s,r}(M)} =  \big( \sum_i \|(\phi_i^{-1})^*(\chi_i u) \|^2_{H_{\sct}^{s,r}(\overline{\R^n})} \big)^{1/2}.
\end{equation}
for $u \in \mathcal{S}(M^\circ)$, and again set $H_{\sct}^{s,r}$ to be the completion of $\mathcal{S}(M^\circ)$ with respect to this norm.
This is a Hilbert space with the inner product defined by
\begin{equation}
\la u, v \ra = \sum_i \la A_i (\phi_i^{-1})^*(\chi_i u), A_i (\phi_i^{-1})^*(\chi_i v) \ra,
\end{equation}
where $A_i$ is a invertible elliptic operator in $\Psi_{\sct}(\overline{\R^n})$ and $\la \cdot , \cdot \ra$ is the $L^2$-pairing with respect to $|dz|$ or $|\frac{d{\msf{x}}dy}{{\msf{x}}^{n+1}}|$ as above, depending on we are near the boundary or not.

The boundedness of $\Psi_{\sct}$-operators on Sobolev spaces are proved by reducing to the $\overline{\R^n}$ case, while the residual part with Schwartz kernel is easy to deal with. 
We only list results below for the convenience of reference afterwards.

\begin{proposition}
Let $A \in \Psi_{\sct}^{m,\el}(M)$, then for any $s,r \in \R$, $A$ is a continuous map from $H_{\sct}^{s,r}(M)$ to $H_{\sct}^{s-m,r-\el}(M)$:
\begin{equation} \label{eq:sc-PsiDO-boundedness-manifold}
\| Au \|_{ H^{s-m,r-\el}_{\sct}(M) }  \leq C \|u\|_{H_{\sct}^{s,r}},
\end{equation}
where the constant $C$ may depend on $A,m,\el,s,r$ and choices we made in the definition of norms.
\end{proposition}

As aforementioned, $A \in \Psi_{\sct}^{m,\el}(M)$ is elliptic if and only if there is a $B \in \Psi_{\sct}^{-m,-r}(M)$ such that
\begin{equation} \label{eq:elliptic-full-parametrix}
AB - \mathrm{Id}, \; BA - \mathrm{Id} \in \Psi^{-\infty,-\infty}_{\sct}(M).
\end{equation}
And this yields the elliptic estimate:

\begin{proposition}
Suppose $A \in \Psi_{\sct}^{m,\el}(M)$ is elliptic, then \eqref{eq:sc-PsiDO-boundedness-manifold} can almost be reversed with an error term:
\begin{equation} \label{eq:sc-elliptic-manifold}
\|u\|_{H_{\sct}^{s,r}}  \leq C \big( \| Au \|_{ H^{s-m,r-\el}_{\sct}(M) } + \|u\|_{H_{\sct}^{-N,-N}} \big).
\end{equation}
Here $N \in \R$ is arbitrary, but the useful case is when $N$ is large: $-N<s,r$.
\end{proposition}

Finally, we briefly introduce the scattering the wavefront sets for distributions and operators.
Let $A \in \Psi_{\sct}^{m,\el}(M)$, then $\WF'_{\sct}(A)$ is the subset in $\partial(\overline{{}^{\sct}T^*M})$ that captures locations and frequencies where $A$ is non-trivial, where trivial means acting with integral kernel that is a Schwartz function. Concretely, let $a$ be the full symbol of $A$ (potentially modulo a $\Psi_{\sct}^{-\infty,-\infty}$-term, according to our definition of the operator class) using the quantization \eqref{eq:sc-quant-global}, then for $q \in \partial(\overline{{}^{\sct}T^*M})$, we say that $q \notin \WF'_{\sct}(A)$ if there exists $\chi \in C^\infty(\overline{{}^{\sct}T^*M})$ such that $\chi(q)=1$ and $\chi a \in S_{\sct}^{-\infty,-\infty}(M)$, which is the space of Schwartz functions on $\overline{{}^{\sct}T^*M}$.
This is equivalent to that there is a $Q \in \Psi_{\sct}^{0,0}$ that is elliptic at $q$ such that $QA \in \Psi_{\sct}^{-\infty,-\infty}$, which means $A$ is acting with an integral kernel that is a Schwartz function microlocally at $q$

For $u \in \mathcal{S}'(M)$, the scattering wavefront set of order $s,r$, denoted by $\WF_{\sct}^{s,r}(u)$ is the subset of $\partial(\overline{{}^{\sct}T^*M})$ that captures locations and frequencies where $u$ fails (to have enough regularity of decay) to lie in $H_{\sct}^{s,r}(M)$. Let $q\in \partial(\overline{{}^{\sct}T^*M})$, we say $q \notin \WF_{\sct}^{s,r}(u)$ if  there is a $Q \in \Psi_{\sct}^{0,0}$ that is elliptic at $q$ such that $Qu \in H_{\sct}^{s,r}(M)$. One can also take union over those finite order wavefront sets to form a set that measures where $u$ fails to be trivial (i.e., being Schwartz):
\begin{align*}
\WF_{\sct}(u) =\overline{\bigcup_{s,r \in \Z} \WF_{\sct}^{s,r}(u)}.
\end{align*}





\subsection{Generalization to vector bundles} \label{subsec:vec-bundle}
The scattering algebra and pseudodifferential algebras constructed 
in previous chapters extends to operators between sections
of vector bundles naturally. We describe this process very briefly for the scattering algebra as given in \cite[Section~3]{Melrose1994}, 
and this transplants to other algebras in a verbatim manner. One major motivation of this generalization is the application to systems of PDEs and geometric rigidity theory.

Let $\pi_E: E \rightarrow M$ and $\pi_F: F \rightarrow M$ be two vector bundles over $M$
of rank $r_E$ and $r_F$ respectively, either complex or real. 
We define the space of scattering pseudodifferential operators from $E$ to $F$ 
with differential order $m$ and decay order $l$ as
\begin{align}
\Psi_{\sct}^{m,\el}(M;E,F):=C^\infty(X;\mathrm{Hom}(E,F)) \otimes \Psi_{\sct}^{m,\el}(M),
\end{align}
where the tensor product is over  $C^\infty(M)$,
viewing $C^\infty(X;\mathrm{Hom}(E,F))$ as a right $C^\infty(M)$-module
and $\Psi_{\sct}^{m,\el}(M)$ as a left $C^\infty(M)$-module.
Equivalently, an element in $\Psi_{\sct}^{m,\el}(M;E,F)$ is a matrix
with entries in $\Psi_{\sct}^{m,\el}(M)$. 
Suppose $e_1,...,e_{r_E}$ and $f_1,...,f_{r_F}$ are local frames of $E$ and $F$
over an open set $O$ respectively, then for any $K \Subset U$, we have
$P_{ij} \in \Psi_{\sct}^{m,\el}(M), 1 \leq j \leq r_E, 1 \leq i \leq r_F$ such that
\begin{align}
P(\sum_{j=1}^{r_E} \phi_j e_j) = \sum_{i,j} P_{ij}(\phi_j)f_i,
\end{align}
where $\phi_i \in C_c^\infty(O)$ and $\supp \phi_i \subset K$. We transposed
the matrix interpretation compared with notations in \cite{Melrose1994} to 
make it more compatible with basic linear algebra.

The space of corresponding symbols, denoted by $S_{\sct}^{m,\el}(M;E,F)$
is the space of $r_F \times r_E$ matrices with entries in
$S_{\sct}^{m,\el}(M)$, and the quantization map sending a symbol
to $\Psi_{\sct}^{m,\el}(M;E,F)$ is applying the quantization map
in the scalar case componentwise. The ellipticity is defined in the same way using \eqref{eq:elliptic-full-parametrix}, except for that now both $A,B$ are endomorphisms of bundles.

The Sobolev spaces $H^{s,r}_{\sct}(M;E)$ is defined using a partition of unity, a local trivialization of $E$, and a scattering connection $\nabla^{\sct}$, which gives the notion of differentiating sections of $E$ along scattering vector fields, i.e., sending sections of $E$ to sections of $E$ satisfying conditions for connections on vector bundles, but with $\mathcal{V}(M)$ replaced by $\mathcal{V}_{\sct}(M)$.
For a positive integer $s$, $H^s_{\sct}(M;E)$ is defined to be the space of sections of $E$ such that
it up to $s$-order derivatives using $\nabla^{\sct}$ and scattering vector fields are square integrable.
Here we assume that there is a fixed Hermitian inner product (and hence volume form) on $E$ to define the integration.
For general $s \in R$, we define for $s \geq 0$ by interpolation, and then use duality to define the $s<0$ case.
The weighted case is done by multiplying ${\msf{x}}^r$ componentwise, where ${\msf{x}}$ is the boundary defining function of $\partial M$.

The concept of the adjoint operator requires us to fix a density $\Omega$, and
define $A^*$ of $\Psi_{\sct}^{m,\el}(M;E,F)$ as an element
in $\Psi_{\sct}^{m,\el}(M;F^* \otimes \Omega, E^* \otimes \Omega)$
using the equation of pairing
\begin{align}
\la Au,v \ra = \overline{ \la u,A^*v  \ra},
\end{align}
where $u,v$ are sections of $E$ and $F^* \otimes \Omega$ respectively.
$a \in S_{\sct}^{m,\el}(M;E,F)$ is said to be elliptic at $(z,\zeta)$ if on a conic neighborhood
of this point there exists $b \in S_{\sct}^{-m,-l}(M;F,E)$ such that $ b \circ a- \Id_E \in S^{-1,-1}_{\sct}(M;E,E)$
and $a \circ b-\Id_F \in S_{\sct}^{-1,-1}(M;F,F)$.
Then the principal symbol the mapping properties, elliptic estimates, wavefront sets are similar to the scalar case.

\subsection{Problems}

\begin{problem}
 Check the claim we made in Section~\ref{subsec:symbol-quantization}: when $M = \overline{\R^n}$, the frame $\partial_{z_i}, \ 1 \leq i \leq n$ will be equivalent (in the sense that linear transformations between them have uniformly bounded and smooth coefficients in both directions) to the frame ${\msf{x}}^2\partial_{\msf{x}}, \; {\msf{x}}\partial_{y_j}, \ 1 \leq i \leq n-1$.   \label{ex:Rn-sc-vec}
\end{problem}

\begin{problem} Verify the claim that a scattering metric as in \eqref{eq:sc-metric} gives a metric on ${}^{\sct}TM$ and it is complete. Can we do something similar to ${}^{\sct}T^*M$?
\end{problem}

\begin{problem} In Section~\ref{ex:sc-bundle-construction}, we claimed that taking $\mathcal{V}_{\sct}(M)$ as local sections uniquely determines a vector bundle. Verify this. (Hint: For any $p \in M$, let $\mathcal{I}_p \subset C^\infty(M)$ be the ideal of smooth functions vanishing at $p$, then define the fiber \begin{align*}
{}^{\sct}T_pM = \mathcal{V}_{\sct}(M)/\mathcal{I}_p  \cdot \mathcal{V}_{\sct}(M).\end{align*} \label{ex:sc-bundle-construction}
Then verify that this gives a vector bundle. Maybe think about what is $C^\infty(M)/\mathcal{I}_p$ first.)
\end{problem}

\begin{problem} Verify the equivalence of two definitions of the ellipticity given in Section~\ref{subsec:sc-basic-property}.  \label{ex:elliptic-equivalence}
\end{problem}

\begin{problem} Derive \eqref{eq: sc Hamilton vector}. \label{ex:sc-Hamiltonian-vec}
\end{problem}

\begin{problem} In terms of the quantization map in \eqref{eq:sc-quant-global}, verify: $\operatorname{Op}(1) = \mathrm{Id}$. And this quantization map is almost surjective: $\Psi_{\sct}^{m,\el}(M) = \operatorname{Op}(S^{m,\el}_{\sct}(M)) + \Psi_{\sct}^{-\infty,-\infty}(M)$.
\end{problem}

\section{Lecture 8: Application to Inverse problems}

In this lecture we will discuss some applications of the scattering calculus to inverse problems.
In particular, we will discuss the result of Uhlmann and Vasy (with an Appendix by Zhou) \cite{uhlmann2016inverse}, but with a slightly different proof, which is closer to the presentation in later papers \cite{SUV2021local-gloabl}\cite{zachos2022inverting}\cite{jia2022tensorial}\cite{vasy2020semiclassical}\cite{jia20202d}. But we avoid introducing a quasi-homogeneous semiclassical calculus, which is used in some of those referred papers.
We can't offer a thorough literature review, but only refer to \cite[Section~1]{SUV2021local-gloabl}\cite{uhlmann2016journey} and references therein. 

\subsection{The X-ray transform}
\label{subsec:X-ray-intro}
Let $(X,g)$ be a Riemannian manifold with boundary. 
The geodesic X-ray transform is a generalization of the Radon transform, and the inverse problem on it can be formulated as follows: 
On a Riemannian manifold $(X,g)$, the information we have is integrals like 
\begin{equation}
(I_0 f)(\gamma(\cdot)) := \int_\gamma  f(\gamma(t))dt, 
\end{equation}
where 
$\gamma$ is a geodesic segment in a neighborhood $O_p$ of a fixed point $p \in \partial X$. Here $0$ stands for viewing $f$ as a tensor of rank $0$. The general case is defined via taking the integrand to be $f(\gamma(t))(\dot{\gamma}(t),...,\dot{\gamma}(t))$ when $f$ is a tensor field.
We should emphasize that here, unlike in Lecture~7, $g$ is a smooth function at the boundary and $X$ is geodesically incomplete.
Since we will be studying the local problem, one may assume that $X$ is compact. In this case, the distance from any interior point to $\partial X$ is finite.
Think of a Euclidean ball with the induced Euclidean metric as a typical example.

In this part, we consider the local geodesic ray transform with weights in dimension at least 3. `Local' means that the geodesic segment we integrate over lies in $O_p$ and has endpoints on $\partial X$, 
see Section \ref{sec_defn_transform} for the more detailed definition.
We show that we can recover the restriction of the function $f$ to $O_p$ from $I_0f$,
which amounts to the injectivity of $I_0$,
by proving an estimate which uses the Sobolev norm of $I_0 f$ (viewed as a function on the 
projective sphere bundle $PSX$) to control the Sobolev norm of $f$. 
In addition,
this estimate is stable under small perturbations of the metric $g$.

\subsection{Notations and results}
\label{notations}
\subsubsection{The set up}
Let $(X,g)$ be a Riemannian manifold with boundary. 
It is convenient to consider a larger region containing $X$. So suppose $X$ is embedded as a strictly convex domain in a Riemannian manifold $(\tilde{X},g)$ (we have used the same notation to indicate the smooth extension of the metric). Here convexity means when a geodesic is tangent to $\partial X$, it is tangent and curving away from $X$. 

Concretely, let $\bar{X}$ be the closure of $X$ in $\tilde{X}$ and let $\rho$ be the boundary defining function of $\bar{X}$, which means $\rho(z)$ vanishes on $\partial X$, $\rho(z)>0$ on $X$, and satisfy the non-degeneracy condition $d\rho \neq 0$ when $\rho=0$.
Using $G$ to denote the dual metric function on $T^*\tilde{X}$ and $H_G$ to denote its Hamilton vector field, the convexity means that if at some $\beta \in T_p^*\tilde{X}\backslash o$ with $p \in \partial X$ and $o$ being the zero section, we have: 
\begin{equation} \label{convexity}
\text{if }(H_G \rho) (\beta)= 0 \text{, then }(H^2_G \rho) (\beta) < 0.
\end{equation}

We will consider local geodesic transform near $p$ in a neighborhood $O_p \subset U$ of $p$ in $X$. See Section \ref{sec_defn_transform} for more details on $O_p$ and the meaning of local here. 
Recall that an initial point and an tangent vector at this point determine a geodesic.
The bundle we use to parametrize geodesics is the \textit{projective sphere bundle}, denoted by $PSX$, whose fibers are $\R \times \mathbb{S}^{n-2}$. 
We also denote the same type bundle extended over $\tilde{X}$ by $PS\tilde{X}$.
It parametrizes geodesics whose initial velocities have unit tangential component, except for those that are normal to the foliation that we will introduce in \eqref{eq:tilde-Sigma-def}, which correspond to $\lambda=\pm \infty$. Those excluded geodesics are irrelevant for our purpose since we will introduce a cut-off function to restrict our analysis to geodesics away from them.

\subsubsection{The foliation condition and the choice of the coordinate system}

If we define the region we consider to be $\{0 \leq \rho(z) \leq c\}$, the region might be non-compact (even when $c$ is small, it might be a long thin strip near the boundary). So we use a modification of $-\rho$ making the level sets less convex to enforce its intersection with $\partial X$ happen in a small compact region close to $p$. 

To this end, we introduce another boundary defining function $\tilde{{\msf{x}}}$ satisfying 
\begin{equation} \label{defn_tildex}
d\tilde{{\msf{x}}}(p)=-d\rho(p), \; \tilde{{\msf{x}}}(p)=0,
\end{equation} 
whose level sets are strictly convex from the sub-level sets $\{\tilde{{\msf{x}}}<-T\}$ for a constant $T>0$,
 which means geodesics tangential to this region will curve away from it. 
This function is used to introduce the artificial boundary, 
which is a level set of $\tilde{{\msf{x}}}: \, \{\tilde{{\msf{x}}}=-c\}$ for $c>0$.
This level set intersects with $\partial X$ and 
together with $\partial X$ it 
encloses a small region on which our discussion happens.
This allows us to conduct analysis locally. 
In terms of this new parameter $c>0$, the region $O_p$ is 
\begin{equation} \label{defn_Omegac}
\Omega_c:=\{ z \in X: \tilde{{\msf{x}}}(z) \geq -c, \rho(z) \geq 0 \}.
\end{equation} 
We can choose $\tilde{{\msf{x}}}$ such that $\bar{\Omega}_c$ is compact for $c$ sufficiently small.  Our proof for the local result is valid for all small $c$.

\begin{figure}
    \centering
    \includegraphics[width=0.5\linewidth]{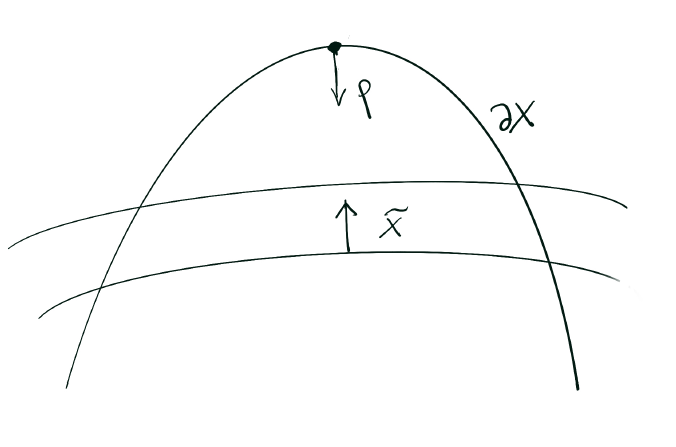}
    \caption{An illustration of the convex foliation. The level sets of $\tilde{x}$ are still convex but less convex than $\partial X$. The arrows show the direction of increase of $\rho$ and $\tilde{\msf{x}}$.}
    \label{fig:convex-foliation}
\end{figure}

We give an explicit construction of $\tilde{{\msf{x}}}$ here to show it exists locally (thus can be used for our Theorem~\ref{thm:inverse-main}), 
but our result is valid for any $\tilde{{\msf{x}}}$ satisfying conditions above.
Shrinking $O_p$ if necessary, we can assume the neighborhood we are working on is entirely in a local coordinate patch. We take
\begin{equation}
\tilde{{\msf{x}}}(z) = -\rho(z) - \epsilon |z-p|^2, \quad  z \in O_p,  \label{tildex}
\end{equation}
where $|\cdot|$ means the Euclidean norm in this coordinate patch, and this term is introduced to enforce the region characterized by $\tilde{{\msf{x}}}$ to be compact. Here $\epsilon>0$ is a fixed small constant chosen before we choose $c$. Taking $c>0$ sufficiently small, $\{\tilde{{\msf{x}}}\geq -c, \rho \geq 0\}$ is compact. This is because $\tilde{{\msf{x}}}\geq -c,\rho \geq 0$ implies $\rho \leq C$ and $|z-p|\leq c/\epsilon$. Since we are in a fixed coordinate patch, topologically this region is a closed subset of a compact Euclidean ball, hence it is compact. 
In addition, by the discussion in \cite[Section~3.1]{uhlmann2016inverse},
each $\Sigma_t$ with $0 \leq t \leq c$ is convex in the sense that any geodesic tangent to it curves away from $\{ \tilde{{\msf{x}}} \leq -t \}$.

We now turn to the \textit{convex foliation condition} we need. 
From now on, we assume $\tilde{{\msf{x}}}$ to be any function that satisfies (\ref{defn_tildex}) 
and convexity condition after it.
Our foliation of the part of $X$ near $\partial X$ is given by level sets of $\tilde{{\msf{x}}}$. 
That is, the family of hypersurfaces 
\begin{equation} \label{eq:tilde-Sigma-def}
 \{\tilde{\Sigma}_t = \tilde{{\msf{x}}}^{-1}(-t),0 \leq t \leq T\}.
\end{equation}
Here we choose $T$ to be a number such that desired properties of $\tilde{{\msf{x}}}$ hold from $\tilde{\Sigma}_0$
up to $\tilde{\Sigma}_T$. 
For Theorem~\ref{thm:inverse-main}, which concerns the local injectivity,
we only use the part of the foliation with $t \leq c$. While for the Corollary,
which concerns the global result, we use the entire foliation up to $\tilde{{\msf{x}}}=-T$.
By our choice of $c$, we may take $T=c$. Taking $T$ larger will make the region 
on which our result holds larger. In fact, one may apply a layer stripping method 
to obtain injectivity result up to $\tilde{\Sigma}_T$.
When $T>c$, we may take $\tilde{\Sigma}_c$ as the `new boundary' and apply our Theorem~\ref{thm:inverse-main}, and then repeat. For more details of the layer stripping method, see
  the discussion after \cite[Corollary]{uhlmann2016inverse}.

Finally, the coordinate system we use is 
\begin{equation} \label{coordinate}
({\msf{x}},y,\tau,\mu), 
\end{equation}
where ${\msf{x}}=\tilde{{\msf{x}}}+c$, and $y$ is the coordinate on $\tilde{\Sigma}_t$. $\tau,\mu$ are fiber variables dual to ${\msf{x}},y$ respectively
in the scattering cotangent bundle $^{\;sc}T^*X$
(i.e., we are writing covectors as $\tau \frac{d{\msf{x}}}{{\msf{x}}^2}+\mu \frac{dy}{{\msf{x}}}$).

We emphasize that introducing $\tilde{X}$ and the artificial boundary $\{\tilde{{\msf{x}}}=-c\}$ brings us convenience in this framework, allowing us to 
restrict our analysis in to this local region and use the scattering pseudodifferential calculus.


\subsubsection{The local geodesic ray transform} \label{sec_defn_transform}
Geodesics below are with respect to the metric $g$. 
Recalling (\ref{defn_Omegac}), we replace $\tilde{{\msf{x}}}$ by ${\msf{x}}=\tilde{{\msf{x}}}+c$, so that ${\msf{x}}$ itself becomes the defining function of the artificial boundary. In an open set $O \subset \bar{X}$, for a geodesic segment $\gamma \subset O$, we call it \textit{O-local geodesic} if its endpoints are on $\partial X \cap O$, and \textbf{all geodesic segments we consider below are assumed to be $O_p$-local} with $O_p$ as in the previous part.

As we mentioned in the introduction, \textit{the local geodesic ray transform} of a function $f$ is defined by:
\begin{equation*}
I_0  f(s) := \int_\gamma  f(\gamma(t))dt,
\end{equation*}
where $s \in PS\tilde{X}$, $\gamma(\cdot)$ is the geodesic determined by $s$.
So our geodesic ray transform is a function on $PS\tilde{X}$. 


\subsubsection{The main result}

Let $\bar{M}_c = \{ \tilde{\msf{x}} \geq -c \} = \{ {\msf{x}} \geq 0 \} \subset \tilde X$, and we may assume that $\bar{M}_c$ is compact by modifying $\tilde{\msf{x}}$ away from $X$. We use exponentially weighted Sobolev spaces: $H^s_F(\bar{M}_c) : = e^{\frac{F}{{\msf{x}}} }H^s(\bar{M}_c) =\{ f \in H_{\mathrm{loc}}^s(\bar{M}_c): e^{-\frac{F}{{\msf{x}}}}f \in H^s(\bar{M}_c) \}$, where the additional subscript $F$ is a positive constant, which indicates the exponential conjugation. 

For exponentially weighted Sobolev spaces on other manifolds, we use the same notation with $\bar{M}_c$ replaced by that manifold. 
Furthermore, $PS\tilde{X}|_{\bar{M}_c}$ is the restriction of the projective sphere bundle to $\bar{M}_c$.
 With all these preparations, the main theorem is:
\begin{theorem} \label{thm:inverse-main} Assume $n \geq 3$, and that $X$, $\tilde{\msf{x}}$ are as above. 
For $p \in \partial X$,  we can choose $c$ small so that the local geodesic ray transform is injective on $H_F^s(\bar{M}_c)$, $s\geq 0$. More precisely, for small $c$ there exists $C>0$ such that for all $f \in H^{s}_F (\bar{M}_c)$ that is supported in $\bar{M}_c \cap \bar{X}$, we have 
\begin{align} \label{est_main}
||f||_{ H^{s}_F (\bar{M}_c)} \leq C || I_0 f ||_{H^{s+1}(PSX|_{\bar{M}_c})}.
\end{align}
\end{theorem}    

We should emphasize that for the purpose of the injectivity of $I_0$ on functions on $X$, it might be most natural to consider the case $s=0$, i.e. $f \in L^2(X)$, since one can just extend $f$ by zero outside $X$ without changing $I_0f$ and keep $f \in L^2(\tilde{X})$.
Then \eqref{est_main} indeed is an estimate on functions on $X$ and shows that $I_0$ is injective. This also gives injectivity on Sobolev spaces with higher regularity since they are contained in $L^2(X)$.

In the corollary below, $X,\tilde{\Sigma}_t$ are defined as above, and in addition we assume that $\bar{X}$ is compact.
\begin{corollary}
If the convex foliation construction $\{\tilde{\Sigma}_t\}$ is valid up to $\{\tilde{{\msf{x}}}=-T\}$ and
$K_T : = X \backslash \cup_{t \in [0,T)} \tilde{\Sigma}_t$ has measure zero, the global geodesic X-ray transform is injective on $L^2(X)$. If $K_T$ has empty interior, the global geodesic transform is injective on $H^s(X)$ for $s > \frac{n}{2}$.
\end{corollary}

\begin{remark}
We added `global' because the functions are not restricted to $O_p$ anymore. 
In addition, our result is stable since all the conditions we need in the proof
are also satisfied by small perturbations of $g$, and the constant $C$ 
in (\ref{est_main}) can be made uniform for small perturbations of $g$,
hence the same result holds for small metric perturbations.
\end{remark}

\begin{proof} Assuming the theorem holds, we prove the corollary. 
For the first part, suppose that $K_T$ has measure zero. For a nonzero $f \in L^2(X)$, $\text{supp}f$ has non-zero measure by the definition of $L^2(X)$.
Consider $\tau:=\text{inf}_{\text{supp}f} (-\tilde{{\msf{x}}})$. If $\tau \geq T$, then $\text{supp} f \subset K_T$, which has measure zero, contradiction. So $\tau < T$ and by definition $f \equiv 0$ on $\tilde{\Sigma}_t$ with $t < \tau$. By the definition of $\tau$, closedness of $\text{supp}f$ and compactness of $\bar{X}$, we know there exists $q \in \tilde{\Sigma}_\tau \cap \text{supp}f$. However, consider the manifold given by $\{ \tilde{{\msf{x}}} < -\tau \}$, to which we can apply our theorem. Since we have local injectivity near $q$, we conclude that $q$ has a neighborhood disjoint with $\text{supp}f$, contradiction.

If $f \in H^s(X),s>\frac{n}{2}$, $f \neq 0$, then $f$ is continuous by the Sobolev embedding theorem and consequently supp$f$ has non-empty interior.
Then apply the local result to a fixed point in $\text{supp}f$ gives the contradiction.
\end{proof}


\subsection{Why scattering calculus?}\label{subsec:why sc}
We are going to use the scattering calculus with $\{ \tilde{x} = -c\}$ being the boundary in the sense of the scattering calculus. We call $\{ \tilde{x} = -c\}$ the `artificial boundary' as it not part of the boundary of the original domain, while the original boundary $\partial X = \{ \rho = 0 \}$ will play a much less significant role. 

At first glance, it seems unnatural to use scattering calculus with a boundary that is not at geometric `infinity'.
But there are good reasons why we need to use scattering calculus instead of Kohn-Nirenberg or H\"ormander type calculus (or Boutet de Monvel type, if one want to run parametrix for this manifold with boundary) here. Before the result of Uhlmann and Vasy \cite{uhlmann2016inverse}, Stefanov and Uhlmann \cite{stefanov2004stability} have already shown that the analogue (without the cutoff we will use) of the normal operator we define below in \eqref{eq:AF-def} is an elliptic pseudodifferential operator (in the usual sense at fibre-infinity). 

However, if we don't restrict to geodesics that are becoming more and more tangent to the boundary of the local region, then the right hand side of the stability estimate will involve Sobolev norms on a slightly larger region, since the X-ray transform and in turn the normal operator will involve information outside this region.
Consequently, this does not give the desired invertibility.


Now if we introduce a cut-off to enforce almost-tangency , as we do in \eqref{eq:L defn} below with the $\chi(\frac{\lambda}{{\msf{x}}})$ factor,  where $\chi(\cdot) \in C_c^\infty(\R)$, then the amplitude of the kernel \eqref{eq:L defn} does not have fixed regularity under the repeated action of $D_{\msf{x}}$, so the resulting normal operator will not a classical elliptic PsiDO as one approaches the artificial boundary. (Think about $\int \chi(\lambda/{\msf{x}}) d\lambda = {\msf{x}}\int \chi(\hat{\lambda})d\hat{\lambda}$ for intuition.)
But it \emph{is} bounded under repeated applications of ${\msf{x}}D_{\msf{x}}$, as is the case for the symbol of a scattering pseudodifferential operator. We will see that the normal operator lies in $\Psi_{\sct}$.

In addition, this cone of geodesics (that is, those with $|\lambda| \lesssim {\msf{x}}$) is becoming degenerate (i.e., thinner and thinner) in the standard tangent bundle. However, it has uniform width as $\msf{x} \to 0$ measured using the scattering structure: see Remark~\ref{rem:geodesic cone} for more explanation of this point.  


Also, this localization is necessary as one will see in the proof. The small parameter $c$, which determines the (small) size of the region $\Omega_c$, works like a `semiclassical' parameter that finally allows us to remove the error term in the elliptic estimate, which is essential for deducing injectivity.

Of course, one can choose other scaling to restrict to geodesics tangent to the artificial boundary. 
For example one can use $\sqrt{{\msf{x}}}$ here as the defining function of the artificial boundary and consider geodesics with $|\lambda| \lesssim \sqrt{{\msf{x}}}$. But that would require the use of some other pseudodifferential algebra in which one needs to treat non-commutative (operator-valued) boundary symbols, which is technically more complicated. 
See \cite[Remark~3.4]{uhlmann2016inverse}.

\subsection{The pseudodifferential property and ellipticity of the normal operator}
\label{subsec:ellipticity}

In this section we prove the pseudodifferential property and ellipticity of the exponentially conjugated microlocalized normal operator. The exponential conjugation is needed because although the Schwartz kernel of the normal operator $A = L \circ I_0$ (see definitions of $L,I_0$ below) behaves well when $X=\frac{{\msf{x}}'-{\msf{x}}}{{\msf{x}}{\msf{x}}'},\, Y=\frac{y'}{{\msf{x}}'}-\frac{y}{{\msf{x}}}$ are bounded, it is not so when $(X,Y) \rightarrow \infty$, and this conjugation gives additional exponential decay to resolve this issue.

Using the notation $({\msf{x}},y,\lambda,\omega) \in PSX$,  
we define the modified adjoint operator of $I_0$ as
\begin{equation}\label{eq:L defn}
(Lv)(\msf{x},y) := {\msf{x}}^{-2} \int \chi(\frac{\lambda}{{\msf{x}}},y) v( \gamma_{{\msf{x}},y,\lambda,\omega}) d\lambda d\omega,
\end{equation}
where $\gamma_{{\msf{x}},y,\lambda,\omega}(t)$ is the geodesic starting at $({\msf{x}},y)$ with initial tangent vector $\lambda \partial_{\msf{x}} + \omega \cdot \partial_y$.
$\chi$ is smooth and compactly supported in the first variable, with $\chi(0,y)=1$ and the size of its support is uniformly bounded. All derivatives with respect to $y$ are also uniformly bounded. 
Here $v$ is a function defined on the space of $O_p$-local geodesic segments, whose prototype is the geodesic ray transform of a function on $X$:
$$
v(\gamma) = I_0 f(\gamma) = \int_\gamma f(\gamma(t)) dt.
$$
By the compactness of $\bar{\Omega}_c$ discussed after (\ref{tildex}) and $|(\lambda,\omega)|\geq 1$ and the convexity assumption we made, \eqref{eq:gamma1-est} below (see also \cite[Lemma~3.1]{vasy2020semiclassical}) shows that there exists a uniform bound $T_g$ of the escape time of $\bar{O}_p$. Thus we assume $|t|\leq T_g$ in arguments below.
$I_0$ is the original geodesic ray transform operator and $L$ is its adjoint if we ignore $\chi$ and assume fast decay conditions on integrands. 

We define the conjugated normal operator $A_F$ as
\begin{equation} \label{eq:AF-def}
A_F = e^{-F/{\msf{x}}}L \circ I_0 e^{F/{\msf{x}}}.
\end{equation}
By the definition of $L$ and $I_0$, we know $A_F$ acts by
\begin{equation}\label{eq:AF kernel}
A_Ff(\msf{x},y):={\msf{x}}^{-2}\int e^{-\frac{F}{{\msf{x}}} +\frac{F}{\gamma^{(1)}_{{\msf{x}},y,\lambda,\omega}(t)}} \chi(\frac{\lambda}{{\msf{x}}},y) f(\gamma_{{\msf{x}},y,\lambda,\omega}(t)) dt |d\nu|,
\end{equation}
where $|d\nu|=|d\lambda d\omega|$ is a smooth density and $\gamma^{(1)}_{{\msf{x}},y,\lambda,\omega}(t)$ is the $\msf{x}$-component of $\gamma_{{\msf{x}},y,\lambda,\omega}(t)$. 

Recall the parameter $c$ in the definition $O_p=\{\tilde{{\msf{x}}} >-c \} \cap \bar{X}$ where we are only concerned with the region $0 \leq {\msf{x}} = \tilde{{\msf{x}}}+c \leq c$, we have:
\begin{theorem} \label{AFproperty}
$A_F$ is a scattering pseudodifferential operator of order $(-1, 0)$ on $\{ \msf{x} \geq 0 \}$ for $F>0$. 
In addition, if we choose $\chi = \chi(\lambda/x) \in C_c^\infty(\R)$ appropriately with $\chi \geq 0, \chi(0)=1$, the principal symbol of $A_F$, including the boundary symbol, is elliptic.
\end{theorem}

\begin{remark}
For the original derivation of the decay property of $A_F$'s Schwartz kernel and consequently the membership $A_F\in \Psi_{\sct}^{-1,0}$, we refer readers to  \cite[Section~3.5]{uhlmann2016inverse}. And see \cite[Section~5]{stefanov2004stability} for the interior of the region. 
\end{remark}

\begin{remark}
It is recommended that the reader solves Problem~\ref{prob:classical Radon} before reading this proof, as it will aid greatly in understanding the calculation. 
\end{remark}
\begin{proof}
Let $z=(1/\msf{x},y/\msf{x})$ be induced `Euclidean' coordinates on $\{ \msf{x} > 0 \}$ and let $w=(\frac{1}{\msf{x}}-\frac{1}{\msf{x}'},\frac{y}{\msf{x}}-\frac{y'}{\msf{x'}})$ denote the difference between the left variable and the right variable. We will denote $\overline\gamma_{{\msf{x}},y,\lambda,\omega}(t)$ to denote $\gamma$ in these Euclidean coordinates; that is, if $\gamma_{{\msf{x}},y,\lambda,\omega}(t) = (\gamma^{(1)}_{{\msf{x}},y,\lambda,\omega}(t),  \gamma^{(2)}_{{\msf{x}},y,\lambda,\omega}(t))$ is the splitting into $(x,y)$-coordinates, then 
$$\overline\gamma_{{\msf{x}},y,\lambda,\omega}(t) =  \Big( \frac1{\gamma^{(1)}_{{\msf{x}},y,\lambda,\omega}(t)}, \frac{\gamma^{(2)}_{{\msf{x}},y,\lambda,\omega}(t)}{\gamma^{(1)}_{{\msf{x}},y,\lambda,\omega}(t)} \Big). 
$$
From the definition of $A_F$, we know
\begin{align*}
K_{A_F}(z,z') = \int & e^{-\frac{F}{{\msf{x}}} +\frac{F}{\gamma^{(1)}_{{\msf{x}},y,\lambda,\omega}(t)}} {\msf{x}}^{-2}\chi(\frac{\lambda}{{\msf{x}}},y)\delta(z'-\overline\gamma_{{\msf{x}},y,\lambda,\omega}(t))  
\   dt |d\nu| \\
 = (2 \pi)^{-n} \int & e^{-\frac{F}{{\msf{x}}} +\frac{F}{{\gamma^{(1)}_{\msf{x},y,\lambda,\omega}}(t)}}   {\msf{x}}^{-2}\chi(\frac{\lambda}{{\msf{x}}},y)
 e^{-i \zeta' \cdot (z'-\overline\gamma_{{\msf{x}},y,\lambda,\omega}(t))} 
 dt |d\nu| |d\zeta '|.
\end{align*}
Using the formula \eqref{eq:symbol from K}, we compute 
\begin{align*}
a_F(z,\zeta) =  &  \int e^{-iw\cdot \zeta} K_{A_F}(z,z-w) \, dw  \\
  =  \int & e^{-\frac{F}{{\msf{x}}} +\frac{F}{{\gamma^{(1)}_{\msf{x},y,\lambda,\omega}}(t)}} {\msf{x}}^{-2}\chi(\frac{\lambda}{{\msf{x}}},y)  e^{-iz \cdot \zeta} e^{i \zeta \cdot \overline\gamma_{{\msf{x}},y,\lambda,\omega}(t)}   dt |d\nu|.
\end{align*}
Or more concretely
\begin{align}  \label{eq3}
a_F(z,\zeta) =  & \int e^{-\frac{F}{{\msf{x}}} +\frac{F}{{\gamma^{(1)}_{\msf{x},y,\lambda,\omega}}(t)}}  {\msf{x}}^{-2}\chi(\frac{\lambda}{{\msf{x}}},y)
e^{i (\tilde{\tau},\mu) \cdot ( \frac{1}{\gamma^{(1)}_{{\msf{x},y,\lambda,\omega}}(t)} - \frac{1}{\msf{x}} , \frac{\gamma^{(2)}_{{\msf{x}},y,\lambda,\omega}(t)}{\gamma^{(1)}_{{\msf{x},y,\lambda,\omega}}(t)}-\frac{y}{\msf{x}})}   dt |d\nu|.
\end{align}
Our goal is to show that this $a_F$ is an elliptic scattering symbol of order $(-1, 0)$. 

Next we investigate the phase function of this oscillatory integral and then apply the stationary phase lemma. We denote components of $\gamma_{z,\nu}(t)$ (recall $\nu=(\lambda,\omega)$) and its derivatives by 
\begin{align}
\begin{split}
&\gamma_{{\msf{x}},y,\lambda,\omega}(0) = ({\msf{x}},y), \quad \dot{\gamma}_{{\msf{x}},y,\lambda,\omega}(0) = (\lambda,\omega), \\
& \ddot{\gamma}_{{\msf{x}},y,\lambda,\omega}(t) = 2 (\alpha({\msf{x}},y,\lambda,\omega,t), \beta({\msf{x}},y,\lambda,\omega,t)),
\end{split}
\label{gamma}
\end{align}
where $\alpha,\beta$ are defined by this equation and are smooth with respect to their variables. In addition, $\alpha$ is a quadratic form in $\omega$, and it is strictly positive definition in $\omega$ for small enough ${\msf{x}},\lambda,t$, which means being close to the starting point at the boundary, by our convexity condition.

Since $\gamma_{{\msf{x}},y,\lambda,\omega}$ starting at $({\msf{x}},y)$ with initial velocity $(\lambda,\omega)$, there exists
smooth functions $\tilde{\Gamma}^{(1)},\Gamma^{(2)}$ such that
$$
\gamma_{{\msf{x}},y,\lambda,\omega}(t) = ( {\msf{x}} + \lambda t + \alpha t^2 + \tilde{\Gamma}^{(1)}({\msf{x}},y,\lambda,\omega,t)t^3, y+ \omega t +  \Gamma^{(2)}({\msf{x}},y,\lambda,\omega,t)t^2),
$$
where we only expanded the second component to the first order and included the $\beta$-term in the definition of $\Gamma^{(2)}$. 
Then we make the change of variables
$$
\hat{t} = \frac{t}{{\msf{x}}}, \quad \hat{\lambda} =\frac{\lambda}{{\msf{x}}}.
$$
By the support condition of $\chi$, we require $\hat \lambda \in \supp \chi$, so the integrand is only non-zero when $\hat{\lambda}$ is in a compact interval. However, the bound on $\hat{t}$ is $|\hat{t}| \leq \frac{T_g}{{\msf{x}}}$, which is not uniformly bounded. We deal with this by treating it in two regions separately. In terms of these new variables we have 
\begin{equation}\label{eq:gamma(i)}\begin{gathered}
\gamma^{(1)}_{{\msf{x}},y,\lambda,\omega}(t) =  {\msf{x}} + {\msf{x}}^2 (\hat \lambda \hat t + \alpha \hat t^2)  + {\msf{x}}^3 \tilde{\Gamma}^{(1)}({\msf{x}},y,\lambda,\omega,t) \hat t^3 , \\
\frac1{\gamma^{(1)}_{{\msf{x}},y,\lambda,\omega}(t)} =  \frac1{{\msf{x}}} - (\hat\lambda \hat t + \alpha \hat t^2) + \msf{x} \Gamma^{(1)}({\msf{x}},y,\lambda,\omega,t) \hat t^3 , \text{ where } \Gamma^{(1)} \text{ is smooth, and} \\
\gamma^{(2)}_{{\msf{x}},y,\lambda,\omega}(t) = y + \msf{x} \hat t \omega + \msf{x}^2 \hat t^2 \Gamma^{(2)}({\msf{x}},y,\lambda,\omega,t), \\
\frac{\gamma^{(2)}_{{\msf{x}},y,\lambda,\omega}(t)}{\gamma^{(1)}_{{\msf{x}},y,\lambda,\omega}(t)} = \frac{y + \msf{x} (\hat t \omega - (\hat \lambda \hat t + \alpha \hat t^2) y) + \msf{x}^2 \hat t^2 \Gamma^{(2)}({\msf{x}},y,\lambda,\omega,t)}{\msf{x}}.
\end{gathered}\end{equation}
Then the oscillatory factor in \eqref{eq3} is $e^{i\phi}$, with
\begin{equation} \label{eq:phi-1}
\phi = \tau(\hat{\lambda}\hat{t} + \alpha \hat{t}^2 + {\msf{x}}\hat{t}^3 \Gamma^{(1)}({\msf{x}},y,{\msf{x}}\hat{\lambda},\omega,{\msf{x}}\hat{t})) + \mu(\omega \hat{t}+{\msf{x}}\hat{t}^2 \Gamma^{(2)}({\msf{x}},y,{\msf{x}}\hat{\lambda},\omega,{\msf{x}}\hat{t})),
\end{equation}
where $\tau = -\tilde{\tau} - \mu \cdot y$. 

The exponent in \eqref{eq:AF kernel} arising from the exponential conjugation is
\begin{align} \label{eq:phase-conjugation}
\begin{split}
-\frac{F}{{\msf{x}}} + \frac{F}{\gamma^{(1)}_{{\msf{x}},y,\lambda,\omega}(t)} = & -F(\lambda t + \alpha t^2 + t^3 \Gamma^{(1)}({\msf{x}},y,{\msf{x}}\hat{\lambda},\omega,{\msf{x}}\hat{t})) \\ 
& \times ({\msf{x}}({\msf{x}}+\lambda t + \alpha t^2 +t^3 \Gamma^{(1)}({\msf{x}},y,{\msf{x}}\hat{\lambda},\omega,{\msf{x}}\hat{t})))^{-1}   \\
=& -F(\hat{\lambda}\hat{t} +\alpha \hat{t}^2 +\hat{t}^3{\msf{x}} \hat{\Gamma}^{(1)}({\msf{x}},y,{\msf{x}}\hat{\lambda},\omega,{\msf{x}}\hat{t})),
\end{split}
\end{align}
where $\hat{\Gamma}^{(i)}$ is introduced when we first express $\gamma^{(1)}_{{\msf{x}},y,\lambda,\omega}(t)$ by variables $t,\lambda$, and then invoke our change of variables, then collect the remaining terms, which is a smooth function of these normalized variables. So this amplitude is Schwartz in $\hat{t}$, hence we take a constant $\epsilon_t>0$ and
deal with regions $|\hat{t}| \geq \epsilon_t$ and $|\hat{t}|< \epsilon_t$ separately. In our later argument, we will take $\epsilon_t$ small
to enforce $\hat{t}=0$ holds for critical points.

The geometric reason that the exponent in \eqref{eq:phase-conjugation} introduced by conjugation is in the favourable direction is that we are restricting to geodesics with tangent vector with at most $O(\msf{x})$ level component on the $(-\msf{x})$-direction. By the convexity assumption (i.e. the geodesic is curving away from the foliation), $\gamma^{(1)}_{\msf{x},y,\lambda,\omega}(t)$ may decrease (i.e. the component of the tangent vector that is normal to the foliation may stay in the $(-\msf{x})$-direction) for $O(\msf{x})$ time, which means it could move at most $O(\msf{x}^2)$ level in the  $(-\msf{x})$-direction. So $-1/x + 1/\gamma^{(1)}_{\msf{x},y,\lambda,\omega}(t) \leq C$ uniformly for all $t$, meaning the exponential factor is uniformly bounded. On the other hand, it produces decay on most of the geodesic. To see this, we claim that, for small enough $c$, 
by the convexity condition \eqref{convexity} we assumed, 
there exists a $\lambda_0>0$ and $C>0$ such that for all $\lambda<\lambda_0$, we have (whenever $\gamma_{{\msf{x}},y,\lambda,\omega}(t)$ is defined):
\begin{equation} \label{eq:gamma1-est}
\gamma_{{\msf{x}},y,\lambda,\omega}^{(1)}(t) \geq {\msf{x}} + \lambda t + Ct^2.
\end{equation}
This follows by the same computation as \eqref{eq:gamma(i)} above, using that ${\msf{x}}\hat{t}^3 = t \times \hat{t}^2$  is $\ll \hat{t}^2$ if we select $c$ to be small so that the uniform upper bound of $t$ is small. So the exponential factor is $O(e^{-C\hat{t}^2})$ for large $\hat{t}$ with some $C>0$. For this reason we shall refer to it as the `damping factor'.  In particular, the exponential decay means that those stationary phase results for compact intervals apply here as well. In terms of Problem~\ref{prob:classical Radon}, this damping factor plays a similar role to the localizing function $\phi$ in that problem, with the unimportant difference that here the damping is exponentially decreasing instead of compactly supported. 
We left the proof of those geometric facts in Exercise~\ref{ex:geodesic-gamma1-est}. They were proven in \cite[Lemma~3.1]{vasy2020semiclassical}.

To summarize, after substituting $\phi$ and the damping factor above into \eqref{eq3}, it is not hard to see, for finite $(\tau,\mu)$ that it remain bounded under repeated application of ${\msf{x}}\partial_{\msf{x}}$ and $\partial_y$. And indeed, the $\chi(\lambda/{\msf{x}})$-factor tell us that we would lose control if we consider the classical pseudodifferential algebra with respect to $({\msf{x}},y)$ here, since then we would need to test regularity under repeated applications of $\partial_{\msf{x}}$, which would obviously lead to blowup. 

In addition, the Gaussian in $\hat{t}$ decay allows the stationary phase argument below to go up to arbitrary many terms, and this also shows the symbolic behavior as $(\tau,\mu) \to \infty$. See Problem~\ref{ex:symbolic-fiber-infinity-inverse}.  

Now we start to show the ellipticity in the differential sense, i.e. as $|(\tau,\mu)| \to \infty$.
We will show that its leading order behavior is a non-zero multiple of $|(\tau,\mu)|^{-1}$ via a stationary phase argument in $(\hat{t},\hat{\lambda})$-variables.
Notice that this also shows the differential order of $A_F$ (as a pseudodifferential operator) is $-1$.

Before considering the critical points of the phase for small ${\msf{x}}>0$, we first consider the critical points of the phase at ${\msf{x}}=0$. This helps us to get rid of 
those $\Gamma^{(i)}$-terms and simplifies the process to solve the equation for the critical points.

When ${\msf{x}}=0$, the phase becomes
\begin{equation}\label{eq:phase x=0}
\tau (\hat{\lambda}\hat{t} + \alpha \hat{t}^2)+ \hat{t}\mu \cdot \omega.
\end{equation}
When $\hat{t} \neq 0$, the derivative with respect to $\hat{\lambda}$ vanishes only when $\tau=0$. 
Since our analysis on the ellipticity is happening away from the zero section, thus $\tau=0$ implies $|\mu| \gtrsim 1$.
Then we consider the $(\hat t, \omega)$-derivative to see that there are no critical points here. So the region $|\hat{t}| \geq \epsilon$ gives rapid decay contribution when ${\msf{x}}=0$, for any $\epsilon > 0$.
Next we consider the region $\hat{t}=0$, still for $\msf{x} = 0$. Clearly this implies that the $\hat \lambda$ and $\omega$-derivatives vanish. For the $\hat t$-derivative to vanish, we require that $(\tau, \mu) \cdot (\hat \lambda, \omega) = 0$. This is a codimension one submanifold of $\R \times S^{n-2} \ni (\hat \lambda, \omega)$. Thus we have stationary points precisely at the submanifold 
$$
 \hat{t} = 0, \quad \tau \hat{\lambda} + \mu \cdot \omega = 0.
$$
It is not hard to see that the Hessian at this submanifold always has rank $2$. If $\tau \neq 0$, then we have nondegenerate stationary phase in the $(\hat \lambda, \hat t)$-variables. On the other hand, if $\tau = 0$, then (as we are looking at $(\tau, \mu)$ large) we have $\mu \neq 0$, so without loss of generality we can assume that $\mu = (\mu_1, 0, \dots, 0)$ with $\mu_1 \neq 0$. Then we have nondegenerate stationary phase in the $(\hat t, \omega_1)$-variables. (In this case, we can move $\omega$ in the direction $(1, 0, \dots, 0)$ since $\omega \cdot \mu = 0$.) 
The phase function vanishes at the stationary points since $\hat t = 0$ there. Performing stationary phase, we obtain a leading asymptotic that is a nonzero multiple of $|(\tau, \mu)|^{-1}$ (where this factor arises from the Hessian, which is a symbol of order $1$ in $(\tau, \mu)$).


 The case ${\msf{x}}>0$ can be dealt with the same method, but with more complicated computation. Notice that, $\alpha,\Gamma^{(i)}$ take $\lambda = {\msf{x}}\hat{\lambda}, t = {\msf{x}} \hat{t}$ as variables, and produces an extra ${\msf{x}}$ factor when we take partial derivatives with respect to $\hat{\lambda},\hat{t}$. Concretely, the derivative with respect to $\hat{\lambda}$ is:
\begin{align*}
\frac{\partial \phi}{\partial \hat{\lambda}} &= \tau \hat{t}(1+{\msf{x}}\hat{t}\partial_\lambda \alpha + {\msf{x}}^2 \hat{t}^2 \partial_\lambda \Gamma^{(1)})+ \mu {\msf{x}}^2\hat{t}^2\partial_\lambda \Gamma^{(2)}\\
&= \hat{t} \Big( \tau (1+t\partial_\lambda \alpha + t^2 \partial_\lambda \Gamma^{(1)})+ x \mu t \partial_\lambda \Gamma^{(2)} \Big).
\end{align*}
Recall that $|t| \leq T_g$ and we can choose $T_g$ to be small, say $\leq \epsilon$ for any given $\epsilon > 0$, by shrinking $O_p$, i.e. by shrinking $c$.
Thus when $\epsilon$ is sufficiently small and  $\tau \geq C \epsilon x |\mu|$, then the first term $\hat t \tau$ dominates,  
so $\frac{\partial \phi}{\partial \hat{\lambda}}$ cannot vanish and there is no critical point in this case.

We are then reduced to considering the case $\tau \leq C \epsilon x |\mu|$. We next compute the $\hat t$-derivative of $\phi$. This is 
$$
\mu \cdot \omega + \tau(\hat \lambda + 2 \alpha \hat t + x \partial_t \alpha \hat t^2) + O( x \tau \hat t^2 + x^2 \tau \hat t^3 + x \mu \hat t + x^2 \mu \hat t^2),
$$
which in the region $\tau \leq C \epsilon x |\mu|$ is 
$$
\mu \cdot \omega + O(\epsilon x |\mu| + \epsilon |\mu|),  
$$
and is therefore nonvanishing whenever $|\mu \cdot \omega| \geq C \epsilon |\mu|$. Finally, in the remaining region $\tau \leq C \epsilon x |\mu|$ and $|\mu \cdot \omega| \leq C \epsilon |\mu|$, assuming without loss of generality that $\mu = (\mu_1, 0, \dots, 0)$ then the $\omega$-derivative of $\phi$ in a direction with a large component in the $\omega_1$ direction is nonzero, since the term $\hat t \mu_1$ will dominate all other terms. 

Thus, critical points only occur at $\hat t = 0$. Here, just as in the previous argument at $x=0$, we have $\partial_{\hat \lambda} \phi = \partial_{\omega} \phi = 0$ when $\hat t = 0$, so we only need to consider the $\hat t$-derivative of $\phi$. This is equal to 
$$
\tau \hat \lambda + \mu \cdot \omega + O(x),
$$
which has nonzero differential in $(\hat \lambda, \omega)$ on $\{ \hat t = 0 \}$, so this yields a codimension-two submanifold of critical points. As with the previous case at $x=0$, we can check that the Hessian of $\phi$ has rank two everywhere on its critical set. When $\tau \neq 0$, then the critical set is given by $\{ \partial_{\hat \lambda}\phi = \partial_{\hat t} \phi = 0 \}$, since $\partial_{\hat \lambda} \phi = \tau \hat t + O(x \hat t^2)$, and it is easy to see from this that one has nondegenerate stationary phase in the $(\hat \lambda, \hat t)$ variables. On the other hand, if $\tau = 0$ (and therefore also in a neighbourhood, $|\tau| \leq \epsilon |\eta|$) then, assuming $\eta = (\eta_1, 0, \dots, )$, the critical set is given by $\{ \partial_{\omega_1}\phi = \partial_{\hat t} \phi = 0 \}$, since $\partial_{\omega_1} \phi = \mu_1 \hat t + O(x \hat t^2)$, and therefore one has nondegenerate stationary phase in the $(\omega_1, \hat t)$ variables.

\vskip 5pt

Next we turn to show boundary part of the principal symbol of $A_F$ is also elliptic (when the fiber variables are finite). 
Evaluating \eqref{eq3} at ${\msf{x}}=0$, with the help of \eqref{eq:phi-1} and \eqref{eq:phase-conjugation}, the boundary principal symbol of $A_F$ is
\begin{align*}
a_F(0,y,\zeta) & = \int e^{-F(\hat{\lambda} \hat{t} + \alpha \hat{t}^2)} \chi(\hat{\lambda},y) e^{i (\tau(\hat{\lambda} \hat{t} + \alpha \hat{t}^2 )+\hat{t} \mu \cdot \omega ) } d\hat{t}  d\hat{\lambda}d\omega.
\end{align*}
Now $\alpha({\msf{x}},y,{\msf{x}}\hat{\lambda},\omega) = \alpha(0,y,0,\omega): = \alpha(y,\omega)$, which is a positive quadratic form in $\omega$, 
hence changing the sign of $\omega$ does not change its value. 
We choose $\chi(\hat{\lambda},y)$ to be a Gaussian density in $\hat{\lambda}$ first, then we use approximation argument to obtain a cut-off function $\chi$ that has compact support in $\hat{\lambda}$. 
We choose $\chi(\hat{\lambda},y,\omega)= e^{-\frac{F\hat{\lambda}^2}{2\alpha(y,\omega)}}$, then we have:
\begin{align} \label{a_F-boundary-3}
\begin{split}
a_F = & \int e^{-F(\hat{\lambda} \hat{t} + \alpha \hat{t}^2)} \chi(\hat{\lambda},y,\omega) e^{i (\tau(\hat{\lambda} \hat{t} + \alpha \hat{t}^2 )+\omega \cdot \mu\hat{t}) } d\hat{t} d\hat{\lambda} d\omega \\
= & \int  (\int e^{-F\hat{\lambda} \hat{t}-\frac{F\hat{\lambda}^2}{2\alpha} + i\tau \hat{\lambda}\hat{t}} d\hat{\lambda} )e^{-F\alpha \hat{t}^2 + i\omega \cdot \mu \hat{t}+i\tau \alpha \hat{t}^2} d\hat{t} d\omega
\end{split}
\end{align}
The integral in $\hat{\lambda}$ is a Fourier transform of a Gaussian, it is $\sqrt{\frac{2\pi \alpha}{F}} e^{\frac{\alpha F \hat{t}^2}{2}-i\tau \alpha \hat{t}^2 -\frac{\alpha}{2F}\hat{t}^2\tau^2}$.
Thus we need to compute:
\begin{align*}
\int e^{-\frac{\alpha}{2F}(F^2+\tau^2) \hat{t}^2 + i\omega \cdot \mu \hat{t}} d\hat{t},
\end{align*}
which is again a Gaussian type integral, and it equals to a constant multiple of 
$$\sqrt{\frac{F}{\alpha}}(F^2+\tau^2)^{-\frac{1}{2}}e^{-\frac{F(\omega\cdot\mu)^2}{2(\tau^2+F^2)\alpha(y,\omega)}},$$ 
which is even in $\mu$. 
Finally, with a constant factor $\hat{C}$, we have:
\begin{align}  \label{symbol_boundary}
a_F(0,y,\tau, \mu) & = \hat{C}\int_{\mathbb{S}^{n-2}} \sqrt{\frac{F}{\alpha}}(F^2+\tau^2)^{-\frac{1}{2}}e^{-\frac{F(\omega \cdot \mu)^2}{2(\tau^2+F^2)\alpha(y,\omega)}} d\omega.
\end{align}
This equation shows that $a_F(0,y,\tau, \mu)$ is bounded below by a positive constant when the fiber variables are uniformly bounded.  
In fact, a more detailed decomposition in $\omega$ also gives upper and lower bound proportional to $|(\tau, \mu)|^{-1}$ as $|(\tau,\mu)| \to \infty$. Indeed, first consider the region $|\tau| \geq 2 |\mu|$. Then the exponential in \eqref{symbol_boundary} is uniformly bounded below, and is of course bounded above by $1$, so this integral is comparable to $|\tau|^{-1}$ for large $|\tau|$. On the other hand, when $|\mu| \geq 2 |\tau|$, then one can look at the region on $\mathbb{S}^{n-2}$ that is within distance $\sim \frac{(|\tau|^2+F^2)^{1/2}}{|\mu|}$ of the great circle perpendicular to $\mu$. This region is roughly a small interval times $\mathbb{S}^{n-3}$ and will have volume $\sim \frac{(|\tau|^2+F^2)^{1/2}}{|\mu|}$ and the factor $e^{-\frac{F(\omega \cdot \mu)^2}{2(\tau^2+F^2)\alpha(y,\omega)}}$ is uniformly bounded below on it. So the integral is $\gtrsim 1/|\mu|$
as $|(\tau,\mu)| \to \infty$. Similar considerations show that this is also an upper bound. This is of course to be expected as we already computed the principal symbol for large $(\tau, \mu)$ at $\msf{x} = 0$ (although for a compactly supported rather than Gaussian $\chi$). 

Notice that the constants in argument above are uniformly bounded for all choices of small $c$, thus we can choose $c$ small and our lower bound on the ellipticity of $a_F$ will be uniform as $c \to 0$. 

Now we fix the compact support issue. Readers may wish to skip this technical part on a first reading. 
Let $\chi$ be a Gaussian as above, which generates an elliptic operator, then we fix a bump function $\chi_0 \in C_c^\infty(\R)$ that is identically $1$ on $[-1,1]$ and set
\begin{equation}
    \chi_m = \chi_0(\hat{\lambda}/m)\chi.
\end{equation}
Then $\chi_m$ is compactly supported in $\hat{\lambda}$ and converge to $\chi$ in the topology of Schwartz space $\mathcal{S}(\mathbb{R})$. 
Using such $\chi_m$, \eqref{a_F-boundary-3} becomes
\begin{align} \label{a_F-boundary-4}
\begin{split}
a_{F,n} = & \int e^{-F(\hat{\lambda} \hat{t} + \alpha \hat{t}^2)} \chi_m(\hat{\lambda},y,\omega) e^{i (\tau(\hat{\lambda} \hat{t} + \alpha \hat{t}^2 )+\omega \cdot \mu\hat{t}) } d\hat{t} d\hat{\lambda} d\omega ,
\end{split}
\end{align}
which can be rearranged as
\begin{align} \label{a_F-boundary-5}
\begin{split}
a_{F,n} = & \int e^{i(\tau+iF)\alpha\hat{t}^2}e^{i((\tau+iF)\hat{\lambda} + \omega\cdot \mu)\hat{t}}
\chi_0(\hat{\lambda}/m)e^{-\frac{\hat{F\lambda^2}}{2\alpha(y,\omega)}}
d\hat{t} d\hat{\lambda} d\omega.
\end{split}
\end{align}
In argument below $F$ is taken as a fixed constant.
Then the $\hat{t}$ integral is the Fourier transform of the Gaussian $e^{i(\tau+iF)\alpha\hat{t}^2}$ evaluated at $(\tau+iF)\hat{\lambda}+\omega\cdot\mu$, which is, up a constant:
\begin{equation}
\alpha^{-1/2}(\tau+iF)^{-1/2}
e^{-i\frac{1}{4\alpha(\tau+iF)}(\tau\hat{\lambda}+\mu\cdot\omega+iF\hat{\lambda})^2},
\end{equation}
where the branch of square root is taken so that it takes positive reals when the variable is a positive real, and we avoid the ray $\{-ir,\; r \in (0,\infty)\}$.
So we have
\begin{align} \label{a_F-boundary-6}
\begin{split}
a_{F,m} & =  C\int (\tau+iF)^{-1/2}
e^{-i\frac{1}{4(\tau+iF)}(\tau\hat{\lambda}+\mu\cdot\omega+iF\hat{\lambda})^2}
\chi_0(\hat{\lambda}/m)e^{-\frac{\hat{F\lambda^2}}{2\alpha(y,\omega)}}
d\hat{\lambda} d\omega
\\ & = C \int (\tau+iF)^{-1/2}
e^{i\frac{(-\tau+iF)}{4\alpha(\tau^2+F^2)}(\tau\hat{\lambda}+\mu\cdot\omega+iF\hat{\lambda})^2}
\chi_0(\hat{\lambda}/m)e^{-\frac{F\hat{\lambda^2}}{2\alpha(y,\omega)}}
d\hat{\lambda} d\omega.
\end{split}
\end{align}
So we only need to show that for any $c>0$, the difference of $a_{F,n}$ in \eqref{a_F-boundary-6} with $a_F$ in \eqref{symbol_boundary} can be bounded by $c\la(\tau,\mu)\ra^{-1}$.
Consider the finite $(\tau,\mu)$ part first.
Since $\chi_0(\hat{\lambda}/m)-1$ is supported in $\hat{\lambda}\gtrsim m$, for any $\epsilon>0$, there is an $N(\epsilon)$ so that 
\begin{align} \label{a_F-boundary-difference}
\begin{split}
|C \int (\tau+iF)^{-1/2}
e^{i\frac{(-\tau+iF)}{4\alpha(\tau^2+F^2)}(\tau\hat{\lambda}+\mu\cdot\omega+iF\hat{\lambda})^2}
(1-\chi_0(\hat{\lambda}/m))e^{-\frac{F\hat{\lambda^2}}{2\alpha(y,\omega)}}
d\hat{\lambda} d\omega| \leq \epsilon
\end{split}
\end{align}
for all $m>N(\epsilon)$.
By the ellipticity proved for the case with $\chi$ being a Gaussian, we know that for any $M>0$, there is a constant $c(M)$ such that the symbol in \eqref{symbol_boundary} is lower bounded by $c(M)$ for $|(\tau,\mu)|\leq M$.
Applying the approximation result above with $\epsilon=\frac{c(M)}{2}$, then we know that for all $M>0$ and \eqref{a_F-boundary-4} with $n \geq N(\frac{c(M)}{2})$, the symbol $a_{F,n}$ with $|(\tau,\mu)| \leq M$ is lower bounded by $\frac{c(M)}{2}$.

Now we turn to the large $(\tau,\mu)$ behaviour of $a_F-a_{F,n}$. 
That is, we need to prove that for any $c>0$, there is an $N(c)$ such that for all $n \geq N(c)$, we have:
\begin{equation}\label{eq:error bound aF}
|\int (\tau+iF)^{-1/2}
e^{i\frac{(-\tau+iF)}{4\alpha(\tau^2+F^2)}(\tau\hat{\lambda}+\mu\cdot\omega+iF\hat{\lambda})^2}
(1-\chi_0(\hat{\lambda}/m))e^{-\frac{F\hat{\lambda^2}}{2\alpha(y,\omega)}}
d\hat{\lambda} d\omega| \leq c\la(\tau,\mu)\ra^{-1},
\end{equation}
where we removed $C$ on the left hand side since the claim is unchanged with $c$ replaced by $c/C$.
We use coordinates in $\omega$ so that $\omega_1|\mu|=\omega\cdot\mu$ and denote the rest components by $\omega'$.
Then the spherical measure becomes $(1-\omega_1^2)^{\frac{n-4}{2}}d\omega_1d\omega'$.
We consider two regions $|\tau| \geq |\mu|$ and $|\tau| \leq |\mu|$ respectively.
When $|\tau|\geq|\mu|$, we rearrange the oscillatory integral as
\begin{multline} \label{eq:11}
\iiint (\tau+iF)^{-1/2}
e^{-i\frac{\tau\hat{\lambda}^2+2\hat{\lambda}\omega_1|\mu|}{4\alpha} }
e^{-i\frac{\tau \omega_1^2|\mu|^2}{4\alpha(\tau^2+F^2)} }
(1-\chi_0(\hat{\lambda}/m))
e^{-F\frac{\omega_1^2|\mu|^2}{4\alpha(\tau^2+F^2)}}e^{-\frac{F\hat{\lambda}^2}{4\alpha}}
d\hat{\lambda} \\ \times  (1-\omega_1^2)^{(n-4)/2}d\omega_1d\omega'.
\end{multline}
Then we apply a non-stationary phase argument in $\hat{\lambda}$.
In this region, the derivative of the phase in $e^{-i\frac{\tau\hat{\lambda}^2+2\hat{\lambda}\omega_1|\mu|}{4\alpha} }$ is 
\begin{equation}
|-i\frac{\hat{\lambda}\tau+\omega_1|\mu|}{2\alpha}|
\geq |\tau|
\end{equation}
for sufficiently large $m$, since we have $|\omega_1| \leq 1$ and $|\tau|\geq|\mu|$.
Then we can integrate by parts in $\hat{\lambda}$ and the integral is bounded by 
\begin{equation}
(|\tau|+|F|)^{-1/2}|\tau|^{-1} \|\partial_{\hat{\lambda}} \Big((1-\chi_0(\hat{\lambda}/m))
e^{-F\frac{\omega_1^2|\mu|^2}{4\alpha(\tau^2+F^2)}}e^{-\frac{F\hat{\lambda}^2}{4\alpha}}\Big)\|_{L^1},
\end{equation}
which is bounded by an arbitrarily small multiple of $|\tau|^{-1}$ for $|\tau| \gtrsim 1$ when $n$ is large.
Notice that this statement is even uniform in large $F$, since $|\hat{\lambda}| \gtrsim m$ and the $F$-loss caused by differentiation can be absorbed by $e^{-\frac{F\hat{\lambda}}{4\alpha}}$.

Now we consider the case with $|\tau|\leq|\mu|$, in particular the interesting case is when $|\tau|\ll|\mu|$. 
We rewrite \eqref{eq:11} as
\begin{multline} \label{eq:12}
\iiint (\tau+iF)^{-1/2}
e^{-i\frac{\tau\hat{\lambda}^2}{4\alpha} }
e^{-i\Big(\frac{\tau \omega_1^2|\mu|^2+2\hat{\lambda}\omega_1|\mu|(\tau^2+F^2)}{4\alpha(\tau^2+F^2)} \Big) }
(1-\chi_0(\hat{\lambda}/m))
e^{-F\frac{\omega_1^2|\mu|^2}{4\alpha(\tau^2+F^2)}}e^{-\frac{F\hat{\lambda}^2}{4\alpha}}
d\hat{\lambda} \\ \times (1-\omega_1^2)^{\frac{n-4}{2}}d\omega_1d\omega'.
\end{multline}

If $|\tau| \geq F$, we decompose the integral in $\omega_1$ into two parts.
After factoring out the large parameter, our imaginary phase is 
\begin{equation}
    \frac{ \omega_1^2+2\hat{\lambda}\omega_1|\mu|^{-1}\tau^{-1}(\tau^2+F^2)}{4\alpha}.
\end{equation}
Then we know that the second order derivative of this is 
\begin{equation}
    \frac{1}{2\alpha}+O(\omega_1),
\end{equation}
where the second term includes terms with $\alpha$ differentiated in $\omega_1$.
So there is a $\delta \in (0,1)$, which is independent of all (potentially) large parameters here, such that when $\omega_1 \in [-\delta,\delta]$,
this second order derivative in $\omega_1$ is $\gtrsim 1$. 

For the part with $|\omega_1| \leq \delta$, we apply the Van der Corput lemma (in the form of \cite[Section~VIII.1.2, Corollary]{Stein-harmonic-textbook}) in the $\omega_1$-integral to see that the left hand side is bounded by (modulo an absolute constant independent of all parameters)
\begin{align}\label{eq:13}
\begin{split}
& \int |\tau+iF|^{-1/2}
(\frac{2|\tau||\mu|^2}{4\alpha(\tau^2+F^2)})^{-1/2}
(1-\chi_0(\hat{\lambda}/m))
\\& 
\Big(1+\|\partial_{\omega_1} \big((1-\omega_1^2)^{\frac{n-4}{2}} e^{-F\frac{\omega_1^2|\mu|^2}
{4\alpha(\tau^2+F^2)}}\big)\|_{L^1[-\delta,\delta]} \Big)
 e^{-\frac{F\hat{\lambda}^2}{4\alpha}} d\hat{\lambda}d\omega',
   \end{split} 
\end{align} 
which is a small multiple of $|\mu|^{-1}$ for $m$ large.

Now for the part with $|\omega_1| \geq \delta$, we rewrite the phase as in \eqref{eq:11} and apply the Van der Corput lemma in $\hat{\lambda}$-integral again, it is bounded by
\begin{align} 
& \int_{|\omega_1| \geq \delta} (\tau+iF)^{-1/2}
|\tau|^{-1/2}
e^{-i\Big(\frac{\tau \omega_1^2|\mu|^2}{4\alpha(\tau^2+F^2)} \Big) }
\|\partial_{\hat{\lambda}} \Big((1-\chi_0(\hat{\lambda}/m)) e^{-\frac{F\hat{\lambda}^2}{4\alpha}} \Big) \|_{L^1}
\\& e^{-F\frac{\omega_1^2|\mu|^2}{4\alpha(\tau^2+F^2)}}
d\hat{\lambda} (1-\omega_1^2)^{\frac{n-4}{2}}d\omega_1d\omega'. 
\end{align}
In the integrated region, $|\omega_1| \gtrsim 1$, hence the $e^{-F\frac{\omega_1^2|\mu|^2}{4\alpha(\tau^2+F^2)}}$ factor allows us to trade-off between $\tau^{-1}$ and $|\mu|^{-1}$, and we conclude that this is a small multiple of $|\mu|^{-1}$ and completes the proof. 

Finally we consider the case $|\tau| \lesssim F$. In this case \eqref{eq:11} is bounded by
\begin{equation}
\|e^{-F\frac{\omega_1^2|\mu|^2}{4\alpha(\tau^2+F^2)}}(1-\omega_1^2)^{\frac{n-4}{2}}\|_{L^1([-1,1])}
 \|(1-\chi_0(\hat{\lambda}/m) e^{-\frac{F\hat{\lambda}^2}{4\alpha}}\|_{L^1(\R)}.
\end{equation}
Now we claim that
\begin{equation}
    \|e^{-F\frac{\omega_1^2|\mu|^2}{4\alpha(\tau^2+F^2)}}(1-\omega_1^2)^{\frac{n-4}{2}}\|_{L^1([-1,1])} \lesssim |\mu|^{-1}.
\end{equation}
We consider the region with $|\omega_1| \leq \frac{1}{2}$ and $|\omega_1| \geq \frac{1}{2}$ respectively.
For the first part, we introduce a change of variables $\tilde{\omega}_1=|\mu|\omega_1$, then this part is bounded by
\begin{equation}
|\mu|^{-1} \int_{\R} e^{-F\frac{\tilde{\omega}_1^2}{4\alpha(\tau^2+F^2)}} (1-(1/2)^2)^{\frac{n-4}{2}} d\tilde{\omega}_1,
\end{equation}
which is $\lesssim |\mu|^{-1}$ since we have $|\tau| \lesssim F$ in this region.
For the contribution from $|\omega_1| \geq \frac{1}{2}$, we again use $e^{-F\frac{\omega_1^2|\mu|^2}{\alpha(\tau^2+F^2)}} \lesssim |\mu|^{-1}$ (in fact we can take any power of $|\mu|^{-1}$), then the contribution is bounded by (modulo an constant independent of $\mu$)
\begin{equation}
|\mu|^{-1}\|(1-\omega_1^2)^{\frac{n-4}{2}}\|_{L^1([-1,-\frac{1}{2}] \cup [\frac{1}{2},1])}  \|(1-\chi_0(\hat{\lambda}/m) e^{-\frac{F\hat{\lambda}^2}{4\alpha}}\|_{L^1(\R)},
\end{equation}
which is a small multiple of $|\mu|^{-1}$ when $m$ is large.

\end{proof}

\begin{remark}\label{rem:geodesic cone}
Returning to a discussion at the beginning of Section~\ref{subsec:why sc}, consider a nearly-tangent geodesic in the sense that its tangent vector is $\lambda\partial_{\msf{x}}+\omega\cdot \partial_y $ with $\hat\lambda = \lambda/x \in \supp \chi$. Then we can write this vector as 
    since ${\msf{x}}^{-1}(\hat{\lambda}{\msf{x}}^2\partial_{\msf{x}}+\omega\cdot {\msf{x}}\partial_y)$, and so after removing a scaling factor ${\msf{x}}^{-1}$, the cone of tangent vectors of allowed geodesics consists of scattering vector fields of the form $\hat{\lambda}{\msf{x}}^2\partial_{\msf{x}}+\omega\cdot {\msf{x}}\partial_y)$, and this cone has --- in terms of scattering geometry --- a fixed width since $\hat \lambda$ lies in a fixed set, and $|\omega| = 1$. Problem~\ref{prob:classical Radon} parts (b) and (c) is concerned with the classical Radon transform with geodesics restricted to a cone of fixed width, and shows that this is sufficient for injectivity. 
\end{remark}


\subsection{The proof of the main theorem, Theorem~\ref{thm:inverse-main}.}
\label{main_proof}
Fix $c_0$ small and apply results in previous sections to $\Omega_{c_0}$, estimates above are uniform with respect to $c \in (0, c_0]$. We let $c$ vary and take $\mathsf{b} \in C(\tilde{X})$ such that ${\msf{x}} \leq \mathsf{b}(c)$ on $\Omega_c$.
For $\Omega_c$, denote $A_F$ in Section~\ref{subsec:ellipticity} constructed for $\Omega_c$ by $B_c$. By Theorem~\ref{AFproperty} and Proposition~\ref{prop:micro-elliptic}, 
we have its parametrix $G_c$ of $B_c$ such that $G_cB_c = \Id+ E_{0c}, \, E_{0c} \in \Psi^{-\infty,-\infty}_{\sct}(\tilde{X})$ on $\Omega_c$.

Consider the map $\Psi_c(\tilde{{\msf{x}}},y)=(\tilde{{\msf{x}}}+c,y)$, and let 
\begin{equation}
A_c =(\Psi_c^{-1})^*B_c\Psi_c^*,\, E_c =(\Psi_c^{-1})^*(E_{0c})\Psi_c^*.
\end{equation}
 This conjugation is introduced to make this family of operators to be defined on a fixed region $\bar{M}_0:=\{ \tilde{{\msf{x}}} \geq 0 \}$. We have an estimate of the error term in terms of $f$. To be more precise, we consider the Schwartz kernel $K_{E_c}$ of $E_c$, which satisfies $|{\msf{x}}^{-N}{\msf{x}}'^{-N}K_{E_c}| \leq C_N$ on $\Omega_c$. Then we insert a truncation factor $\phi_c$ compactly supported in $\Omega_c$, and being identically 1 on smaller compact set $K_c$, such that $|\phi_c({\msf{x}},y)\phi_c({\msf{x}}',y')K_{E_c}| \leq C'_Nc^{2N}{\msf{x}}^{n+1}({\msf{x}}')^{n+1}$ for all $N$. (The $(n+1)$-power factors are introduced to deal with the scattering density, which is as in \eqref{eq:quant-sc}.) Then we apply Schur's lemma on the integral operator bound 
 to conclude that 
\begin{equation}
 ||\phi_cE_c\phi_c||_{L^2_{\sct}(\bar{M}_0)
 \to L^2_{\sct}(\bar{M}_0)} \leq C''_Nc^{2N}. 
\end{equation}
 In particular, we can take $c_0$ so that this norm $<1$ when $c \in (0,c_0]$. Since those conjugations are invertible, this guarantees that $\phi_cG_cB_c \phi_c =   \Id + \phi_cE_{0c}\phi_c$ is invertible. So for the functions supported on $K_c$, $B_c$ is injective. $K_c$ can be arbitrary compact subset of $\Omega_c$ for arguments up to now. 

 Define $\bar{M}_c:=\{ \tilde{{\msf{x}}}+c \geq 0 \}$ and choose $K_c = \bar{M}_c \cap \{\rho \geq 0\}=\Omega_c$. $K_c$ is compact by our choice of $\tilde{{\msf{x}}}$.
 Support conditions are encoded by subscripts below. For example, $H^{s,r}_{\sct}(\bar{M}_c)_{K_c}$ is the space consists of those functions in $H^{s,r}_{\sct}(\bar{M}_c)$ which have support in $K_c$.  We have
$$
||v||_{H^{s,r}_{\sct}(\bar{M}_c)_{K_c}}  \leq C ||B_cv||_{H^{s+1,r}_{\sct}(\bar{M}_c)}.
$$
If we express this in terms of $A = L \circ I_0$, this is (with $f = e^{\frac{F}{{\msf{x}}}}v$):
$$
||f||_{e^{\frac{F}{{\msf{x}}}}H^{s,r}_{\sct}(\bar{M}_c)_{K_c}}  \leq C ||Af||_{e^{\frac{F}{{\msf{x}}}}H^{s+1,r+1}_{\sct}(\bar{M}_c)}.
$$

We can get rid of the $r$-indices with the cost of increasing the power of left hand side to $e^{\frac{F+\delta}{{\msf{x}}}}$. That is:\\
\begin{equation}\label{eq:inj est 1}
||f||_{e^{\frac{F+\delta}{{\msf{x}}}}H^{s}_{\sct}(\bar{M}_c)_{K_c}}  \leq C ||Af||_{e^{\frac{F}{{\msf{x}}}}H^{s+1}_{\sct}(\bar{M}_c)},
\end{equation}
which proves injectivity of the local geodesic ray transform. 

Finally we consider the boundedness of operators involved. We consider the decomposition $A = L \circ I_0$, and show that $L$ is bounded from $H^s(PSX|_{\bar{M}_c})$ to $\msf{x}^{-s+1}H^s(\bar{M}_c)$. In order to prove this, we decompose $L$ into $L=M_2 \circ \Pi \circ M_1$, with $M_2,\Pi,M_1$ being 
\begin{align*}
 M_1:& \ H^s([0,+\infty)_{\msf{x}} \times \mathbb{R}_y^{n-1} \times \mathbb{R}_\lambda \times \mathbb{S}^{n-2}_{\omega}) \rightarrow H^s([0,+\infty)_{\msf{x}} \times \mathbb{R}_y^{n-1} \times \mathbb{R}_\lambda \times \mathbb{S}^{n-2}_{\omega}), \\
   &(M_1u)({\msf{x}},y,\lambda,\omega) = {\msf{x}}^s \chi(\frac{\lambda}{{\msf{x}}},y)u({\msf{x}},y,\lambda,\omega),\\
 \Pi: & \ H^s([0,+\infty)_{\msf{x}} \times \mathbb{R}_y^{n-1} \times \mathbb{R}_\lambda \times \mathbb{S}^{n-2}_{\omega} ) \rightarrow H^s([0,+\infty)_{\msf{x}} \times \mathbb{R}_y^{n-1}),
 \\&(\Pi u)({\msf{x}},y) = \int_{\mathbb{R}_\lambda \times \mathbb{S}^{n-2}_{\omega}} u({\msf{x}},y,\lambda, \omega)d\lambda d\omega,\\
 M_2:& \ H^s([0,+\infty)_{\msf{x}} \times \mathbb{R}_y^{n-1}) \rightarrow {\msf{x}}^{-(s+1)}H^s([0,+\infty)_{\msf{x}} \times \mathbb{R}_y^{n-1}), \\ 
 &(M_2f)({\msf{x}},y)={\msf{x}}^{-(s+1)}f({\msf{x}},y).\\
\end{align*}

Consider the boundedness of $M_1$ when $s \in \N$ first. The general case follows from interpolation. Consider derivatives of ${\msf{x}}^s \chi(\frac{\lambda}{{\msf{x}}},y)u({\msf{x}},y,\lambda,\omega)$ up to order $s$. Each order of differentiation on $\chi$ gives an ${\msf{x}}^{-1}$ factor, which is canceled by ${\msf{x}}^s$ and the remaining part belongs to $L^2$ by smoothness of $\chi$ and $u\in H^s$. 
$M_2$ is bounded by the definition of the space on the right hand side. The operator $\Pi$ is a pushforward map, integrating over $|\lambda| \leq C |{\msf{x}}|$ (notice the support condition after we apply $M_1$), hence bounded via Minkowski inequality.
So we have boundedness of $L$ from $H^s(PSX|_{\bar{M}_c})$ to $\msf{x}^{-s+1}H^s(\bar{M}_c)$.
Since the scattering Sobolev spaces and the usual Sobolev spaces only differ by a polynomial weight and such weight, together with $\msf{x}^{-s+1}$ above, can be absorbed into the exponential weight, and we know that $L$ is bounded from $H^{s+1}(PSX|_{\bar{M}_c})$ to $e^{\frac{F}{{\msf{x}}}}H^{s+1}_{\sct}(\bar{M}_c)$.

On the other hand, $I_0$ itself is a bounded operator. This comes from the decomposition $I_0 = \tilde{\Pi} \circ \Phi^*$, where $\Phi$ is the geodesic coordinate representation $\Phi(z,\nu,t)=\gamma_{z,\nu}(t)$ and $\tilde{\Pi}$ is integrating against $t$, which is bounded as a pushforward map. Because the initial vector always has length 1 on the tangent component, the travel time is uniformly bounded. 
By the property of the exponential map, after potentially shrinking the region to make the travel time short, 
$\Phi$ has surjective differential, hence the pull back is bounded. Consequently $I_0$ is bounded.

The boundedness of $L$ gives us an estimate
\begin{equation}\label{eq:inj est 2}
 ||Af||_{e^{\frac{F}{{\msf{x}}}}H^{s+1}_{\sct}(\bar{M}_c)} \leq C_1 ||I_0 f||_{H^{s+1}(PSX|_{\bar{M}_c})},
\end{equation}
where we require $f$ to have support in $K_c$. 
Combining \eqref{eq:inj est 1} and \eqref{eq:inj est 2} completes the proof of Theorem~\ref{thm:inverse-main}.

\subsection{Problems}

\begin{problem}\label{prob:classical Radon}
It is instructive to first consider the `flat model' for the normal operator. So, let us consider the classical (localized) Radon transform on $\R^n$, $n \geq 2$ which takes $f \in \Schw(\R^n)$  to the function $I_0 f$ of $(z, \omega)$, where $z \in \R^n$, and $\omega \in S^{n-1}$, given by the integral along the corresponding straight line (geodesic):
\begin{equation}
I_0f (z, \omega) = \int f(z + t \omega) \phi(t) \, dt. 
\end{equation}
Here we will take $\phi$ to be a real, nonnegative, even function, positive for $|t| \leq 1$ and supported where $|t| \geq 2$. 
The Schwartz kernel of this operator is $K(x, \omega; y) = \int \delta(y-z-t\omega) \phi(t) dt$. 
The adjoint operator $L$  is 
\begin{equation}
    L v(y) = \int v( z, \omega) \delta(y - z - t\omega) \phi(t) \, dt \, dz \, d\omega = 
    \int v( y - t\omega, \omega) \phi(t) \, dt \,  d\omega. 
\end{equation}
The composition is then the localized normal operator: 
\begin{equation}\label{eq: A defn prob}
    A = L \circ I_0, Af(y) = \iiint f(z + t_1\omega - t_2 \omega) \phi(t_1) \phi(t_2) \, dt_1 \, dt_2 \, d\omega 
    = \int f(z - t\omega) \tilde \phi(t) dt d\omega
\end{equation}
where $\tilde \phi = \phi * \phi$ is the convolution of $\phi$ with itself. 
So, $Af$ at $y \in \R^n$ is given by integrating $f$ over all straight lines through $y$ (localized using $\tilde \phi$ so we only consider points at distance $\leq 4$ from $y$), and then integrating this over all possible directions $\omega$. We will also consider a variant of this where we don't integrate over all possible directions, but only those with $\omega_1$ small, where we write $\omega = (\omega_1, \sqrt{1 - \omega_1^2} \omega')$ with $\omega' \in S^{n-2}$. This is achieved by inserting an additional cutoff $\chi(\omega_1)$ into the integral for $A$. 

(a) Assume that $n \geq 2$. Without the cutoff function $\chi(\omega_1)$, show that $A$ is an elliptic scattering pseudodifferential operator of order $(-1, 0)$, with positive principal symbol. Hint: first show that the kernel takes the form
$$
K(y, z) = K(y-z) =  \tilde \phi(|y-z|) |y-z|^{-(n-1)}. 
$$
Show that the Fourier transform of $\tilde \phi$ is a nonnegative function, and hence show that the Fourier transform of $K$ in the variable $y-z$ is strictly positive and asymptotic to $c |\xi|^{-1}$ with $c > 0$ as $|\xi| \to \infty$. 
Alternatively, return to \eqref{eq: A defn prob} and express in terms of a delta function, write this delta function in terms of a Fourier integral and hence compute the symbol.

(b) Assume now that $n \geq 3$. Assume that a cutoff function $\chi(\omega_1)$ satisfying $\chi(0) > 0$, $\chi \geq 0$ has been inserted into the definition of $A$ (that is, we multiply the Schwartz kernel by $\chi(\omega_1)$).  Show that this new $A$ is still an elliptic scattering pseudodifferential operator. Hint: use the nonnegativity of both $\chi$ and $\widehat{\tilde \phi}$. Why is $n \geq 3$ a necessary condition for this result? Explain clearly the geometric significance of $n \geq 3$ that allows us to deduce ellipticity in this case. 

(c) In either case, deduce the injectivity of $I_0$. 
\end{problem}

\begin{problem} Verify that $\tilde{{\msf{x}}}$ given by \eqref{tildex} has properties we want.
\end{problem}

\begin{problem} What is the expression of the damping factor introduced by the conjugation in terms of $X = \frac{{\msf{x}}-{\msf{x}}'}{{\msf{x}}{\msf{x}}'}$? Why this is good for us?
\end{problem}

\begin{problem} Prove \eqref{eq:gamma1-est}, and verify the claim about the uniform upper bound of $t$. \label{ex:geodesic-gamma1-est}
\end{problem}

\begin{problem} Prove that $a_F$ in \eqref{eq3} is indeed symbolic with respect to scattering frequencies $(\tau,\mu)$. (Hint: Fix the order of derivatives you want to deal with, then perform stationary expansion up to some level depending on this order.)  \label{ex:symbolic-fiber-infinity-inverse}
\end{problem}


\begin{problem}  (i)Why this proof does not work in two dimension? \\
(ii) What conclusion can we say in two dimension using this proof? \label{ex:inverse-2D-gap}
\end{problem}

\section{Lecture 9: Helmholtz operator}
Let $g$ be an asymptotically Euclidean metric on $\R^n$. We assume that 
$$
g_{ij} - \delta_{ij} \in S^{-1}(\R^n). 
$$
The  metric Laplacian $\Delta_g$ , given by 
\begin{equation}
    \Delta_g f = \sum_{i,j} \frac1{\sqrt{g(x)}} D_{x_i} (g^{ij}(x) \sqrt{g(x)} D_{x_j} f),\quad g(x) = \det g_{ij}(x). 
\end{equation}
is the operator associated to the Dirichlet form
$$
Q(f) = \int_{\R^n} |\nabla_g f|^2 d\mathrm{vol}(g) = \int_{\R^n} g^{ij}(x) D_{x_i} f \overline{D_{x_j} f} \sqrt{g(x)} \, dx. 
$$
(Recall that the Riemannian measure induced by $g$ is $\sqrt{g(x)} dx$.)
That is, 
$$
Q(u) = \ang{u, \Delta_g u}_{L^2(\sqrt{g(x)} dx)}. 
$$
Let $\lambda > 0$ be a spectral parameter, let $V \in S^{-1}(\R^n)$ be a real-valued potential function, and let  
\begin{equation}
    P = \Delta_g + V(x) - \lambda^2.
\end{equation}
We call $P$ a `Helmholtz operator'. (We remind the reader here that our sign convention is that $\Delta_g$ is positive as an operator, which is opposite to the standard PDE convention.) The operator $P_0 = \Delta - \lambda^2$ is called the free, or flat, Helmholtz operator (where $\Delta$ is the free Laplacian, associated to the flat metric $\delta_{ij}$).  We study the free case first. 

\subsection{Free Helmholtz operator}
The free Helmholtz operator $P_0$ is unitarily equivalent, via the Fourier transform, to the multiplication operator $|\xi|^2 - \lambda^2$ on $\R^n_{\xi}$. It is easy to see that $|\xi|^2$, as a multiplication operator, has continuous spectrum on $[0, \infty)$, and each spectral projection $E_I$ corresponding to a nontrivial interval $I$ has infinite dimensional range. It follows that $|\xi|^2 - \lambda^2$ is \emph{not} Fredholm on $L^2$, for any $\lambda \in [0, \infty)$. We will, nevertheless, find Hilbert spaces $\SX$, $\SY$ betweeen which $P_0$ (and $P$) is Fredholm (and, in fact, invertible) --- see Lecture 12. 

It is clear that $|\xi|^2$, as a multiplication operator, has no $L^2$ eigenfunctions. In fact $|\xi|^2 - \lambda^2$ vanishes on a set of measure zero, so if $f \in L^2(\R^n_\xi)$ satisfies $(|\xi|^2 - \lambda^2) f = 0$, then $f$ vanishes outside a set of measure zero and is zero in $L^2(\R^n_\xi)$. However, there are certainly many generalized (non-$L^2$) eigenfunctions, which comprise the \emph{scattering theory} of $P_0$. In the Fourier picture, these are simply distributions on $\R^n_\xi$ that vanish when multiplied by $|\xi|^2 - \lambda^2$. Such distributions are the product of the delta function $\delta(|\xi| - \lambda)$ with another distribution, but one has to be careful as the product of two distributions is not always well-defined. However, as $\delta(|\xi| - \lambda)$ is in $H^{-1/2 - \epsilon}(\R^n_\xi)$ for any $\epsilon > 0$, we can multiply it by $u \in H^{1/2 + \epsilon}(\R^n_\xi)$. This is because
$$
\ang{\delta_\lambda u, \phi} = \ang{\delta_\lambda, u \phi}, \quad \phi \in \Schw(\R^n_{\xi}), \quad \delta_\lambda = \delta_{|\xi| - \lambda}, 
$$
and $u \phi$ is in $H^{1/2 + \epsilon}$ which allows pairing with $\delta_\lambda$. One can also note that by Sobolev trace theorems, any $u \in H^{1/2 + \epsilon}(\R^n_\xi)$ can be restricted to the sphere $\{ |\xi| = \lambda \}$ as an $H^\epsilon$ function. Then we can take the tensor product with the $\delta$ function in the transverse direction. 

Observe that the spectral projection, for the operator multiplication by $|\xi|^2$, corresponding to the interval $[0, \lambda^2]$ is given by multiplication by $1_{|\xi| \leq \lambda}$. Formally differentiating in $\lambda$, we see that the spectral measure $dE(\lambda^2)$ is (formally) given by $\delta(|\xi| - \lambda) d\lambda$. From the discussion in the previous paragraph, the spectral measure maps $H^{1/2 + \epsilon}(\R^n_\xi)$ to $H^{-1/2 - \epsilon}(\R^n_\xi)$. In terms of weighted Sobolev spaces on $\R^n_\xi$, we can put any `spatial' (in terms of $\xi$!) order that we like, as the behaviour near $|\xi| = \infty$ is not relevant. Thus, the spectral measure for $P_0$ on the Fourier side maps $H^{1/2 + \epsilon, S}(\R^n_\xi)$ to $H^{-1/2 - \epsilon, S'}(\R^n_\xi)$ for any $S, S'$.
Conjugating by the Fourier transform, we see that the spectral measure for $P_0$ in physical space maps $H^{S, 1/2 + \epsilon}(\R^n_\xi)$ to $H^{S', -1/2 - \epsilon}(\R^n_\xi)$ for any $S, S'$. By abuse of notation we usually write this as $H^{S, 1/2 + \epsilon}(\R^n_\xi)$ to $H^{\infty, -1/2 - \epsilon}(\R^n_\xi)$ for any $S$. It turns out that \emph{all} nontrivial generalized eigenfunctions of $P_0$, or more generally $P$, fail to lie in $H^{\infty, -1/2}(\R^n_x)$, but there is an infinite-dimensional family that lie in $H^{\infty, -1/2 - \epsilon}(\R^n_x)$ (and many more, such as the plane waves $e^{i\lambda x \cdot \theta}$, where $|\theta| = 1$, that have slower decay than this). This special value $-1/2$ of the spatial order is called a `threshold' and will play a key role in the microlocal theory of scattering in Lectures 11 and 12. 

Let's consider the spectral measure $dE(\lambda^2)$ acting on a Schwartz function $f(x)$. 
On the Fourier side, this is the Schwartz function $\hat f(\xi)$ times $\delta_\lambda$. The inverse Fourier transform is smooth, but non-$L^2$, so $dE(\lambda^2)f$ must have nontrivial behaviour as $|x| \to \infty$. We can compute the asymptotic of the inverse Fourier transform,
$$
(2\pi)^{-n} \int_{S^{n-1}} e^{ix \cdot \lambda \theta} \hat f(\lambda \theta) \lambda^{n-1} d\theta, 
$$
using the stationary phase lemma. The general stationary phase lemma says that if $\beta$ is a large parameter and $\Phi$ is a smooth function of $y \in \R^{n-1}$ that has a nondegenerate critical point at $y_0$, that is, $d\Phi(y_0) = 0$ but its Hessian is nondegenerate, and no other critical points on $\supp \ a$ where  $a(y)$ is smooth and compactly supported, then 
\begin{multline}\label{eq:spl}
\int e^{i\beta \Phi(y)} a(y) dy \\ \sim \big( \frac{2\pi}{\beta} \big)^{(n-1)/2} |\det d^2\Phi(y_0)|^{-1/2} e^{i \beta \Phi(y_0) + i\pi \signum d^2 \Phi(y_0)/4} a(y_0) + O(\beta^{-(n+1)/2}). 
\end{multline}
Moreover, if there is no stationary point, then the integral is $O(\beta^{-N})$ for every $N$. The intuition is that the rapidly oscillating phase gives rise to cancellations everywhere except at a critical point. In an $O(\beta^{-1/2})$ neighbourhood of the critical point, the phase is approximately constant and leads to a contribution at order $(\beta^{-1/2})^{(n-1)}$, which is proportional to the volume of a ball of radius $\beta^{-1/2}$. 

Let us use \eqref{eq:spl} to compute the asymptotic of $dE(\lambda^2)f$ for Schwartz $f$. 
Notice that $\lambda |x|$ can be viewed as a large parameter multiplying the phase function $\hat x \cdot \theta$. 
As a function of $\theta$ with $\hat x$ fixed, this phase function is nonstationary except at the two points $\theta = \pm \hat x$, where we can only move $\theta$ in directions orthogonal to $\hat x$. To compute the Hessian of the phase function at these critical points, as we need for the stationary phase expansion, we assume (without loss of generality, due to the symmetry of the sphere) that $\hat x = (1, 0, \dots, 0)$. Then 
$$
\hat x \cdot \theta = \theta_1 = \sqrt{1 - \sum_{j=2}^n \theta_j^2} = 1 - \frac1{2} \sum_{j=2}^n \theta_j^2 + O\big( (\sum_{j=2}^n \theta_j^2)^2 \big). 
$$
The Hessian is thus $-\Id_{n-1}$ when $\theta = \hat x$ and $\Id_{n-1}$ when $\theta = -\hat x$. The stationary phase expansion then gives us 
\begin{multline}\label{eq:smoothefns}
    (2\pi)^{-n} \int_{S^{n-1}} e^{ix \cdot \lambda \theta} \hat f(\lambda \theta) \lambda^{n-1} d\theta 
    \\ 
    \sim   (2\pi)^{-n} \big( \frac{2\pi}{\lambda} \big)^{(n-1)/2} \sum_\pm  e^{\mp i\pi(n-1)/4}  \frac{e^{\pm i \lambda |x|} \hat f(\pm \lambda \hat x) }{|x|^{(n-1)/2}}, \quad |x| \to \infty. 
\end{multline}
(In fact, there is a complete expansion with the same oscillatory term $e^{\pm i \lambda |x|}$ and decreasing powers of $|x|$.) These generalized eigenfunctions take the form of a spherical wave, $e^{\pm i \lambda |x|}$ with a decay rate $|x|^{-(n-1)/2}$, and an amplitude $\hat f(\pm \lambda \hat x)$, up to fixed constants depending only $\lambda$ and dimension. Thus we see that the expansion, to leading order, only depends on $\hat f(\xi)$ on the sphere $\{ |\xi| = \lambda \}$, consistent with the earlier discussion. 
We see from this explicit expansion that, as long as $\hat f$ does not vanish identically on $\{ |\xi| = \lambda \}$, \eqref{eq:smoothefns} is in $H^{\infty, -1/2 - \epsilon}(\R^n_x)$ for every $\epsilon > 0$, but not in $H^{\infty, -1/2}(\R^n_x)$. 

We can compute the scattering wavefront set of spherical waves $e^{\pm i\lambda |x|}$, by noting that they are annihilated by $D_r \mp \lambda$ and $D_{y_j}$, where we use polar coordinates $(r, y)$ and the $(y_1, \dots, y_{n-1})$ are angular coordinates, that is, functions of $\hat x$. The scattering wavefront set is therefore in the intersection of the characteristic varieties of these operators (according to \eqref{eq:WF defn}). We use scattering frequency coordinates $(\tau, \mu_1, \dots, \mu_{n-1})$ dual to $-D_r, r^{-1} D_{y_j}$, consistent with the usage in Lecture 7. The intersection of the characteristic varieties is
$$
\{|x|^{-1} = 0,  \tau = \mp \lambda, \mu = 0 \}. 
$$
We call these sets \emph{radial sets}, and denote them $\Rin$ (where $\tau = +\lambda$) and $\Rout$  (where $\tau = -\lambda$). Notice that $\mu = 0$ is equivalent to saying that $\hat x = \pm \hat \xi$ are parallel.  More on the radial sets soon.

Notice that we can also find the scattering wavefront set by first calculating, on the Fourier side, the usual wavefront set of $\delta_\lambda$ times a smooth function, and then applying the action of the Fourier transform on the scattering cotangent bundle. See exercises. 

\subsection{General Helmholtz operator} 
As a general Helmholtz operator $P$ is not translation invariant, we cannot use the Fourier transform as with the free Helmholtz operator; instead we use the scattering calculus. We note that the principal symbol of $P$ is elliptic at frequency infinity, and is equal to the principal symbol of $P_0$ at spatial infinity, where it is not elliptic. Thus, $\Char(P) = \Char(P_0)$, and the Hamilton vector field of $P$ on $\Char(P)$ coincides with the Hamilton vector field of $P_0$, which makes it very easy to compute. 

Recall that the characteristic set of $P_0$, and hence also $P$, is given by the sphere $|\xi| = \lambda$ over each point $\hat x$ at spatial infinity. That is, the chacteristic set is an $(n-1)-$sphere of radius $\lambda$ in each fibre of phase space lying over spatial infinity, which is itself an $(n-1)$-sphere.

Now we compute the Hamilton vector field of $P_0$ (or equivalently, $P$) on $\Char(P)$. In Cartesian coordinates, $(x, \xi)$, this takes the form $2\xi \cdot \partial_x$. This is nothing more than straight line motion (in physical space) in the direction $\hat \xi$, with speed $2\lambda$. The coordinate $\xi$ does not change. We want to understand what this looks like rescaled by a factor of $|x|$ or equivalent rescaling, and then restricted to $\Char(P)$ at spatial infinity. (In Lecture 5 we rescaled an operator of order $(2,0)$ such as $P$ by a factor $|x| |\xi|^{-1}$, but we can dispense with the scaling factor in $\xi$ since $\Char(P)$ does not reach frequency infinity. It is also sometimes convenient to replace the factor $|x|$ by some other factor that is homogeneous of degree one in $x$, such as $x_j$.) 

We use projective coordinates: assuming that $x_1$ is a dominant variable locally, meaning that $x_1 \geq c |x|$ for some $c > 0$, we can use $\rho := \pm 1/x_1$ as a boundary defining function for spatial infinity locally (with the sign chosen so that $\rho_s$ is nonnegative), and $y_j := x_j/x_1$, $j \geq 2$, as the angular coordinates. In preparation for studying the Hamiltonian flow near the radial sets, we observe that near the radial sets, if $x_1$ is a dominant variable then so is $\xi_1$, since $\hat x = \pm \hat \xi$ at the radial sets; in particular, $\xi_1$ will not vanish so we may divide by it locally. So we may define smooth coordinates  $\omega_j := \xi_j/\xi_1$, $j \geq 2$, in a neighbourhood of the radial sets and complete the coordinate system with $p = |\xi|^2 - \lambda^2$, which is easily checked to have differential linearly independent from $d\omega_j$.  In the coordinates $(\rho, y, p, \omega)$ we can compute
\begin{equation}\begin{gathered}
\rho^{-1} \xi_1 \partial_{x_1} =  \xi_1 \Big( \mp \rho \partial_\rho \pm \sum_j y_j \partial_{y_j} \Big), \\
\rho^{-1} \sum_{j \geq 2} \xi_j \partial_{x_j} = \pm \xi_1 \omega_j \partial_{y_j}, \\
\Longrightarrow \rho^{-1} H_p = \pm \xi_1 \Big( - \rho \partial_\rho + (\omega_j - y_j) \partial_{y_j} \Big). 
\end{gathered}\end{equation}
We see that this vanishes at $\rho = 0, \ \omega_j = y_j$, which is to say, exactly when $x$ and $\xi$ are parallel. Further changing variable from $y_j$ to $v_j = \omega_j - y_j$ we we can express this as 
\begin{equation}
    \rho^{-1} H_p = \mp \xi_1 \Big( \rho \partial_\rho + \sum_j v_j \partial_{v_j}\Big). 
    \label{eq:misc_a}
\end{equation}
We see from this that the radial sets $\{ \rho = 0, v = 0 \}$ are either sources or sinks of the Hamilton vector field, depending on the sign $\mp \xi_1$. This has the same sign as $-(\signum x)(\signum \xi)$, which is positive on $\Rin$ and negative on $\Rout$. We see that $\Rin$ is a source, and $\Rout$ a sink, of the Hamilton flow. 

Next we show that the two radial sets are \emph{the global source and sink} of the Hamilton flow. To do this we note that the two coordinate systems above (one where $x_1 > 0$, one where $x_1 < 0$) cover the whole of $\Char(P)$ except on the hypersurface where $x_1/|x| = 0$, where $x_1$ is not dominant and we must switch to a different dominant variable. We continue to assume that $\xi_1 \neq 0$ (as $\xi$ is constant under the flow), and assume without loss of generality that $x_2$ is a dominant variable (locally, on a particular trajectory) where $x_1/|x| = 0$. We switch to the new boundary defining function $\rho' = 1/x_2$. We find then that 
$$
(\rho')^{-1}  H_p \big( \frac{x_1}{x_2} \big)  = H_p(x_1) - \frac{x_1}{x_2} H_p(x_2) = 2\xi_1 - 2 \frac{x_1}{x_2} \xi_2 .
$$
This is nonzero when $x_1/x_2$ is close to zero, since $\xi_1 \neq 0$. Thus, the trajectory passes from the coordinate chart where $x_1 > 0$ to the coordinate chart where $x_1 < 0$ (or vice versa, depending on the sign of $\xi_1$). We see that every trajectory converges to $\Rin$ as the `time' parameter tends to $-\infty$ and to $\Rout$ as the parameter tends to $+\infty$. 

There is an alternative, more geometric way of thinking about the Hamilton, or bicharacteristic, flow. The flow in the interior of phase space is just straight line motion. If we fix $\xi_0$ with $|\xi_0| = \lambda$ and consider all the trajectories with $\xi = \xi_0$, these form a pencil of parallel lines with direction $\hat \xi_0$. Each such line is determined by its intersection with the hyperplane $\Pi_0$ orthogonal to $\xi_0$. Assume that this point of intersection is not at the origin. If we perform a dilation, centred at the origin, so that this point of intersection tends to infinity in $\Pi_0$, we obtain a limiting trajectory at spatial infinity. Moreover, if we scale the Hamilton flow by a factor of $|x|$, or something homogeneous of degree $1$ in $x$, then the flow is homogeneous of degree zero and hence has a limit on the spatial sphere at infinity. It is intuitively clear that each such trajectory tends to $\hat x = - \hat \xi_0$ (which is on $\Rin$) as `time' tends to $-\infty$ and to $\hat x = +\hat \xi_0$, on $\Rout$, as it tends to $+\infty$. Each trajectory is half a great circle connecting these two antipodal points (we can think of it as the intersection of a 2-plane through the origin, namely the dilation of a given line not through the origin, with the sphere at infinity). 

We can divide the phase space (more precisely, the boundary $\partial \comphase$ of phase space) into four disjoint sets:

\begin{itemize}
    \item 
First there is the elliptic set, $\Ell(P)$. Here we understand regularity perfectly: $u$ is in $H^{s+2, r}$ microlocally iff $Pu$ is in $H^{s, r}$ microlocally. 

\item Next there is 
$$
\Char(P) \setminus (\Rin \cup \Rout),
$$
on the characteristic variety and away from the radial sets. Here the rescaled Hamilton vector field $|x| H_p$ is nonvanishing, and the propagation theorems from Lecture 5 apply. That is, if $Pu$ is microlocally in $H^{s, r}$ on this set, then $\WFsc^{s+1, r-1}(u)$ (or equivalently, its complement, where $u$ is microlocally in $H^{s+1, r-1}$) is a union of bicharacteristics. 

\item Finally, we have the two radial sets $\Rin$ and $\Rout$. At present, we cannot say much about the radial sets, although we have seen --- at least for the  the flat model $P_0$ --- that they are precisely where the wavefront set of `nice' generalized eigenfunctions live, so they are clearly important! This will be the topic of Lectures 11 and 12, where we will explain that there are microlocal propagation estimates of a more subtle kind at the radial points. In particular, whether the spatial order is above or below the `threshold' value of $-1/2$ governs what sort of microlocal propagation can happen. 
\end{itemize}

We can re-express the geometry of the Hamilton vector field near the radial sets in a more analytic form, which is directly useful for positive commutator estimates. Let us define a \emph{quadratic defining function} for a submanifold $S$ to be a function $q \geq 0$ such that $S = \{ q = 0\}$, and $q$ vanishes to precisely second order at $S$. For example, if $S$ is given locally by the vanishing of $f_1, \dots, f_k$, where $df_j$ are linearly independent, then $q := \sum f_j^2$ is a quadratic defining function for $S$. 

\begin{proposition}
    There exist quadratic defining functions $q_+$ for $\Rin$ and $q_-$ for $\Rout$, regarded as submanifolds of $\comphase$, with the property that 
    \begin{equation}
\scH_p^{2,0}(q_\pm) = \pm \iota q_\pm \pm F + E + \rho G,
    \end{equation}
where $\iota > 0$, $F \geq 0$, $E$ vanishes cubically at $\Rin$ and $\Rout$, $\rho$ is a boundary defining function for spatial infinity, and $G$ vanishes linearly at $\Rin$, $\Rout$. 
\end{proposition}

\begin{proof} The calculation \eqref{eq:misc_a} shows that, locally, we can take $\rho^2 + |v|^2$ as a quadratic boundary defining function and we have the required form, with $E = F = G = 0$. Next we glue together the local quadratic boundary defining functions to a global one. Thus let $\chi_i$ be a partition of unity; we form 
$$
q_\pm = \sum \chi_i q_i. 
$$
Let us check that, if $q_i$ all satisfy the required conditions, then so does $q_\pm$. We take the $+$ sign for definiteness. Then
$$
\scH_p^{2,0}(q_+) = \sum_i \chi_i \scH_p^{2,0}(q_i) + q_i \scH_p^{2,0}( \chi_i). 
$$
We take $\iota = \min_i \iota_i$; then the first term $\chi_i \scH_p^{2,0}(q_i)$ satisfies the required condition with a positive $\iota$, and an additional leftover positive part $F$ if the $\iota_i$ are not all equal.  For the second term, notice that $\scH_p^{2,0}( \chi_i)$ vanishes at $\Rin$ since $\scH_p^{2,0}$ itself vanishes there. Thus, this term vanishes cubically at $\Rin$, so contributes to the $E$ term. (In this case we do not require a $G$ term, but for more general operators we may have such a term, and it is an `acceptable' term for the purpose of positive commutator estimates.)
\end{proof}

\subsection{Exercises}
\begin{problem}\label{prob:smoothefns}

(a) Find the usual wavefront set of $\delta(|\xi| - \lambda) f(\xi)$, where $f \neq 0$ is smooth. Then apply the mapping properties of the Fourier transform on the scattering wavefront set to deduce the scattering wavefront set of generalized eigenfunctions as in \eqref{eq:smoothefns}. 

(b) More precisely, find the values of $(s, r)$ for which $\WF_{\sca}^{s,r}(u)$ is nonempty, where $u$ is the inverse Fourier transform of $\delta(|\xi| - \lambda) f(\xi)$ as in part (a). You should find that the necessary and sufficient condition is that $\WF_{\sca}^{s,r}(u) \neq \emptyset \Longleftrightarrow r \geq -1/2$. 

(c) Do the same for a delta function at one point $\delta_{\xi_0}$ where $|\xi_0| = \lambda$. 
\end{problem}

\begin{problem} (i) Consider the tensor product $u = \delta(|\xi| - \lambda) \otimes f(\hat \xi)$, where $f$ is a distribution on the $(n-1)$-sphere. Show that if $f \in H^s(S^{n-1})$ for $s \geq 0$, then $u \in H^{-1/2 - \epsilon}(\R^n_\xi)$, but if $f \in H^s(S^{n-1})$ for $s < 0$, then we can only say that $u \in H^{-1/2  + s}(\R^n)$ (for general $f$). Interpret these results in terms of the Fourier transform of $u$, which is a solution to the free Helmholtz equation.

(ii) Suppose that $f \in H^s(S^{n-1})$ for $s > 0$. Show that if $(\xi, \hat x)$ is in the cosphere bundle of $\R^n_\xi$ and $|\xi| = \lambda$, and in addition if $\hat x$ is not parallel to $\xi$, then $u= \delta(|\xi| - \lambda) \otimes f(\hat \xi)$ is microlocally in $H^{s-1/2}$ near $(\xi, \hat x)$. This result improves on that in part (i) except for the points where $\hat x$ is parallel to $\xi$. Interpret the condition that $\hat x$ is parallel to $\xi$ in terms of the radial sets of the Helmholtz equation. 
\end{problem}

\begin{problem}\label{prob:1dFredholm}
Compute the Hamilton vector field $\scH^{1,0}_p$ where $p = \xi_1$ is the symbol of $D_{x_1} \in \Psisc^{1,0}(\R^n)$, on the boundary of $\comphase$. Find the `radial sets' (that is, where the rescaled Hamilton vector field $\scH^{1,0}_p$ vanishes on the boundary of $\comphase$), and show that they have a source/sink structure similar to the Helmholtz operator.   
\end{problem}

\begin{problem}
Consider geodesic flow on the sphere bundle $M$ of a hyperbolic manifold. Locally it is a piece of the sphere bundle of the hyperbolic plane. We can use the upper half plane model of the hyperbolic plane, and use coordinates $x, y$ for the hyperbolic plane itself, and $\theta$ for the sphere direction ($\theta$ can be taken to be the angle between the vector and $\partial_y$, say). 
In these coordinates, the generator of geodesic flow takes the form 
\begin{equation}
P = - y\sin \theta D_x +  y\cos \theta D_y  + \sin \theta D_\theta.
\label{vfH}
\end{equation}
We let consider the induced Hamilton flow on $T^* M$. 

(i) Using the coordinates $(x, y, \theta)$ and their dual coordinates, $(\xi, \eta, \omega)$, we have 
$$
h := \sigma(P) = -y \sin \theta \xi +  y\cos \theta \eta  + \sin \theta \omega.
$$
We also define dynamical variables 
\begin{align}
 u &= y \cos \theta \xi + y \sin \theta \eta + (1-\cos \theta) \omega \\
s &= y \cos \theta \xi + y \sin \theta \eta - (1 + \cos \theta) \omega. 
\end{align}
(Remark: $s$ and $u$ are the generators of forward and backward horocycle flow, respectively. Also, $s$ and $u$ are linear coordinates on the stable and unstable bundles, respectively, with respect to the geodesic flow, which is hyperbolic in the dynamical sense.) 
Show that the Hamilton flow in coordinates $x, y, \theta, h, s, u$ is given by 
$$
- y\sin \theta \partial_x +  y\cos \theta \partial_y  + \sin \theta \partial_\theta + u \partial_u - s \partial_s. 
$$
In particular, the flow decouples into the original flow on $M$ in the $(x, y, \theta)$ variables, and a simple flow in $(u, s)$ variables.

(ii) Find two sets in the boundary of phase space (we only have the boundary at fibre-infinity here since we are working on a compact manifold) such that one is a source, and the other a sink, for the Hamilton flow. Remark: the situation is a little different from the case of the Helmholtz operator, as here the vector field does not vanish on the two sets, but is rather tangent to the sets. 

\end{problem}

\section{Lecture 10: Variable order operators and spaces}
We need one more technical ingredient in order to realize the Helmholtz operator as a Fredholm operator, and that is variable order 
Sobolev spaces. As we only require the spatial order to be variable (as $\Char(P)$ is contained in the spatial boundary), we only treat this case below. That is, we will take the differential order to be constant, and the spatial order to be variable. 

\subsection{Variable order symbols and operators} We start by defining variable order symbols. A variable order $\ml$ is by definition a smooth function on $\comphase$, or equivalently, a classical symbol of order $(0,0)$. A variable order will always be indicated with a bold font, to distinguish it from a constant order. Let $m \in \R$. We will define a space of symbols of order $(m, \ml)$. A typical elliptic element of such a space should be 
$$
\ang{\xi}^m \ang{x}^{\ml(x, \xi)}.
$$
But there is an issue: if we differentiate this expression then we get log factors:
\begin{align}
    D_{x_j} \ang{x}^{\ml(x, \xi)} &= \ml(x, \xi) \ang{x}^{\ml(x, \xi)-1 } \frac{x_j}{\ang{x}} + (D_{x_j} \ml(x, \xi)) \log \ang{x} \ang{x}^{\ml(x, \xi)}, \label{eq:varorder1} \\
    D_{\xi_j} \ang{x}^{\ml(x, \xi)} &= (D_{\xi_j} \ml(x, \xi)) \log \ang{x} \ang{x}^{\ml(x, \xi)}. \label{eq:varorder2}
\end{align}
Since $\ml$ is a symbol of order $(0,0)$, $(D_{x_j} \ml(x, \xi)) = O(\ang{x}^{-1})$, and $(D_{\xi_j} \ml(x, \xi)) = O(\ang{\xi}^{-1})$, giving us the expected additional decay rates. However, due to the log factors, these derivatives of $\ang{x}^{\ml(x, \xi)}$ are overall logarithmically larger than would be the case if the analogous symbol estimates to our constant-order spaces were obeyed. This is a minor issue: as you may know, there are many variants of symbol spaces, such as H\"ormander's $S^m_{\rho, \delta}$ symbol classes, in which $\xi$ derivatives only improve the decay by a power $\xi^{-\rho}$ where $\rho$ can be smaller than $1$, and $x$ derivatives can worsen the decay by a factor $\ang{\xi}^\delta$ where $\delta$ can be positive. Here, we define a class of symbols that has symbol estimates that are very slightly weaker than the estimates we have used thus far, and will accommodate the log losses that we observe above. 

We express our symbol estimates in terms of fixed growth under repeated applications of vector fields. Recall that last week, we showed that the symbol estimates with fixed orders could be expressed in terms of membership in $\ang{\xi}^m \ang{x}^l L^\infty(T^* \R^n)$ under the repeated application of smooth vector fields tangent to the boundary of $\comphase$, which we refer to as b-vector fields, and denote the space of such vector fields $\SV_b(\comphase)$. Here, we modify the space to 
\begin{equation}\label{eq:vbtilde}
\tilde \SV_b(\comphase) = \cup_{\delta > 0} \ang{x}^{-\delta} \SV_b(\comphase). 
\end{equation}
These vector fields have an (arbitrarily slight) amount of extra decay at the spatial boundary, compared to b-vector fields, which will compensate for the logarithmic factors. We then define
\begin{multline}
    S^{m, \ml}_{\sca}(T^* \R^n) = \big\{ a \in \ang{\xi}^s \ang{x}^{\ml(x, \xi)} L^\infty(T^* \R^n) \mid \forall \, k \in \mathbb{N}, \forall \, V_1, \dots, V_k \in \tilde \SV_b(\comphase), \\ V_1 \dots V_k a \in \ang{\xi}^s \ang{x}^{\ml(x, \xi)} L^\infty(T^* \R^n) \big\}. 
\end{multline}
It is clear that, by applying \eqref{eq:varorder1}, \eqref{eq:varorder2} iteratively that $\ang{\xi}^s \ang{x}^{\ml(x, \xi)}$ is indeed in this space, as we intend. The symbol space $S^{m, \ml}_{\sca}(T^* \R^n)$ is a Fr\'echet space, although that is slightly less obvious than the case of constant order. We observe that in the intersection over $\delta > 0$ in \eqref{eq:vbtilde}, it suffices to take an intersection over a \emph{sequence} of $\delta$ decreasing to $0$. If we do this, and also note that it suffices to take only vector fields from a finite generating set (over $C^\infty(\comphase)$), then we can realize $S^{m, \ml}_{\sca}(T^* \R^n)$ as a space defined by countably many norms. 

These variable order symbol classes have similar basic properties as we saw in Lecture 1, but note the slight modification in the second item:
\begin{itemize}
    \item If $m' \leq m$ and $\ml' \leq \ml$ then $S^{m',\ml'}_{\sca}(\R^n \times \R^n) \hookrightarrow S^{m,\ml}_{\sca}(\R^n \times \R^n)$ is continuous. 
    \item $D_x^\alpha D_\xi^\beta$ maps $S^{m, \ml}_{\sca}(\R^n \times \R^n) \to S^{m-|\beta|,\ml-|\alpha|+ \delta}_{\sca}(\R^n \times \R^n)$ continuously, for all $\delta > 0$. 
    \item Pointwise multiplication is continuous 
    $$
S^{m, \ml}_{\sca}(\R^n \times \R^n) \times S^{m', \ml'}_{\sca}(\R^n \times \R^n) \to S^{m+m', \ml + \ml'}_{\sca}(\R^n \times \R^n).     
    $$
    \item Density. The residual space of symbols $S^{-\infty, -\infty}_{\sca}$ is not dense in $S^{m, \ml}_{\sca}$. However, if $a \in S^{m, \ml}_{\sca}$, then there exists a sequence $a_j \in S^{-\infty, -\infty}_{\sca}$ that is uniformly bounded in $S^{m, \ml}_{\sca}$ and which converges to $a$ in the (slightly weaker) topology of $S^{m', \ml'}_{\sca}$ for any $s' > s$ and $\ml' > \ml$. (Notice that $\ml' > \ml$ automatically implies there exists a constant $\epsilon$ such that $\ml' \geq \ml + \epsilon$, by compactness of $\comphase$.) 
    \item Asymptotic summation: given a sequence of orders $(m_j, \ml_j), j \geq 0$ with $m_j \searrow -\infty$, $\ml_j \searrow -\infty$, and a sequence of scattering symbols $a_j \in S^{m_j, \ml_j}_{\sca}$, there exists $a \in S^{m_0. \ml_0}_{\sca}$ such that 
    $$
    a - \sum_{j=0}^{J-1} a_j \in S^{m_J, \ml_J}_{\sca}.
    $$
    We call $a$ an asymptotic sum of the $a_j$; it is unique modulo $S^{-\infty, -\infty}_{\sca}$.
\end{itemize}

We quantize symbols in the standard way (using left quantization) to obtain the class of variable order pseudodifferential operators $\Psisc^{m, \ml}(\R^n)$. By item 1 above, $\Psisc^{m, \ml}(\R^n)$ is contained in a fixed order space, $\Psisc^{m, L}(\R^n)$, provided that $L$ is \emph{strictly} greater than $\sup \ml$. We deduce immediately that variable order operators map Schwartz functions to Schwartz functions, tempered distributions to tempered distributions, and so on. 

\subsection{Composition, principal symbol map and elliptic parametrix}
Composition is straightforward to prove, using again the observation that $\Psisc^{m, \ml}$ is contained in a fixed order space of operators $\Psisc^{m, L}$, provided that $L$ is \emph{strictly} greater than $\sup \ml$. Therefore, the composition of $A \in \Psisc^{m, \ml}$ and $B \in \Psisc^{m', \ml'}$ is a scattering pseudodifferential of order $(m+m', L + L')$. We want to show it is a variable order operator in the class $\Psisc^{m+m', \ml + \ml'}$. 
This follows from the asymptotic expansion of the symbol of the composition: using items 2 and 3 above, we see that the $k$th term in the expansion is in $S^{m+m'-k, \ml + \ml' - k + \delta}$, while the remainder after $k$ terms is in $\Psisc^{m+m'-k, L + L' - k}$. So each term in the expansion is in the expected class $\Psisc^{m+m', \ml + \ml'}$ and the remainder is as well, for a sufficiently large $k$. 

Conclusion: composition works in the expected way, namely
\begin{equation}\label{eq:varopcomp}
    \Psisc^{m, \ml} \circ \Psisc^{m', \ml'} \subset \Psisc^{m+m', \ml + \ml'}
\end{equation}
and moreover, the left-reduced symbol of the composition has the usual expansion in terms of the left-reduced symbol of the factors, with the only difference that the spatial order of the $k$th term drops by $k - \delta$ instead of $k$. 

It follows that the principal symbol map must be slightly adjusted: the range of the principal symbol map is no longer $S^{m, \ml}_{\sca} / S^{m-1, \ml - 1}_{\sca}$, but $S^{m, \ml}_{\sca}$ quotiented by the intersection  $\cap_{\delta > 0}S^{m-1, \ml - 1+\delta}_{\sca}$. It is usually more convenient to fix a specific small $\delta$ and view the principal symbol as taking values in $S^{m, \ml}_{\sca} / S^{m-1, \ml - 1 + \delta}_{\sca}$. Then we have the principal symbol short exact sequence:
 \begin{equation}
        0 \rightarrow \Psisc^{m-1, \ml-1+\delta} \rightarrow \Psisc^{m,\ml} \rightarrow S^{m,\ml}_{\sca} / S^{m-1,\ml-1 + \delta}_{\sca} \rightarrow 0,
    \end{equation}
where the third arrow is the principal symbol map, and is an algebra homomorphism. 

Moreover, if we consider the commutator of two variable order operators $A \in \Psisc^{m, \ml}$ and $B \in \Psisc^{m', \ml'}$, it is still true that the principal symbol of $i[A, B]$ is $H_a(b) = - H_b (a)$, but this now lives in $ \cap_{\delta > 0}\Psisc^{m+m'-1, \ml + \ml' -1+\delta}$. 

We next turn to ellipticity. We say that $A \in \Psisc^{m, \ml}$ is elliptic if its symbol is elliptic, and $a = \sigma_L(A)$ is elliptic if there exists $b \in S^{-m, -\ml}_{\sca}$ such that $ab - 1 \in S^{-1, -1 + \delta}_{\sca}$ for any $\delta > 0$. For example, the symbol 
$$
\ang{\xi}^m \ang{x}^{\ml(x, \xi)} \in S^{m, \ml}_{\sca} \text{ is elliptic.}
$$
The elliptic parametrix construction goes through, although in the iterative procedure we only gain $1 - \delta$ in the spatial decay, rather than the gain of $1$ in the fixed order case. This difference is inconsequential. We find:
\begin{proposition} Let $A \in \Psisc^{m, \ml}$ be elliptic. Then there exists a parametrix $B \in \Psisc^{-m, -\ml}$ such that 
$$
AB - \Id \in \Psisc^{-\infty, -\infty} \ni BA - \Id. 
$$
\end{proposition}
As before, we immediately deduce that the null space of $A$ consists of Schwartz functions. Moreover, acting on $L^2$, say, this shows that $BA$ is Fredholm, so has finite dimensional null space. Therefore, $A$ also has finite dimensional null space. 

\subsection{Variable order Sobolev spaces, boundedness}
We now come to variable order spaces. Recall that the fixed order Sobolev spaces are defined so that 
\begin{equation}
    u \in H^{s, r} \Longleftrightarrow \ang{D}^s \ang{x}^r u \in L^2. 
\end{equation}
We can think of $\ang{D}^s \ang{x}^l$ as an invertible operator in $\Psisc^{m,l}$. So the condition is that an invertible operator, with the same orders as the Sobolev space, maps to $L^2$. We can use the same definition for variable order spaces, except that it is slightly less obvious that there exists an invertible operator when the order is variable. Still this is not too difficult, and is left as an exercise. However, it is slightly easier to just make use of an elliptic operator $A$. So we make the following definition:

\begin{definition}
    Let $s \in \R$ and let $\mr \in C^\infty(\comphase)$ be a variable spatial order. We fix a constant order $L \leq \mr$. Let $A \in \Psisc^{s, \mr}$ be elliptic, and let $\Lambda = \ang{\xi}^s \ang{x}^L$. 
    The Sobolev space $H^{s, \mr}$ is defined to be the set of $u \in \Schw'(\R^n)$ such that $Au \in L^2$ and $\Lambda u \in L^2$, and the norm is defined to be 
    \begin{equation}\label{eq:var order Sob space}
        \| u \|_{H^{s, \mr}}^2 = \| Au \|_{L^2}^2 + \| \Lambda u \|_{L^2}^2. 
    \end{equation}
\end{definition}

The role of the $\Lambda$ term is simply to ensure that the RHS of \eqref{eq:var order Sob space} is strictly positive for every $u \neq 0$, in the case that $A$ has nontrivial null space.  If $A$ is known to have trivial null space, the $\Lambda$ term can be discarded. 

We leave as an exercise to check two important properties of this definition: first, the space $H^{s, \mr}$ as a set is independent of the choice of $A$ and $L$, and different choices lead to equivalent norms. Second, this is a Hilbert space --- in particular, it is complete. 

Then it is straightforward to check that the expected boundedness properties hold. That is, if $B \in \Psisc^{m, \ml}$ then $B$ maps $H^{s, \mr} \to H^{s-m, \mr-\ml}$ boundedly. See exercises. We see from this that elliptic operators are Fredholm mapping between the appropriate variable order Sobolev spaces. Notice that, as the null space consists of Schwartz functions, the null space does not depend on which Sobolev space is chosen as the domain of the operator. 

\subsection{Variable order propagation}
We next consider microlocal propagation in variable order spaces. Here the logarithmic factor when we differentiate the variable order has a more crucial effect. 

Let's revisit the proof of propagation from Lecture 6. We constructed a commutant $A$ with symbol (in the case of general orders, so $A$ should have orders $(2s, 2r)$ --- but we shall assume that our bicharacteristic $\gamma$ is contained in the spatial boundary and at finite frequency, so the differential order $s$ is irrelevant and we shall ignore it)
$$
a = \rho^{-2r} \chi_1(z_1)^2 \chi(z_1) \psi(z')^2, \quad \rho = 1/|x|, 
$$
where we recall that $\psi$ localizes to points close to $\gamma$, $\chi_1$ is the turn-on function and $\chi$ is the turn-off function which we took to have the specific form 
\begin{equation}
\chi_0(s_0 + \epsilon - z_1), \quad \chi_0(t) = \begin{cases} 0, t \leq 0 \\ e^{-\digamma /t}, t \geq 0 \end{cases}  \quad .
\end{equation}
This explicit form builds in a large parameter $\digamma$ which we take sufficiently large to dominate other terms that arise in the construction. 
In the proof of propagation, we need to compute $H_p a + p_1 a$, where $p_1 = i\sigma_L(P^* - P)$ is a symbol of order $(0,0)$. The presence of $\rho^{-2r}$ contributes a term of the form 
$$
(-2r) \rho H_p(\rho^{-1})
$$
which is another symbol of order $(0,0)$, so can be combined with $p_1$. But look at what happens when we replace the constant order $l$ with the variable order $\mr$: we get another additional term from $H_p$ hitting the variable order $\mr$, and this is 
\begin{equation}\label{eq:logterm}
-2 H_p(\mr) \log \rho. 
\end{equation}
Since $\mr$ is in $C^\infty(\comphase)$, so is $H_p \mr$. So this is logarithmically \emph{larger} than the other terms that we get when calculating $H_p a + p_1 a$. Now recall that we divided the terms up according to their sign. The negative terms we collected into $-b^2$, and the positive term became $e$. It was important that $e$ was supported in a small neighbourhood of $\gamma(0)$, since then the $E$ term is finite if we have a priori control of $u$ microlocally near $\gamma(0)$. If \eqref{eq:logterm} has a positive sign, this will completely ruin our strategy, since this term could have support over the whole of $\gamma$ if $\mr$ is strictly increasing there. 

On the other hand, if $\mr$ is nondecreasing, then the sign of \eqref{eq:logterm} is negative, so it could in principle be combined with the $-b^2$ term. However, it is problematic to literally include it in $-b^2$ since it is important that $b$ is a symbol, which requires smoothness properties. The definition of $b$ involves a square root, and we used the freedom to choose $\digamma$ large so that the the argument of the square root was smooth, positive and bounded away from zero. The smoothness would be problematic if we try to include the logarithmic term. The problem is that $H_p \mr$ might vanish sometimes, so sometimes the logarithmic factor is present, sometimes not. In our applications, we'll want to use variable orders that are constant on open sets and then strictly decreasing at other places. 

So instead, we treat the log term as a separate term which we discard as being negative. There are two approaches to this. One is to assume that the variable order has been chosen so that $H_p \mr$ is the square of a smooth function (recall this is true for the function $\chi_0$). In that case, the log term can be written $-(b')^2$ for some (non-classical) symbol $b'$, which can be quantized to $-B'^* B'$. This term, having the good sign, can be discarded, and the argument proceeds as before, except that instead of gaining $1/2$ in the order at each step, we only gain $1/2 - \delta$ (which is inconsequential). The second approach is to apply the sharp G{\aa}rding inequality which applies to operators with nonnegative symbols (or nonpositive, in this case), and shows that the operator is positive up to an error in a weaker Sobolev space. Again, this leads to a gain of $1/2 - \delta$ at each step. 

Since we have not discussed the sharp G{\aa}rding inequality, we take the first approach. This approach generally works, because usually \emph{we get to choose} the variable order, and it is usually straightforward to choose it such that it has the property that $H_p \mr$ is the square of a smooth function. We have thus sketched a proof of the following result:

\begin{theorem}\label{thm:mpe variable}
Let $P \in \Psi^{1,1}(\R^n)$ be as in Theorem~\ref{thm:mpe}. (It is easy to adjust the theorem for an operator with arbitrary orders.)
Let $\mr$ be a variable spatial order that is nonincreasing along the Hamilton flow of $P$, i.e. $H_p \mr \leq 0$. Let $\gamma$ be a null bicharacteristic. Then we can find $B, E, G$ as in Theorem~\ref{thm:mpe} with respect to $\gamma$, and for any $s, N$ a constant $C$ such that for all $u \in \Schw'(\R^n)$, 
\begin{equation}\label{eq:mpe-lec10}
\| Bu \|_{H^{s,\mr}} \leq C \Big( \| GPu \|_{H^{s, \mr}} + \| Eu \|_{H^{s,\mr}} + \| u \|_{H^{-N, -N}} \Big). 
\end{equation}
This inequality holds in the strong sense that if the RHS is finite, then so is the LHS, and the inequality holds. In particular it means that if $u \in H^{-N, -N}$, $Pu$ is in $H^{s, \mr}$ microlocally in $U$, and $u$ is in $H^{s,\mr}$ microlocally in $U_0$, then $u$ is in $H^{s,\mr}$ microlocally on $\Ell(B)$, in particular on $\gamma([0, s_0])$. 
\end{theorem}

\begin{remark} A moment's thought shows that it \emph{should} be impossible to prove the result for an order that increases along the Hamilton flow. In fact, if such propagation held, then one could consider a solution to $Pu = 0$, such that $u$ is in  $H^{s, l}$ microlocally at $\gamma(0)$ and propagate forward, finding that $u$ along the bicharacteristic was in a better space (i.e. with a higher spatial order) at $\gamma(s_0)$. Then one could switch to a different variable order that was decreasing, and propagate backward along the bicharacteristic, gaining even more spatial decay, concluding that $u$ was, microlocally near $\gamma(0)$, strictly better than initially assumed! This is clearly absurd. 
\end{remark}

\subsection{Exercises}
  \begin{problem} The simplest variable order space. Consider the operator $$P = \sqrt{g(x)} D_x \sqrt{g(x)}$$ on the real line, where $g(x) = |x|$ for $|x| \geq 1$ and is strictly positive for all $x$. The operator  $P$ is a scattering pseudodifferential operator of order $(1,1)$ and is elliptic except at spacetime infinity, i.e. $x = \pm \infty$, where $\xi = 0$. The variable order is only interesting on the characteristic set; we choose an order that is $\epsilon$ when $x = -\infty$ and $-\epsilon$ when $x = +\infty$. So let $\mr = \mr(x)$ be equal to $\epsilon$ for $x \leq -1$, $-\epsilon$ for $x \geq 1$ and monotone in between. 

    (a) Prove a `strict' semi-Fredholm estimate (i.e. with no error term) of the form 
    $$
    \| u \|_{H^{0, \mr}} \leq C \| Pu \|_{H^{0, \mr}}
    $$
    for all $u \in H^{0, \mr}$ such that $Pu \in H^{0, \mr}$. Notice that this condition implies that in fact $u \in H^{1, \mr}$ since $P$ is elliptic at frequency infinity. Hint: choose a positive function $a^2 \sim |x|^{2\epsilon}$ for $x \to -\infty$ and $a^2 \sim x^{-2\epsilon}$ for $x \to +\infty$, and compute $\ang{u, i[P, a^2]u}$. Notice that $a$ can (and should) be chosen to be \emph{globally} monotonic. This is a simple model for our Fredholm results for more interesting operators. 
    
    (b) The `dual' estimate, where we replace each Hilbert space by its dual and reverse the direction of $P^*$ compared to $P$,  is 
$$
\| u \|_{H^{0, -\mr}} \leq C \| P^*u \|_{H^{0, -\mr}}.
$$
Here $P^* = P$ due to the divergence-form structure of $P$, but we write $P^*$ to make the duality more evident. Show that this estimate can be proved in exactly the same way; we just need to choose $a$ so that $a^2 \sim |x|^{2\epsilon}$ for $x \to +\infty$ and $a^2 \sim x^{-2\epsilon}$ for $x \to -\infty$. 

(c) Observe that $P$ has a one-dimensional null space (viewed as acting on tempered distributions, say) but the choice of variable order $\pm \mr$ eliminates this null space in  $H^{0, \pm \mr}$. 

(d)  How should the orders, and the commutant $a$, in the previous problem be adjusted for the operator $g(x) D_x$ instead of $P$? How about for the operator $D_x$? 
    \end{problem}

\begin{problem}
    Suppose that two variable orders are equal at the spatial boundary. Show that the variable order Sobolev spaces they define are equal, with equivalent norms. The conclusion is that the only `important' property of a variable order is its restriction to the boundary of phase space. It nevertheless needs to be defined throughout all of phase space, so that quantities like $\ang{x}^{\ml(x, \xi)}$ make sense. 
    \end{problem}
    \begin{problem}\label{prob:var invertible}
     Prove that there exists an invertible operator $\Lambda$  in any variable order space $\Psisc^{m, \ml}$, with inverse in $\Psisc^{-m, -\ml}$. Hint: initially choose $\Lambda$ to be elliptic and formally self-adjoint. Then, show it has at worst finite-dimensional null space contained in $\Schw(\R^n)$; modify it by adding a judiciously chosen finite-rank operator to make it invertible.
     \end{problem}
    \begin{problem}
    Prove boundedness of variable order operators between (appropriate) variable order spaces. Hint: make use of the result of the previous exercise.
    \end{problem}
    \begin{problem}\label{prob:var compact embedding}
    Suppose that $\mr < \mr'$, and $s < s'$. Show that $H^{s', \mr'}$ is compactly embedded in $H^{s, \mr}$. Note the subtlety: we only require that $\mr < \mr'$ pointwise, that is, $\mr(q) < \mr'(q)$ for all $q \in \comphase$ (and it would be enough to ask for this inequality just at spatial infinity); we do not need $\max \mr < \min \mr'$. Hint: make use of Problem~\ref{prob:var invertible} by writing $\Id = \Lambda^{-1} \circ \Id \circ \Lambda$ for an invertible $\Lambda \in \Psisc^{s', \mr'}$. 
    \end{problem}

    \begin{problem} Suppose that $\mr < \mr'$, and $s < s'$, and let $N$ be a (large) positive number. Show that for all $\epsilon > 0$ there is a constant $C(\epsilon)$ such that 
    \begin{equation}\label{eq:var interp}
\| u \|_{H^{s,r}} \leq \epsilon \| u \|_{H^{s',r'}} + C(\epsilon) \| u \|_{H^{-N, -N}}.
    \end{equation}
Hint: first show the inequality for $\delta, \eta > 0$
$$
\| u \|_{H^{s,r}}^2 \leq C \| u \|_{H^{s+\delta, r + \eta}} \| u \|_{H^{s-\delta, r - \eta}}
$$
by expressing the LHS as an inner product (up to a constant - use Problem~\ref{prob:var invertible} again). Then iterate this to get an inequality of the form 
$$
\| u \|_{H^{s,r}} \leq C \| u \|_{H^{s+\delta, r + \eta}}^{m/(m+1)} \| u \|_{H^{s-m\delta, r - m\eta}}^{1/(m+1)}.
$$
Finally apply Young's inequality to obtain \eqref{eq:var interp}.
\end{problem}

\section{Lecture 11: Radial point estimates}

Previously, we proved the microlocal version of energy estimates, known as propagation estimates. When studying a PDE $Pu=f$ of real principal type, if we have microlocal control of $u$ somewhere, then we can propagate that control along  integral curves of the associated Hamiltonian vector field $\mathsf{H}_p$, as long as we stay within the region of phase space where we have control on the forcing $f$. This is an important result, but two defects limit its usefulness: 
\begin{itemize}
    \item In order to apply a propagation estimate, we need to assume control of $u$ somewhere. But this is usually something which we want to \emph{prove}, not assume. 
    \item The propagation estimate fails at vanishing points of $\mathsf{H}_p$. We cannot propagate control into such a point, even if we have control everywhere else, and even if $f$ is Schwartz.
\end{itemize}
At first glance, the latter point might seem to be a mere technical issue --- in realistic situations, $\mathsf{H}_p$ will only vanish at a set of points with positive codimension, which means propagation works almost everywhere.

Unfortunately, it is often at such points that the sc-wavefront sets of important solutions of the PDE lie.

Recall from Lecture 9 the discussion of solutions of $(\triangle - \lambda^2) u = 0$. We saw that a family of solutions is given by the integral in \eqref{eq:smoothefns}, which can be viewed as the Fourier transform of the delta function on the sphere of radius $\lambda$ in frequency space, multiplied by a smooth factor $f(\lambda \theta)$. We saw that the scattering wavefront set of such solutions is contained in the union of two `radial sets' $\Rin \cup \Rout$, where the rescaled Hamilton vector field --- a smooth vector field on $\comphase$ tangent to the boundary --- vanishes. It is not hard to check that if $f(\lambda \theta)$ has full support on the sphere, then the scattering wavefront set is precisely equal to $\Rin \cup \Rout$; this follows from the asymptotic expansion in \eqref{eq:smoothefns}. In particular, the scattering wavefront set of such solutions is generically nonempty.

Evidently it is important to understand what is going on at $\Rin \cup \Rout$. This is the subject of this lecture.

\begin{remark}[Terminology]
    The vanishing sets of $\mathsf{H}_p$ are called \emph{radial sets}, hence the notation $\Rin, \Rout$. The estimates in this section are correspondingly called ``radial point estimates.''
    This terminology is a vestigial remnant of the ``conic perspective'' in ordinary microlocal analysis, in which one works directly with the Hamiltonian vector field $H_p$ on the cotangent bundle, rather than $\mathsf{H}_p$ on the cosphere bundle. Vanishing points of $\mathsf{H}_p$ correspond to lines in the cotangent bundle where $H_p$ is pointed radially.
\end{remark} 

The radial point estimates, in the form presented here, were introduced by Melrose \cite{Melrose1994}. Melrose viewed the estimates as  microlocalized versions of a global positive commutator estimate of Mourre \cite{Mourre}, who used the estimate to show, among other things, absolute continuity of the spectrum of Laplace-type operators. The terminology ``Microlocal Mourre estimate'' was used in \cite{melrose1994spectral}, but radial point estimate has become more standard.

In the rest of this lecture, we consider the Helmholtz operator $P=\triangle-1$. Actually, we could let $P$ be a short-range perturbation of $\triangle-1$, but this would not affect the arguments below, so we ignore this possibility and just take 
\[
    P=\triangle-1.
\]
exactly. Note that radial point estimates can be proved in a more general framework; see for example \cite[Section 5.4.7]{vasy-monicourse}. 

\subsection{Propagation of control into the radial set}
In the theorems below, $\mathcal{R}$ denotes one of $\Rin$ or $\Rout$, and $\mathcal{R}^*$ denotes the other radial set. Our first theorem concerns below-threshold values of the spatial order $r$, that is, $r < -1/2$. 

\begin{theorem}[Below threshold radial point estimate]\label{thm:below}
    Let $B,G,E \in \Psi_{\mathrm{sc}}^{0,0}$ denote microlocalizers such that
    \begin{itemize}
        \item $\operatorname{WF}'_{\mathrm{sc}}(B)\cap \mathcal{R}^* = \varnothing$, 
        \item $\operatorname{WF}'_{\mathrm{sc}}(B)\subseteq \operatorname{Ell}_{\mathrm{sc}}^{0,0}(G)$, 
        \item for every point $q$ in $(\operatorname{WF}'_{\mathrm{sc}}(B) \cap \Char(P))\setminus \mathcal{R}$, there is a bicharacteristic segment $\gamma : [0, s_0] \to \Char(P)$ such that $\gamma(0) = q$, $\gamma(s_0) \in \operatorname{Ell}_{\mathrm{sc}}^{0,0}(E)$, and $\gamma([0, s_0])$ lies entirely within $\operatorname{Ell}_{\mathrm{sc}}^{0,0}(G)$. In other words, one can travel along a bicharacteristic from any characteristic point of $\operatorname{WF}'_{\mathrm{sc}}(B)$ (other than points in $\mathcal{R}$, which are of course fixed under bicharacteristic flow) to the elliptic set of $E$, while remaining within the elliptic set of $G$. 
    \end{itemize}
    Then, for each $s,r\in \mathbb{R}$ with $r<-1/2$, and for each $N\in \mathbb{R}$, there exists a constant $C>0$ such that 
    \begin{equation}\label{eq:below mpe}
        \lVert Bu \rVert_{H_{\mathrm{sc}}^{s,r} } \leq C(\lVert GPu \rVert_{H_{\mathrm{sc}}^{s-2,r+1} } + \lVert Eu \rVert_{H_{\mathrm{sc}}^{s,r} } + \lVert u \rVert_{H_{\mathrm{sc}}^{-N,-N} })
    \end{equation}
    holds for all $u\in \mathcal{S}'$. 
\end{theorem}  

If $\WF'_{\mathrm{sc}}(B)$ is disjoint from $\mathcal{R}$ then this result follows from Theorem~\ref{thm:mpe}, since then the Hamilton vector field is nonvanishing along $\gamma$. 
 The interesting case of Theorem~\ref{thm:below} is that $B$ is elliptic on $\mathcal{R}$ and $\WF'(E)$ is disjoint from $\mathcal{R}$. In this case we can interpret the theorem as saying that microlocal regularity flows into the radial set from a punctured neighbourhood of it (despite the fact that the bicharacteristic flow can never reach $\mathcal{R}$ from the outside in finite parameter time):  

\begin{theorem}
For any $m\in \mathbb{R}$ and $r<-1/2$, if $\operatorname{WF}_{\mathrm{sc}}^{s,r}(u)$ is disjoint from a punctured neighborhood of $\mathcal{R}$ (i.e.\ around $\mathcal{R}$, but without stipulating anything at $\mathcal{R}$ itself), and if $\operatorname{WF}_{\mathrm{sc}}^{s-2,r+1}(Pu)$ is disjoint from a whole neighborhood of $\mathcal{R}$, then $\operatorname{WF}_{\mathrm{sc}}^{s,r}(u)$ is disjoint from $\mathcal{R}$.
\label{thm:radial_point_out}
\end{theorem}

So, we can propagate below-threshold (meaning $H^{s,r}$ for $r<-1/2$) control from an annular region around the radial set into the radial set, assuming that $Pu$ is sufficiently nice. Equivalently, very bad wavefront set (meaning $\WF^{s,r}(u)$ with $r < -1/2$) has to escape the radial set.

\begin{remark}
    Note the orders $(s-2,r+1)$ on the $GPu$ term. 
    If this were an elliptic estimate, then, since the Helmholtz operator lies in $\Psi_{\mathrm{sc}}^{2,0}$, we would have $\lVert GP u \rVert_{H_{\mathrm{sc}}^{s-2,r} }$ on the right-hand side. Having the stronger norm 
    \[ 
        \lVert GP u \rVert_{H_{\mathrm{sc}}^{s-2,r+1} }
    \] 
    on the right-hand side means our estimate requires stronger hypotheses  --- just like with propagation estimates, radial-point estimates require one more order of control on the forcing. 
    We also remark that if we were treating the Klein-Gordon operator, which is not elliptic at fibre-infinity as well as at spatial infinity, then the orders on the $GPu$ term (with $P$ now representing the Klein-Gordon operator) would be $(s-1, r+1)$ --- i.e. we would need one order of extra control, relative to an elliptic estimate, at both fibre-infinity and spatial infinity. 
\end{remark}

\begin{example}
    Consider the solution \eqref{eq:smoothefns} of the free Helmholtz equation $Pu=0$. We saw that its scattering wavefront set is contained in $\Rin \cup \Rout$, so it follows immediately from the theorem above that $\operatorname{WF}_{\mathrm{sc}}^{s,r}(u)$ is empty for all $r<-1/2$, which we already knew by direct means (see Problem~\ref{prob:smoothefns}). 
\end{example}

\begin{example}
    Consider the plane wave $u=e^{ix_1}$. This lies only in $H^{\infty,-d/2-\varepsilon}_{\mathrm{sc}}$ for $\varepsilon>0$. So, assuming that $d\geq 2$, 
    \[\operatorname{WF}_{\mathrm{sc}}^{m,-1/2-\varepsilon}(u)\] 
    is nonempty for $\varepsilon$ sufficiently small. Theorem~\ref{thm:below} tells us that this wavefront set, if intersecting $\mathcal{R}$, must escape the radial set and propagate along the integral curves of the flow. Indeed, we have already seen that the sc-wavefront set of a plane wave is not confined to the radial sets.
\end{example}

We only prove \Cref{thm:radial_point_out} for $u\in \mathcal{S}$. This will suffice for our intended application in Lecture 12: proving a global Fredholm estimate.  The proof below loosely follows \cite{Hintz-book}. The estimate is valid in the `strong sense' that it holds for all $u\in \mathcal{S}'$. To show this requires a regularization argument of the sort discussed in Lecture 6. 

\subsection{Propagation of control out of the radial set}
In the next theorem, we assume that the spatial order $r$ is \emph{above} threshold, that is, $r > -1/2$. Here are two versions of the above-threshold radial point estimate. 

\begin{theorem}[Above threshold radial point estimate]\label{thm:above Schw}
    Let $B,G \in \Psi_{\mathrm{sc}}^{0,0}$ denote microlocalizers such that
    \begin{itemize}
        \item 
        $\operatorname{WF}'_{\mathrm{sc}}(B)\cap \mathcal{R}^* = \varnothing$, 
        \item $\operatorname{WF}'_{\mathrm{sc}}(B)\subseteq \operatorname{Ell}_{\mathrm{sc}}^{0,0}(G)$, 
        \item for every $q \in \operatorname{WF}'_{\mathrm{sc}}(B) \cap \Char(P)$, there is a bicharacteristic segment $\gamma([0, \infty))$ such that $\gamma(0) = q$ and $\gamma(s) \to \mathcal{R}$ as $s \to \infty$, while remaining  within  $\operatorname{Ell}_{\mathrm{sc}}^{0,0}(G)$. 
    \end{itemize}
    Then, for each $s,r\in \mathbb{R}$ with $r>-1/2$, and for each $N\in \mathbb{R}$, there exists a constant $C>0$ such that for all \textbf{Schwartz} $u$ we have 
    \begin{equation}\label{eq:mpe radial above}
        \lVert Bu \rVert_{H_{\mathrm{sc}}^{s,r} } \leq C \Big( \lVert GPu \rVert_{H_{\mathrm{sc}}^{s-2,r+1} } + \lVert u \rVert_{H_{\mathrm{sc}}^{s-1/2,r - 1/2} } \Big).
    \end{equation}
    
\end{theorem}

\begin{remark} There are a number of variants of this estimate. For example, 
the final term on the RHS may be microlocalized to $\lVert Gu \rVert_{H_{\mathrm{sc}}^{s-1/2,r - 1/2} }$, at the cost of another (global) term $\lVert u \rVert_{H_{\mathrm{sc}}^{-N, -N} }$ measured in an arbitrarily weak Sobolev norm. Moreover, if $r-1/2$ is still above threshold, then one could iterate the estimate to reduce the orders of the $Gu$ term until the spatial order is below threshold. 
\end{remark}

Unlike \eqref{eq:below mpe}, the estimate \eqref{eq:mpe radial above} does \emph{not} hold in the strong sense. In fact, the solutions \eqref{eq:smoothefns} to $Pu = 0$ are such that, for $-1/2 < r < 0$, the RHS is finite, and the LHS is infinite! (See Problem~\ref{prob:strong}.)
We next give a strong version of this radial point estimate, valid for all $u \in \Schw '$. In the strong form of the estimate, we need an extra term on the RHS that imposes above-threshold behaviour of $u$ at the radial point, at some arbitrary $r_0 > -1/2$ unrelated to $r$: 
\begin{theorem}\label{thm:above dist}
 In addition to the microlocalizers $B, G$, let $\tilde E \in \Psi_{\mathrm{sc}}^{0,0}$ be a microlocalizer that is elliptic at $\mathcal{R}$. 
 Then, for each $s,r\in \mathbb{R}$ with $r>-1/2$, for each $s_0 \in \R$ and $r_0 > -1/2$, and for each $N\in \mathbb{R}$, there exists a constant $C>0$ such that for all tempered distributions $u$ we have 
    \begin{equation}\label{eq:mpe radial above 2}
        \lVert Bu \rVert_{H_{\mathrm{sc}}^{s,r} } \leq C \Big(\lVert GPu \rVert_{H_{\mathrm{sc}}^{s-2,r+1} } + \lVert \tilde E u \rVert_{H_{\mathrm{sc}}^{s_0,r_0} } + \lVert u \rVert_{H_{\mathrm{sc}}^{-N, -N} } \Big) .
    \end{equation}
\end{theorem}

So, if $Pu$ is nice, microlocally near $\mathcal{R}$, and if $u$ is known \emph{a priori} to have sufficient -- that is, above threshold -- decay on $\mathcal{R}$ (as provided by the $\tilde Eu$ term, microlocally in $H_{\sca}^{s_0, r_0}$ near $\mathcal{R}$ with $r_0$ above threshold if the RHS of \eqref{eq:mpe radial above} is finite), then $u$ is \emph{automatically} nice on $\mathcal{R}$ as well. 
Notice that, unlike the previous `propagation of regularity' theorems we have seen, we don't require regularity of the same order elsewhere to obtain above-threshold regularity at $\mathcal{R}$. In that sense, we can interpret Theorems~\ref{thm:above Schw} and \ref{thm:above dist} as saying that regularity `flows out of the radial set', as opposed to Theorem~\ref{thm:below}, in which regularity `flows into the radial set' from outside. 


\begin{example}
    Suppose that $(\triangle -1) u = f $ for $f\in \mathcal{S}(\mathbb{R}^n)$. Using the Fourier transform, we can construct many solutions of the form
    \begin{equation}
        u \in \langle r \rangle^{-(d-1)/2} e^{i \langle r \rangle} A + \langle r \rangle^{-(d-1)/2} e^{-i \langle r \rangle} B, \;  r = |x|
    \end{equation}
    for $A,B\in C^\infty(\overline{\mathbb{R}^n})$. What the theorem above says in this case is that if $A|_{\infty \mathbb{S}^{n-1}} = 0$, then in fact $A\in \mathcal{S}(\mathbb{R}^n)$, in which case we can choose $A=0$. 
\end{example}

\subsection{\texorpdfstring{Proof of the radial point estimates (for Schwartz $u$)}{Proof of the radial point estimates (for Schwartz u)}} 
\label{subsec:proof-radial-est}

We will now prove the radial point estimates assuming that $u$ is Schwartz. We will see in Lecture 12 that this is sufficient for a global Fredholm estimate, due to the density of Schwartz functions in the appropriate function space --- see Lemma~\ref{prop:Schwdense}.

As a reminder, we are only discussing the free Helmholtz operator $P= \triangle+1$, but the discussion below works essentially verbatim for all $P$ such that 
\[
    P-(\triangle+1)\in \operatorname{Diff}_{\mathrm{sc}}^{2,-2}.
\]
Small modifications allow one to handle elliptic $P$ that differ from $\triangle+1$ by an element of $\Psi_{\mathrm{sc}}^{2,-1}$. See \cite{Hintz-book} or \cite{vasy-monicourse} for details on that. An anti-symmetric subprincipal part shifts the thresholds in the radial point estimate.

\begin{remark}\label{rem:warning}
    Warning: proofs of radial point estimates always involve an excessive amount of notation. To avoid some of this, we will define some operators $B,E,\dots$ below which have a different meaning than in the theorem statements.
\end{remark}

    To simplify the notation a bit, we will only specify the form of symbols near the characteristic set, thus ignoring fiber infinity (which lies in the elliptic region) and the interior of $\partial(\overline{\mathbb{R}^n}\times \overline{\mathbb{R}^n} )$. Consequently, we will write a `$*$' where fiber infinity orders should be when it does not matter. 

    We use the notation: 
    \begin{itemize}
        \item $\varrho$ is a quadratic-defining-function of $\mathcal{R}$ in the characteristic set. 
        
        \item $\mathsf{p} = (\xi^2-1)/(\xi^2+1)$ is the principal symbol of $(\triangle -1)/(\triangle+1)$.
        \item $\beta_0 \in C^\infty(\overline{\mathbb{R}^n}\times \overline{\mathbb{R}^n})$ is defined such that $\mathsf{H}_p \rho = \beta_0 \rho$ near $\mathcal{R}$, where $\rho$ is a boundary-defining-function for base infinity (e.g. $\rho=\langle r \rangle^{-1}$).
        \item  $\beta_1 \in C^\infty(\overline{\mathbb{R}^n}\times \overline{\mathbb{R}^n})$ is defined such that $\mathsf{H}_p \varrho = \beta_1 \varrho$ near $\mathcal{R}$.
    \end{itemize}
    The exact formula for $\beta_0,\beta_1$ depends on the choice of boundary-defining-functions used in weighting $H_p$ to define $\mathsf{H}_p$. 
    Also the sign depends on the sign convention used in defining the Hamiltonian flow, which has no intrinsic significance. What matters is that each of $\beta_0,\beta_1$ has a definite sign on $\mathcal{R}$, and also what their relative signs are.  Moreover, these definite signs do not depend on any choices made. So, it suffices to work them out in local coordinates, using whatever locally valid coordinates, boundary-defining-functions, etc.\ we like. 
    Also, using the symmetry of the problem, it suffices to work in a region where $x_1$ is a dominant variable, i.e. $x_1 \geq c |x|$, with $c > 0$, so we can use $\rho=1/x_1$ as a valid boundary-defining-function, rather than $1/r$. Working with either would result in the same signs --- the computation is just easier this way.

    Then \cref{eq:misc_a} reads $\mathsf{H}_p = \mp \xi_1 (\rho \partial_\rho + \sum_j v_j \partial_{v_j})$ for $v_j = \omega_j - y_j$ (where the top sign refers to $\Rout$ and the bottom sign to $\Rin$). Then, if $\varrho = \sum_j v_j^2$,  
    \begin{align}
        \beta_0 &= \mp \xi_1, \\ 
        \beta_1 &= \mp 2 \xi_1 . 
    \end{align}
    In the region under consideration, $\xi_1$ is a dominant variable near the radial set when $x_1$ is a dominant variable, since $x$ and $\xi$ are parallel on the radial set. So the radial set lies away from $\{\xi_1=0\}$, and this means $\beta_1,\beta_0$ have a definite sign, indeed the same definite sign. For definiteness, we will discuss the case of $\Rout$, so that $\xi_1>0$ on it, and $\beta_0, \beta_1$ are both strictly negative.

    The proof is via a positive commutator estimate of the sort used in previous lectures to prove a propagation estimates. 
    Our goal is to cook up a pseudodifferential operator $A$, the ``commutant,'' such that $i[P,A]$ has a definite sign (modulo lower order terms) modulo some term which we assume to be under control. This means constructing a symbol $a$ such that the derivative $H_p a$ of $a$ along the Hamiltonian flow has a definite sign. But, because $\mathsf{H}_p$ vanishes at $\mathcal{R}$, this is easier said than done.  The symbol cannot be smooth at $\mathcal{R}$ --- or else we have to weaken what we mean when we say that $H_p a$ has a definite sign. The simplest sort of non-smooth symbol is one with some factors of the boundary-defining-function $\rho$ thrown in.
    
    Indeed, suppose $a=\rho^{-r'}$ near the radial set $\mathcal{R}$. Then 
    \[
        \mathsf{H}_p a \sim -r' \beta_0  a.
    \]
    So, as long as $r'\neq 0$ (which would make $a$ smooth at $\mathcal{R}$), then $a^{-1}\mathsf{H}_p$ has a definite sign at $\mathcal{R}$. The proof of radial point estimates involves cutting off $a$ away from $\mathcal{R}$ in an amenable way, and then running through the positive-commutator argument to see what pops out.

    Here are the details: let\footnote{When doing this carefully, there should be cutoffs localizing near $\rho=0$ and away from fiber infinity, but as we already mentioned, these are  of tertiary importance.}
    \begin{equation}
        a = \rho^{-(2r+1)}  \phi( \mathsf{p})^2 \psi(\varrho)^2,
    \end{equation}
    where $\phi,\psi \in C_{\mathrm{c}}^\infty(\mathbb{R};[0,1])$
    are both identically $=1$ near $0$ and $\psi$ satisfies 
    \begin{equation}\label{eq:psi smoothness}
        \sqrt{-\psi' \psi} \in C^\infty( [0,\infty) ). 
    \end{equation}
    (We are taking $r'=2r+1$ for later convenience.)
    See Lecture 6 for why such a $\psi$ exists (it is easy to construct using the function $\chi_0$ from \eqref{eq:chi_0} and \eqref{eq:chi_0 2}). The purpose of $\phi$ is just to localize near the characteristic set --- its role in the argument is of secondary importance. The purpose of the $\psi(\varrho)$ term is to further localize near the radial set. This term is of primary importance.

    We can now compute 
    \begin{multline}
        H_p a = \rho \mathsf{H}_p a = \rho^{-2r} ( \beta_0 (-2r-1) \phi(\mathsf{p})^2 \psi(\varrho)^2 + \\
        + 2 (\mathsf{H}_p \varrho) \phi(\mathsf{p})^2 \psi'(\varrho) \psi(\varrho) +2 \mathsf{q} \mathsf{p} \phi'(\mathsf{p}) \phi(\mathsf{p}) \psi(\varrho)^2   ) ,
    \end{multline}
    where $\mathsf{q} \in C^\infty$ is defined by $\mathsf{H}_p \mathsf{p} = \mathsf{q}\mathsf{p}$. 
    On the right-hand side:
    \begin{enumerate}
        \item The first term has a definite sign on the radial set (unless $r=-1/2$, exactly, which is the reason this value of $r$ is explicitly excluded in both Theorems~\ref{thm:below} and \ref{thm:above Schw}). The sign is determined by whether $r<-1/2$ or $r>-1/2$. 
        \item The second term also has a semi-definite sign on the characteristic set (we do not have control over that sign), and it is essentially supported in a small annular region encircling $\mathcal{R}$. 
        \item The final term, since it involves $\phi'(\mathsf{p})$, is supported in the elliptic set of $P$ (which we understand completely).
    \end{enumerate}
    If $r>-1/2$, then the first two terms on the right-hand side have the \emph{same} sign, while if $r<-1/2$ they have the opposite sign. (Recall that $\psi'<0$.)

    Let us discuss the $r<-1/2$ case first. We will then mention the differences in the $r>-1/2$ case, which amount to some consequential sign switches.

    In the $r < -1/2$, i.e. below-threshold, case, we make the following observation: it is enough to prove the estimate for a $B$ that is elliptic at $\mathcal{R}$ and microsupported in an arbitrarily small neighbourhood of $\mathcal{R}$. Indeed, if we prove the estimate in this case, then it also follows for any other $B$ operator (as in the Theorem statement) using microlocal propagation estimates, Theorem~\ref{thm:mpe}. This is where the first and third hypotheses of the theorem are crucial. This observation is useful as it means that we need only prove the estimate for a specific $B$ of our choice. 
    
    Let us give a name to the various terms appearing above:
    \begin{align}
            b&= \rho^{-r} \phi(\mathsf{p}) \psi(\varrho) \sqrt{ \beta_0 \big(2r+1 -\underbrace{2\delta \beta_0^{-1} \phi(\mathsf{p})^2 \psi(\varrho)^2}_{\text{For later use}} \big)} \label{eq:b defn} \\
            e&= \rho^{-r}  \phi(\mathsf{p}) \sqrt{2 (\mathsf{H}_p \varrho) \psi'(\varrho)\psi(\varrho) } \\ 
            h &= 2 \rho^{-2r} \mathsf{q} \phi'(\mathsf{p}) \phi(\mathsf{p}) \psi(\varrho)^2.
    \end{align}
    Here $\delta > 0$ is a small parameter to be chosen later, and the indicated term in \eqref{eq:b defn} is added in for later advantage. Notice that it is precisely the strict positivity of $-\beta_0(2r+1)$ that allows us to `squeeze in' this additional term. (More abstractly, it is crucial that $\varrho^{-1}H_p \varrho$ has a definite sign near the radial set.)
    As long as the supports of $\phi,\psi$ and $\delta$ are sufficiently small, then the definitions of these symbols all make sense in the sense that the square roots are smooth, due either to strict positivity in \eqref{eq:b defn} or due to \eqref{eq:psi smoothness}.  (First take the supports of $\phi,\psi$ sufficiently small, then $\delta$ relative to those.)
    
    So, 
    \begin{equation}
        H_p a = -2 \delta \rho^{2r+2} a^2-b^2+e^2+hp.
    \end{equation}
    Now we quantize: let $A= \operatorname{Op}(a)$, $B=\operatorname{Op}(b)$, etc. Thus, $A \in \Psisc^{*, 2r+1}, B \in \Psisc^{*, r}$, $E \in \Psisc^{*,r}$ and $H \in \Psisc^{*, 2r}$. (Recall that the differential orders here are not relevant; all these operators are in fact of order $-\infty$ in the differential sense, since we are working near the radial set which is disjoint from fibre-infinity). Notice that $B$ is not quite the same operator that appears in the Theorem statement, since there $B$ had order $(0,0)$; however, this is an inessential point, as we can adjust for this by adjusting the norm in which we measure the functions. Here we want to use the $L^2$ inner product, so it is convenient to put the weight onto the operator rather than the norm.  We also  let $\Lambda=\operatorname{Op}(\rho^{r+1}) \in \Psisc^{0, r+1}$. Then, 
    \begin{equation}
        i[P,A] =  -2 \delta (\Lambda A)^* (\Lambda A) - B^*B + E^* E+ HP + R 
        \label{eq:misc_11}
    \end{equation}
for some $R\in \Psi_{\mathrm{sc}}^{*,2r-1}$. Note that $B$ is elliptic at the radial set. We also observe that $B$ and $E$ satisfy the third condition in Theorem~\ref{thm:below} since $E$ is elliptic on a small annular region around $\mathcal{R}$, and in particular, the elliptic set of $E$ intercepts every bicharacteristic of $P$ other than the stationary bicharacteristics sitting at $\mathcal{R}$ and $\mathcal{R}^*$.  For convenience, we can arrange that $A$ is self-adjoint, by replacing it with $(A+A^*)/2$.

We now proceed as in the propagation estimates, namely we apply the operators in \eqref{eq:misc_11} to $u$ and then take the inner product with $u$, an operation that is certainly well-defined for $u \in \Schw$:
\begin{equation}\label{eq:misc_12}
    \langle u, i[P,A] u\rangle = -2\delta \lVert  \Lambda A u \rVert^2 - \underbrace{\lVert Bu \rVert^2}_{\text{Thing we want to control}} + \lVert Eu \rVert^2 + \langle u,HP u \rangle + \langle u,R u \rangle.  
\end{equation}
The thing we want to control is $\lVert Bu \rVert$. The term $-2\delta \lVert  \Lambda A u \rVert^2$ on the right-hand side has the same sign, so is \emph{helping us}. The $Eu$ term, however, has the `wrong' sign, so it has to be retained on the RHS of the estimate.  The term $\langle u,HP u \rangle$ can be understood via elliptic regularity (since $h$ was essentially supported in the elliptic set), and $\langle u,R u \rangle$ should be ``lower order,'' so unimportant. 

The left-hand side $\langle u, i[P,A] u\rangle$ will be controlled in terms of $Pu$. This last step is the same as in the proof of Theorem~\ref{thm:mpe}:\footnote{As a reminder, we are working with $P$ such that $P=P^*$. If $P\neq P^*$, then $P-P^*$, which is lower order, will enter. }
\begin{equation}\label{eq:misc_13}
    |\langle u, i[P,A] u\rangle| = |\langle u,  PA u\rangle+\langle u,  AP u\rangle| = |\langle Pu,  A u\rangle+\langle Au,  P u\rangle| = 2 |\Im \langle Au,Pu \rangle|. 
\end{equation}
So, rearranging \eqref{eq:misc_12} and using \eqref{eq:misc_13}, we have 
\begin{equation}
    2\delta \lVert  \Lambda A u \rVert^2 + \lVert Bu \rVert^2  \leq 2 |\Im \langle Au,Pu \rangle|+ | \langle u,HP u \rangle | +  \lVert Eu \rVert^2 + |\langle u,R u \rangle|. 
\end{equation}
Using the Peter--Paul inequality\footnote{The Peter-Paul inequality is a frivolous name for the inequality $2|xy| \leq \epsilon x^2 + y^2/\epsilon$, for any real $x, y$ and $\epsilon > 0$. The choice of $\epsilon$ allows us to have a small constant in front of the $x^2$ term at the cost of a large constant in front of $y^2$. We rob Peter to pay Paul.} to estimate $\langle Au,Pu \rangle = \langle \Lambda Au, \Lambda^{-1} Pu \rangle$, we get a term $\varepsilon \lVert \Lambda A u \rVert^2$, where $\varepsilon$ can be taken arbitrarily small, and one $\varepsilon^{-1} \lVert \Lambda^{-1} Pu \rVert^2$. That $\varepsilon \lVert A u \rVert^2$ term can be absorbed into the left-hand side if $\varepsilon <\delta$, and that is the \emph{raison d'\^{e}tre} of the $2\delta \lVert  \Lambda A u \rVert^2$ term. So: 
\begin{equation}
    \lVert Bu \rVert^2 \lesssim \lVert \Lambda^{-1} Pu \rVert^2+ | \langle u,HP u \rangle | + \lVert Eu \rVert^2  + |\langle u,R u \rangle|. 
\end{equation}

It is now ``straightforward'' to go from this estimate to the final result, since we have a bound of $Bu$ (the thing we want to control) in terms of $Pu$, some term controlled via an elliptic estimate, and some ``lower-order'' term. Since the details are not enlightening (and since they are covered in e.g.\cite{Hintz-book}), we summarize the key points:  $|\langle u,R u \rangle|$ is controlled by $u$ near $\mathcal{R}$ in 
\[ 
H_{\mathrm{sc}}^{*,r-1/2},
\]
and $\lVert \Lambda^{-1} Pu \rVert^2$ is controlled by $Pu$ in \[H^{*,r+1}_{\mathrm{sc}}.
\]
The $\langle u,HP u \rangle$ can be estimated by $\| u \|_{*, r-1} \| Pu \|_{*, r+1} \lesssim \| u \|_{*, r-1}^2 + \| Pu \|_{*, r+1}^2$, two terms that are already controlled. 
So, we can control $\lVert Bu \rVert_2^2$ if we have a local $H_{\mathrm{sc}}^{*,r-1/2}$ bound on $u$ and a global $H_{\mathrm{sc}}^{*,r+1}$ bound on $Pu$. Since $B\in \Psi_{\mathrm{sc}}^{*,r}$, this means that we can conclude that $u$ is in $H_{\mathrm{sc}}^{*,r}$ on the elliptic set of $B$. 
This is almost what we want, but sub-optimal in two ways:
\begin{itemize}
    \item We want to only make use of local control on $Pu$; but this is easy to arrange, because when we bounded $\langle Au,Pu \rangle$, because $A$ has operator wavefront set near $\mathcal{R}$, we really only needed to control $Pu$ nearby. Indeed, if we choose $G$ as in Theorem~\ref{thm:below}, in addition to being microlocally equal to the identity\footnote{This means that $G - \Id$ has operator wavefront set disjoint from $\WF'(A)$} on $\WF'(A)$, then we have $[A, P] = [A, GP] + [A, (\Id - G)P]$, but the factor $[A, (\Id - G) P]$ is residual if $G$ is microlocally equal to the identity on $\WF'(A)$. Thus, we can replace $P$ by $GP$ in the argument, at the cost of an error term 
    $\ang{u, [A, (\Id - G)P] u}$ that can be estimated by $\| u \|_{-N, -N}^2$, which is the final term on the RHS of \eqref{eq:below mpe}. 
    \item We don't want to assume control on $u$, even microlocally, in $H_{\mathrm{sc}}^{*,r-1/2}$. But we can iterate the argument (with $r$ replaced by $r-1/2$) to estimate this term by a (weaker) norm of $Pu$ and a still weaker norm of $u$ in $H_{\mathrm{sc}}^{*,r-1}$. This iteration can be continued until the norm of $u$ is reduced to order $(-N, -N)$, at which point this term can be absorbed in the final term on the RHS of \eqref{eq:below mpe}. 
\end{itemize}

That completes the discussion of the below threshold estimate $r < -1/2$. What about $r>-1/2$? Similarly to the below-threshold estimate, it is sufficient (using the first and third conditions in the Theorem statement and microlocal propagation, Theorem~\ref{thm:mpe}) to prove for a single choice of $B$ that is elliptic at $\mathcal{R}$ and microsupported in a small neighbourhood of $\mathcal{R}$. 
The estimate works in a similar fashion, except one has to change the definition of $b$ in \eqref{eq:b defn} so that the argument of the square-root is positive: we redefine 
$$
b= \rho^{-r} \phi(\mathsf{p}) \psi(\varrho) \sqrt{ -\beta_0 \big(2r+1 +2\delta \beta_0^{-1} \phi(\mathsf{p})^2 \psi(\varrho)^2 \big)} . 
$$
This changes the sign of the $B^*$ and $(\Lambda A)^* (\Lambda A)$ terms, so now instead of \eqref{eq:misc_11}  one gets
    \begin{equation}
        i[P,A] =  2 \delta (\Lambda A)^* (\Lambda A) + B^*B + E^* E+ HP + R 
    \end{equation}
 One then proceeds as before. The main difference is that now the $B^*B$ and $E^*E$ terms have the \emph{same sign}, so the $E^*E$ term can be discarded. This explains why there is no $E$ term in \eqref{eq:mpe radial above}. 

 \

We next make some brief remarks about the proof of Theorem~\ref{thm:above dist}. In this case, if $u$ is not in $H_{\sca}^{s_0, r_0}$ for some $s_0 \in \R$ and some $r_0 > -1/2$, then the result is vacuously true as the RHS is $+\infty$. So we may assume that $u \in H_{\sca}^{s_0, r_0}$. To prove this result, we regularize, similar to the procedure in Lecture 6. Again we take $\tilde r$ as a regularizing parameter, and define, similar to \eqref{eq:reg x}, 
\begin{equation}\label{eq:reg above}
\tilde a = a (1 + \tilde r \rho^{-2})^{-M}. 
\end{equation}
Then we find that 
$$
H_p \tilde a = (1 + \tilde r \rho^{-2})^{-M} \Big( H_p a +  \frac{2\tilde r M \beta_0 \rho^{-2}}{1 + \tilde r \rho^{-2}} a \Big). 
$$
This new term is the `wrong' sign, that is, opposite in sign to the good positive term $-\beta_0 (2r+1) a$ that allows us to define $b$ by taking a square-root. We can see that $2M$ must be taken less than $2r+1$ in order for this needed positivity not to be wiped out completely. A moment's thought shows that this is to be expected, as the regularizing factor has the effect of reducing the growth rate of the symbol of $a$ from $\rho^{-2r-1}$ to $\rho^{-2r-1+2M}$. However, we also know that the positive term on which the entire estimate relies requires that this power is negative, otherwise the sign of this term changes sign (as in the below-threshold estimate). So, we see intuitively that the condition $2M < 2r+1$ is to be expected. Then, since the overall power of $\rho$ can be taken arbitrarily close to zero (while remaining negative), and taking into account the gain of one in the order of the commutator $i[P, A]$ relative to the order of $A$, we see that $u$ being in an above-threshold space microlocally, i.e. in $H_{\mathrm{sc}}^{*,r_0}$ for some $r_0 > -1/2$, is a sufficient condition for the calculation above to go through with the regularizing factor. The regularizer is then removed following the argument in Lecture 6.  (Actually, there is a further subtlety involved in calculating the commutator, which requires a second regularization step. This is explained in  \cite[discussion below (5.62)]{vasy-monicourse}; we do not pursue this point further.)

\subsection{Problems}

\begin{problem}
    Consider $P= x D_x\in \Psi_{\mathrm{sc}}^{1,1}(\mathbb{R})$. 
    \begin{enumerate}[label=(\alph*)]
        \item What is $\operatorname{Char}_{\mathrm{sc}}^{1,1}(P)$? 
        \item State the radial-point estimates for the radial sets of $P$ over spatial infinity. Hint: you may use without proof that the threshold decay rate $s$ is the one such that $H_{\mathrm{sc}}^{*,s-\varepsilon}$ contains $\ker P$ for $\varepsilon>0$ but not $\varepsilon=0$.    
        \item Deduce them from the radial-point estimates of $P$ over $0$, which you may cite from \cite{Hintz-book}.
        \item   Prove them directly. 
    \end{enumerate}
\end{problem}

\begin{problem}\label{prob:strong}
    Consider the estimate from Theorem~\ref{thm:above dist}:
    \begin{equation}\label{eq:above dist}
        \lVert Bu \rVert_{H_{\mathrm{sc}}^{s,r} } \leq C(\lVert GPu \rVert_{H_{\mathrm{sc}}^{s-2,r+1} } + \lVert \tilde Eu \rVert_{H_{\mathrm{sc}}^{s_0,r_0} } + \lVert u \rVert_{H_{\mathrm{sc}}^{-N,-N} })
    \end{equation}
    for all $u\in \mathcal{S}'$. Here, $-1/2<r_0<r$ and $s_0 < s$. 
    \begin{enumerate}[label=(\alph*)]
        \item Show that (for appropriate $\tilde E,B$ satisfying the hypotheses of that theorem),  for each $\varepsilon>0$ we can bound 
    \begin{equation}
    \lVert \tilde Eu \rVert_{H_{\mathrm{sc}}^{s_0,r_0} } \leq \varepsilon \lVert Bu \rVert_{H_{\mathrm{sc}}^{s,r} } + C \| Pu \|_{H_{\mathrm{sc}}^{s-2,r+1} }+ C(\varepsilon) \lVert u \rVert_{H_{\mathrm{sc}}^{-N,-N} } 
    \end{equation}
    for some $C(\varepsilon)>0$. 
    \item  Consequently,  if we ``absorb the  $\varepsilon \lVert Bu \rVert_{H_{\mathrm{sc}}^{s,r} }$ term into the left-hand side,'' 
    we get an estimate like 
    \begin{equation}\label{eq:above nonstrong}
        \lVert Bu \rVert_{H_{\mathrm{sc}}^{s,r} } \lesssim  \lVert GPu \rVert_{H_{\mathrm{sc}}^{s-2,r+1} }  + \lVert u \rVert_{H_{\mathrm{sc}}^{-N,-N} }.
    \end{equation}
    Explain what assumption was implicitly made in order to deduce \eqref{eq:above nonstrong}, and how we know that this estimate does \emph{not} hold for all $u\in \mathcal{S}'$. 
    \end{enumerate}
    
\end{problem}

\begin{problem}
    Fix $\mathcal{R} \in \{\Rin, \Rout \}$. 
    Prove directly, without using radial-point estimates, that if a solution 
    \begin{equation}
        u(x)  = \int_{\mathbb{R}^d} e^{ix\cdot \xi} \delta(\lVert \xi \rVert - 1) A(\xi ) \, d\xi,\quad A\in C^\infty(\mathbb{R}^d)
    \end{equation}
    to the free Helmholtz equation satisfies $\operatorname{WF}_{\mathrm{sc}}^{s,r}(u)\cap \mathcal{R}= \varnothing$ for some $m\in \mathbb{R}$ and $r>-1/2$, then $\operatorname{WF}_{\mathrm{sc}}(u)\cap \mathcal{R}= \varnothing$.
\end{problem}

\begin{problem} Explain clearly why it is that, in the regularization step  needed to prove Theorem~\ref{thm:above dist}, see \eqref{eq:reg above}, it is only possible to regularize by a fixed amount $2M \leq 2r+1$; while for the below-threshold case, one can regularize by an arbitrarily large amount (that is, $M$ may be taken arbitrarily large in that case). 
\end{problem}

\begin{problem}
    Consider a `formal series solution' to the free Helmholtz equation $(\Delta - \lambda^2) u = 0 $ of the form
    $$
    e^{i\lambda |x|} \sum_{j=0}^\infty |x|^{-\alpha - j} f_j(\hat x), \quad  \hat x = \frac{x}{|x|}. 
    $$
    Here we assume that $f_j$ are $C^\infty$ functions on the sphere. 
    
    Show that such a formal series exists if $\alpha = (n-1)/2$, but not for any other value of $\alpha$ (if $f_0$ is not identically zero). Relate this fact to the threshold value $-1/2$ for the spatial order for $\Delta - \lambda^2$. 
\end{problem}

\begin{problem}\label{prob:wave} Let $P$ be the free wave operator $D_t^2 - \Delta$ on $\R^{n+1}$, with symbol $\tau^2 - |\xi|^2$. Compute the characteristic variety and the radial set of this operator (as an operator in the scattering calculus). Show that $P$ cannot be treated using the analysis presented in this lecture, since the characteristic variety meets the zero section $\{ (\xi, \tau) = 0 \}$ over the boundary of the `light cone' at spacetime infinity, and $\beta_0$ (as defined below Remark~\ref{rem:warning}) is equal to zero at the zero section. This means that one cannot define the symbol $b$ as we have done above, as the argument of the square-root in \eqref{eq:b defn} will not be strictly positive. 

Remark: The solution to this issue is to `blow up at zero frequency' at the spacetime boundary, which leads to the b-calculus, as the front face created by this blowup coincides with the boundary of the b-cotangent bundle. The Fredholm method can be implemented in the b-calculus (and this is where it was first introduced by Vasy; see the Introduction for a brief discussion), but there is a catch: in the b-calculus, instead of the symbol extending to spacetime infinity, there is a normal operator, which is global on the boundary. In the b-calculus setting one needs invertibility of this normal operator, which is much more difficult to arrange than invertibility of the principal symbol (at the spacetime boundary) in the scattering calculus. The latter, which we carried out in this Lecture, simply requires checking that a certain function we construct (the function $b$ in our case, the principal symbol of $B$) is nonzero. For this reason, the Fredholm method in the scattering calculus is technically simpler than the analogous construction in the b-calculus. It is for this reason that the authors elected to use the scattering calculus in these Lecture notes. 
\end{problem}

\section{Lecture 12: Global (semi-)Fredholm estimate for Helmholtz operator}


In this lecture, we combine estimates obtained previously to establish our global (semi-)Fredholm estimates. We will call them Fredholm estimates below since the same estimate for their adjoints can be proved in the same manner and combining them indeed establishes the Fredholm property.

We will consider the Helmholtz operator $P=\Delta_g-\lambda^2$. 
Most results in this section holds on more general asymptotically conic manifolds, see \cite{Melrose1994}\cite{Helmholtz-1}. But the commuting property as in Proposition~\ref{prop:cusp-Laplacian-commutator} and the test module used will be more complicated.

\subsection{Global (semi-)Fredholm estimates for Helmholtz operators: variable order spaces}

We derive global Fredholm estimates for 
\begin{equation}
P = \Delta_g -\lambda^2
\end{equation}
in this part. The treatment here follows \cite[Section~3]{Helmholtz-1} and \cite[Section~5.2]{gell2022propagation}.

Here $g$ is a short-range perturbation of the Euclidean metric.
That is, denoting the Euclidean metric by $g_0$ and the boundary defining function of $\partial\overline{\R^n}$ by $\msf{x}$, then
\begin{equation} \label{eq:metric-assumption}
g - g_0 \in {\msf{x}}^2 C^\infty(\overline{\R^n}, S^2T^*\overline{\R^n}),
\end{equation}
where $S^2T^*\overline{\R^n}$ is the bundle of symmetric 2-tensors. Here we required smoothness on $\overline{\R^n}$ for simplicity, but readers that are familiar with conormality should be aware that almost all results below generalizes to the situation where one has conormal (or symbolic) metric and potential perturbations. 
The Fredholm estimates presented in this Lecture would hold equally well if we considered long-range (that is, $O({\msf{x}})$-level) perturbations, but we make this stronger decay condition on the perturbation because invertibility (as in Theorem~\ref{thm:Helmtholz-invertible}) would require additional assumptions. Also, we still keep the non-trapping assumption on the geodesic flow of $g$.

As analyzed in previous lectures, the Hamilton flow associated the the principal symbol of $P$ has a nice source-sink structure and we denote the source by $\Rin$ and the sink by $\Rout$.

\begin{definition} \label{definition:variable-order-admissible}
Let $\mathsf{r}_\pm \in C^\infty({}^{\sct}T^*\overline{\R^n})$ be variable spatial orders. We call $\mr_+$ Feynman admissible, if $\mathsf{r}_+$ is a constant larger than $-\frac{1}{2}$ near $\Rin$, a constant less than $-\frac{1}{2}$ near $\Rout$ and it is non-increasing along the Hamilton flow associated to $P$. Similarly we call $\mr_-$ anti-Feynmann admissible if it is a constant less than $-\frac{1}{2}$ near $\Rin$, a constant larger than $-\frac{1}{2}$ near $\Rout$ and it is non-decreasing along the Hamilton flow associated to $P$. 

Given such variable orders, we will denote by $\Uin, \Uout$ neighbourhoods of $\Rin, \Rout$ on which these variable orders are constant. 
\end{definition}

We first introduce the space on which we establish our Fredholm theory.
We define $\mathcal{X}^{s,\mathsf{r}_\pm}$ to be the spaces
\begin{equation}\label{eq:calX-defn}
\mathcal{X}^{s,\mathsf{r}_\pm} = \{ u \in H_{\sct}^{s,\mathsf{r}_\pm}: Pu \in  H_{\sct}^{s-2,\mathsf{r}_\pm+1} \},
\end{equation}
where $\mr_\pm$ are Feynman/anti-Feynman admissible,  equipped with the `graph norm' defined via
\begin{equation} \label{eq:cal-X-norm}
\| u \|_{\mathcal{X}^{s,\mathsf{r}_\pm} }^2 =
\| u \|_{s,\mathsf{r}_\pm}^2 + \| Pu \|_{s-2,\mathsf{r}_\pm+1}^2.
\end{equation}

\begin{remark}
Notice that the definition of this $\mathcal{X}^{s,\mathsf{r}_\pm}$ depends only on the principal symbol (in all senses, not only the differential sense!) of $P$. Since the requirement on $Pu$ does not change under lower order perturbations (due to the requirement on $u$ itself).
\end{remark}

\begin{remark}
In the case of a hyperbolic operator, such as the Klein-Gordon operator $P_{KG}$, we would define instead of \eqref{eq:calX-defn}
\begin{equation}\label{eq:calX-defn-hyp}
\mathcal{X}^{s,\mathsf{r}_\pm} = \{ u \in H_{\sct}^{s,\mathsf{r}_\pm}: P_{KG}u \in  H_{\sct}^{s-1,\mathsf{r}_\pm+1} \},
\end{equation}
where the differential order on the $P_{KG}u$ term is $s-1$, not $s-2$, reflecting the non-ellipticity of $P_{KG}$ in this case at fibre-infinity. Notice that, in our case where $P$ is elliptic at fibre-infinity, we could equally well define $\mathcal{X}^{s,\mathsf{r}_\pm}$ by requiring $u \in H_{\sct}^{s-1,\mathsf{r}_\pm}(\R^n)$ and $Pu \in  H_{\sct}^{s-2,\mathsf{r}_+}(\R^n)$, since microlocal elliptic regularity at fibre-infinity implies that  $u \in H_{\sct}^{s,\mathsf{r}_\pm}(\R^n)$ automatically. So this works the same as the hyperbolic case, except for the (slightly unfortunate) discrepancy of $1$ in our choice of index $s$. 
\end{remark}


Now we summarize basic properties of $\mathcal{X}^{s,\mathsf{r}_\pm}$ before diving into microlocal analysis.

\begin{proposition}
$\mathcal{X}^{s,\mathsf{r}_\pm}$ is a Hilbert space, using inner products from $H_{\sct}^{s,\mathsf{r}_\pm}$ and $H_{\sct}^{s-2,\mathsf{r}_\pm}$ for $u$ and $Pu$ respectively.
\end{proposition}

\begin{proof}
It is clear that this gives a inner product on $\mathcal{X}^{s,\mathsf{r}_\pm}$ that is compatible with our norm above and the only thing left to verify is the completeness.

Suppose $\{ u_j \}$ is a Cauchy sequence in terms of the norm in \eqref{eq:cal-X-norm}, then $\{ u_j\}$ is Cauchy in $H_{\sct}^{s,\mathsf{r}_\pm}$-norm and $\{ Pu_j \}$ is Cauchy in $H_{\sct}^{s-2,\mathsf{r}_\pm+1}$-norm. 
So we can find $u_\infty , \, f_\infty$ such that
\begin{equation}
u_j \to u_\infty \text{ in } H_{\sct}^{s, \mathsf{r}_\pm} \quad Pu_j \to f_\infty \text{ in } H_{\sct}^{s-2,\mathsf{r}_\pm+1}.
\end{equation}

Since $P$ is continuous $H_{\sct}^{s,\mathsf{r}_\pm} \to H_{\sct}^{s-2,\mathsf{r}_\pm }$, we know $Pu_j \to Pu_\infty$ in $H_{\sct}^{s-2,\mathsf{r}_\pm }$. This shows $Pu_\infty = f_\infty \in H_{\sct}^{s-2,\mathsf{r}_\pm}$ and indeed 
$\{u_j\}$ converges to $u_\infty$ in $\mathcal{X}^{s,\mathsf{r}_\pm}$.

\end{proof}

We can also run density argument using nice functions in this $\mathcal{X}^{s,\mathsf{r}_\pm}$. 
For each $u \in \mathcal{X}^{s,\mathsf{r}_\pm}$, we will construct a regularized family $u_r \in \mathcal{S}(\R^n)$ that converges to $u$ as $r \to 0$. We first prove the following lemma, which offers us those regularizers.
\begin{lemma}  \label{lemma:regularizer-Tr}
There exists a family of $T_r \in \Psi_{\sct}^{-\infty,-\infty}$ for $r \in (0,r_0]$ with some $r_0>0$ such that
\begin{equation} \label{eq:Tr-approximate-Id}
T_r \to \mathrm{Id}, 
\end{equation}
as $r \to 0$ with respect to the strong operator topology (for $H_{\sct}^{s,\mathsf{r}_\pm}(\R^n) \to H_{\sct}^{s,\mathsf{r}_\pm}(\R^n)$), and
\begin{equation} \label{eq:Tr,P-commutator-to-0}
 [T_r,P] \to 0
\end{equation}
as $r \to 0$ with respect to the strong operator topology (for $H_{\sct}^{s,\mathsf{r}_\pm}(\R^n) \to H_{\sct}^{s-1,\mathsf{r}_\pm+1}(\R^n)$).

\end{lemma}

\begin{proof}
Let $\chi \in C_c^\infty(\R^n)$ be even, identically $1$ near the origin and monotonically decreasing on $[0,\infty)$.
Then we set $T_r = \chi(r\la z \ra) \operatorname{Op}(\chi(\la r\zeta \ra))$.
We leave it to readers to verify \eqref{eq:Tr-approximate-Id}. (Hint: $\|(\mathrm{Id}-T_r)u\|_{s,\mathsf{r}_+}$ is  upper-bounded by the Sobolev norm of $u$ computed outside a large ball with radius $\sim r^{-1}$, on both the spatial and frequency side.)
We prove the slightly trickier \eqref{eq:Tr,P-commutator-to-0} here. 
A sketch of the idea is that, we can require $T_r$ being uniformly bounded in $\Psi_{\sct}^{0,0}$ and $T_r \to \Id$ in the operator topology of $\Psi_{\sct}^{\delta,\delta}$ for some small $\delta>0$ and it follows that $[T_r,P] \to 0$ in $\Psi_{\sct}^{1+\delta,-1+\delta}$ and $[T_r,P]$ being uniformly bounded in $\Psi_{\sct}^{1,-1}$ and this is sufficient for the desired strong convergence. 
This can be shown by noticing that $H_{\sct}^{s+\delta,\mathsf{r}_\pm+\delta}$ is dense in $H_{\sct}^{s,\mathsf{r}_\pm}$. So for each $u \in H_{\sct}^{s,\mathsf{r}_\pm}$, we take a sequence $u_j \in H_{\sct}^{s+\delta,\mathsf{r}_\pm+\delta}$ such that $u_j \to u$ in $H_{\sct}^{s,\mathsf{r}_\pm}$. Then we decompose $[T_r,P]$ as
\begin{equation}
[T_r,P]u = [T_r,P]u_j + [T_r,P](u-u_j).
\end{equation}
The right hand side converges to $0$ in $H_{\sct}^{s-1,\mathsf{r}_\pm+1}$ by the following reason.
For the first term, we use $[T_r,P] \to 0$ in $\Psi_{\sct}^{1+\delta,-1+\delta}$; while for the second term, we use uniform boundedness of $[T_r,P]$ in $\Psi_{\sct}^{1,-1}$ and $u-u_j \to 0$ in $H_{\sct}^{s,\mathsf{r}_\pm}$.

An alternative (almost equivalent) way is to further introduce a regularizer $T_{r'}$ and decompose $[T_r,P]$ as
\begin{equation}
[T_r,P] = 
[T_r-\mathrm{Id},P]T_{r'} + [T_r,P](\mathrm{Id}-T_r'),
\end{equation} 
for any fixed $r'$, where we used $[\mathrm{Id},P]=0$.  
Now we only concern $r \ll r'$ and send $r \to 0$ first. With $r \ll r'$, derivatives of the symbol of $T_r$ vanishes on the support of the (left) symbol of $T_{r'}$. 

For the first term, we need to show
\begin{align*}
\| [T_r-\mathrm{Id},P]T_{r'} u \|_{s-1,\mathsf{r}_\pm + 1} \to 0,
\end{align*}
which is equivalent to 
\begin{align*}
\| \la z \ra [T_r-\mathrm{Id},P]T_{r'} u \|_{s-1,\mathsf{r}_\pm} \to 0,
\end{align*}

The left full symbol $[T_r-\mathrm{Id},P]T_{r'}$ is a sum of products of symbols of $T_r, P,T_{r'}$ but with at least one order derivative hitting that of $T_r$ and $p$, which produces a $r$-factor. In addition, on the essential support of this, we have $\la z \ra \lesssim (r')^{-1}$. One can then check that each semi-norm of the left symbol of $\la z \ra[T_r-\mathrm{Id},P]T_{r'}$ in $S^{1,0}_{\sct}$ is $O(r/r')$. So sending $r \to 0$ gives the desired convergence.



On the other hand, we have $(\mathrm{Id}-T_r')u \to 0$ in $H^{s,\mathsf{r}_\pm}$ as $r' \to 0$ and $[T_r,P]$ is uniformly bounded in $\Psi_{\sct}^{1,-1}$.
So taking $r'\to 0$ (while keeping $r/r' \to 0$) gives $[T_r,P](\mathrm{Id}-T_r') \to 0$ in $H_{\sct}^{s-1,\mathsf{r}_\pm}$, which completes the proof.

\end{proof}

\begin{proposition}\label{prop:Schwdense}
The Schwartz function class $\mathcal{S}(\R^n)$ is dense in $\mathcal{X}^{s,\mathsf{r}_\pm}$ in terms of the topology from the $\mathcal{X}^{s,\mathsf{r}_\pm}$-norm.
\end{proposition}

\begin{proof}
For any $u \in \mathcal{X}^{s,\mathsf{r}_\pm}$, we set
\begin{equation}
u_r = T_ru,
\end{equation}
where $T_r$ is as in Lemma~\ref{lemma:regularizer-Tr}. Then we know $u_r \in \mathcal{S}(\R^n)$ since $T_r \in \Psi_{\sct}^{-\infty,-\infty}$.

To prove the conclusion, we need to show $T_ru \to u$ in $H_{\sct}^{s,\mathsf{r}_\pm}$ and $PT_ru \to Pu$ in $H_{\sct}^{s-2,\mathsf{r}_\pm+1}$.
The first claim follows from \eqref{eq:Tr-approximate-Id}. 
For the latter one, we write
\begin{align*}
PT_ru = T_rPu + [P,T_r]u.
\end{align*}
Then we apply \eqref{eq:Tr-approximate-Id} to the first term and apply \eqref{eq:Tr,P-commutator-to-0} to the second term.
\end{proof}

Now we state the main result in this part. Observe that $\rm_+$ is Feynman admissible if and only if $-\rm_+ - 1$ is anti-Feynman admissible. We assume that $\rm_- = -\rm_+ - 1$ in the next theorem: 
\begin{theorem}\label{thm:Fredholm-Helmholtz}
Let $\mathsf{r}_\pm$ be variable orders as above, then for any $s, M,N \in \R$ and $u \in \mathcal{X}^{s,\mathsf{r}_\pm}$, we have:
\begin{equation} \label{eq:Helmholtz-global-Fredholm}
\| u \|_{s,\mathsf{r}_\pm} \lesssim \| Pu \|_{s-2,\mathsf{r}_\pm}
+ \| u \|_{M,N}.
\end{equation}
The interesting case is when $M<\min\{s,2-s\}$ and $N < \min\{\mathsf{r}_+,\mathsf{r}_-\} = \frac{1}{2}-\delta$. With either sign $\pm$, this implies that
\begin{equation} \label{eq:Helmholtz-mapping}
(\Delta_g-\lambda^2):\; \mathcal{X}^{s,\mathsf{r}_\pm}
\to H_{\sct}^{s-2,\mathsf{r}_\pm+1}
\end{equation}
is Fredholm.
\end{theorem}

\begin{proof}
As we have discussed in Proposition~\ref{prop:Fred}, the last statement on the Fredholm property follows from \eqref{eq:Helmholtz-global-Fredholm} and the same type estimate for $P^*$. So we only need to prove the estimate here.

For definiteness we prove the case for spaces with order $\mathsf{r}_+$; the case for $\mathsf{r}_-$ can be proved in the same way.
	Recall the definitions of $\Uin, \Uout$ from Definition~\ref{definition:variable-order-admissible}. We begin by choosing $Q_i\in\Psi_{\sct}^{0,0}$, $1 \leq i \leq 4$, such that 
	\begin{enumerate}
		\item $P$ is elliptic on $\WF'(Q_1)$,
        \item $\WF'(Q_2)$ is disjoint from  $\mathcal{R}_\pm$,
		\item $\WF'(Q_3)\subset \Uout$,
		\item $\WF'(Q_4) \subset\Uin$,
		\item Every $\alpha \in \Char(P)$ in $\WF'(Q_3) \setminus \Rout$ lies on a forward bicharacteristic $\gamma$ from a point $\alpha'\in \Ell(Q_2)$, 
		\item Every $\alpha\in  \Char(P)$ in $\WF'(Q_2)$ lies on a forward bicharacteristic $\gamma$ from a point $\alpha'\in\Ell(Q_4)$,  \label{condition:Q4-control-Q2}
		\item $Q_1+Q_2+Q_3+Q_4=\mathrm{Id}$. \label{condition:Qi-partition}
	\end{enumerate}
A set of operators $\{ Q_i \}$ satisfying \eqref{condition:Qi-partition} is called a microlocal partition of unity. They can be constructed by quantizing a partition of unity $\{ \chi_i \}$ on ${}^{\sct}\overline{T}^*\R^n$. That is, $Q_i = \operatorname{Op}(\chi_i)$.

First consider $Q_1 = \operatorname{Op}(\chi_1)$. It deals with the elliptic region. So let $\chi_1$ be a function that vanishes on a small neighborhood of $\Sigma$ and is identically $1$ outside a slightly larger neighborhood of $\Sigma$.
Then we set $\chi_3$ to be a function that is identically $1$ on a small neighborhood of $\Rout$ and is supported in a slighly larger neighborhood. Similarly, let $\chi_4$ be a function that is identically $1$ on a small neighborhood of $\Rin$ and is supported in a lightly larger region.
Then we set $\chi_2 = 1-\chi_1-\chi_3-\chi_4$. In particular, $\chi_2$ will be supported away from $\mathcal{R}_+ \cup \mathcal{R}_-$ and close to $\Sigma$.

In the elliptic region, the elliptic estimate can be weakened to be:
	\begin{equation}
		\label{eq:fred1}
		\|Q_1u\|_{s,\mathsf{r}_+}\leq C(\|Pu\|_{s-2,\mathsf{r}_+ +1}+\|u\|_{M,N})
	\end{equation}
so that the spatial order of the norm of $Pu$ agrees with the norms arising from positive commutator estimates. (Here and below, we remove the microlocalizer $G$ in front of $Pu$, which of course we are entitled to do as $G \in \Psisc^{0,0}$ is bounded on all Sobolev spaces.) 

Now we first apply the above-threshold radial point estimate,
	using the operator $Q_4$. We can replace the spatial order $r$ with the variable weight $\mathsf{r}_+$ in this estimate, as $\mathsf{r}_+$ is constant (equal to $r$) in $\WF'(G_4)$. This yields the estimate
	\begin{equation}
		\label{eq:fred4}
		\|Q_4 u \|_{s,\mathsf{r}_+} \leq C(\|Pu\|_{s-2,\mathsf{r}_++1}+\|u\|_{s-1/2, \mr_+ - 1/2}).
	\end{equation} 
	Similarly, we apply the below-threshold radial point estimate 
	to the operators $Q_3,Q_2$ to give
	\begin{equation}
		\label{eq:fred3}
		\|Q_3u\|_{s,\mathsf{r}_+}\leq C(\|Q_2 u\|_{s,\mathsf{r}_+}+\|Pu\|_{s-2,\mathsf{r}_++1}+\|u\|_{M,N}).
	\end{equation}
	Away from the radial sets, we can control $\|Q_2u\|$ using Theorem~\ref{thm:mpe}, using $\WF'(Q_4)$ as a source of regularity, given the dynamical condition \eqref{condition:Q4-control-Q2}.
	
	Consequently we have an estimate
	\begin{equation}
		\label{eq:fred2}
		\|Q_2u\|_{s,\mathsf{r}_+}\leq C(\|Q_4u\|_{s,\mathsf{r}_+}+\|Pu\|_{s-2,\mathsf{r}_++1}+\|u\|_{M,N}).
	\end{equation}

Without loss of generality, we can assume that the constants $C$ in estimates \eqref{eq:fred4} --- \eqref{eq:fred1} are equal, and exceed $1$. Then, we 
 estimate 
 $$
 \|u\|_{s,\mathsf{r}_+}\leq \|Q_1u\|_{s,\mathsf{r}_+}+ 2C\|Q_2u\|_{s,\mathsf{r}_+}+\|Q_3u\|_{s,\mathsf{r}_+}+4C^2 \|Q_4u\|_{s,\mathsf{r}_+}
 $$ and combine the estimates in \eqref{eq:fred1}, \eqref{eq:fred2},  \eqref{eq:fred3} and \eqref{eq:fred4}. This combination allows us to absorb the $Q_2u$ and $Q_4u$ terms on the RHS by those on the LHS. 
	This gives (with a new constant $C$) 
	\begin{equation}
		\label{eq:fred5}
		\|u\|_{s,\mathsf{r}_+}\leq C(\|Pu\|_{s-1,\mathsf{r}_++1}+\|u\|_{s-1/2, \mr_+ - 1/2}+\|u\|_{M,N}).
	\end{equation}
    Then, we can use the interpolation inequality \eqref{eq:var interp} to estimate 
$$
C \|u\|_{s-1/2, \mr - 1/2} \leq \frac1{2} \|u\|_{s,\mathsf{r}_+} + C' \|u\|_{M,N}
$$
which allows us to remove the $\|u\|_{s-1/2, \mr - 1/2}$ term on the RHS of \eqref{eq:fred5} (at the cost of increasing $C$). 
This yields estimate \eqref{eq:Helmholtz-global-Fredholm} which is a semi-Fredholm estimate for $M < s$ and $N < \inf \rm_+$, thanks to the compact embedding $H_{\sca}^{s, \mr_+} \hookrightarrow H_{\sca}^{M, N}$ (see Problem~\ref{prob:var compact embedding}).

The final step to conclude Fredholm property is to establish that the cokernel is finite-dimensional.  
To do this, we identify the cokernel with the annihilator, in $H_{\sca}^{2-s, -\mr_+ - 1}$ (which we view as the adjoint of $H_{\sca}^{s-2, \mr_+ + 1}$ via $L^2$-pairing, as in Proposition~\ref{prop:duality}; note this is also valid for variable order spaces) of the range of $P$ on $\mathcal{X}^{s,\mathsf{r}_\pm}$. Let $v$ lie in this annihilator subspace of $H_{\sca}^{2-s, -\mr_+ - 1}$. Then, for all $u \in \mathcal{X}^{s,\mathsf{r}_+}$ we have $\ang{Pu, v} = 0$. In particular this holds for all Schwartz $u$. Since $P$ is formally self-adjoint we have $\ang{u, Pv} = 0$ for all Schwartz $u$, which implies that $Pv = 0$. Thus, $v$ lies in $\mathcal{X}^{2-s,-\mathsf{r}_+-1} = \mathcal{X}^{2-s,\mathsf{r}_-}$ (using $\mr_- = - \mr_+ - 1$) and is in the kernel of $P$ acting from $\mathcal{X}^{2-s,\mathsf{r}_-} \to H_{\sca}^{-s, \mathsf{r}_-+1}$. 

It thus suffices to show that the kernel of $P : \mathcal{X}^{2-s,\mathsf{r}_-} \to H_{\sca}^{-s, \mathsf{r}_-+1}$ has finite-dimensional kernel. This follows from the corresponding semi-Fredholm estimate for $P$ acting between these spaces, which is proved exactly as above, but with $\mr_-$ replacing $\mr_+$ and with $\Rin$ and $\Rout$ changing roles. 

\end{proof}

After proving Fredholm estimates, the final step to `solve'\footnote{In the sense that we have existence and uniqueness of solutions in terms of spaces in \eqref{eq:Helmholtz-mapping}.} a PDE is to show invertibility of the operator; then the inverse is the desired solution operator. 
For $P = \Delta_g -\lambda^2$, this is obtained via boundary pairing after extracting the leading order asymptotic of solutions with so-called `module regularity' as we sketch below. See also \cite{melrose1994spectral}, \cite[Section~3.2]{Helmholtz-1}. The term module regularity was introduced in \cite{HMV2004} but the idea has a long tradition in PDE. For example, Klainerman's vector field method \cite{klainerman1985} is essentially the same idea. 

For equations with conserved quantities (like linear and some non-linear Schr\"odinger equations), this can be achieved via exploiting certain conservation laws. See for example \cite[Section~6]{gell2022propagation}.

We summarize the result `solving' the inhomogeneous PDE below.

\begin{theorem} \label{thm:Helmtholz-invertible}
Let $\mathsf{r}_\pm$ be as in Definition~\ref{definition:variable-order-admissible}, and $\lambda > 0$, then 
\begin{equation}\label{eq:P Fred}
P = (\Delta_g-\lambda^2):\; \mathcal{X}^{s,\mathsf{r}_\pm}
\to H_{\sct}^{s-2,\mathsf{r}_\pm+1}
\end{equation}
is invertible and has bounded inverse:
\begin{equation}
\| u \|_{s,\mathsf{r}_\pm} \lesssim \| Pu \|_{s-2,\mathsf{r}_\pm+1}.
\end{equation}
We denote those bounded inverses by $P_\pm^{-1}$.
\end{theorem}

\begin{proof} Since we already know that $P$, acting as in \eqref{eq:P Fred}, is Fredholm, it only remains to show that the kernel and cokernel are both trivial. But as we saw in the previous proof, actually the cokernel can be identified with the kernel of the  operator $P$ acting on $\mathcal{X}^{2-s,\mathsf{r}_\mp}$. So it suffices to show that the kernel of $P$ is trivial, for all $s$ and for both signs $\mr_\pm$. 

Assume that $Pu = 0$, $u \in \mathcal{X}^{s,\mathsf{r}_+}$. We will show that $u = 0$. However, to do this we need two extra ingredients. The first is the notion of `module regularity', that is, extra regularity that may be possessed by a solution $u$ to $Pu = f$. This is the topic of the next subsection. The second is the existence of a leading asymptotic of solutions to $Pu = 0$ at spatial infinity, which we will prove in the next lecture (see Theorem~\ref{thm:leading-asymptotic-Helmholtz}).
Thus, the remainder of the proof is postponed to after the proof of Theorem~\ref{thm:leading-asymptotic-Helmholtz}. We will also need a `boundary pairing' formula, which we give here as Problem~\ref{prob:bpf}. 

\end{proof}

\subsection{Fredholm estimates for module regularity spaces}
Next we discuss some analogous Fredholm estimates on module regularity spaces that we will introduce below, which we will use in Lecture 13 to obtain asymptotics of solutions to $Pu = f$ for $f$ with suitable `module regularity'. Module regularity is extra regularity (additional to, say, membership in a Sobolev space $H_{\sca}^{s, \mr}$) with respect to a Lie algebra of vector fields. In our case, the extra regularity will be with respect to so-called `cusp' vector fields\footnote{The terminology arose from hyperbolic manifolds with cusps, but this is not relevant to our discussion.}
We first define cusp vector fields, and discuss some intertwining properties of cusp vector fields and our operator $P$. 

Let $\mathcal{V}_{\cu}(\R^n)$, the Lie algebra of cusp vector fields on $\overline{\R^n}$, be the $C^\infty(\overline{\R^n})$-module generated by 
\begin{equation}
\mathcal{G} = \{ D_{x_i}, \quad 1 \leq i \leq n, \quad x_iD_{x_j}-x_jD_{x_i}, \quad 1 \leq i,j \leq n, \quad \mathrm{Id} \}.
\end{equation}
We can think of this as the asymptotic symmetry group of our operator; if the metric is the flat Euclidean metric, then these vector fields genuine symmetries of $\Delta_0$, namely translations and rotations. 
We are only interested in the module in a neighbourhood of spatial infinity. There, we can discard all the partial derivatives $D_{x_i}$ and retain only a radial derivative $D_r$, where $r = |x|$, as the others are $C^\infty(\overline{\R^n})$-linear combinations of $D_r$ and angular derivatives $x_iD_{x_j}-x_jD_{x_i}$. 
In terms of coordinates $(\msf{x}, y_1, \dots, y_{n-1})$ on $\overline{\R^n}$ with $(y_1,..,y_{n-1})$ being a coordinate system on $\partial \overline{\R^n}$, as in Lecture 7, an equivalent generating set is  
\begin{equation} \label{eq:generating-set-2}
{\msf{x}}^2D_{\msf{x}},\quad D_{y_1},..D_{y_{n-1}}, \quad \Id. 
\end{equation} 
Then we set $\mathrm{Diff}_{\cu}^k$ to be the $C^\infty(\overline{\R^n})$-module generated by $k$-fold products of elements in $\mathcal{G}$.


Our intertwining property is as follows.
\begin{proposition} \label{prop:cusp-Laplacian-commutator}
Let $L \in \mathrm{Diff}_{\cu}^k \subset \Psi_{\sct}^{k,k}$ be a $k$-fold product of elements in $\mathcal{G}$, we have
\begin{equation} \label{eq:module-commutator-extra-vanishing}
[L,P] \in \Psi_{\sct}^{k+1,k-2}.
\end{equation}

That is, it has one extra order of spatial decay compared with the `trivial' membership, reflecting that $P$ is asymptotically symmetric with respect to the flow generated by (factors of) those vector fields.
\end{proposition}
\begin{remark}
This one order extra decay for all cusp vector fields is the main reason that our asymptotically Euclidean setting is easier than the general asymptotically conic setting in \cite{Melrose1994}, in which case one needs to deal with certain elliptic combination of cusp vector fields instead of individual ones. 
Geometrically, this is because the metric (and hence operator) to which we  are asymptotic  is rotationally symmetric.
\end{remark}

\begin{proof}
We only need to consider the case $k=1$ since we have the composition law and the derivation property:
\begin{equation}
[V_1V_2,P] = V_1[V_2,P] + [V_1,P]V_2,
\end{equation}
for $V_i \in \mathcal{G}$.

Let $\Delta_0$ be the Laplacian of the Euclidean space, then for all $V \in \mathcal{G}$, we have $[D_{y_i},\Delta_0] = 0$ and $[D_r,\Delta_0] \in \Psi_{\sct}^{1,-2}$.

We can write
\begin{equation}
[V,\Delta_g-\lambda^2]  = [V,\Delta_g-\Delta_0]+[V,\Delta_0].
\end{equation}
Under our assumption \eqref{eq:metric-assumption}, we have (the computation is left as Problem~\eqref{ex:Laplacian-approximate})
\begin{equation} \label{eq:Laplacian-approximate}
\Delta_g - \Delta_0 \in \Psi_{\sct}^{2,-1},
\end{equation}
and this gives \eqref{eq:module-commutator-extra-vanishing} for $k=1$, which in turn finishes the proof as we have explained.
\end{proof}

Then our Fredholm estimate generalizes to those module regularity spaces below.

\begin{definition}
Let $\mathcal{G}^k$ be the set of $k$-fold products of elements in $\mathcal{G}$. We define the $k$th-order module regularity space, denoted by $H_{\sct}^{s,\mathsf{r}_\pm;k}(\R^n)$, to be the space of functions such that
\begin{equation}
\{u:   L u  \in   H_{\sct}^{s,\mathsf{r}_\pm}, \forall L \in \mathcal{G}^k \}.
\end{equation}
Of course, this already requires $u \in   H_{\sct}^{s,\mathsf{r}_\pm}$.  We equip this space with the norm (in fact the corresponding inner product as well):
\begin{equation} \label{eq:k-module-norm-defn}
\| u \|_{s,\mathsf{r}_\pm;k}^2 = \sum_{L \in \mathcal{G}^k} \| Lu \|_{s,\mathsf{r}_\pm}^2.
\end{equation}
\end{definition}
Then scattering pseudodifferential operators are bounded between those Sobolev spaces.
\begin{proposition}  \label{prop:boundedness-module-regularity}
Let $Q \in \Psisc^{m,l}$, then
\begin{equation}
    \| Q u \|_{s-m,\mathsf{r}_\pm-l;k} \lesssim \| u \|_{s,\mathsf{r}_\pm;k}.
\end{equation}
\end{proposition}

\begin{proof}
By the definition of the $H_{\sct}^{s,\mathsf{r}_\pm;k}$-norm, we only need to show
\begin{equation}
    \sum_{L \in \mathcal{G}^k} \| L Q u \|_{s-m,\mathsf{r}_\pm-l} \lesssim \sum_{L \in \mathcal{G}^k} \| L  u \|_{s,\mathsf{r}_\pm},
\end{equation}
    which in turn can be implied by 
\begin{equation} \label{eq:module-est-L0}
   \| L_0 Q u \|_{s-m,\mathsf{r}_\pm-l} \lesssim \sum_{L \in \mathcal{G}^k} \| L  u \|_{s,\mathsf{r}_\pm},
\end{equation}    
for a fixed $L_0 \in \mathcal{G}^k$. Now we apply an induction on $k$ and the estimate above for $k=0$ follows from the boundedness of pseudodifferential operators between usual weighted (variable order) Sobolev spaces we discussed.
Suppose the estimate is true with $k$ replaced by $k-1$ and now we prove \eqref{eq:module-est-L0}.
Writing $L_0 = V_1...V_k$, we have
\begin{equation} \label{eq:L0Q-decomposition}
    L_0Qu = QL_0u + [L_0,Q] u .
\end{equation}
For the contribution from the first term, we apply the boundedness of $Q$ on Sobolev spaces without module regularity to obtain
\begin{equation}
    \| QL_0u \|_{s-m,\mathsf{r}_\pm-l} \lesssim \|  L_0 u \|_{s,\mathsf{r}_\pm}  \leq \sum_{L \in \mathcal{G}^k} \| L  u \|_{s,\mathsf{r}_\pm}.
\end{equation}
For the contribution from the second term in \eqref{eq:L0Q-decomposition}, we notice that $[L_0,Q]$ can be written as a sum of terms of the form $V_1...V_{j-1}[V_{j},Q]V_{j+1}...V_k$ and we have
\begin{equation}
   \| [L_0,Q] u \|_{s-m,\mathsf{r}_\pm-l}
   \leq \sum_{j=1}^k \| V_1...V_{j-1}[V_{j},Q]V_{j+1}...V_k u \|_{s-m,\mathsf{r}_\pm-l}.
\end{equation}
For each individual term, we have
\begin{multline}
   \| V_1...V_{j-1}[V_{j},Q]V_{j+1}...V_ku \|_{s-m,\mathsf{r}_\pm-l} \lesssim  \| [V_{j},Q]V_{j+1}...V_ku \|_{s-m,\mathsf{r}_\pm-l;j-1}
  \\  \lesssim \| V_{j+1}...V_ku \|_{s,\mathsf{r}_\pm;j-1}
  \lesssim \| u \|_{s,\mathsf{r}_\pm;k-1},
\end{multline}
where the first inequality follows from the definition of module regularity norms, the second inequality follows from the induction hypothesis since $[V_j,Q] \in \Psisc^{m,l}$, and the last inequality again follows from the definition of module regularity norms. This completes the proof. 
\end{proof}

The analogue of the $\mathcal{X}^{s,\mathsf{r}_\pm}$-space in this module regularity is the following:
\begin{equation}
\mathcal{X}^{s,\mathsf{r}_\pm;k}:=
\{ u \in H^{s,\mathsf{r}_\pm;k}\mid Pu \in  \mathcal{X}^{s-2,\mathsf{r}_\pm;k} \}.
\end{equation}
And it is equipped with the norm (and inner product given by)
\begin{equation} \label{eq:module-reg-norm}
\| u \|_{\mathcal{X}^{s,\mathsf{r}_\pm;k}}^2
= \| u \|_{s,\mathsf{r}_\pm;k}^2 + \| Pu \|_{s-2,\mathsf{r}_\pm;k}^2.
\end{equation}

Using the same argument as before, this is a Hilbert space. Now we have the following global Fredholm estimate for the module regularity spaces.
\begin{theorem}\label{thm:Fredholm-Helmholtz-module case}
Let $\mathsf{r}_\pm$ be variable orders as before, then for any $s,M,N \in \R$ and $u \in \mathcal{X}^{s,\mathsf{r}_\pm;k}$, we have (we ignore the module regularity in the error term, since they can be absorbed into $M,N$):
\begin{equation} \label{eq:module-reg-global-Fredholm}
\| u \|_{s,\mathsf{r}_\pm;k} \lesssim \| Pu \|_{s-2,\mathsf{r}_\pm+1;k}+ \| u \|_{M,N}.
\end{equation}
The interesting case is when $M<\min\{s,2-s\}$ and $N < \min\{\mathsf{r}_+,\mathsf{r}_-\} = \frac{1}{2}-\delta$. With either sign $\pm$, this implies that
\begin{equation} \label{eq:Helmholtz-mapping-module}
(\Delta_g-\lambda^2):\; \mathcal{X}^{s,\mathsf{r}_\pm;k}
\to H_{\sct}^{s-2,\mathsf{r}_\pm+1;k}
\end{equation}
is Fredholm.
\end{theorem}

\begin{proof}
We prove this using \eqref{eq:Helmholtz-global-Fredholm} and the intertwining property above. 
We make an induction on $k$. The case $k=0$ is \eqref{eq:Helmholtz-global-Fredholm}.
Suppose \eqref{eq:Helmholtz-mapping-module} is true with $k$ replaced by $k-1$, then the induction amounts to show
\begin{equation}
\| V u \|_{s,\mathsf{r}_\pm;k-1} \lesssim \| V Pu \|_{s-2,\mathsf{r}_\pm+1;k-1}+ \| Pu \|_{s-2,\mathsf{r}_\pm+1;k-1}+ \| V u \|_{M,N},
\end{equation}
for any $V \in \mathcal{G}$. To show this, we notice that
\begin{align} \label{eq:VP-commutator}
VPu = PVu + [V,P]u.
\end{align}
Applying the induction hypothesis to $Vu$, we have
\begin{align*}
\| V u \|_{s,\mathsf{r}_\pm;k-1} & \lesssim \| PVu \|_{s-2,\mathsf{r}_\pm+1;k-1} + \| Vu \|_{M,N}
\\ & \leq   \| VPu \|_{s-2,\mathsf{r}_\pm+1;k-1} + \|[P,V]u\|_{s-2,\mathsf{r}_\pm+1;k-1} + \| Vu \|_{M,N}
\\ & \lesssim \| VPu \|_{s-2,\mathsf{r}_\pm+1;k-1} + \|u\|_{s,\mathsf{r}_\pm;k-1} + \| Vu \|_{M,N},
\end{align*}
where in the last line we used \eqref{eq:module-commutator-extra-vanishing} with $L$ there being a single vector and the boundedness in Proposition~\ref{prop:boundedness-module-regularity}. For the $\|u\|_{s,\mathsf{r}_\pm;k-1}$-term, we can apply the induction hypothesis again to control it by $\| Pu \|_{s-2,\mathsf{r}_\pm+1;k-1}$ with an error term that can be combined into the last term. 

\end{proof}

\begin{proposition}\label{prop:inf module reg}
 Assume that $Pu$ is Schwartz, and $u \in H_{\sca}^{s, \mr_\pm}$. Then $u \in H_{\sca}^{s, \mr_\pm; k}$ for all $k \in \mathbb{N}$.    
\end{proposition}

\begin{remark}
 We first remark on the relationship between this Proposition and Theorem~\ref{thm:Fredholm-Helmholtz-module case}. Given $Pu = f$, we can ask, if $f$ has additional module regularity, does $u$ also have additional module regularity? Furthermore, if the answer is yes, can we get estimates on the module regularity norms of $u$ in terms of the module regularity norms of $f = Pu$, plus weaker norms of $u$? 

 Theorem~\ref{thm:Fredholm-Helmholtz-module case} \emph{assumes} the answer to the first question is yes, and provides an estimate. By contrast, this Proposition answers the first question affirmatively, \emph{under the additional assumption that $u$ is above threshold at one of the radial sets}. It is easy to see that the answer to the first question is negative without this assumption; for example, there are many solutions to $Pu = 0$, including plane waves such as $e^{i\lambda x_1}$ in the free case, that have no additional module regularity. 
\end{remark}

\begin{proof}
It suffices to prove the result just for $\mr_+$, as the argument for $\mr_-$ is strictly analogous. 

 We prove this by induction on $k$. We consider the scattering wavefront set of $u$. 
Since $Pu$ is Schwartz, the wavefront set of $u$ is contained in $\Char(P)$. Next, by assumption, $u$ is above threshold microlocally near $\Rin$ since $\mr_+ > -1/2$ there. Thus using the above-threshold radial point estimate, Theorem~\ref{thm:above dist}, we see that $u$ is Schwartz microlocally near $\Rin$, or equivalently, that $\Rin \notin \WFsc(u)$. Next, using Theorem~\ref{thm:mpe}, this regularity propagates throughout $\Char(P)$ except for the outgoing radial set $\Rout$. Thus $\WFsc(u) \subset \Rout$. 

 Next we consider $Lu$ for $L \in \mathcal{G}$.  Clearly $\WFsc(Lu) \subset \Rout$, since this is true of $u$, and the operator $L$ is microlocal in the sense of Proposition~\ref{prop:microlocality}. We investigate the microlocal regularity of $Lu$ near $\Rout$. We compute
 \begin{equation}\label{eq:1 commutator}
 PLu = LPu + [P, L]u = [P, L] u \text{ mod } \mathcal{S}. 
 \end{equation}
Since $\mr_+$ is below threshold near $\Rout$, we can apply the below-threshold radial point estimate, Theorem~\ref{thm:below}. According to Proposition~\ref{prop:boundedness-module-regularity}, $[P, L] \in \Psisc^{2, -1}$, so we see that $[P, L] u \in H_{\sca}^{s-2, \mr_+ + 1}$. Hence, using \eqref{eq:1 commutator}, $P(Lu) \in H_{\sca}^{s-2, \mr_+ + 1}$ and therefore, Theorem~\ref{thm:below} shows that $Lu \in H_{\sca}^{s, \mr_+}$ microlocally near $\Rout$. As we already observed that $\WFsc(Lu) \subset \Rout$, this means that $Lu \in H_{\sca}^{s, \mr_+}$ globally. Thus, $u \in H_{\sca}^{s, \mr_+; 1}$. That is, we have proved the $k=1$ case of the Proposition.

To prove for $k=2$, we compute, for $L_1, L_2 \in \mathcal{G}$,  
\begin{multline}
 PL_1 L_2u = L_1P L_2 u + [P, L_1]L_2 u 
 = L_1 L_2 Pu + L_1[P, L_2]  u + [P, L_1]L_2 u \\
 = [P, L_2] L_1 u + [L_1, [P, L_2]] u  + [P, L_1]L_2 u\text{ mod } \mathcal{S}. 
 \end{multline}
\end{proof}
We claim that the RHS is in $H_{\sca}^{s-2, \mr_+ + 1}$. To see this, we have already shown that $L_i u \in H_{\sca}^{s, \mr_+}$, and $[P, L_j] \in \Psisc^{2, -1}$, so the first and the third terms are in  $H_{\sca}^{s-2, \mr_+ + 1}$. As for the second term, the double commutator $[L_1, [P, L_2]]$ is in $\Psisc^{2, -1}$ showing that the second term is also in this space, proving the claim. Then Theorem~\ref{thm:below} shows that $L_1 L_2u \in H_{\sca}^{s, \mr_+}$ microlocally near $\Rout$, which proves the $k=2$ case of the Proposition. This argument can be iterated to prove the claim for all $k$. 

\begin{remark}
In the analysis of the Helmholtz equation, the characteristic variety is only at finite frequency over spatial infinity. So only the geodesic flow `at spatial infinity' is relevant to our analysis and no non-trapping condition is needed. If we were instead considering Klein-Gordon equations or Sch\"odinger equations, then we would have characteristic variety also at fibre-infinity over the interior. Establishing a Fredholm estimate would require microlocal propagation estimates at fiber infinity over the interior of spacetime. To have a global source-sink structure for the flow, as we showed for the Helmholtz operator, would require the metric to be \emph{non-trapping}. That means that every bicharacteristic (geodesic, in this case) would reach spatial infinity both forward and backward. This would not happen if there were a periodic geodesic for example. 

But we should point out that, even in the presence of mild (normally hyperbolic) trapping, one can still have propagation estimates with certain loss. See \cite{dyatlov2016spectral}\cite{hintz2021normally}\cite{jia2022propagation}. In addition, we can obtain a estimate that propagates regularity out of the trapped set and then again glue them using \cite{Datchev-Vasy-glue} to obtain a global Fredholm estimate.
In fact, this is the situation in \cite{vasy2013asympthyp}, which is where this non-elliptic Fredholm theory was invented.
\end{remark}


\subsection{Problems}

\begin{problem}
 Verify \eqref{eq:Tr-approximate-Id} and also this difference is in $r^{-2}\mathrm{Diff}_{\cu}^2$.
\end{problem}

\begin{problem}
 Prove $\mathcal{X}^{s,\mathsf{r}_\pm;k}$ with norm \eqref{eq:module-reg-norm} is a Hilbert space.
\end{problem}

\begin{problem}
In fact, we have the following generalization (\cite[Lemma~5.10]{vasy-monicourse}): suppose we have a family $A_j \in \Psi_{\sct}^{m_j,\el_j}(\R^n),\, j = 1,2,...N$ and set
\begin{equation}
\mathcal{X} = \{u \in H_{\sct}^{s,r}: \;  A_j u \in H_{\sct}^{s_j,r_j}, \, j = 1,2,...,N  \}
\end{equation}
and equip it with the norm
\begin{equation}
\|u\|_{\mathcal{X}}^2 = \|u\|_{s,r}^2 + \sum_{j=1}^N \| A_j \|^2_{s_j,r_j},
\end{equation}
then this is a Hilbert space.
\end{problem}

\begin{problem} Failure of density. Let $\ml_\pm$ be as above, and choose a constant $L > 1/2$. Show that Schwartz functions are \emph{not} dense in the space 
$$
\SX = \{ u \in H^{s, \ml_\pm} \mid P_0u \in H^{s, L} \}, \quad P_0 = \Delta - \lambda^2. 
$$
Hint: let $u = P_0^{-1} \phi$, where $P_0^{-1}$ is the outgoing resolvent, and $\phi$ is Schwartz. Assume that $\omega_j$ is a sequence of Schwartz functions converging to $u$ in the topology of the space $\SX$ above. Pass to the Fourier transform, and consider the operation of restriction to the sphere of radius $\lambda$ to derive a contradiction. This counterexample shows that the density result of Proposition~\ref{prop:Schwdense} is quite delicate. 
\end{problem}

\begin{problem}
(Complex absorption) Suppose $Q \in \Psi^{0,0}_{\sct}$ has real principal symbol elliptic near $\mathcal{R}_-$ , and that $Q$ is micro-supported in a slightly larger neighborhood of $\mathcal{R}_-$. Let $\mathsf{r}$ be a variable order that is not necessarily monotonic near $\mathcal{R}_-$, prove that for $\tilde{P}=P-iQ$, the sub-elliptic estimate like the above threshold radial point estimate still hold near $\mathcal{R}_-$ (we don't have a threshold to restrict $\mathsf{r}$ now!), which in turn gives the global Fredholm estimate if the assumption on $\mathsf{r}$ (monotonic along the $H_p$-flow and below threshold there) remain valid near $\mathcal{R}_+$.
\end{problem}

\begin{problem}\label{prob:bpf} (Boundary pairing formula). Let $u_1, u_2$ be two solutions to $Pu_i \in \Schw$. \emph{Assume} that each $u_i$ has an asymptotic expansion 
\begin{multline}
u_i = r^{-(n-1)/2} \Big( e^{i\lambda r} f_{i,+}(\omega) + e^{-i \lambda r} f_{i,-}(\omega) \Big) + u_i', \quad r = |x| \to \infty, \quad \omega = \frac{x}{|x|}, \\
u_i', D_{x_j} u_i' = O(r^{-(n-1)/2-\epsilon})
\end{multline}
where $f_{i, \pm} \in L^2(S^{n-1})$ and the $O(r^{-(n-1)/2-\epsilon})$ term is measured in $L^2(S^{n-1})$. 
Prove  the formula 
\begin{equation} \label{eq:boudary-pairing}
\int_{\R^n} \big(u_1 \overline{Pu_2} - (Pu_1)\overline{u_2} \big) \mathrm{dvol}_g = 2i\lambda \int_{\mathbb{S}^{n-1}} \big( f_{1,+}\overline{f_{2,+}}-f_{1,-}\overline{f_{2,-}} \big) \mathrm{dvol}_{\mathbb{S}^{n-1}}.
\end{equation}
Hint: apply Green's formula on the ball of radius $R$, $B_R = \{ |x| \leq R \}$, and show that the boundary term tends to the RHS of \eqref{eq:boudary-pairing} as $R \to \infty$. It might help to do the case of the free Helmholtz operator $P_0$ first. 
\end{problem}

\begin{problem}
Generalize Proposition~\ref{prop:inf module reg} as follows: instead of $Pu \in \Schw$, assume the weaker condition that $Pu \in H_{\sca}^{s-2, \mr_\pm + 1; k}$. Show that if $u \in H_{\sca}^{s, \mr_\pm}$, then $u \in H_{\sca}^{s, \mr_\pm; k}$. 
\end{problem}

\section{Lecture 13: Scattering theory}
We discuss scattering theory for the Helmholtz operator $P= \Delta_g - \lambda^2$ in this lecture. We use the term `scattering theory' in the sense of the `geometric scattering theory' of Melrose \cite{Melrose-geometric-sc}, rather than the traditional approach in terms of wave operators, as expounded in e.g. Reed and Simon \cite{RS-III}. We study the structure of the generalized eigenfunctions of $P$ using the inverses $P_{\pm}^{-1}$ constructed in Theorem~\ref{thm:Helmtholz-invertible}. However, we first have to complete the proof of this Theorem, for which we need Theorem~\ref{thm:leading-asymptotic-Helmholtz}.

\subsection{The Poisson operator}


Let $P = \Delta_g - \lambda^2$ be a Helmholtz operator. 
The Poisson operator $\Poi(\pm \lambda)$ is (the extension of) the operator sending $f_- \in C^\infty(\partial \overline{\R^n})$ to the solution $u$ of 
\begin{align*}
Pu = 0
\end{align*}
that is of the form (we use $r = |x|$ to denote the radial length in $\R^n$)
\begin{align} \label{eq:typical-expansion-Helmholtz}
r^{-(n-1)/2}e^{-i\lambda r} f_-(y) + r^{-(n-1)/2}e^{i\lambda r}f_+(y)+ o(r^{-(n-1)/2}).
\end{align}
Here $f_-$ is called the `incoming' data while $f_+$ is called the `outgoing' data.\footnote{Think about multiplying \eqref{eq:typical-expansion-Helmholtz} by $e^{-i\lambda t}$. This becomes a solution to the wave equation, and the factor $e^{-i\lambda(r+t)}$ is a spherical incoming wave, while $e^{i\lambda(r-t)}$ is a spherical outgoing wave.} 

Of course this definition only makes sense if we can establish the existence and uniqueness of such a function $u$. We proceed to show this. 
The next theorem below tells us that solutions in our $\mathcal{X}^{s,\mathsf{r}_\pm}$-type spaces are indeed like this, but with only one piece of the expansion, that is, only the $e^{\pm i\lambda r}$ oscillation. We need to assume in addition 2 orders of module regularity as introduced in Lecture 12 (see Definition~\ref{eq:k-module-norm-defn}). For simplicity, we state the result only for $\mr_+$. 


\begin{theorem} \label{thm:leading-asymptotic-Helmholtz}
Suppose $u \in \mathcal{X}^{s,\mathsf{r}_+;2}$ with $s\geq0$, with $\min \mathsf{r}_+ = -\frac{1}{2}-\epsilon$ for some small $\epsilon>0$ and
\begin{equation}
P u = f \in \la r \ra^{-2}L^2(\R^n),
\end{equation}
and $\WF_{\sct}(f) \cap \Rin = \emptyset$, then we have $a_0(y) \in L^2(\mathbb{S}^{n-1})$ such that
\begin{equation}\label{eq:leading-asymptotic-Helmholtz}\begin{gathered}
u = r^{-(n-1)/2} e^{i\lambda r}a_0(y) + u', \\
u', D_{x_j} u' = O(r^{-(n+1)/2 + \epsilon'}), \ \forall \epsilon' > \epsilon, 
\end{gathered}\end{equation}
where the $O(\cdot)$-notation is in terms of an $L^2(\mathbb{S}^{n-1})$-valued function of $r$.
\end{theorem}

\begin{proof}
Notice that in the polar coordinates, we have
\begin{align} \label{eq:Laplacian-r-coordinate}
\Delta_g = D_r^2 -i(n-1)r^{-1}D_r + r^{-2}\Delta_{\mathbb{S}^{n-1}} + R,
\end{align}
with $R \in r^{-1}\mathrm{Diff}_{\sct}^2(\R^n) \cap r^{-2} \mathrm{Diff}_{\mathrm{cu}}^2(\R^n)$.\footnote{See \cite[Section~3]{Melrose1994}.} 
Let $v = r^{(n-1)/2}u \in H_{\sca}^{s, \mr_+ - (n-1)/2; 2}$, then we have
\begin{align*}
(\Delta_g-\lambda^2) u = r^{-(n-1)/2}\big((r^{(n-1)/2}\Delta_gr^{-(n-1)/2}) - \lambda^2\big)v.
\end{align*}
The conjugation has the effect:
\begin{align*}
r^{(n-1)/2}D_rr^{-(n-1)/2} = D_r + i \frac{n-1}{2}r^{-1}.
\end{align*}

So we have
\begin{align} \label{eq:v-factored}
(D_r+\lambda)(D_r-\lambda)v = r^{(n-1)/2}f+
r^{-2}\Delta_{\mathbb{S}^{n-1}}v + \tilde{R}v,
\end{align}
with $\tilde{R} \in r^{-1}\mathrm{Diff}^2_{\sct}(\R^n) \cap r^{-2}\mathrm{Diff}^2_{\mathrm{cu}}(\R^n)$.
It is the latter containment, $\tilde{R} \in r^{-2}\mathrm{Diff}^2_{\mathrm{cu}}(\R^n)$ that is important here, as second-order cusp differential operators are generated by $\mathcal{G}^2$; since by assumption we have two orders of module regularity, applying $\tilde R$ to $v$ on the RHS of \eqref{eq:v-factored} gains us two orders of spatial decay. Moreover, the same is true of the operator $r^{-2} \Delta_{\mathbb{S}^{n-1}}$, since the spherical Laplacian is generated by vector fields $D_{y_j}$ which are also module generators. 

Then we use this condition to obtain the regularity and decay of $v$ by a bootstrapping argument. In summary, we first use the given information about $u$ (or equivalently $v$) and substitute it into the right hand side of \eqref{eq:v-factored}. We then use this equation to improve the conclusion, the key idea being to invert $(D_r+\lambda)$ microlocally near $\Char(P)$ and away from $\Rin$.
We start with the hypotheses of the theorem, giving 
 $u \in H_{\sct}^{0,-\frac{1}{2}-\epsilon;2}$, and hence $v \in H_{\sct}^{0,-\frac{n}{2} -\epsilon;2}$. It follows that the RHS of \eqref{eq:v-factored} is in $r^{-2+\frac{n}{2}+\epsilon}L^2(r^{n-1}dr, L^2(\mathbb{S}^{n-1})) = r^{-3/2+\epsilon}L^2(dr, L^2(\mathbb{S}^{n-1}))$.

Now we try to improve the decay property of $v$ using \eqref{eq:v-factored}. There are two regions two consider: the elliptic region (i.e. away from $\Char(P)$) and the non-elliptic region (i.e. near $\Char(P)$).
Take $G \in \Psisc^{0,0}$ such that $G$ is microlocally equal to the identity near $\Char(P)$ and in the region $r \lesssim 1$, which means
$\WF'(\Id - G)$ (defined in Section~\ref{subsec:op-WF}) is away from $\Char(P) = \Char(r^{(n-1)/2} P r^{-(n-1)/2})$ and the region $r \lesssim 1$. Then by the (microlocal) elliptic estimate, we have, for $M, N$ sufficiently positive, 
\begin{multline}\label{eq:(Id - G)v est}
\| (\Id-G)v \|_{2,-n/2 + 5/2} \lesssim \| r^{(n-1)/2} P r^{-(n-1)/2} v \|_{0, -n/2+5/2} + \| v \|_{-M, -N}  \\
= \| r^{(n-1)/2} P u \|_{0, -n/2+5/2} + \| u \|_{-M, -N'} = \| f \|_{0, 2} + \| u \|_{-M, -N'}  < \infty. 
\end{multline}
In particular, \eqref{eq:(Id - G)v est} tells us that
\begin{equation}\label{eq: (Id - G) v decay}
(\Id-G)v,  D_r(\Id-G)v \in r^{n/2 - 5/2} L^2(r^{n-1}dr, L^2(\mathbb{S}^{n-1})) = r^{-2} L^2(dr, L^2(\mathbb{S}^{n-1})),
\end{equation}
and by Cauchy-Schwarz this space is contained in $$r^{-3/2 + \epsilon} L^1(dr, L^2(\mathbb{S}^{n-1})).$$ Integrating the $r$-derivative, we see that $(\Id-G)v$ is continuous in $r$ with values in $L^2(\mathbb{S}^{n-1})$, with a limit at infinity. This limit has to vanish to be consistent with \eqref{eq: (Id - G) v decay}; moreover, we then see that 
\begin{equation}
   \|  (\Id-G)v(r) \|_{L^2(\mathbb{S}^{n-1})} \leq C r^{-3/2 + \epsilon}. 
\end{equation}
Moreover, in the passage from \eqref{eq:(Id - G)v est} to \eqref{eq: (Id - G) v decay} we gave up one derivative, so the same estimate can be made for the partial derivatives of $(\Id-G)v(r)$. 

Next we consider $Gv$. 
Notice that $u$ satisfies the above-threshold condition at $\Rin$ and $f$ is microlocally Schwartz there, then we can apply the above threshold radial point estimate discussed at the end of Section~\ref{subsec:proof-radial-est} to conclude that $u$ (hence $v$) is microlocally Schwartz near $\Rin$, similar to the proof of Proposition~\ref{prop:inf module reg}. 
More precisely, we further decompose $Gv$ into 
\begin{equation}
    Gv = G_1v+G_2v,
\end{equation}
where $\WF'(G_1)$ is contained in a neighborhood of $\Rin$ while $\WF'(G_2)$ is away from $\Rin$ and $(D_r+\lambda)$ is (uniformly) elliptic on a neighborhood of it. This is possible, since the only part of $\Char(P)$ where $D_r + \lambda$ is non-elliptic is $\Rin$. 
Then argument above says that $G_1v$ is Schwartz, so it only remains to  consider $G_2v$.
By the triangle inequality, 
\begin{equation}
\| (D_r - \lambda)G_2v \|_{1,2-\frac{n}{2}-\epsilon}
\leq  \| G_2 (D_r - \lambda)v \|_{1,2-\frac{n}{2}-\epsilon} 
  +  \| [(D_r - \lambda),G_2]v \|_{1,2-\frac{n}{2}-\epsilon} .    
\end{equation}
The commutator term can be treated as above, since $\WFsc'([(D_r - \lambda),G_2])$ is contained in the elliptic set of $P$ or in a small neighbourhood of $\Rin$, so the previous arguments and estimates apply to this term. 
The first term on the RHS can be estimated using microlocal elliptic regularity, as $D_r + \lambda$ is elliptic on the operator wavefront set of $G_2$:
\begin{equation}\label{eq:G2 1}
  \| G_2 (D_r - \lambda)v \|_{1,2-\frac{n}{2}-\epsilon} 
  \lesssim \| (D_r+\lambda)(D_r - \lambda)v \|_{0,2-\frac{n}{2}-\epsilon} + \|v\|_{-M,-N}   < \infty
\end{equation}
where we used \eqref{eq:v-factored} and the discussion below it to deduce the finitude of the first norm on the RHS. The second norm is finite if  $M,N \in \R$ are taken sufficiently positive. Thus we have 
\begin{multline}\label{eq: G2 2}
  (D_r - \lambda)G_2v \in r^{n/2 + \epsilon - 2} L^2(r^{n-1} dr, L^2(\mathbb{S}^{n-1}))  = 
  r^{-3/2 + \epsilon} L^2(dr, L^2(\mathbb{S}^{n-1})) \\
  \subset r^{-1 + \epsilon'} L^1(dr, L^2(\mathbb{S}^{n-1})), \quad \epsilon' > \epsilon. 
\end{multline}
Solving this ODE gives 
\begin{align*}
G_2 v = e^{i\lambda r}a_0(y) + O(r^{-1 + \epsilon'}).
\end{align*}
Similarly to last time, we gave up one derivative in the passage from \eqref{eq:G2 1} to \eqref{eq: G2 2}, so we can obtain the same error estimate on the partial derivatives of $G_2 v$. 
Since we have proved that $(\Id-G)v$ and $G_1v$ only contribute an $O(r^{-3/2 + \epsilon})$ error, we conclude that  
\begin{align*}
 v - e^{i\lambda r}a_0(y), D_{x_j} \big( v - e^{i\lambda r}a_0(y) \big)  = O(r^{-1 + \epsilon'}).
\end{align*}
This completes the proof since $u=r^{-(n-1)/2}v$.
\end{proof}

\begin{proof}[Completion of the proof of Theorem~\ref{thm:Helmtholz-invertible}]
Assume that $u \in H_{\sca}^{s, \mr_+}$ and $Pu = 0$.  Then by Proposition~\ref{prop:inf module reg}, $u$ has infinite module regularity. We may also assume that $s \geq 0$ without loss of generality, due to the ellipticity of $P$ at fibre-infinity. Thus we can apply Theorem~\ref{thm:leading-asymptotic-Helmholtz}, to deduce that $u$ has an expansion as in \eqref{eq:leading-asymptotic-Helmholtz}. Thanks to this expansion, we can apply the boundary pairing formula \eqref{eq:boudary-pairing} from Problem~\ref{prob:bpf} with $u_1 = u_2 = u$. We obtain 
$ 0 = 2i\lambda \| a_0 \|_{L^2}^2$, which implies that $a_0$ vanishes identically. However, this means that $u$ is also above threshold at $\Rout$. Repeating the argument in Proposition~\ref{prop:inf module reg} we see that in fact $u$ has no scattering wavefront set at $\Rout$, which means that the scattering wavefront set is empty, i.e. $u$ is Schwartz. 

Once we know that $Pu = 0$ and $u$ is Schwartz, then the triviality of $u$ follows from standard results in PDE, e.g. the 
uniqueness results  in \cite[Theorem~17.2.8]{hormander2007analysis} or the absence of $L^2$-eigenfunctions in \cite{froese1982absence}.
\end{proof}

The Theorem above also gives a characterization of being `above threshold' in terms of leading order expansion. Suppose $Pu \in \mathcal{S}(\R^n)$, and $u$ is above threshold at $\Rin$. Then, as we have just seen, by Proposition~\ref{prop:inf module reg}, $u$ has infinite module regularity and has wavefront set only at $\Rout$. Theorem~\ref{thm:leading-asymptotic-Helmholtz} applies to show that the limit 
\begin{equation}\label{eq: L limit}
L(\lambda)u = \lim_{r \to \infty} r^{(n-1)/2}e^{-i\lambda r} u
\end{equation}
exists and is some function $a_0 \in L^2(\mathbb{S}^{n-1})$ (actually it will be $C^\infty$ due to the infinite module regularity --- see Problem ****) playing the role of the asymptotic profile.
The operator $L(\lambda)$ is in fact (on a suitable domain, and modulo some constants and $\lambda$-factor) the adjoint of $P(-\lambda)$. On the other hand, being above threshold at $\Rout$ should mean that the limit above is $0$, after microlocalizing near $\mathcal{R}_+$, which discussion we postpone to after the microlocal decomposition of solutions.



One of the standard technique in scattering theory is to construct an approximate solution given a leading order asymptotic $a_0$ as in Theorem~\ref{thm:leading-asymptotic-Helmholtz}.  

\begin{proposition} \label{prop:approximate-sol}
Given $f_- \in C^\infty(\partial \overline{\R^n})$, then one can construct $\tilde{u}$ such that it takes $r^{-(n-1)/2}e^{i\lambda r}f_-(y)$ as the leading term and $(\Delta_g-\lambda)\tilde{u} \in \mathcal{S}(\R^n)$. Moreover, $\tilde u \in \mathcal{X}^{s, \mr_-}$ for any admissible $\mr_-$. 
\end{proposition}

\begin{proof}

Using \eqref{eq:Laplacian-r-coordinate}, a direct computation shows that $r^{-(n-1)/2}e^{-i\lambda r} f_-(y) $ solves the PDE to the leading and subleading order:
\begin{align*}
(\Delta_g-\lambda^2)( r^{-(n-1)/2}e^{-i\lambda r} f_-(y) ) = r^{-(n+3)/2}e^{-i\lambda r}g(y)+O(r^{-(n+5)/2}),
\end{align*}
for some $g(y) \in C^\infty(\partial \overline{\R^n})$. All $O(\cdot)$ notation in this part means the uniform control for fixed $C^k$-norm or some Sobolev norm in $y$ as $r\to\infty$ (or ${\msf{x}} \to 0$ below).
In fact, more generally, in terms of $({\msf{x}},y)$-coordinates near $\partial \overline{\R^n}$ with ${\msf{x}}=r^{-1}$, we have
\begin{align}\label{eq:formal iteration step}
(\Delta_g-\lambda^2)({\msf{x}}^p e^{-i\lambda/{\msf{x}}}a(y)) = i\lambda (2p-n+1){\msf{x}}^{(p+1)}e^{-i\lambda/{\msf{x}}}a(y)+O({\msf{x}}^{p+2}).
\end{align}
Notice that the first term on the RHS of \eqref{eq:formal iteration step} vanishes if and only if $p =  (n-1)/2$. 
Using this iteratively, starting with the power $p = (n-1)/2$ (that is, one prescribes the next term so that its leading output cancels the error from the previous step), one can prescribe a formal series 
\begin{align}\label{eq:incoming exp}
\tilde{u} = {\msf{x}}^{(n-1)/2}e^{-i\lambda/{\msf{x}}} \sum_{j=0}^\infty {\msf{x}}^j a_j(y)
\end{align}
with $a_0=f_-$ formally solving the equation $(\Delta_g-\lambda^2)\tilde{u} = 0$. Borel summing the series gives us an honest function $\tilde u$ satisfying  
\begin{align*}
(\Delta_g-\lambda^2)\tilde{u} \in \mathcal{S}(\R^n)
\end{align*}
and with the prescribed leading asymptotic $f_-$. 

To show that $\tilde u \in \mathcal{X}^{s, \mr_-}$, we first note that since $Pu \in \Schw$, the wavefront set of $\tilde u$ is contained in $\Char(P)$. Second, we note that $\tilde u = O(\ang{r}^{-(n-1)/2})$, so $u \in H_{\sca}^{s, -1/2 - \epsilon}$ for any $\epsilon > 0$. Moreover,  $(D_r + \lambda)^j u \in H_{\sca}^{s, -1/2 - \epsilon + j}$ for any $\epsilon > 0$, so $u$ is itself in this space microlocally on the elliptic set of $(D_r + \lambda)^j$. On $\Char(P)$, this includes everything except for $\Rin$. Since $j$ is arbitrary, it follows that $\tilde u$ has wavefront set only at $\Rin$. Combining these facts shows that $\tilde u \in \mathcal{X}^{s, \mr_-}$ for any admissible $\mr_-$ (since by definition of admissibility, $\mr_- < -1/2$ at $\Rin$, and its values elsewhere are irrelevant). 

\end{proof}

Now we turn to construct the actual solution with prescribed asymptotic data.
Recall that  from Theorem~\ref{thm:Helmtholz-invertible} we have bounded inverse $P_\pm^{-1}$ of
\begin{equation}
P = (\Delta_g-\lambda^2):\; \mathcal{X}^{s,\mathsf{r}_\pm}
\to H_{\sct}^{s-2,\mathsf{r}_\pm+1},
\end{equation}
which will be called the outgoing ($+$) and incoming ($-$) resolvent. And indeed, they are just classical outgoing/incoming resolvents that are denoted by $R(\lambda \pm i0)$.

According to Proposition~\ref{prop:approximate-sol} above, we have $\tilde{u}$ with prescribed data near $\mathcal{R}_-$.
Now we set
\begin{align*}
u = \tilde{u} - P_+^{-1}\big( (\Delta_g-\lambda^2)\tilde{u} \big)
\end{align*}
and this is an \emph{exact} solution to $Pu=0$. Notice that the first term is in $\mathcal{X}^{s, \mr_-}$ and the second term is in $\mathcal{X}^{s, \mr_+}$. Moreover it is nontrivial provided that $f_- \neq 0$, since the first term is below threshold at $\Rin$ while the second term is above threshold at $\Rin$, so the sum is below threshold at $\Rin$. Indeed, the second term, being above threshold at $\Rin$, and therefore having no wavefront set there as in the proof of Proposition~\ref{prop:inf module reg}, does not affect the asymptotic expansion \eqref{eq:incoming exp} in any way; it merely contributes a corresponding expansion with the `other' oscillation $e^{+i\lambda r}$ instead of $e^{-i\lambda r}$. 

So the $u$ constructed above indeed is a solution in $\mathcal{X}^{s,\mathsf{r}_+}+\mathcal{X}^{s,\mathsf{r}_-}$ with prescribed incoming data $f_-=a_0$ and we define
\begin{equation}
\Poi(\lambda)f_- = u.
\end{equation}
This is well-defined since such solution is unique. Because if there are two such solutions $u_1$ and $u_2$, then they we know that $u=u_1-u_2$ solves the equation with zero incoming data. Using the boundary pairing in \eqref{eq:boudary-pairing} again, we know that the outgoing data of $u$ vanishes as well (of course it should, since energy should be conserved). Then we are again in a situation where it is above threshold at both radial set, applying the proof in Theorem~\ref{thm:Helmtholz-invertible} concludes that $u \equiv 0$.

There is another way of decomposing solutions of $Pu = 0$ which is more general and conceptually more satisfying. To do this, let $\mathrm{Id}=A_++A_-$ be a microlocal partition of unity (that is, a quantization of a partition of unity on ${}^{\sct}T^*\R^n$) such that $A_\pm$ is micro-supported away from $\mathcal{R}_\mp$, which is equivalent to requiring $A_\pm$ to be microlocally equal to the identity near $\mathcal{R}_\pm$.

\begin{proposition} \label{prop:u-microlocal-decomposition}
Let $u \in \mathcal{X}^{s, \mr_-} + \mathcal{X}^{s, \mr_+}$ satisfy $Pu = 0$  and let $A_\pm$ be as above. Then we have a decomposition:
\begin{equation}
u = u_+ + u_-,
\end{equation}
where 
\begin{equation}
u_\pm = A_\pm u = P_\pm^{-1}[P,\;A_\pm]u.
\end{equation}
\end{proposition}
\begin{remark}
See \cite[Proposition~3.3]{Helmholtz-2} for the same statement with sharp module regularity.
\end{remark}

\begin{proof}
First of all, we have $[P,A_+] = -[P,A_-]$ since $[P,A_++A_-] = [P,\mathrm{Id}]=0$.
Since $Pu=0$, we have
\begin{align}
\begin{split}
& P_+^{-1}[P,A_+]u + P_-^{-1}[P,A_-]u
\\ = & P_+^{-1}PA_+u + P_-^{-1}PA_-u,
\end{split}
\end{align}
If we know $A_\pm u \in \mathcal{X}^{s,\tilde{\mathsf{r}}_\pm}$ for some $s$, and some Feynman/anti-Feynman admissible orders $\tilde{\mathsf{r}}_\pm$ then we know that $P_\pm^{-1}PA_\pm u = A_\pm u$ and the proof is finished.
To ensure that the statement $A_\pm u \in \mathcal{X}^{s,\tilde{\mathsf{r}}_\pm}$ is true, we take $\tilde{\mathsf{r}}_\pm$ to be an arbitrary constant greater than $-1/2$ in a small neighbourhood of the above-threshold radial point set (the value is irrelevant since $A_\pm u$ is microlocally trivial there), and equal to $\min (\mr_+, \mr_-)$ outside a slightly larger  neighbourhood of the above-threshold radial set. 
\end{proof}

Now we can come back to the limiting process we discussed earlier in \eqref{eq: L limit}. Let $u = |poi(\lambda) f_-$, where $f_- \in C^\infty(\mathbb{S}^{n-1})$.  In terms of the decomposition of Proposition~\ref{prop:u-microlocal-decomposition},  $u_-$ contributes the $f_-$-term in the expansion in \eqref{eq:typical-expansion-Helmholtz} while $u_+$ contributes the $f_+$-term in the expansion.
In terms of $L(\lambda)$ (and its analogue $L(-\lambda)$ defined via shifting the sign of the complex phase), we have
\begin{equation} \label{eq:L(lambda)-convergence}
L(\lambda)u_+ = f_+,\quad L(\lambda)u_- = 0, \quad L(-\lambda)u_+ = 0, \quad L(-\lambda)u_- = f_-.
\end{equation}
One needs to interpret these limits in certain $L^2$-based Sobolev spaces.

\subsection{The scattering matrix}
Let $u$ be a solution to $Pu=0$ as in \eqref{eq:typical-expansion-Helmholtz}, then the scattering matrix $S(\lambda)$ is the (extension of the) map $C^\infty(\mathbb{S}^{n-1}) \to C^\infty(\mathbb{S}^{n-1})$ sending $f_-$ to $f_+$.
It can be interpreted as the `boundary value' of our Poisson operator\footnote{This perspective was used in \cite{melrose1996scattering}, where $\Poi(\lambda)$ was constructed as certain Legendre distribution and $S(\lambda)$ was shown to be a Fourier integral operator of order $0$  associated to the time $\pi$ geodesic flow on the boundary.}.

Apart from this perspective, $S(\lambda)$ has a simple formula in terms of objects we introduced before:
\begin{equation} \label{eq:S-matrix-formula}
S(\lambda) = (2i\lambda)^{-1}\Poi(-\lambda)^* [P,A_+] \Poi(\lambda),
\end{equation}
where $\Poi(\pm \lambda)^*$ is the adjoint of $\Poi(\pm \lambda)$, which is we will analyze below.
The goal of this part is understanding this, and other identities connecting various objects that we have introduced.

We first give a characterization of $\Poi(\lambda)^*$ via the boundary pairing. Let $v \in \mathcal{S}(\R^n)$ and set $u_\pm = \pm P_\pm^{-1} v$, then we apply \eqref{eq:boudary-pairing} to $u_-$ and $P(\lambda)a$ with $a \in C^\infty(\partial \overline{\R^n})$ to get
\begin{align*}
\int (\Poi(\lambda)a) \overline{v} dV_g = 2i\lambda \int_{\partial \overline{\R^n}} a \overline{f_-} dV_{\partial \overline{\R^n}},
\end{align*}
where we have used the fact that $u_-$ has zero outgoing data and $f_-$ is the incoming data of $u_-$.
This gives
\begin{align}
\Poi(\lambda)^* v = 2i\lambda f_-.
\end{align}
Similarly, let $f_+$ be the incoming data of $u_+$, then we have
\begin{align}
\Poi(-\lambda)^* v = 2i\lambda f_+.
\end{align}

Applying $\Poi(\pm \lambda)$ to $f_\mp$, we obtain a function that solves $Pu=0$ with prescribed incoming/outgoing data equal to $f_\mp$, which is therefore $u_-+u_+$. So we have
\begin{align*}
\Poi(\pm \lambda)\Poi(\pm \lambda)^* = 2i \lambda  (P_+^{-1} - P_-^{-1}),
\end{align*}
which connects the Poisson operator to the spectral measure as given by Stone's theorem (the right hand side can be interpreted as the difference between the outgoing and incoming resolvents at the point $\lambda^2$ of the spectrum of $\Delta_g$).

Now for $u = \Poi(\lambda)f_-$, we use $v = [P,A_+]u$, which is Schwartz, then $u=u_-+u_+$ in terms of notation, which also coincides with the microlocal decomposition we had in Proposition~\ref{prop:u-microlocal-decomposition}. Then we can write
\begin{align*}
u = (2i\lambda) \Poi(-\lambda)\Poi(-\lambda)^*[P,A_+]u = (2i\lambda) \Poi(-\lambda)\Poi(-\lambda)^*[P,A_+] \Poi(\lambda)f_-.
\end{align*}
Writing it as
\begin{align*}
u = \Poi(-\lambda)\big( (2i\lambda)  \Poi(-\lambda)^*[P,A_+] \Poi(\lambda)f_-  \big),
\end{align*}
we know (since $u = \Poi(-\lambda) f_+$ and  by the uniqueness of the outgoing data)
\begin{align*}
f_+ = (2i\lambda)  \Poi(-\lambda)^*[P,A_+] \Poi(\lambda)f_- ,
\end{align*}
which proves \eqref{eq:S-matrix-formula}.

Applying the boundary pairing \eqref{eq:boudary-pairing} again, we see
\begin{align*}
\int_{\mathbb{S}^{n-1}} S(\lambda)f \overline{S(\lambda)f}  \ \mathrm{dvol}_{\mathbb{S}^{n-1}} - \int_{\mathbb{S}^{n-1}} f \overline{f} \  \mathrm{dvol}_{\mathbb{S}^{n-1}} = 0,
\end{align*}
which shows that $S(\lambda)$ is a unitary operator.

\subsection{Problems}

\begin{problem} Verify \eqref{eq:Laplacian-approximate}.  \label{ex:Laplacian-approximate}
\end{problem}

\begin{problem} The Huygens principle: Write down the spectral measure of $\Delta_0$ (or more precisely, of $\sqrt{\Delta_0}$), which in turn gives the propagator $e^{it\sqrt{\Delta_0}}$. What can you say about the wave front set of solutions to (say, initial value problems of) wave equation?
(Hint: You can use Fourier transform to derive the spectral measure, which ``diagonalizes'' the differentiation.)
\\(Hint 2/Solution: Verify that $\big( \frac{\lambda^{n-1}}{(2\pi)^n} \int_{\mathbb{S}^{n-1}} e^{i(z-z') \cdot \lambda \omega} d\omega \big) d\lambda$ is the spectral measure.)
\end{problem}

\begin{problem} (The free scattering theory) 
Write down the Poisson operator.
Also, write down the scattering operator for $\Delta_0-\lambda^2$ and convince yourself that this is a Fourier integral operator associated to the time $\pi$ geodesic flow. 
(Hint/Solution: the Poisson operator is a constant multiple of $\lambda^{(n-1)/2}e^{i\lambda z \cdot y}$, where we use $y \in \mathbb{S}^{n-1} \hookrightarrow \R^n$ to write its coordinates.)
\end{problem}

\begin{problem} According to the expression of $P(\lambda)$ for the free case and assume similar expression holds for in more general settings, what happens if we want to study the low energy behavior as $\lambda \to 0$? What kind of resolution is needed?
\end{problem}

\begin{problem} Use the expression for Poisson operator above to give another derivation of the expression for the spectral measure.
\end{problem}

\begin{problem} Give a example of topology in which the convergence in \eqref{eq:L(lambda)-convergence} will not hold in general.
\end{problem}

\begin{problem} (Scattering theory in one dimension) Consider the time-independent Schr\"odinger equation $(D_x^2+V-\lambda^2)\psi = 0$ in one dimension with real valued smooth $V$.
Suppose we have $\psi \sim e^{i\lambda x}+\mathbf{r}e^{-i\lambda x}$ as $x \to -\infty$ while $\psi \sim \mathbf{t}e^{ikx}$ as $x \to \infty$. Prove:
\begin{equation}
|\mathbf{r}|^2+|\mathbf{t}|^2=1.
\end{equation}
Here $\mathbf{r}$ is called the reflection coefficient and $\mathbf{t}$ is called the transmission coefficient.
(Hint: consider the current, or Wronskian in terms of the ODE theory, $J = \overline{\psi}\psi'-\psi\overline{\psi'}$, then $D_xJ=0$.)
Why this is our unitary property of $S(\lambda)$? (Be careful that $r=-x$ as $x \to -\infty$!)
\end{problem}

\begin{problem} (Potential perturbations) Suppose we consider $P=\Delta_g+V-\lambda^2$ instead, with $V \in \la r \ra^{-2}C^\infty(\overline{\R^n})$, verify main results in this section. (Impose extra decay conditions if you face some troubles.)
\end{problem}

\section{Lecture 14: The Klein--Gordon equation}

So far, we have taken the Helmholtz operator on $\mathbb{R}^n$ as our main example. We have stressed that the arguments given only depend on the principal part of the operator, so the same results apply in asymptotically Euclidean settings, or when a potential is present. Nevertheless, all of the operators considered are asymptotically Helmholtz. 

In this lecture we will consider a very different application, the Klein--Gordon equation $(\square+1) u =0$ on $\mathbb{R}^{1+n}_{t,x}$, where $\square = D_t^2-\triangle = -\partial_t^2 + \sum_{j=1}^n  \partial_{x_j}^2$ is the d'Alembertian, a.k.a ``wave operator''. (In some sources $\square$ denotes the negative of this operator; there is no consensus in the literature on this point.)

\begin{remark}
    The Klein--Gordon equation was introduced by Klein, Fock, and Gordon (in that order) as a model for relativistic quantum mechanics. In nature, it describes certain atomic nuclei, such as Helium-4, as well as pions, certain mesons which act as a carrier of the strong nuclear force between nucleons. It also serves as a model for the Higgs boson in the limit where its self-interactions are ignored.
\end{remark}

Notice that the operator $P=\square-1$ lies in $\Psi_{\mathrm{sc}}^{2,0}$, like the Helmholtz operator. Like the Helmholtz operator, we will see that the Hamiltonian vector field $\mathsf{H}_p$ 
coming from $p=\sigma_{\mathrm{sc}}^{2,0}(P)$ induces on the characteristic set a source-to-sink flow.
\emph{Unlike} the Helmholtz operator, $P$ is not elliptic at fiber infinity, which causes some issues. But these issues are (for the most part) minor, and the Helmholtz story applies in the expected way, with the expected modifications.

To simplify our discussion, we will work in the $n=1$ case. This means working on two-dimensional spacetime $\mathbb{R}^{2}_{t,x}$.

\begin{remark}
    As in our discussion of the Helmholtz operator, we will only discuss the case of $P=\square-1$, the exact Klein--Gordon operator, but the arguments use only the principal form of $P$ (and dynamical things like global hyperbolicity), so they apply for more general metrics that are \emph{asymptotically Minkowski} or for operators $P$ which differ from the Klein--Gordon operator $\square_g-1$ of an asymptotically Minkowski metric $g$ by lower-order terms.
\end{remark}

\subsection{The dynamics}

Let $\tau$ denote the frequency coordinate dual to $t$ and $\xi$ the frequency coordinate dual to $x$.

The operator $P$ is classical, so its characteristic set is the subset of the boundary $\partial (\overline{\mathbb{R}^2}\times \overline{\mathbb{R}^2} )$ of compactified phase space on which 
\begin{equation}
    \mathsf{p} = \frac{\tau^2-\xi^2 - 1 }{ \tau^2+\xi^2+1}
\end{equation}
vanishes.  

Over spatial infinity, at finite frequency, this consists of the hyperbola $\{\tau^2  = \xi^2+1\}$. At infinite frequency, this consists of the ``endpoints'' of the hyperbola. 

The characteristic set $\Sigma_{\mathrm{KG}}=\operatorname{Char}_{\mathrm{sc}}^{2,0}(P)$ has two components, one in the $\{\tau>0\}$ half of the phase space and one in the other half. The situations in the two halves are comparable, so we will restrict attention to the $\{\tau>0\}$ half. 
We can picture 
\begin{equation}
    \operatorname{Char}_{\mathrm{sc}}^{2,0}(P) \subset \partial (\overline{\mathbb{R}^2}\times \overline{\mathbb{R}^2} )
\end{equation}
as the boundary of a cylinder. The side of the cylinder consists of all of the copies of the hyperbola $\{\tau^2 +1 = \xi^2\}$ over spacetime infinity $\infty \mathbb{S}^1$, one hyperbola for each point on $\infty \mathbb{S}^1$. The top and bottom of the cylinder are the portion of the characteristic set at fiber infinity. One of the two, say the top, contains left-moving wavefront set, and the other, the bottom, contains right-moving wavefront set.

Unfortunately, it is hard to draw a cylinder. So, instead, we puncture the cylinder at its base and then splay out its sides, turning the base inside-out in the process, and then laying the whole thing flat. The result is a large flat disk, whose center is the disk that formed the original top, surrounded by an annulus which consists of the side of the cylinder, surrounded by another annulus consisting of the punctured, stretched, and flipped base.

\begin{center}
    \begin{tikzpicture} 
    \filldraw[fill=lightgray!50] (0,-50pt) ellipse (50pt and 10pt);
    \fill[fill=lightgray!50] (-50pt, 0) -- (50pt,0) -- (50pt,-50pt) -- (-50pt,-50pt) -- cycle;
    \draw (50pt,0) -- (50pt,-50pt);
    \draw (-50pt,0) -- (-50pt,-50pt);
    \filldraw[fill=lightgray!50] (0,0) ellipse (50pt and 10pt);
    \draw[->]  (0,-90pt) node[below] {Puncture} -- (0,-65pt);
\end{tikzpicture}
\begin{tikzpicture} 
    \filldraw[fill=lightgray!50] (-50pt,-50pt) -- (-30pt,-75pt) to[out=-30, in=210] (30pt,-75pt) -- (50pt,-50pt) -- cycle;
    \filldraw[fill=lightgray!50] (0,-50pt) ellipse (50pt and 10pt);
    \fill[fill=lightgray!50] (-50pt, 0) -- (50pt,0) -- (50pt,-50pt) -- (-50pt,-50pt) -- cycle;
    \draw (50pt,0) -- (50pt,-50pt);
    \draw (-50pt,0) -- (-50pt,-50pt);
    \filldraw[fill=lightgray!50] (0,0) ellipse (50pt and 10pt);
    \draw[->]  (0,-110pt) node[below] {Stretch} -- (0,-90pt);
\end{tikzpicture}
\end{center}

\begin{center}
    
\begin{tikzpicture} 
    \filldraw[fill=lightgray!50] (-50pt,-50pt) -- (-70pt,-75pt) to[out=-10,in=190] (70pt,-75pt) -- (50pt,-50pt) -- cycle;
    \filldraw[fill=lightgray!50] (0,-50pt) ellipse (50pt and 10pt);
    \fill[fill=lightgray!50] (-50pt, 0) -- (50pt,0) -- (50pt,-50pt) -- (-50pt,-50pt) -- cycle;
    \draw (50pt,0) -- (50pt,-50pt);
    \draw (-50pt,0) -- (-50pt,-50pt);
    \filldraw[fill=lightgray!50] (0,0) ellipse (50pt and 10pt);
    \draw[->] (60pt,-30pt) -- (80pt,-30pt) node[right] {flatten and rotate};
\end{tikzpicture}
\end{center}
Final result:
\begin{center}
    \begin{tikzpicture}[scale=.5]
        \fill[fill=lightgray!50] (0,0) circle (120pt);
        \filldraw[fill=lightgray!50] (0,0) circle (90pt);
        \filldraw[fill=lightgray!50] (0,0) circle (50pt);
    \end{tikzpicture}
\end{center}

Our next goal is to prove (or sketch a proof of):
\begin{proposition}
The Hamiltonian vector field $\mathsf{H}_{p} = (1+x^2+t^2)^{1/2} (1+\tau^2+\xi^2))^{-1/2} H_{p}$ induces on $\Sigma_{\mathrm{KG}}$ a source-to-sink flow.
\end{proposition}  

First, let us compute the Hamiltonian vector field $H_{p}$ of $p = \tau^2 - \xi^2 - 1$. The Hamiltonian flow is 
\[
H_{p} = (\partial_\tau p) \partial_t + (\partial_\xi p) \partial_x =  2 \tau \partial_t - 2 \xi \partial_x. 
\]

The quickest way to find those points in  $\infty \mathbb{S}^1 \times \mathbb{R}^2$ where $\mathsf{H}_{p}$ vanishes is to note that it consists of points of the form $(\theta,\zeta)\in \infty \mathbb{S}^1 \times  \mathbb{R}^2$ such that, at finite spacetime points $(t,x) \neq (0,0)$ in spacetime direction $\theta = \operatorname{arctan}(t/x)$, the Hamiltonian vector field $H_{p}$ is parallel or anti-parallel to $\theta$. This is just translating between the compactified perspective and the conic perspective.\footnote{The reason this works is that $H_{p}$ is homogeneous of degree $-1$ with respect to spacetime dilations. [Exercise]}

The Hamiltonian vector field is pointed in the spacetime direction $(\tau,-\xi)$. This is parallel/anti-parallel to $(t,x)$ when $-\tau x = t \xi$. Now:
\begin{itemize}
    \item If $\theta=\operatorname{arctan}(t/x)$ is between $45^\circ$ and $135^\circ$, then $(\tau,-\xi)$ should be between those same angles, which picks out exactly one point in the hyperbola $\{\tau^2 = \xi^2+1\}$ (which defines the characteristic set).
    \item  Likewise if $\theta$ is between $225^\circ$ and $315^\circ$. 
    \item For other $\theta$, there are no points $(\tau,-\xi)$ in the characteristic set with that same angle.
\end{itemize}
So, the radial set $\mathcal{R}$ consists, in the component of the characteristic set to which we have been restricting attention, of two components, one over the future timelike cap, $\mathcal{R}_+$, and one over the back timelike cap, $\mathcal{R}_-$, each of which consists of one point over each point at spacetime infinity.

Since $H_{p}$ is flowing forwards in time on the $\{\tau>0\}$ component of the characteristic set to which we have been restricting attention, this is highly suggestive of a source-to-sink flow, with the radial set over the past timelike cap as a source and the radial set over the future timelike cap as a sink:
\begin{center}
    \includegraphics[scale=1]{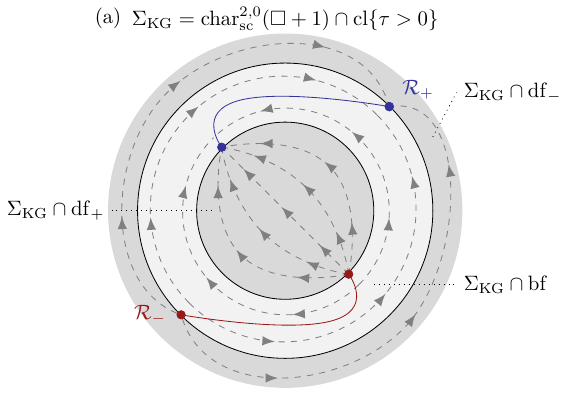}
\end{center}
This is in fact what the flow looks like. Let us verify that $\mathcal{R}$ is a sink, at least away from fiber infinity. 

\begin{remark}
    The figure above shows the Hamiltonian flow induced by $\mathsf{H}_{p}$ moving backwards in time along $\infty \mathbb{S}^{n}$. This might seem inconsistent with the fact that $H_{p}$ is going forwards in time. The reason is that even if $H_{p}$ is going forwards in time, if it is decreasing the spatial radius $r$ at a faster rate, then $H_{p}$ is bring us closer to the time-axis. This means that the induced flow $\mathsf{H}_{p}$ on the corresponding portion of the boundary is doing the same thing.  
\end{remark}

Let us use $\rho=1/t$ as a boundary-defining-function of base infinity (valid in a neighborhood of the subset of $\infty \mathbb{S}^1$ over which $\mathcal{R}_+$ lies). In addition, instead of using $x/t$ as a coordinate, let us use 
\[
v = \frac{x \tau}{t} + \xi .
\]
This vanishes precisely on $\mathcal{R}_+$. 
So, our coordinate system of choice is $(\rho=1/t, v, \tau,\xi)$. Above, we wrote $H_{p}$ in terms of the coordinate system $(t,x,\tau,\xi)$. 
In order to carry out the coordinate change from the latter coordinate system to the former, use 
\begin{align}
\begin{split} 
   \frac{\partial}{\partial t} &= \frac{\partial \rho}{\partial t}\frac{\partial}{\partial \rho} + \frac{\partial v}{\partial t}\frac{\partial}{\partial v} \\ 
    &= - \rho^2 \frac{\partial}{\partial \rho }  -  \rho (v-\xi) \frac{\partial}{\partial v} \\ 
    \frac{\partial}{\partial x} &= \frac{\partial v}{\partial x}\frac{\partial}{\partial v} = \rho \tau \frac{\partial}{\partial v}.
\end{split}
\end{align}
So, 
\begin{multline}
    H_{p} = 2\tau\partial_t - 2 \xi \partial_x  =  -2 \tau \Big(\rho^2 \frac{\partial}{\partial \rho }  +  \rho (v-\xi) \frac{\partial}{\partial v} \Big) - 2 \xi  \rho \tau \frac{\partial}{\partial v} \\ 
    = - 2 \tau \rho^2 \frac{\partial}{\partial \rho} - 2 \tau \rho v \frac{\partial}{\partial v}. 
\end{multline}
So: 
\begin{equation}
    \mathsf{H}_{p} = \rho^{-1} H_{p} = - 2 \tau \rho \frac{\partial}{\partial \rho}   - 2 \tau v \frac{\partial}{\partial v} .
\end{equation}
From this, we can see that $\{v=0,\rho=0\}$ is a sink on the component of the characteristic variety with $\tau>0$.

\subsection{Theorems}

Let $\mathsf{r}$ denote a variable order monotonic under the Hamiltonian flow (it is allowed for $\mathsf{r}$ to be decreasing on one component of the characteristic set and increasing on the other), constant near each radial set, and, within each component of the characteristic set, $\mathsf{r}<-1/2$ on one radial set and $\mathsf{r}>-1/2$ on the other. Now let 
\begin{align}
    \mathcal{X}^{s,\mathsf{r}} &= \{ u \in H_{\mathrm{sc}}^{s,\mathsf{r}} : Pu \in H_{\mathrm{sc}}^{m-1,\mathsf{s}+1}  \} , \\ 
    \mathcal{Y}^{s,\mathsf{r}} &= H_{\mathrm{sc}}^{s,\mathsf{r}}.
\end{align}
Then, the Klein--Gordon operator $P=\square-1$ satisfies:
\begin{theorem}
The map $P:\mathcal{X}^{s,\mathsf{r}} \to \mathcal{Y}^{m-1,\mathsf{s}+1}$ is Fredholm.  
\end{theorem} 

In this way we get \emph{four} Fredholm problems. The reason why we get four instead of two (like for the Helmholtz problem) is that the characteristic set has two components, and we can choose the propagation direction of each. 

If $\mathsf{s}$ is high in the past, then this is called the advanced problem. If $\mathsf{s}$ is high in the future, then this is called the retarded problem. The other two problems are called the Feynman and anti-Feynman problems.

\subsection{Exercises}

\begin{problem}
    Show that $\mathcal{R}_+$ continues to be a sink at fiber infinity. Do this by: 
    \begin{enumerate}[label=(\alph*)] 
        \item Change coordinates from $(\rho,v,\tau,\xi)$ to $(\rho,\tilde{v}, 1/\tau,\hat{\xi})$ for $\hat{\xi} = \xi/\tau$ and $\tilde{v} = x/t+\hat{\xi}$. 
        \item Why does this coordinate chart suffice?
    \end{enumerate}
\end{problem}

For the next two problems, you can use the Fourier transform if you like.
\begin{problem}
    Suppose that $\mathsf{s}$ is increasing on both components of the characteristic set. (Or decreasing on both.) Show that  $P:\mathcal{X}^{s,\mathsf{r}} \to \mathcal{Y}^{m-1,\mathsf{s}+1}$ is invertible. 
\end{problem}

\begin{problem}
    Suppose that $\mathsf{s}$ is increasing on one component of the characteristic set and decreasing on the other. Show that  $P:\mathcal{X}^{s,\mathsf{r}} \to \mathcal{Y}^{m-1,\mathsf{s}+1}$ is invertible. 
\end{problem}

\section{Lecture 15:  Another way to realize Fredholmness, and the  Schr\"odinger operator}

\subsection{Another way to realize Fredholmness} There are a variety of ways to construct of Fredholm realization of (pseudo-)differential operators. One relatively nontechnical way was introduced recently by Baskin, Doll and Gell-Redman \cite{BDGR}. It uses constant-order Sobolev spaces, with an additional condition at the above-threshold radial set. We will discuss this just in the context of the Helmholtz operator, denoted $P$, here. 

To define these spaces, let us choose $\Rin$ to be the above-threshold radial set (of course we can equally well choose $\Rout$, and it should be clear how to modify the construction for this choice). We then choose two microlocal cutoffs $Q, Q' \in \Psisc^{0,0}$, which are both microlocally equal to the identity near $\Rin$ (in particular, they are elliptic near $\Rin$), and microlocally trivial near $\Rout$ in the sense that their operator wavefront set is disjoint from $\Rout$. In addition, we ask that $\WFsc'(Q) \subset \Ell(Q')$. We then choose a differential order $s$ and two constant spatial orders $l_-, l_+$ satisfying $l_- < -1/2 < l_+ < l_- + 1$.  We then define our $\SX$ and $\SY$ spaces by 
\begin{equation}\begin{gathered}
    \SY^{s, l_-, l_+} = \{ u \in H^{s, l_-} \mid Q'u \in H^{s, l_+} \}, \\
    \| u \|_{\SY^{s, l_-, l_+}} := \| u \|_{H^{s, l_-}} + \| Q'u \|_{H^{s, l_+}}
\end{gathered}\end{equation}
and 
\begin{equation}\begin{gathered}
    \SX^{s, l_-, l_+} = \{ u \in H^{s, l_-} \mid Qu \in H^{s, l_+}, Pu \in \SY^{s-2, l_-+1, l_+ + 1} \}, \\
    \| u \|_{\SX^{s, l_-, l_+}} := \| u \|_{H^{s, l_-}} + \| Qu \|_{H^{s, l_+}} + \| Pu \|_{H^{s-2, l_-+1}} + \| Q'Pu \|_{H^{s-2, l_++1}}. 
\end{gathered}\end{equation}
We see that if $u \in \SX^{s, l_-, l_+}$ then $u$ has above-threshold decay at $\Rin$ but can be below threshold at $\Rout$. Moreover, $Pu$ satisfies similar conditions, with the expected change of order to $(s-2, l_\pm + 1)$, but using the $Q'$ cutoff rather than the $Q$ cutoff for $u$. 

\begin{proposition}\label{prop:BDGR} The operator $P : \SX^{s, l_-, l_+} \to \SY^{s-2, l_-+1, l_++1}$ is Fredholm. 
\end{proposition}

\begin{proof}
We first prove a semi-Fredholm estimate for $P$, i.e. we show 
\begin{equation}
    \| u \|_{\SX^{s, l_-, l_+}} \leq C \Big( \| Pu \|_{\SY^{s-2, l_-+1, l_++1}} + \| u \|_{H^{-N, -N}} \Big). 
\end{equation}
It is enough to show that
\begin{equation}
    \| u \|_{H^{s, l_-}} + \| Qu \|_{H^{s, l_+}} \leq C \Big( \| Pu \|_{H^{s-2, l_-+1}} + \| Q'Pu \|_{H^{s-2, l_++1}} \Big). 
\end{equation}
In fact, since $\| Qu \|_{H^{s, l_-}}$ is clearly bounded by $\| Qu \|_{H^{s, l_+}}$, it is enough to show 
\begin{equation}\label{eq:wtsBDGR}
    \| (\Id - Q) u \|_{H^{s, l_-}} + \| Qu \|_{H^{s, l_+}} \leq C \Big( \| Pu \|_{H^{s-2, l_-+1}} + \| Q'Pu \|_{H^{s-2, l_++1}} \Big). 
\end{equation}
We start with the above threshold radial point estimate, which allows us to estimate $\| Qu \|_{H^{s, l_+}}$ by 
\begin{equation}
\| Qu \|_{H^{s, l_+}} \leq C \Big( \| Q'Pu \|_{H^{s-2, l_++1}} + \| Eu \|_{H^{s', l'}} + \| u \|_{H^{-N, -N}} \Big),
\end{equation}
where $-1/2 < l' < l_+$ and $s' < s$. The $E$ operator initially would have (from the proof of the above threshold radial point estimate) have operator wavefront set a bit bigger than that of $Q$. However, using the propagation estimate, we can estimate $Eu$ by $\tilde E u$ where $\tilde E$ is elliptic at $\Rin$ but can have operator wavefront set as close as desired to $\Rin$:
\begin{equation}
\| Eu \|_{H^{s', l'}} \leq C \Big( \| \tilde Eu \|_{H^{s', l'}} + \| Q'Pu \|_{H^{s'-2, l'+1}} + \| u \|_{H^{-N, -N}} \Big),
\end{equation}
and the norm of $Q' Pu$ here is weaker than the one in \eqref{eq:wtsBDGR}, so it can be absorbed. Then the $\| \tilde Eu \|_{H^{s', l'}}$ norm can be absorbed in the LHS, assuming as we may that $\WFsc'(\tilde E) \subset \Ell(Q)$, up to another $\| u \|_{H^{-N, -N}}$ term, as explained in previous lectures. We thus obtain, for $u \in \SX^{s, l_-, l_+}$, the estimate  
\begin{equation}\label{eq:term2}
    \| Qu \|_{H^{s, l_+}} \leq C \Big( \| Q'Pu \|_{H^{s-2, l_++1}} + \| u \|_{H^{-N, -N}} \Big). 
\end{equation}
Note that this estimate does \emph{not} hold in the strong sense: we need to know a priori that $u \in \SX^{s, l_-, l_+}$, which means in particular that both the LHS and RHS are finite. (Exercise: show that for a general distribution $u$, the RHS may be finite and the LHS infinite.) 

We have therefore estimated the second term in \eqref{eq:wtsBDGR}. To estimate the first term, we use the below threshold radial point estimate combined with the propagation estimate. It shows that 
\begin{equation}
    \| (\Id - Q) u \|_{H^{s, l_-}} \leq C \Big( \| Qu \|_{H^{s, l_-}} + \| Pu \|_{H^{s-2, l_-+1}} + \| u \|_{H^{-N, -N}}  \Big). 
\end{equation}
Of course we can trivially weaken this by taking a stronger norm on $Qu$ on the RHS:  
  \begin{equation}\label{eq:term1}
    \| (\Id - Q) u \|_{H^{s, l_-}} \leq C \Big( \| Qu \|_{H^{s, l_+}} + \| Pu \|_{H^{s-2, l_-+1}} + \| u \|_{H^{-N, -N}}  \Big). 
\end{equation}  
Now we add together \eqref{eq:term1} and \eqref{eq:term2} and use \eqref{eq:term2} again on the RHS to eliminate the $Qu$ term on the RHS, and arrive at \eqref{eq:wtsBDGR}. 

\vskip 8pt

The space $\SX^{s, l_-, l_+}$ is compactly embedded in $H^{-N, -N}$, provided that $N$ is sufficiently large. This shows that $P$, acting as in Proposition~\ref{prop:BDGR}, has a finite dimensional kernel and closed range. 
We need to show that the range has finite codimension. It suffices to show that the annihilator of the range of $P$ is finite dimensional. The annihilator lies in the dual space of $\SY^{s-2, l_-+1, l_+ + 1}$, so let's pause to consider what this dual space actually is. 

Note that $\SY^{s-2, l_-+1, l_+ + 1}$ is an intersection of spaces, that is, for $v \in \SY^{s-2, l_-+1, l_+ + 1}$ we require both $v \in H^{s-2, l_- + 1}$ and $Q' v \in H^{s-2, l_+ + 1}$. The dual space is therefore a sum: it consists of $f = f_1 + Q'f_2$ where $f_1 \in H^{2-s, -l_- - 1}$ and $f_2 \in H^{2-s, -l_+ -1}$ (exercise). It will be more transparent if we assume that $l_+ + l_- = -1$, that is, these two orders are symmetric around $-1/2$. Then $-l_- - 1 = l_+$ and $-l_+ - 1 = l_-$. So we require that $f = f_1 + Q'f_2$ where $f_1 \in H^{2-s, l_+}$ and $f_2 \in H^{2-s, l_-}$. This space of functions is similar to $\SX^{s, l_-, l_+}$ but with different microlocalizing functions, and more importantly, with the above and below threshold radial sets reversed: it is now below threshold at $\Rin$ and above threshold at $\Rout$. 

Moreover, for $f$ to be in the annihilator, then it must satisfy 
$$
\ang{P\phi, f} = 0 \text{ for all } \phi \in \Schw(\R^n). 
$$
For Schwartz $\phi$ there is no impediment to `integrating by parts', so this implies that $\ang{\phi, P^* f} = 0$ for all $\phi \in \Schw(\R^n)$, which in turn implies that $P^* f = 0$. So, the annihilator is identified with the null space of $P^* = P$ on a space similar to $\SX^{s, l_-, l_+}$ but with the threshold condition reversed. Exactly the same semi-Fredholm estimate as above can be made for this space and shows the finite dimensionality of the annihilator. Therefore, $P$ is Fredholm acting between the spaces in Proposition~\ref{prop:BDGR}. 
\end{proof}

\subsection{Parabolic calculus}

Finally we discuss how the ``Fredholm philosophy" can be applied to the Schr\"odinger operator \cite{gell2022propagation}. The free Schr\"odinger operator is
$$
P_0 = D_t + \Delta \text{ on } \R^{n+1}_{t,x}, \quad D_t = -i \partial_t, 
$$
and is the basic operator of nonrelativistic quantum mechanics. We consider perturbations $P$ given by 
\begin{equation}
P =     D_t + \Delta_{g(t)} + V(t, x), 
\end{equation}
where $g(t)$ is a family of metrics on $\R^n$, for simplicity a compactly supported perturbation of the flat metric, and such that $g(t)$ is flat for  $|t|$ large (this strong assumption can certainly be weakened to a finite decay rate at spacetime infinity). Similarly, we assume that the potential $V$ is smooth, real-valued, and compactly supported in spacetime (similarly, this could be weakened to a decay rate). Under our assumptions, $P = P_0$ near spacetime infinity.

We are going to consider a different compactification to our standard radial compactification. Why? -- because we need to respect the scaling of the operator, in which one time derivative is ``worth" two spatial derivatives. Note that the symbol of $P_0$ is $\tau + |\xi|^2$. We consider a parabolic scaling in the cotangent fibres:
$$
(\tau, \xi) \mapsto (a^2 \tau, a \xi), \quad a > 0.
$$
We call this map the `parabolic dilation with scale $a$'. 
We form a compactification of $\R^{n+1}_{\tau, \xi}$ in which the points at infinity correspond to all the orbits under parabolic dilation, in the same way that the points at infinity under \emph{radial} dilation correspond to all the orbits under the standard, linear dilation. To do this formally, we define the `parabolic sphere' in $\R^{n+1}$:
\begin{equation}\label{eq:parsphere}
    S^n_{\parb} = \{ (\tau, \xi) \mid \tau^2 + |\xi|^4 = 1 \}, 
\end{equation}
thus the parabolic sphere of scale $R$ is the set $\{ (\tau, \xi) \mid \tau^2 + |\xi|^4 = R^4 \}$. We then define the parabolic compactification of $\R^{n+1}_{\tau, \xi}$ to be 
\begin{equation}\begin{gathered}
 \overline{\R^{n+1}_{\parb}} :=    \R^{n+1}_{\tau, \xi} \sqcup S^n_{\parb} \times  [0, 1)_s / \sim ,  \\
 (\tau, \xi) \sim (\omega, s) = ((\omega_\tau, \omega_\xi), s) \iff (\tau, \xi) = (s^{-2}\omega_\tau, s^{-1} \omega_\xi).
\end{gathered}\end{equation}
The boundary of the parabolic compactification is at $s=0$, and $s = (\tau^2 + |\xi|^4)^{-1/4}$ is a boundary defining function. Let $\aang{R} : = (1 + \tau^2 + |\xi|^4)^{1/4}$ (the analogue of $\ang{\xi}$ in the linear case).  We then define classical symbols of order $s$ in $(\tau, \xi)$ to be smooth functions $a$  in the interior of $\overline{\R^{n+1}_{\parb}}$ such that $\aang{R}^{-s} a$ is $C^\infty$ on $\overline{\R^{n+1}_{\parb}}$. The corresponding definition of non-classical symbol of order $s$ is a function $a(\tau, \xi)$ such that, for all $j$ and all multi-indices $\beta$ of length $n$, we have 
\begin{equation}\label{eq:parsymbols}
  \Big|   D_\tau^j D_\xi^\beta a(\tau, \xi) \Big| \leq C \aang{R}^{s - 2j - |\beta|}. 
\end{equation}
Notice that $\xi_j$ is a symbol of order $1$, but $\tau$ is a symbol of order two. That means that $D_{x_j}$ will be an operator of differential order $1$ in the parabolic calculus, but $D_t$ will be an operator of order two. That is, $D_t$ and $\Delta$ have been brought to the same order, which fits with the intuitive idea that these two components of $P$ have `the same strength'. Notice also that $\tau$-derivatives in \eqref{eq:parsymbols} reduce the order of a symbol by two, which is again consistent (and in fact required) by the parabolic scaling. 

Our compactification of phase space $\comparphase$ is defined to be the Cartesian product of the radial compactification of spacetime, $\overline{\R^{n+1}_{t, x}}$, with $\overline{\R^{n+1}_{\parb}}$, and symbols of order $(s, l)$ will be functions $a(t, x, \tau, \xi)$ such that 
\begin{equation}
  \Big| D_{x, t}^\alpha  D_\tau^j D_\xi^\beta a(t, x, \tau, \xi) \Big| \leq C \ang{x, t}^{l-|\alpha|} \aang{R}^{s - 2j - |\beta|}. 
\end{equation}
This is no different to the previous compactification near spacetime infinity where $(\tau, \xi)$ are finite, but it differs considerably when the frequency variables tend to infinity, as we shall see shortly. 

We then define our parabolic calculus by quantizing, in the usual way (left quantization), this class of parabolic symbols. This gives us the parabolic calculus, $\Psipar^{s, l}(\R^{n+1})$. We can also define spaces of operators and symbols with variable spacetime order, although, for reasons of limited time, we do not elaborate here. This works just as well as in the usual calculus, and in fact is scarcely different as we have the usual dilation structure in the spacetime variables. 

From the parabolic calculus, we get weighted parabolic Sobolev spaces $\Hpar^{s, l}$, defined as the distributions that are mapped to $L^2$ by all operators in $\Psipar^{s,l}$. 

\subsection{Schr\"odinger operator}
Next, as microlocal analysts, we ask the following universal questions about the operator $P$:
\begin{itemize}
    \item What is the characteristic set $\Char(P)$, and how does it intersect the boundary of (parabolically) compactified phase space $\comparphase$?
    \item What is the Hamilton vector field at the boundary of $\comparphase$?
    \item Where are the radial sets (where the Hamilton vector field at infinity vanishes), and what is the linearization of the Hamilton vector field at the radial sets?
\end{itemize}
The answers are as follows, at least for the free operator $P_0$ for simplicity (but the geometry for general $P$ is only inessentially different):

 $\bullet$ The characteristic set $\Char(P)$ is given by $\{ \tau + |\xi|^2 = 0 \}$, that is, a paraboloid. This is a smooth manifold with boundary on the parabolic compactification, and it meets the boundary transversally. This is quite different to the geometry under the radial compactification, on which the continuous extension to the boundary is not a manifold. This already indicates that the parabolic compactification is the `correct' compactification. 
 
$\bullet$ The Hamilton vector field, for the free operator $P_0$, is $H_p = \partial_t + 2 \xi \cdot \partial_x$. This is straight line motion in $\R^{n+1}_{t,x}$ with fixed velocity $(1, 2\xi)$. We examine the behaviour of the Hamilton vector field at the boundary of the compactification. To do that, we need to rescale by multiplying by $\ang{t, x} \aang{R}^{-1}$. But since we only are interested in the Hamilton vector field on $\Char(P_0)$, it is more convenient to replace $\aang{R}^{-1}$ with $\ang{\xi}^{-1}$; and at spacetime infinity, we can multiply by either $\ang{t}$ or $\ang{x}$, depending on whether $t$ or $|x|$ is dominant. 

We first consider the situation at frequency infinity, in the interior of spacetime. In this case, we can forget about the factor $\ang{t, x}$ and only apply the factor $|\xi|^{-1}$. Then the limit of the Hamilton vector field $H_p$ is 
$$
\scH^{2, 0}_p = 2 \hat \xi \cdot \partial_x. 
$$
(This is a slight abuse of notation: we shall rescale the Hamilton vector field in different ways in different regions, so what we have is not precisely $\scH^{2, 0}_p$ but a smooth, bounded multiple of it. This is inconsequential as we are only interested in the bicharacteristics which are invariant under such operations.) 
Notice that the $\partial_t$ coefficient is equal to zero at frequency infinity. This reflects the fact that at frequency infinity, we have `infinite propagation speed', i.e. each bicharacteristic is contained in a single moment in time. Indeed, the form of $H_p$ shows that the speed of propagation at frequency $\xi$ is $2|\xi|$ which tends to infinity as $|\xi|$ tends to infinity. 

Second, at spacetime infinity, and at finite frequency, suppose that $t$ is dominant. Then we can use coordinates $\rho_b = 1/|t|$ and $y_j = x_j/t$,
as well as $(\tau, \xi)$ where these variables are finite. In this coordinates, the rescaled Hamilton vector field is 
$$
\scH^{2, 0}_p = |t| H_p = \signum(t) \Big( - \rho_b \partial_{\rho_b} - \sum_j y_j \partial_{y_j} + 2\sum_j \xi_j \partial_{y_j} \Big).  
$$
We see that this vanishes where $2\xi = y$. Moreover, to lie in $\Char(P)$ we must have $\tau = |\xi|^2$.  That is, there is exactly one point in each cotangent fibre on the spacetime boundary where this rescaled Hamilton vector field vanishes, at least when $t$ is dominant. When $x$ is dominant, we can assume without loss of generality that $x_1$ is a dominant variable, and define $s = t/x_1$, $\rho_b = 1/|x_1|$, $v_j = x_j/x_1$, $j \geq 2$. We note that when $|x|/|t| \to \infty$, on $\Char(P_0)$ we have $|\xi| \to \infty$, and $\xi$ is parallel to $x$. So we simultaneously need to work near frequency infinity, and we can assume that $\xi_1$ is a dominant variable. We define $\rho_f = 1/|\xi_1|$ and let $\sigma = \tau/\xi_1^2$, $\omega_j = \xi_j/\xi_1$ for $j \geq 2$. In these coordinates, we have (good exercise to check)
\begin{multline}
    \scH^{2, 0}_p = \frac{|x_1|}{|\xi_1|} H_p \\ = (\signum x_1)(\signum \xi_1) \Big( ((\signum \xi_1) \rho_f - 2s) \partial_s - 2 \rho_b \partial_{\rho_b} + 2\sum_j (\omega_j - v_j) \partial_{v_j} \Big). 
\end{multline}
We see that the rescaled Hamilton vector field vanishes where $(\signum \xi_1)\rho_f = 2s$ and $\omega_j = v_j$. This defines a smooth manifold with boundary, contained in the spacetime boundary $\rho_b = 0$ and meeting frequency infinity, where $\rho_f = 0$, transversally. 

$\bullet$ If we change variable to $\tilde s = 2s - \signum(\xi_1) \rho_f$ and $\tilde v_j = v_j - \omega_j$, this can be written in the form
\begin{equation}
    \scH^{2, 0}_p  = (\signum x_1)(\signum \xi_1) \Big( -2 \tilde s \partial_{\tilde s} - 2 \rho_b \partial_{\rho_b} -  2\sum_j (\tilde v_j) \partial_{\tilde v_j} \Big). 
\end{equation}
This vector field vanishes at $\tilde s = \tilde v = \rho_b = 0$, and it is a sink for $\signum x_1 = \signum \xi_1$, and a source for $\signum x_1 = - \signum \xi_1$. Notice that $2s = \signum(\xi_1) \rho_f = 1/\xi_1$ has a fixed sign. Thus there are two components to the radial set, which we label $\Rin$ and $\Rout$ where $\pm$ is the sign of $(\signum x_1)(\signum \xi_1)$ on each component. The component $\Rin$, lies over the `southern hemisphere' where $t/|x|$ is nonpositive, and $\Rout$ lies over the northern hemisphere where $t/|x|$ is nonnegative. Each is a graph over the corresponding open hemisphere, but the radial set reaches frequency infinity over the `equator' where $|t|/|x| = 0$, that is, the graph `turns vertical' at the equator which is the boundary of both hemispheres.  

It is straightforward to show that, for the free operator $P_0$, the radial sets are a global source and sink for the flow. That is, every bicharacteristic tends in the backward direction to $\Rin$ and in the forward direction to $\Rout$. Whenever the frequency is finite, time tends to $\pm \infty$ along the bicharacteristic but for infinite frequency, the bicharacteristic is in a level set of time. For the perturbed operator $P$, the same is true provided that each metric $g(t)$ is \emph{nontrapping}. 

This means that the geometry is very similar to that of the Klein-Gordon operator! Consequently, the same analysis applies. As for Helmholtz or Klein-Gordon, the threshold value of the spacetime order is $-1/2$, due to the fact that the operator (Helmholtz/Klein-Gordon/Schr\"odinger) has spatial/spacetime order $0$ in each case.  We choose a variable spatial order $\ml_+$ such that $\ml_+$ is monotone nonincreasing, above threshold at $\Rin$ and below threshold at $\Rout$, with the reverse true for $\ml_- := -1 - \ml_+$. Then, we define Hilbert spaces 
\begin{align}
    \mathcal{X}^{s,\ml_\pm} &= \{ u \in \Hpar^{s, \ml_\pm} : Pu \in \Hpar^{s-1,\ml_\pm+1}  \} \\ 
    \mathcal{Y}^{s, \ml_\pm} &= \Hpar^{s, \ml_\pm}.
\end{align}
Completely analogous reasoning leads to a semi-Fredholm estimate of the form 
\begin{equation}
    \| u \|_{\SX^{s, \ml_\pm}} \leq C \Big( \| u \|_{\Hpar^{s-1, \ml_\pm + 1}} + \| u \|_{\Hpar^{-N, -N}} \Big). 
\end{equation}
Indeed, the key geometric property that makes this work is the uniform source/sink structure of the rescaled Hamilton vector field near the radial sets --- allowing us to prove above- and below-threshold radial point estimates analogous to Theorems~\ref{thm:below} and \ref{thm:above Schw} for the Schr\"odinger operator --- plus the condition that these are the global source and sink of the flow. 

As before, the orthogonal complement (with respect to the $L^2$-pairing) of the range of $P : \SX^{s, \ml_\pm} \to \Hpar^{s-1, \ml_\pm + 1}$ can be identified with the kernel of  $P$ acting on $\SX^{1-s, \ml_\mp}$. Indeed, if $v \in (\Hpar^{s-1, \ml_\pm + 1})^*$ is in the orthogonal complement of the range of $P$ acting on $\SX^{s, \ml_\pm}$, then in particular it is orthogonal to $P\phi$ for all $\phi \in \Schw(\R^{n+1})$. Since $P$ is formally self-adjoint this means that $\ang{Pv, \phi} = 0$ for all $\phi \in \Schw(\R^{n+1})$, showing that $Pv = 0$. It follows that $v \in \SX^{1-s, \ml_\mp}$ (using the identity $\ml_\mp := -1 - \ml_\pm$). The semi-Fredholm estimate for $P$ on $\SX^{1-s, \ml_\mp}$ proves that $P$ on $\SX^{s, \ml_\pm}$ has finite-codimensional range, thus showing that $P$ is Fredholm. 

Moreover, for the Schr\"odinger operator $P$, invertibility is rather easy to prove, since any solution $v$ to $Pv = 0$ satisfies mass conservation: that is, the spatial $L^2$ norm is constant in time. This is incompatible with the above threshold condition on $v$ unless $v=0$. Thus the null space is trivial, and this shows both the null space and (using the adjoint as above) the cokernel are trivial, i.e. $P$ is invertible. 

From this we can get obtain results on global solutions to $Pu = 0$ that are completely analogous to the case for Helmholtz scattering. The solutions are found in the sum of spaces, $\SX^{s, \ml_+} + \SX^{s, \ml_-}$, since they must be below threshold at both radial sets. In particular, given any function $f_- \in \Schw(\R^n_\xi)$, there is a solution $u$ to $Pu = 0$ with the asymptotic
$$
u = (4\pi it)^{-n/2} e^{i|x|^2/4t} f_-\big( \frac{x}{t} \big) + O(t^{-n/2-1}), \quad t \to -\infty,
$$
and this solution has a similar asymptotic as $t \to +\infty$:
$$
u = (4\pi it)^{-n/2} e^{i|x|^2/4t} f_+\big( \frac{x}{t} \big) + O(t^{-n/2-1}), \quad t \to +\infty,
$$
where $f_+$ is also Schwartz. This defines a map, the `Poisson operator' in this context, from $f_-$ to $u$. This map extends to tempered distributions. The scattering map $S$ may then be defined as the map $S : f_- \to f_+$. A microlocal analysis of $S$ shows that $S$ is a Fourier integral operator and preserves a scale of weighted Sobolev-type spaces. Since the incoming and outgoing data $f_\pm$ are functions of $x/t \in \R^n$, one might guess that the standard weighted Sobolev spaces, such as we have been using in our study of the scattering calculus, would be an appropriate scale of spaces. It turns out, however, that the scale of weighted `one-cusp' Sobolev spaces is the best scale to use.\footnote{The `one-cusp' calculus is a calculus associated to a Lie Algebra of vector fields, somewhat similar to the scattering calculus, but instead of being generated, in the notation of Lecture 7, by the vector fields $\mathsf{x}^2 \partial_{\mathsf{x}}, \mathsf{x} \partial_{y_1}, \dots, \mathsf{x} \partial_{y_{n-1}}$, one uses generators $\mathsf{x}^3 \partial_{\mathsf{x}}, \mathsf{x} \partial_{y_1}, \dots, \mathsf{x} \partial_{y_{n-1}}$. See \cite{zachos2022inverting}.}  Indeed, the first two authors have recently shown that $S$ is a one-cusp Fourier integral operator, associated to a symplectic map determined by the bicharacteristic flow from the incoming to the outgoing radial set \cite{HJscatteringmap}, and preserves the scale of weighted one-cusp Sobolev spaces. This is the type of result one expects to be able to prove, if the underlying geometry has been properly elucidated: see \cite{melrose1996scattering}, \cite{alexandrova2005} for antecedents. Further discussion of this is unfortunately well beyond the scope of these lecture notes!

\subsection{Exercises} 
\begin{problem}
    Show that the dual space of $\SY^{s, l_-, l_+}$ is 
    $$
    \{ f = f_1 + Q' f_2 \mid f_1 \in H^{-s, -l_-}, f_2 \in H^{-s, -l_+} \}
    $$
    with norm
    $$
    \| f \| := \inf \| f_1 \|_{H^{-s, -l_-}} + \| f_2 \|_{H^{-s, -l_+}}
    $$
    where the infimum is over all decompositions of $f = f_1 + Q'f_2$ with $f_i$ in the above spaces. Hint: for the more difficult direction, which is showing the dual space is contained in the space above, use the fact that Hilbert spaces such as $\SY^{s, l_-, l_+}$ are their own dual spaces with respect to the inner product on that Hilbert space, then write the duality in terms of $L^2$-pairing. 
\end{problem}

\begin{problem}
    Show that \eqref{eq:parsymbols} is equivalent to the following condition: $a$ is in $\aang{R}^s L^\infty$, and remains in this space under the repeated application of smooth vector fields on $\overline{\R^{n+1}_{\parb}}$ tangent to the boundary. That is, the `new' symbol estimates are just our familiar  conormal estimates, as in \eqref{eq:conormal reg}, on the parabolic compactification of phase space. 
\end{problem}

\begin{problem}
Show that the radial set $\Rin$ for the Schr\"odinger operator $P_0$ is a global source for the rescaled Hamilton flow, and the radial set $\Rout$ is a global sink. Hint: use the same method as was used in Lecture 9 for the Helmholtz operator. 
\end{problem}

\begin{problem} The symbol of $P_0$ is homogeneous with respect to parabolic scaling. But the Hamilton vector field is not, as shown by the fact that the $\partial_t$ component of the Hamilton vector field, when rescaled, vanishes at frequency infinity, although the rest of the Hamilton vector field has a nonzero limit. How could that be? 
\end{problem}

\begin{problem} For the free Schr\"odinger operator $P_0$, use the Fourier transform to express `nice' solutions of $P_0 u = 0$. Perform the inverse Fourier transform and use the stationary phase lemma to get an asymptotic expansion of these solutions as $t \to \pm \infty$. Do the same for the Klein-Gordon operator. 
\end{problem}

\bibliographystyle{plain}
\bibliography{lecture_notes}

\end{document}